\documentclass[10pt]{amsart}

\usepackage{amsfonts,amsmath,latexsym,amssymb,verbatim,amsbsy,amsthm}
\usepackage{dsfont,bm}
\usepackage{multirow}

\usepackage[top=1in, bottom=1in, left=1in, right=1in]{geometry}

\usepackage[dvipsnames]{xcolor}

\usepackage[colorlinks=true, pdfstartview=FitV, linkcolor=RoyalBlue,citecolor=ForestGreen, urlcolor=blue]{hyperref}

\theoremstyle{plain}
\newtheorem{THEOREM}{Theorem}[section]

\newtheorem{theorem}[THEOREM]{Theorem}
\newtheorem{corollary}[THEOREM]{Corollary}
\newtheorem{lemma}[THEOREM]{Lemma}
\newtheorem{proposition}[THEOREM]{Proposition}

\theoremstyle{definition}

\newtheorem{definition}[THEOREM]{Definition}

\theoremstyle{remark}

\newtheorem{remark}[THEOREM]{Remark}
\newtheorem{claim}[THEOREM]{Claim}

\newtheorem{example}[THEOREM]{Example}

\newcommand{\thm}[1]{Theorem~\ref{#1}}
\newcommand{\lem}[1]{Lemma~\ref{#1}}
\newcommand{\defin}[1]{Definition~\ref{#1}}
\newcommand{\rem}[1]{Remark~\ref{#1}}
\newcommand{\cor}[1]{Corollary~\ref{#1}}
\newcommand{\prop}[1]{Proposition~\ref{#1}}

\newcommand{\exam}[1]{Example~\ref{#1}}

\newcommand{\sect}[1]{Section~\ref{#1}}

\newcommand{\appx}[1]{Appendix~\ref{#1}}

\def \a {\alpha}
\def \b {\beta}
\def \g {\gamma}
\def \d {\delta}
\def \e {\varepsilon}
\def \f {\varphi}
\def \k {\kappa}
\def \l {\lambda}
\def \n {\nabla}
\def \s {\sigma}
\def \t {\tau}

\def \w {\omega}

\def \D {\Delta}
\def \G {\Gamma}

\def \L {\Lambda}
\def \Th {\Theta}
\def \O {\Omega}

\def \be {{\bf e}}

\def \bk {{\bf k}}

\def \bp {{\bf p}}

\def \bv {{\bf v}}

\def \bx {{\bf x}}

\def \bzero {{\bf 0}}

\def \bF {{\bf F}}

\def \oH { \overline{H} }
\def \oW { \overline{W} }
\def \oO { \overline{\O} }
\def \oS { \overline{S} }

\def \oC { \overline{C} }

\def \cA {\mathcal{A}}
\def \cB {\mathcal{B}}

\def \cD {\mathcal{D}}
\def \cE {\mathcal{E}}
\def \cF {\mathcal{F}}
\def \cG {\mathcal{G}}
\def \cH {\mathcal{H}}
\def \cI {\mathcal{I}}

\def \cL {\mathcal{L}}
\def \cM {\mathcal{M}}
\def \cN {\mathcal{N}}
\def \cO {\mathcal{O}}
\def \cP {\mathcal{P}}

\def \cR {\mathcal{R}}

\def \barx {\bar{x}}

\def \barv {\bar{v}}
\def \baru {\bar{u}}

\def \barho {{\bar{\rho}}}

\def \orho {\bar{\rho}}

\newcommand{\N}{\ensuremath{\mathbb{N}}}   
\newcommand{\Z}{\ensuremath{\mathbb{Z}}}   
\newcommand{\R}{\ensuremath{\mathbb{R}}}   
\newcommand{\T}{\ensuremath{\mathbb{T}}}   
\newcommand{\E}{\ensuremath{\mathbb{E}}}   
\newcommand{\I}{\ensuremath{\mathbb{I}}}   
\renewcommand{\P}{\ensuremath{\mathbb{P}}} 

\def \loc {\mathrm{loc}}

\def \spec {\mathrm{spec}}

\def \sign {\mathrm{sgn}}
\def \one {{\mathds{1}}}

\def \Lip {\mathrm{Lip}}

\def \ext {\mathrm{ext}}

\newcommand{\jap}[1]{\left\langle #1 \right\rangle}
\newcommand{\ave}[1]{ \left[ #1 \right]}

\newcommand{\rest}[2]{#1\raisebox{-0.3ex}{\mbox{$\mid_{#2}$}}}

\DeclareMathOperator{\Ave}{Ave} %

\DeclareMathOperator{\supp}{supp} %

\DeclareMathOperator{\conv}{conv} %
\DeclareMathOperator{\diver}{div} %
\DeclareMathOperator{\diam}{diam} %
\DeclareMathOperator{\diag}{diag} %
\DeclareMathOperator{\Tr}{Tr} %
\DeclareMathOperator{\Id}{Id} %
\DeclareMathOperator{\re}{Re} %

\def \lan {\langle}
\def \ran {\rangle}

\def \p {\partial}
\def \ra {\rightarrow}
\def \ss {\subset}

\def \GL {Gr\"onwall's Lemma}
\def \HI {H\"older inequality}

\def \CK{Csisz\'ar-Kullback inequality}
\def \CS{Cauchy-Schwartz inequality}

\renewcommand{\geq}{\geqslant}

\renewcommand{\leq}{\leqslant}

\def \da  {\, \mbox{d}\alpha}

\def \df  {\, \mbox{d}f}
\def \dg  {\, \mbox{d}\gamma}
\def \dx  {\, \mbox{d}x}
\def \dbx  {\, \mbox{d}\bar{x}}
\def \dxi  {\, \mbox{d}\xi}
\def \dt  {\, \mbox{d}t}

\def \dy  {\, \mbox{d}y}
\def \dz  {\, \mbox{d}z}

\def \dr  {\, \mbox{d}r}
\def \domega  {\, \mbox{d}\omega}
\def \ds  {\, \mbox{d}s}
\def \dw  {\, \mbox{d}w}

\def \dB  {\, \mbox{d}B}

\def \dmu  {\, \mbox{d}\mu}
\def \dnu  {\, \mbox{d}\nu}
\def \dk  {\, \mbox{d}\kappa}
\def \domega  {\, \mbox{d}\omega}

\def \drho  {\, \mbox{d}\rho}

\def \dbv  {\, \mbox{d}\bar{v}}
\def \dv  {\, \mbox{d} v}

\def \ddt  {\frac{\mbox{d\,\,}}{\mbox{d}t}}

\def \dd  {\mbox{d}}

\def \domain {{\O \times \R^n}}

\def \uFavre {u_{\mathrm{F}}}
\def \uF {u_{\mathrm{F}}}
\def \vF {v_{\mathrm{F}}}

 \def \st {\mathrm{s}}
  
 \def \rmb {\mathrm{b}}
  \def \rmw {\mathrm{w}}

 \def \en {\mathrm{e}}
 
\def \rmB {\mathrm{B}}
\def \rmA {\mathrm{A}}

\def \LFP {\cL_{\mathrm{FP}}}

\def \cMglob {\cM_{\mathrm{glob}}}

\def \cMcs {\cM_{\mathrm{CS}}}
\def \cMcstopo {\cM_{\mathrm{CS}}^{\mathrm{topo}}}
\def \cMmt {\cM_{\mathrm{MT}}}
\def \cMmttopo {\cM_{\mathrm{MT}}^{\mathrm{topo}} }
\def \cMbtopo {\cM_{\mathrm{\beta}}^{\mathrm{topo}} }
\def \cMfmt {\cM_{\mathrm{\phi}}}

\def \cMff {\cM_{\mathrm{\phi \phi}}}
\def \cMfp {\cM_{\mathrm{\phi, p}}}
\def \cMcond {\cM_{\mathrm{cond}}}

\def \cMseg {\cM_{\mathrm{seg}}}
\def \cMF {\cM_{\cF}}
\def \cMb {\cM_{\b}}

\begin{document}

\title{Environmental averaging}

\author{Roman Shvydkoy}

\address{851 S Morgan St, M/C 249, Department of Mathematics, Statistics and Computer Science, University of Illinois at Chicago, Chicago, IL 60607}

\email{shvydkoy@uic.edu}

\subjclass{35Q84, 35Q35, 92D25, 92D50}

\date{\today}

\keywords{collective behavior, emergence, alignment, Cucker-Smale system, Motsch-Tadmor system, Fokker-Planck equation, hypocoercivity, mean-field limit, hydrodynamic limit}

\thanks{\textbf{Acknowledgment.}  
	This work was  supported in part by NSF
	grant  DMS-2107956. The author thanks E. Tadmor and C. Imbert for fruitful discussions, and acknowledges hospitality of Isaac Newton Institute for Mathematical Sciences during the preparation of this paper.}

\begin{abstract}
 Many classical examples of models of self-organized dynamics, including the Cucker-Smale, Motsch-Tadmor, multi-species, and several others, include an alignment force that is based upon density-weighted averaging protocol. Those protocols can be viewed as special cases of `environmental averaging'. In this paper we formalize this concept and introduce a unified framework for systematic analysis of alignment models. 

A series of studies are presented including the mean-field limit in deterministic and stochastic settings, hydrodynamic limits in the monokinetic and Maxwellian regimes, hypocoercivity and global relaxation for dissipative kinetic models, several general alignment results based on chain connectivity and spectral gap analysis. These studies cover many of the known results and reveal new ones, which include  asymptotic alignment criteria based on connectivity conditions,  new estimates on the spectral gap of the alignment force that do not rely on the  upper bound of the macroscopic density, uniform gain of positivity for solutions of  the Fokker-Planck-Alignment model based on smooth environmental averaging.  As a consequence, we establish unconditional relaxation result for global solutions to the Fokker-Planck-Alignment model, which presents a substantial improvement over previously known perturbative results. 
 
\end{abstract}

\maketitle 

\tableofcontents

\section{Introduction} Many mathematical models of swarming behavior reflect the tendency of every agent to align its velocity to an averaged direction of motion of the crowd around. Although the rules that describe the average may not be given explicitly, most adhere to a few basic principles. First, agents react more to the closest neighbors, and second, the density of the swarm plays constructive role in defining a particular communication protocol. Such rules, in a broad sense, give rise to what is called {\em environmental averaging}.  

Early computer simulations that incorporated an alignment mechanism along with other interaction forces produced first realistic visualizations of flocks and schools, see \cite{Aok1982,Rey1987}. A wide variety of applications ranging from swarming behavior of animals to technological implementations, see these sources \cite{ABFHKPPS,Axel97,Ben2005,Jackson2008,Edel2001,VZ2012,MT2014,MP2018,Sbook,Tadmor-notices} and references therein, has ignited  mathematical inquiries into theoretical foundation of alignment dynamics. 

A prototypical example of a static averaging model arises in opinion dynamics, where each agent labeled by index $i\in [1,N]$ has a set of other agents $\cN_i$ to which it is connected. The opinion vector $\bp_i$ aligns to the opinions of connected agents via
\begin{equation}\label{e:opinion}
\dot{\bp}_i= \l \sum_{j\in \cN_i} a_{ij}(t)(\bp_j-\bp_i) + \bF_i, \qquad \sum_j a_{ij}(t)=1.
\end{equation}
Here, $\bF_i$ incorporate all other forces such as adherence to convictions or random noise. If the graph of players is connected then the system naturally reaches the total consensus $\bp_i \to \bar{\bp}$. Forces may lead to non-trivial limiting  states, such as Nash equilibria, see \cite{MT2014,DeGroot,OS2006,LRS-friction}.

In swarming dynamics the pioneering work of Vicsek el al \cite{VCBCS1995} introduced a discrete  model of self-propelled particles with local interactions 
\begin{equation}\label{e:Viscek}
\left\{
\begin{split}
\bv_i(k+1) & = v_0  \frac{\sum_{j: |x_j - x_i|<r_0} \bv_j}{\left|\sum_{j: |x_j - x_i|<r_0} \bv_j 
	\right|}+ \s \xi_n,\\
\bx_i(k+1) &= \bx_i(k) + \bv_i(k+1).
\end{split}\right.
\end{equation}
where $\xi_n$ are random variables and $\s>0$ is the noise intensity. The Vicsek averaging is spatially local and includes normalization to reflect the tendency of agents to adhere to a fixed characteristic speed. The model produces a number of emergent phenomena developing into global patterns such as  mills or periodically rotating chains. Solutions undergo phase transitions from ordered to disordered states depending on the noise level, see \cite{VZ2012} for discussion. 

A growing number of studies on flocking behavior is based on the Cucker-Smale system introduced in \cite{CS2007a,CS2007b},
\begin{equation}\label{e:CSintro}
\left\{
\begin{split}
\dot{x}_i  &= v_i,  \\
\dot{v}_i & = \sum_{j=1}^N m_j \phi(x_i-x_j)(v_j - v_i ),
\end{split}\right. \qquad  (x_i,v_i) \in \R^n\times \R^n, \ i = 1,\ldots,N.
\end{equation}
Here, $\phi$ is a smooth radially symmetric and decreasing kernel, originally $\phi(r) = \frac{\l}{(1+ r^{2})^{\b/2}}$, where $\l,\b>0$. The model provides a well-defined mathematical framework which admits justifiable kinetic and macroscopic descriptions, see \cite{HT2008,HL2009,CFRT2010,BCC2011,FK2019,Sbook,TT2014}. It appeared, in part, in response to the need for a model whose long time behavior is not associated with perpetual connectivity assumptions on the flock as in prior studies. In fact, a simple  criterion for alignment can be stated solely based on rate of decay of the kernel.

\begin{theorem}[\cite{CS2007a,CS2007b}]\label{t:CSinto}
If $\b \leq 1$, all solutions to \eqref{e:CSintro} align exponentially fast to the mean velocity $\bar{v} = \frac{1}{ \sum_{j=1}^N m_j} \sum_{j=1}^N m_j v_j$, while flock remains bounded  
\[
\max_{i=1,\ldots,N} |v_i - \bar{v}| \leq C e^{-\d t}, \qquad \max_{i,j=1,\ldots,N} |x_i - x_j| \leq \bar{D},
\]
where $C,\d,\bar{D}$ depend only on the initial condition and parameters of the kernel. If $\b >1$ there are solutions that do not align.
\end{theorem}
Since its inception the Cucker-Smale system has seen numerous applications. A remarkable implementation to satellite navigation was proposed in \cite{Darwin}, where value of $\b = 0.4$ was found to be most optimal for the purposes of the mission. Adaptations to control problems are addressed in \cite{Bo2015,CEPT2015,CKPP2019}. Interacting agents immersed in an incompressible fluid lead to hybrid systems  with Cucker-Smale component modeling the alignment force, \cite{HKK2014}.  Multi-scale and multi-species flocks have been studied in \cite{HeTmulti,STmulti}. An important modification of the system with thermodynamic features was proposed in \cite{HRCST2017}, see also \cite{ABFHKPPS}. Flocking analysis can be extended to nonlinear alignment protocols as well \cite{Tadmor-pressure,P-ES2015,HHK2010,Mar2018}. A comprehensive review of various other features of the Cucker-Smale dynamics based on hierarchy, angle of vision, and emergence of leaders can be found in \cite{CFTV2010}. In the context of alignment dynamics which includes potential attraction/repulsion or Rayleigh frictions forces, the emergent behavior has not yet been fully understood, although it is clear from these studies \cite{DOrsogna,ShuT2019,ShuT2019anti,STmulti,LRS-friction}, that the effect of such forces on collective outcomes could be dramatic. In particular, the quadratic confinement potential drives the system to an aggregated harmonic oscillator, \cite{ShuT2019}. Some general $N$-dependent results in this direction can be achieved for the 3Zone model of Reynolds \cite{Rey1987} with the use of the corrector method introduced in \cite{DS2019}, see \cite{Sbook}. Lastly, we mention that the alignment criterion itself stated in \thm{t:CSinto} does not require the kernel to have any explicit form and has seen numerous extensions to include general fat-tail kernels and kernels with degenerate communication in short range, see  \cite{DS2019,HL2009} and \sect{s:CST} below.

It is insightful to rewrite the Cucker-Smale system as follows
\begin{equation}\label{e:ABS}
\begin{split}
\dot{x}_i & = v_i,  \hspace{0.7in} x_i   \in \O, \\
 \dot{v}_i & = \st_i ( \ave{v}_i - v_i ), \quad v_i  \in \R^n
  \end{split} \qquad  i = 1, \ldots, N.
\end{equation}
where $\O$ is an environment  (for most of our discussion either $\T^n$ or $\R^n$), $\ave{v}_i$ is an averaging protocol of the $i$th agent, $v = (v_1,\ldots,v_N)$, and $\st_i$ is a specific communication strength.
Here,
\begin{equation}\label{e:CSMintro}
\st_i = \sum_{j=1}^N m_j \phi(x_i-x_j), \qquad  \ave{v}_i = \frac{\sum_{j=1}^N m_j \phi(x_i-x_j) v_j}{\sum_{j=1}^N m_j \phi(x_i-x_j)}.
\end{equation}
This form highlights  two separate structural components of an alignment model -- the averaging and communication strength. Varying  these  components allows to adapt the system to a particular modeling scenario. For example, it is argued in \cite{MT2011,MT2014} that if a flock consists of clusters with unbalanced sizes it is more realistic to incorporate a static strength parameter $\st_i = \l >0$, leading to what is called the Motsch-Tadmor model
 \begin{equation}\label{ }
 \dot{v}_i  = \l (\ave{v}_i - v_i).
 \end{equation}
Analysis of this model presents many challenges related to the lack of symmetry and momentum conservation. However, the analogue of \thm{t:CSinto} still holds,  \cite{MT2014}. A modification that restores the symmetry was proposed in  \cite{Sbook},
\begin{equation}\label{e:Mfintro}
\st_i = 1, \qquad \ave{v}_i = \int_{\R^n} \phi(x_i - \xi) \frac{\sum_{j=1}^N m_j \phi(\xi-x_j) v_j}{\sum_{j=1}^N m_j \phi(\xi-x_j)} \dxi.
\end{equation}
This particular averaging appears instrumental in several other studies of flocking such as hydrodynamic limits \cite{Sbook}, relaxation and hypocoercivity in kinetic dynamics \cite{S-hypo}, see also Sections~\ref{s:hypo} and \ref{s:hydrolim}. Its continuous variant emerged in the analysis of non-homogeneous turbulence in \cite{LS2016}.

Another interesting example of a non-Galilean invariant environmental averaging is given by a class of segregation models. Let $\{g_l\}_{l=1}^L$ be a smooth partition of unity $\sum_{l=1}^L g_l = 1$ subordinated to an open cover $\cup_{l=1}^L \O_l = \O$, where $\O$ is a compact environment.  Let
\begin{equation}\label{e:segintro}
\st_i =1, \qquad \ave{v}_i = \sum_{l=1}^Lg_l(x_i) \frac{\sum_{j=1}^N m_j  v_j g_l(x_j) }{\sum_{j=1}^N m_j g_l(x_j)}.
\end{equation}
Here, the agents communicate predominantly in their own communities and exchange of information is facilitated through the borders. Consensus can be reached provided the border is sufficiently populated at all times, see \sect{s:flocking} for rigorous formulation.  Many more examples are discussed in \sect{s:basic}.

In the large crowd limit as $N\to \infty$ the components $\st_i, [\cdot]_i$ take macroscopic forms, which makes them suitable for statistical description of the alignment systems. For example,  denoting $f_\phi = f\ast \phi$ for a distribution $f$, we can see that the Cucker-Smale model corresponds to
\[
 \st_\rho = \rho_\phi, \qquad \ave{u}_\rho = \frac{(u \rho)_\phi}{\rho_\phi}.
 \]
This averaging rule is also known as the Favre filtration, \cite{Favre}, which was introduced in the context of  numerical simulations of turbulent flow. In the same manner, the averaging of \eqref{e:Mfintro} is given by the over-mollification of the Favre filtration 
\begin{equation}\label{ }
\ave{u}_\rho = \left(\frac{(u \rho)_\phi}{\rho_\phi}\right)_\phi,
\end{equation}
and  the averaging of \eqref{e:segintro} becomes
\begin{equation}\label{ }
\ave{u}_\rho(x) = \sum_{l=1}^L g_l(x)\frac{ \int_\O u g_l \rho \dy}{ \int_\O  g_l \rho \dy}.
\end{equation}
All the operations above make mathematical sense for any probability measure 
$\rho \in \cP(\O)$ and any bounded field  $u\in L^\infty(\rho)$.  In particular, we can go back to the discrete analogues by applying averaging on  empirical pairs
\begin{equation}\label{e:emp}
\begin{split}
\rho^N = \sum_{i=1}^N m_i \d_{x_i}, & \quad u^N = \sum_{i=1}^N v_i \one_{x_i},\\
\ave{v}_i : = \ave{u^N}_{\rho^N} (x_i), & \quad \st_i : = \st_{\rho^N}(x_i).
\end{split}
\end{equation}
It is therefore more inclusive to define averaging rules via macroscopic formulas.

Physical features of the system \eqref{e:ABS} are intimately connected to analytical properties of the  pair $(\st_\rho, \ave{\cdot}_\rho)$. In most situations those properties are more naturally expressed in terms of the strength measure given by $\dk_\rho = \st_\rho \drho$. Thus, the preservation of $\k$-momentum 
\[
\int_\O \ave{u}_\rho \dk_\rho = \int_\O u \dk_\rho ,
\]
implies conservation of the physical hydrodynamic momentum, $\ddt \int_\O u \drho = 0$. The symmetry
\[
\int_\O  v\cdot \ave{u}_\rho \dk_\rho =\int_\O  \ave{v}_\rho \cdot u \dk_\rho
\]
implies a natural energy dissipation law 
\begin{equation}\label{e:enlawintro}
\ddt  \int_{\O} |u|^2 \drho(x) = - \frac12 \int_{\O \times \O} \phi_\rho(x,y) |u(x) - u(y)|^2 \drho(x)\drho(y),
\end{equation}
where $\phi_\rho$ is a communication kernel representing a given averaging, see \sect{s:repkernel}. The long time behavior analysis becomes connected to coercivity and positive-definiteness of the averaging, see \sect{s:flocking}.

 In order to get more insight into such connections, it is useful to disassociate the averaging/strength pair $(\k_\rho, \ave{\cdot}_\rho)$  from any particular differential law they are involved in, and take a `birds eye' look on its kinematic properties.  For this purpose,  we will delegate the concept of an {\em environmental averaging model} to a family of pairs
\[
\cM = \{ (\k_\rho, \ave{\cdot}_\rho): \rho \in \cP(\O)\},
\]
parametrized by probability measures $\rho \in \cP(\O)$, and satisfying a list of continuity assumptions stated below in \sect{s:basic}. Through the study of such models it appears possible to build a unifying framework for many flocking and regularity results that have appeared scattered before, and to find substantially new ones that, otherwise, are obscured by specificity of a particular model. This will be the main objective of the present work. So, let us give a brief overview of the studies undertaken here.

\medskip
(I) First, we develop basic functional analysis of the averaging models.  Here we focus primarily on those properties that have direct physical interpretation in terms of dynamics of a particular system they are involved in. Those include representability (existence of a communication kernel), conservation, symmetry, and most importantly a quantitative version of positive definiteness  -- ball positivity, see \sect{s:classes}.  We also describe  regularity conditions on the pairs $\st_\rho, \ave{u}_\rho$ necessary for developing a meaningful well-posedness theory for kinetic models done in Sections \ref{s:mfldet} and \ref{s:WP}.

(II) Next we address the classical flocking result of Cucker and Smale for general environmental averaging models. We choose  the kinetic description in the context of measure-valued solutions:
\begin{equation}\label{e:VAintro}
\p_t f + v \cdot \n_x f =  \n_v \cdot ( \st_\rho (v - \ave{u}_\rho) f ).
\end{equation}
Here  $\rho$ and $u\rho$ are the macroscopic density and momentum, respectively. 
It is the most inclusive framework since it encapsulates the microscopic system \eqref{e:ABS} if applied to empirical measures $f = \sum_{i=1}^N m_i \d_{x_i} \otimes \d_{v_i}$, and the pressureless hydrodynamic system if applied to mono-kinetic solutions $f= \rho(x,t) \d_{0}(v - u(x,t))$, see \eqref{e:EASpresless}.  For global communication kernels the analogue of the original Cucker-Smale alignment criterion is stated in \thm{t:CS}, see also Carrillo et al \cite{CFRT2010} for the first result of this kind in kinetic formulation.

In the case of local communication, which is our primary focus, all alignment criteria can be sorted into two types -- ones that rely on a chain-connectivity of the flock, and ones that make use of the spectral gap condition.  The former approach is dynamic. It is based on the idea that connected misaligned components of the flock lose energy through the law \eqref{e:enlawintro} until full alignment is achieved. For the classical Cucker-Smale  and topological singular models this was addressed in \cite{ST-topo,MPT2018}. Here we present a new result stated in \thm{t:conn} which gives a sufficient condition of ball-thickness, see \eqref{e:thball}: as long as the flock is connected at a local communication scale $r$ of the kernel, and $\orho_r(\supp \rho) \gtrsim \frac{1}{t^{1/4}}$ in the open space or $\orho_r(\O) \gtrsim \frac{1}{t^{1/2}}$ on the torus, the flock aligns. No control on the upper bound of the density is necessary.

The spectral gap approach is kinematic in nature. It relies on finding efficient bounds on the spectral gap of the   averaging operator set in a proper function space. In fact, spectral gaps are relevant to  flocking behavior in several contexts including relaxation problem for the Fokker-Planck-Alignment model. So, it will be our primary focus in \sect{s:lowenergy}.  A criterion proved in \cite{Tadmor-notices} states that a symmetric model aligns provided $\int_0^\infty \l(t) \dt = \infty$, where 
\begin{equation}\label{e:varprobintro}
\l = \inf_{ u \in L^2_0(\rho)} \frac{ (u, \cL_\rho u)_\rho}{ (u,u)_\rho}, \quad \cL_\rho u = \st_\rho(u - \ave{u}_\rho),
\end{equation}
and $(u,v)_\rho = \int_\O u \cdot v \drho$. In  \prop{p:enalign} we present an extension of this result to the non-symmetric case. For the Cucker-Smale model the bound $\l \gtrsim \frac{\rho_-^2}{\rho_+}$ was proved in the same work \cite{Tadmor-notices}, see also Remark~\ref{r:spgap}. This result is consistent with the chain-connectivity criterion stated above provided $\rho_+$ remains bounded.  For systems with a singular kernel a similar result was established in \cite{ST-topo}. With a view towards the relaxation problem, where  reliance on $\rho_+$ is prohibitive, it is imperative to find bounds on the spectral gap  independent of $\rho_+$.

To this end we propose a somewhat different methodology -- one that focuses directly on the averaging $\ave{\cdot}_\rho$ in the framework of $\k_\rho$-weighted spaces:
\begin{equation}\label{e:egapintro}
(u, \ave{u}_\rho)_{\k_\rho} \leq (1-\e) \|u\|^2_{L^2(\k_\rho)}.
\end{equation}
We introduce the {\em low energy method} tailored to finding estimates on $\e$ solely in terms of $\rho_-$. The method applies to a special, but quite broad class of so called {\em ball-positive models}, see \prop{p:gaps}. These include the segregation \eqref{e:segintro}, the overmollified Motsch-Tadmor variant \eqref{e:Mfintro}, and most notably the classical Cucker-Smale model \eqref{e:CSMintro} provided the defining kernel is  Bochner-positive: $\phi = \psi \ast \psi$ for some $\psi \geq 0$. 
In particular, if applied to the Cucker-Smale model the method gives the following bound:
\begin{equation}\label{e:gaprho3}
\e \gtrsim \orho_r^3(\O).
\end{equation}

\smallskip
(III) The next study is dedicated to justifying the kinetic description through a mean-filed limit in both deterministic and stochastic contexts. As the number of agents grows $N\to \infty$, the microscopic system settles in the weak sense to a solution to the kinetic {\em Vlasov-Alignment} equation \eqref{e:VAintro}
\begin{equation*}\label{e}
\mu^N = \sum_{i=1}^N m_i \d_{x_i} \otimes \d_{v_i} \to f.
\end{equation*}
So far the limit has been rigorously justified for the Cucker-Smale and \eqref{e:Mfintro}-models,  \cite{HT2008,HL2009,Sbook}. In Section~\ref{s:mfldet} we establish a much broader result which covers models  with certain uniform regularity properties, see \defin{d:r}, and is insensitive to the symmetry of a model. It applies in particular, to the Motsch-Tadmor and other similar models. 

When system \eqref{e:ABS} is supplemented with density-weighted stochastic forces 
\begin{equation}\label{e:ABSstoch}
\begin{split}
\dx_i & = v_i \dt,   \\
 \dv_i & = \st_i ( \ave{v}_i - v_i ) \dt +  \sqrt{2\s \st_i } \dB_i,
  \end{split} \qquad  i = 1, \ldots, N.
\end{equation}
where $B_i$'s are independent Brownian motions in $\R^n$, the limit `in law'  settles to a solution of the {\em Fokker-Planck-Alignment} equation
\begin{equation}\label{e:FPAintro}
\p_t f + v \cdot \n_x f = \s \st_\rho \D_v f + \n_v \cdot ( \st_\rho (v - \ave{u}_\rho) f ).
\end{equation}
  For the additive noise and general convolution-type models the result was proved in \cite{BCC2011}, while the present  case is treated in \sect{s:mflstoch}.  The non-homogeneous diffusion requires a separate consideration, and is introduced for two reasons. First, it makes physical sense to put stochasticity where communication actually happens and is proportional to its strength. Random deviations get stronger with more active communication, so $\st_i$ acts as a thermalization parameter. Second, it ensures that the kinetic model \eqref{e:FPAintro} has a natural Maxwellian equilibrium. This will be instrumental in the study of relaxation.

\smallskip
(IV) Reading off the evolution of macroscopic quantities from  \eqref{e:VAintro} we obtain the hydrodynamic {\em Euler-alignment} system (EAS)
\begin{equation}\label{e:EAS}
\begin{split}
\p_t \rho + \n\cdot (u \rho) & = 0, \\
\p_t (\rho u) + \n \cdot (\rho u \otimes u) +  \n \cdot \cR & = (\ave{u}_\rho - u)\dk_\rho ,
\end{split}
\end{equation}
where $\cR$ is the Reynolds stress given by 
\[
\cR = \int_{\R^n}  ( v - u) \otimes (v-u)  f \dv.
\]
Here, we encounter the classical closure problem. One can achieve a specific form of $\cR$ by introducing various scaling regimes. This has been addressed in two situations. The monokinetic regime $f \to  \rho(x,t) \d_{0}(v-u(x,t))$ results in the pressureless EAS, $\cR = 0$, and the analysis of this limit for the classical Cucker-Smale model goes back to \cite{MV2008,KV2015,FK2019} see also \cite{Sbook}. The convergence was established quantitatively in Wasserstein-1 metric.  In \sect{s:mono} we produce a general result and upgrade the convergence to Wasserstein-2 under mild continuity assumptions on $\cM$. It applies, in particular, to all the models listed here, including non-symmetric ones such as \ref{MT}.  

By incorporating a strong penalization force of Fokker-Planck type one can achieve another regime where $f$ settles to a Maxwellian. This results in the Euler-alignment system with isothermal pressure tensor $\cR = \rho \Id$. The Cucker-Smale model was analyzed in \cite{KMT2013,KMT2014,KMT2015}, and \eqref{e:Mfintro} was analyzed in \cite{S-hypo}, see also \cite{CK2023} for a new development in the mildly singular case. \sect{s:Max} presents a general result. 

We note that kinetic closure is not the only way to model flocking on the macroscopic level. A general class of systems with entropic pressure introduced in \cite{Tadmor-pressure}, which includes kinetic ones as a particular example,  is amenable to flocking analysis as well. 

\smallskip
(V)    The most comprehensive study in this present work is related to well-posedness and relaxation of the Fokker-Planck-Alignment model \eqref{e:FPAintro} on the periodic environment $\O = \T^n$.  The motivation for this study is rooted in the original question of {\em emergence} -- formation of collective outcome from purely local interactions. On the periodic domain, if the communication kernel $\phi$ has a short reach, $\supp \phi \ss [0,r_0]$, then there exists a family of unaligned solutions where agents rotate along parallel geodesics with various velocities (or even perpendicular geodesics with mutually rational velocities). These are  called {\em locked states}.  Such solutions form a measure-zero set in the ensemble of initial data $(x_1,\ldots,x_N, v_1,\ldots,v_N)$.  No deterministic approach to establishing alignment based on generic data that avoids locked states has been explored yet, except for 1D case \cite{DS2019}.  It is natural, however, to look into this problem in stochastic settings of  \eqref{e:ABSstoch}, where locked states are being disrupted  instantly. One can expect a collective outcome in two limiting steps: first $t\to \infty$, then $\s \to 0$. For large crowd distributions governed by \eqref{e:FPAintro} this can be viewed as a relaxation problem: on the first step we obtain convergence to Maxwellian  
\begin{equation}\label{e:relaxintro}
f \to \mu_{\s,\bar{u}} = \frac{1}{|\O|(2\pi \s)^{n/2}} e^{- \frac{|v - \bar{u}|^2}{2\s}},
\end{equation}
which in turn aggregates on the monokinetic  state  $\d_0(v-\bar{u})\otimes \dx $ as $\s \to 0$. The latter represents 
a perfectly aligned configuration.

This program has seen some success in the past. The relaxation itself for the linear problem is a classical and well-understood subject, see \cite{Villani} and references therein. With the nonlinear alignment force the works \cite{DFT2010,Choi2016} establish relaxation for perturbative solutions near equilibrium in the case of the Cucker-Smale and purely local models, respectively.  The first global result was proved in \cite{S-hypo} in the context of the \eqref{e:Mfintro}-model, where linear technique was adapted to the nonlinear problem enabled by special cancelations in the alignment forcing. 

In \sect{s:hypo} we extend this technique further and prove a much more general result that pertains to a wide variety of models. \prop{p:mainrelax} lists a set of functional requirement on a given solution to imply exponential relaxation. This applies in particular to perturbative solutions, but the main application manifests itself in global relaxation for  ball-positive models. It comes in conjunction with the detailed well-posedness theory for the Fokker-Planck-Alignment equations developed in \sect{s:WP}.  We prove that most models $\cM$ with good regularity properties facilitate the classical kinetic diffusion effect --  spread of positivity of solutions expressed by the instant gain of Gaussian tails 
\begin{equation}\label{e:gaussintro}
f(t,x,v) \geq b e^{-a |v|^2}.
\end{equation}
The spread of positivity is a well-known result observed  in many kinetic equations, see \cite{Carl1933,AZ2021,GIMV2019,HST2020,Mouhot2005,IMS2020} and references therein. The novel additional aspect of our result stated in \prop{p:Gauss} is that the constants $a,b$ depend only on the entropy and $L^\infty$-bound on the drift $\st_\rho \ave{u}_\rho$. For many models, including the Cucker-Smale,  the latter two can be controlled by initial condition only. Consequently, for those models we obtain uniform control over the lower bound on the density, and hence, the spectral gap through  \eqref{e:gaprho3}.  In such cases relaxation result is unconditional. Let us summarize the result specifically for the original Cucker-Smale model.

\begin{theorem}\label{t:CSrelaxintro}
Any classical solution $f$ to \eqref{e:FPAintro} based on the Cucker-Smale model  with Bochner-positive kernel $\phi$ relaxes exponentially fast to the global Maxwellian \eqref{e:relaxintro}.
\end{theorem}

\thm{t:relax} contains the full list of models to which a similar result applies. We note again that previously this result was established only in perturbative regime by Duan et al \cite{DFT2010}.  Concerning other models, in particular non-symmetric models such as Motsch-Tadmor, we obtain relaxation near equilibirum in Fisher information sense. The complete statement is given in  \thm{t:relaxlocal}.

Finally, let us comment on what is not included in our study and what would be highly desirable to address in the near future.  First, we include no forces, focusing mainly on the core alignment mechanism. Potential forces, such as confinement, attraction/repulsion etc, have a great impact on collective  outcomes and play major role in applications, \cite{ShuT2019,ShuT2019anti,CCP2017,CFTV2010,DOrsogna}. Second, we treat only linear couplings in the alignment force. Several recent studies \cite{Tadmor-pressure,P-ES2015,HHK2010,Mar2018} highlight the importance of non-linear couplings as well. In our general framework nonlinearity $\G$ can be incorporated by considering the system
\[
 \dot{v}_i = \st_i  \ave{\G(v - v_i)}_i.
 \]
 Developing regularity and relaxation theory, say, for the kinetic counterpart would be crucial to understanding more intricate nonlinear phenomena of self-organization.   Third, our framework does not presume communication to be singular, either mildly or strongly.  Such models were introduced in \cite{Pe2015,ST1,ST2,ST3,DKRT2018,ST-topo} to analyze the effects of enhanced local communication and its role in emergent dynamics, see the survey \cite{MP2018}.  Finally, we leave the analysis of hydrodynamic models in our general framework to future research as it shifts the focus far from the thread of this work, see \cite{TT2014,CCTT2016,Sbook,LS-uni1,MT2014,HeT2017} and the literature therein. However, we will share a new prospective on modeling macroscopic alignment in \sect{s:EASnew}.

\section{Basic concept and examples}\label{s:basic}

Let $\O$ denote an $n$-dimensional {\em environment}. We mostly focus on the cases when $\O$ is either the open space $\R^n$, periodic domain $\T^n$, a finite set of points, or Cartesian products of the above. Denote by $\cP(\O)$ the set of probability measures on $\O$. An {\em  environmental averaging model} is a family of pairs
\[
\cM = \{ (\k_\rho, \ave{\cdot}_\rho): \rho \in \cP(\O)\}
\]
satisfying the following functional requirements:
\begin{itemize}
\item[(ev1)] For every $\rho \in \cP(\O)$, $ \k_\rho$ is a finite positive measure on $\O$. We call it communication strength.
\item[(ev2)] $\ave{\cdot}_\rho$ is a linear bounded operator on the weighted space $L^2(\O,\dk_\rho): = L^2(\k_\rho)$.
\item[(ev3)] $\ave{\cdot}_\rho$ is a linear bounded operator on $L^\infty(\k_\rho)$, with the properties ($\k_\rho$-a.e.)
\begin{equation}\label{e:order}
\ave{u}_\rho \geq 0 \text{ for all } u \geq 0, \text{ and } \ave{\one_\O}_\rho =\one_\O.
\end{equation}
\end{itemize}
Here and throughout $\one_A$ denotes the characteristic function of a set $A$. If $u=(u_1,\ldots,u_m)$ is a vector field (where $m$ may be unrelated to the dimension $n$) we assume that the operator $\ave{u}_\rho$ is acting on each coordinate:
 \begin{equation}\label{e:coordave}
\ave{u}_\rho = ( \ave{u_1}_\rho, \dots, \ave{u_m}_\rho).
\end{equation}

Although the averaging models are generally assumed to be defined over all densities $\rho \in \cP(\O)$, to fulfill further regularity assumptions on the averaging operation it may be necessary to restrict the probabilities $\rho$ to a narrower admissible class $\cD \ss \cP$. The most encountered examples include ``thick" flocks, see \sect{s:rth}.

Most natural models are {\em material} - a property of adherence to the support of the flock. Namely, we say that the model $\cM$ is material if
\begin{itemize}
\item[(ev4)] there exists  bounded family of non-negative functions $\st_\rho\in L^\infty_+(\O)$ with $\sup_{\rho \in \cP(\O)}\| \st_\rho \|_{L^\infty(\O)} \leq \oS$ such that $\k_\rho = \rho \st_\rho$. We also call  $\st_\rho$  a (specific) strength function.
 \item[(ev5)] $\ave{u}_\rho = 0$, provided $u = 0$ $\rho$-almost everywhere.
\end{itemize}

 On the microscopic level one considers discretely distributed density and velocity fields associated to a set of $N$ agents $\{ x_i \}_{i=1}^N$
\begin{equation}\label{e:rhov}
\rho^N = \sum_{i=1}^N m_i \d_{x_i}, \quad u^N = \sum_{i=1}^N v_i \one_{x_i}.
\end{equation}
Assuming that the model is material we can unambiguously compute the values of the average and strength at the agents' locations
\begin{equation}\label{e:NM}
\ave{v}_i : = \ave{u^N}_{\rho^N} (x_i),\quad \st_i : = \st_{\rho^N}(x_i).
\end{equation}
The agent-based system \eqref{e:ABS} is stated precisely in terms of these discrete components.

\subsection{Examples}\label{ss:examples} Let us list several classical examples, and some new ones, and show how they fit into the definition of environmental averaging.

\begin{example}
The most obvious example is the global averaging
\begin{equation}\label{global}
\st_\rho = 1, \qquad \ave{u}_\rho =  \int_\O u \rho \dx. \tag{$\cMglob$}
\end{equation}
and the system \eqref{e:ABS} in this case expresses alignment with all-to-all communication
\[
\dot{v}_i =  \sum_{j=1}^N m_j (v_j - v_i).
\]

The extreme opposite is the pure identity model
\begin{equation}\label{Id}
\st_\rho = 1, \qquad \ave{u}_\rho = u\, \one_{\supp \rho}.\tag{$\cM_{\mathbb{I}}$}
\end{equation}
The agent-based version obviously leads to a stalled system. However, the utility of this model in the kinetic formulation will present itself in the study of hydrodynamic limits, see  \sect{s:hydrolim}. 
\end{example}
\begin{example}
The classical Cucker-Smale system has been discussed in detail in the introduction. Let us recall that in this case the pair is given by
\begin{equation}\label{CS}
\st_\rho = \rho_\phi, \qquad \ave{u}_\rho = \frac{(u \rho)_\phi}{\rho_\phi}.\tag{$\cMcs$}
\end{equation}
Here and throughout we denote for short $f_\phi = f \ast \phi$, and $\phi$ is assumed to be infinitely smooth.  In this case the averaging $\ave{u}_\rho = \uF$ is also known as the Favre filtration used in large eddy simulations of compressible turbulence, \cite{Favre}. Its remarkable property comes from the fact that if $\rho$ satisfies the continuity equation
\[
\p_t \rho + \n \cdot( u\rho) = 0,
\]
then the filtered density $\rho_\phi$ satisfies the continuity equation relative to the Favre-filtered velocity field
\[
\p_t \rho_\phi + \n \cdot( \uF \rho_\phi) = 0.
\]
An important implication of this equation will be discussed in \sect{s:EASnew}.

Properties (ev1) and (ev3) are obvious here. To verify (ev2) we notice that $\k_\rho = \rho \rho_\phi$. So using that for any $\rho\in \cP(\O)$
\begin{equation}\label{e:Jenex}
|(u\rho)_\phi|^2 \leq (|u|^2 \rho)_\phi \rho_\phi,
\end{equation}
we obtain
\[
\int_\O | \ave{u}_\rho |^2 \dk_\rho = \int_\O |(u\rho)_\phi|^2 \frac{\drho}{\rho_\phi} \leq \int_\O (|u|^2 \rho)_\phi \drho = \int_\O |u|^2 \rho_\phi  \drho = \|u\|^2_{L^2(\k_\rho)}.
\]
We can see that the \ref{CS}-model is contractive. The contractivity generally holds even in $L^p$-spaces for any conservative model, see \lem{l:conscontr}.
\end{example}

\begin{example} If we set $\st_\rho = 1$, the example above turns into another well-known model, so called  Motsch-Tadmor model   \cite{MT2011,MT2014}:
\begin{equation}\label{MT}
\st_\rho = 1, \qquad \ave{u}_\rho = \frac{(u \rho)_\phi}{\rho_\phi}.\tag{$\cMmt$}
\end{equation}
The model was introduced to mediate some issues arising in application of the Cucker-Smale averaging to multi-scale flocks, where a large and distant sub-flock overpowers the dynamics of a smaller sub-flock, see also \cite{STmulti,Sbook} for more discussion.

The only non-trivial property (ev2) holds for the admissible class of thick densities $\cD = \{ \rho \in \cP: \inf \rho_\phi > 0 \}$ under no assumption on the kernel $\phi$. However if the kernel $\phi$ is local and compactly supported, i.e.
\begin{equation}\label{e:kerRr}
c \one_{|x|<r_0} \leq  \phi(x) \leq C \one_{|x| < R_0}, \quad R_0>r_0,
\end{equation}
(the latter holds automatically on compact $\O$), then the $L^2$-boundedness holds for any $\rho\in \cP(\O)$ uniformly over $\cP(\O)$. Indeed, using \eqref{e:Jenex},
\[
\int_\O | \ave{u}_\rho |^2 \drho \leq \int_\O (|u|^2 \rho)_\phi \frac{\drho}{\rho_\phi} =  \int_\O |u|^2 \left( \frac{\rho}{\rho_\phi}\right)_\phi  \drho.
\]
According to \cite[Lemma 5.2]{KMT2013}, and see also the Appendix, under the condition \eqref{e:kerRr} we have
\begin{equation}\label{e:KMTrho}
\left( \frac{\rho}{\rho_\phi}\right)_\phi  \leq C,
\end{equation}
where $C$ depends only on the constants the appear in  \eqref{e:kerRr} and dimension. This implies the desired.
\end{example}
\begin{example}
We can interpolate between  \ref{CS} and \ref{MT} and consider a general power law for the specific strength function:
\begin{equation}\label{Mb}
\st_\rho = \rho_\phi^\b, \qquad \ave{u}_\rho = \frac{(u \rho)_\phi}{\rho_\phi},\quad  \b \geq 0.\tag{$\cMb$}
\end{equation}

All these models satisfy the requirements (ev1) and (ev3) obviously, and (ev2) follows as above provided we have the following generalization of \eqref{e:KMTrho}, which is proved in the Appendix: under   \eqref{e:kerRr} 
\begin{equation}\label{e:KMTrhobeta}
\left( \frac{\rho}{\rho^{1-\b}_\phi}\right)_\phi  \leq C \rho_\phi^\b, \quad \forall\ 0\leq \b\leq 1,
\end{equation}
where $C$ depends only on the constants the appear in  \eqref{e:kerRr}, $\b$, and the dimension.

\end{example}
\begin{example} More suitable for modeling local communication, a symmetric version of the Motsch-Tadmor model can be defined by applying extra convolution to the Favre filtration:
\begin{equation}\label{Mf}
\st_\rho = 1, \qquad  \ave{u}_\rho = \left( \frac{(u \rho)_\phi}{\rho_\phi} \right)_\phi .\tag{$\cMfmt$}
\end{equation}
This gives rise to the discrete averaging given by \eqref{e:Mfintro}. Here we assume as always that  $\phi \in C^\infty$ and it is a mollifier: $\phi\in L^1_+(\O)$ with $\int \phi \dx = 1$.

The model was introduced in \cite{S-hypo,Sbook} and played various roles. It was proved to define a globally hypocoercive kinetic dynamics, and  was also used to extend Figalli and Kang's hydrodynamic limit in the monokinetic regime \cite{FK2019}  to flocks with compact support, see also \sect{s:hydrolim}.

More  versions of \ref{Mf} can be obtained by looking into different strengths  by analogy with the \ref{Mb}-model, or by replacing $\rho$ with a more general baratropic pressure:
\begin{equation}\label{Mfp}
\k_\rho = p(\rho), \qquad \ave{u}_\rho = \left( \frac{(u p)_\phi}{p_\phi} \right)_\phi,\tag{$\cMfp$}
\end{equation}
where $p \geq 0$ is a function of $\rho$. Here, the support of the strength function may not coincide with $\rho$, or $\st_\rho$ may be unbounded, which makes it a non-material model. Also the class of admissible densities $\cD$ may be restricted depending on the pressure law $p(\rho)$. For example, in the ideal gas case $p = \rho^\g$ we naturally assume $\cD = L^\g(\O)$.

 One interesting case  is obtained when $p=1$, resulting in 
\begin{equation}\label{Mff}
\k_\rho = 1\ , \qquad \ave{u}_\rho  = u_{\phi  \ast \phi}. \tag{$\cMff$}
\end{equation}
In this case the average and the strength do not depend on the density at all, and consequently define a non-material model.
\end{example}

\begin{example}[Topological models]
A new way of modeling interactions which implement topological, rather than Euclidean measure of distance, has long been advocated by many empirical studies \cite{Shang2014,NIIZATO201462,Bal2008,CCGPS2012}. The first symmetric topological model  was introduced in \cite{ST-topo}, see also \cite{LRS-topo,RS-topoloc,MMP2020}, although it incorporated singular communication.  Its smooth variant fits within our framework of environmental averaging.  

To define such a model let us consider a basic symmetric domain $\cO_0 = \cO(-\be_1,\be_1)$ connecting two points $-\be_1$ and $\be_1$, and for any pair $(x,y)$, let  $\cO(x,y)$ be the domain connecting $x$ and $y$ obtained by rotation and dilation  of $\cO_0$. Let $\chi_{\cO(x,y)}$ be some mollification of the characteristic function $\one_{\cO(x,y)}$.  We introduce the topological ``distance" given by 
\begin{equation}\label{ }
\dd_\rho(x,y) = \int_\O \chi_{\cO(x,y)}(\zeta) \drho(\zeta) 
\end{equation}
Now let $\phi(d,z): \R_+ \times \O \to \R_+$ be a smooth non-negative kernel, radial in $z$. We define
\begin{equation}\label{e:topoker}
\phi_\rho(x,y) = \phi( \dd_\rho(x,y), x-y ).
\end{equation}
The kernel incorporates both metric and topological distances. Note that due to the symmetry of the domain $\cO(x,y)$, the kernel is also symmetric.

Let us define 
\begin{equation}\label{CStopo}
\st_\rho(x) = \int_\O \phi_\rho(x,y) \drho(y), \qquad \ave{u}_\rho = \frac{ \int_\O \phi_\rho(x,y) u(y) \drho(y)}{ \int_\O \phi_\rho(x,y) \drho(y)}. \tag{$\cMcstopo$}
\end{equation}
This is the full topological variant of \ref{CS}.  As these models bear relevance to biological applications it makes most sense to assume inverse dependence on the topological distance. For example,
\begin{equation}\label{e:kertopo}
\phi(d,z) = \frac{\psi(z)}{(\e + d^2)^{\a /2}}, \a \geq 0,
\end{equation}
where $\psi$ is a smooth kernel and $\e>0$ is a parameter ($\e = 0$ would correspond to the fully singular case).

 By analogy we can also define a topological version of \ref{MT}:
\begin{equation}\label{MTtopo}
\st_\rho(x) = 1, \qquad \ave{u}_\rho = \text{same}, \tag{$\cMmttopo$}
\end{equation}
or the $\beta$-model
\begin{equation}\label{Mbtopo}
\st_\rho(x) = \left( \int_\O \phi_\rho(x,y) \rho(y) \dy \right)^{\beta}, \qquad \ave{u}_\rho  = \text{same}. \tag{$\cMbtopo$}
\end{equation}

There is no reasonable topological counterpart of the mollified model \ref{Mf},  since there is no way to guarantee that $\phi_\rho$ integrates to $1$ at all times.
\end{example}

\begin{example}[Models with strict Segregation]\label{ }
A family of examples with segregated alignment protocol can be built by setting $\st_\rho=1$, fixing a $\s$-algebra $\cF$ of Borel subsets of $\O$ and considering the conditional expectation $\E_\rho (f|\cF)$ relative to $\drho$. Define
\begin{equation}\label{Mcond}
\ave{u}_{\rho,\cF} =  \E_\rho (u | \cF). \tag{$\cMcond$}
\end{equation}
For a given filtration $\{ \O,\emptyset\} \ss \cF_1 \ss \cF_2 \ss \dots \to \cB$ we can define a martingale chain of averages
\[
\ave{u}_{\rho,n} = \E_\rho (u | \cF_n)
\]
which naturally connects the global averaging model with the purely local one, as $\ave{u}_{\rho,n} \to u$ in any $L^p(\rho)$, $1\leq p<\infty$.

Such an averaging operation models strict segregation between disjoint subalgebras of $\cF$, so-called ``neighborhoods". Let us consider one specific example. Suppose $\cF$ is the algebra spanned by a partitioning  of $\O$ into subsets $A_1,\dots,A_L$. Then
\begin{equation}\label{MF}
\ave{u}_{\rho,\cF} = \sum_{l=1}^L \frac{\one_{A_l}}{\rho(A_l)} \int_{A_l} u \rho \dx. \tag{$\cMF$}
\end{equation}
If $u_0 = u^l_0$ within each cube $A_l$, and initial density $\rho_0$ is stays away from the borders $\p A_l$,  then  for a short period of time the solution satisfies a pure transport equation
\[
\rho_t + u_0^l \cdot \n_x \rho = 0
\]
on each $A_l$. So, the flock will travel with constant velocity within each neighborhood and will remain segregated until one piece reaches the boundary of its neighborhood and starts communicating with others. 
\end{example}

\begin{example}[Smooth Segregation]

Since in practice there is always a gradual transition between neighborhoods, it makes sense to consider a smooth version of the model above, which is also more amenable to analysis. Let us assume that $\O$ is compact, and consider any smooth partition of unity $g_l \in C^\infty(\O)$, $g_l \geq 0$, and $\sum_{l=1}^L g_l = 1$.  Most naturally, such a partition can be obtained by subordinating it to an open cover $\{ \cO_l \}_{l=1}^L$ of $\O$,  so that $\supp g_l \ss \cO_l$. We define the model by setting all $\st_\rho = 1$, and 
\begin{equation}\label{Mseg}
\ave{u}_\rho(x) = \sum_{l=1}^L \frac{g_l(x)}{  \rho(g_l)} \int_\O u g_l \rho \dy, \quad \rho(g_l) = \int_\O g_l \rho \dx. \tag{$\cMseg$}
\end{equation}
In this model the boundaries are not sharp as in the previous version, and there is some exchange of information that occurs across the adjacent neighborhoods.

\end{example}

There are ways to combine several averaging models into one that describe evolution of a multi-flock. Here ``multi" may mean several things -- either multiple subflocks with their own communication rules combine into a mega-flock with some global communication between subflocks, or it could mean the use of several communication rules within and between subgroups which we call `species'. Both of these variants were studied in \cite{HeTmulti,STmulti}.

\begin{example}[Multi-species] When a big flock contains groups of agents with distinct characteristics, communication between different groups may be facilitated by different rules, or communication kernels $\phi^{\a\b}$.  A model that accommodates such various communication rules was introduced in \cite{HeTmulti} :
\begin{equation}\label{e:HTmulti}
\begin{split}
\dot{x}^\a_i &= v^\a_i, \qquad i=1,\dots,N^\a, \ \a = 1,\dots,A,\\
\dot{v}^\a_i & = \sum_{\b=1}^A \sum_{j=1}^{N^\b} m^\b_j \phi^{\a\b}(x^\b_j - x^\a_i)(v^\b_j - v^\a_i).
\end{split}
\end{equation}
Here, each communication protocol is of Cucker-Smale type.

Such multi-species models can be generalized and fit into the framework of environmental averaging we discuss here. To do that, suppose we have an array of $A^2$ material models $\cM_{\a\b}$, $\a,\b =1,\dots,A$ defined over the same environment $\O$.  We can combine them into a new multi-model on the product space $\O\times A$. To account for possible variations of masses of sub-flocks, we fix a set of masses $\{M^\a\}_\a$ with the total mass being  $M = \sum_\a M^\a$, and encode them into the set of admissible densities $\cD^A$ over $\O\times A$. Namely, we say that  $\rho \in \cD^A$ is admissible if 
\[
\rho = \frac{1}{M} \sum_{\a=1}^A M^\a \rho^\a \otimes \d_\a,
\]
where $\rho^\a \in \cP$.  
We define a  cumulative strength function by
\begin{equation*}\label{}
\st_\rho(x,\a) = \sum_{\b = 1}^A M^\b \st^{\a\b}_{\rho^\b}(x).
\end{equation*}
The corresponding averaging of a function $u = \{ u^\a\}_\a$ is given by
\begin{equation}\label{ }
\ave{u}_\rho(x,\a) = \frac{1}{\st_\rho(x,\a)} \sum_{\b = 1}^A  M^\b \st_{\rho^\b}^{\a\b}(x) \ave{u^\b}_{\rho^\b}^{\a\b}(x).
\end{equation}
In terms of this average one can see directly, that the model \eqref{e:HTmulti} takes the canonical form
\[
\dot{v} = \st_\rho (\ave{v}_\rho -v).
\]

\end{example}

\begin{example}[Multi-flocks]  Let us recall the multi-flock model introduced in \cite{STmulti}
\begin{equation}\label{e:multiflock}
\begin{split}
\dot{x}^\a_i&=v^\a_i,\\
\dot{v}^\a_i& =   \sum_{j = 1}^{N_\a} m^\a_{j}  \phi^\a(x^\a_i - x^\a_j) (v^\a_j-v^\a_i) + \e \sum_{\substack{\b=1 \\ \b \neq \a}}^{A} M^\b\psi(X^\a,X^\b) (V^{\b}-v^\a_i).
\end{split}
\end{equation}
The model represents $A$ groups of agents evolving according to their own communication, Cucker-Smale type in this particular case, while communication between groups is facilitated through another protocol which involves a kernel $\psi$ and alignment to macroscopic parameters of each subflock, namely their center of masses
\[
X^\a =\frac{1}{M^\a}  \sum_{i=1}^{N_\a} m_i^\a x_i^\a, \quad M^\a =  \sum_{i=1}^{N_\a} m_i^\a,
\]
and momenta
\[
V^\a = \frac{1}{M^\a}  \sum_{i=1}^{N_\a} m_i^\a v_i^\a.
\]
This idea can be made more formal via an asymptotic analysis detailed in \cite{STmulti}.

In general, let $\{\cM^\a\}_{\a=1}^A$ be a family of material models defined over the same environment $\O$.  We define the admissible set of densities $\cD^A$ as in the previous example. For any $\rho=\{\rho^\a\}_\a \in \cD^A$ we define the strength function by
\[
\st_\rho(x,\a) = M^\a \st^\a_{\rho^\a}(x),
\]
and for $u = \{ u^\a\}_\a$ the average is given by
\[
\ave{u}_\rho(x,\a) = \ave{u^\a}^\a_{\rho^\a}(x).
\]

So far this model incorporates only internal flock communications. To combine these into an interactive multi-flock we assume that the communication between sub-flocks is facilitated through another averaging model $(\st^\ext_\rho, \ave{u}^\ext_\rho)$. The multiflock model \eqref{e:multiflock} can be written as a system over $\O\times A$ :
\[
\dot{v} = \st_{\rho} (\ave{v}_{\rho} - v) + \e  \st^\ext_{\rho^A}( \ave{V}_{\rho^A}^\ext - v),
\]
where $\rho^A = \sum_{\a = 1}^A M^\a \d_{X^\a}$, and $V = \sum_{\a = 1}^A \one_{V^\a}$.

\end{example}

\begin{example}[Models on finite sets] \label{ex:finite}

The last but not least example on our list  is the family of models  on finite environments $\O = \{ x_1, \dots, x_N\}$. These will be an essential tool to prove results about continuous models, see Appendix~\ref{a:fm}. Finite models illustrate a situation when all the agents are planted in their places and simply play the role of labels. They do not give rise to any inertial systems of type \eqref{e:ABS}. However, they do give rise to families of first order linear systems for $v_i = v(x_i)\in \R^m$,
\[
 \dot{v}_i = \st_i ( \ave{v}_i - v_i ),
\]
for each distribution of masses $\rho = (m_1,\dots, m_N)$. Since the averages act coordinate-wise, \eqref{e:coordave}, the systems for each coordinate decouple and we can assume that $v_i$ are scalars. In this case the properties of the model can be reduced to the properties of the corresponding reproducing matrix associated with the average:
\[
A = ( a_{ij} )_{i,j = 1}^N, \qquad a_{ij} = \ave{\one_{x_j}}(x_i).
\]
 Property (ev3) implies that $A$ has non-negative entries, and $A \one = \one$, i.e. $A$ is right-stochastic.  
\end{example}

\section{Classes of models and their properties} \label{s:classes}

 In this section we will systematize functional properties of environmental averaging models without association with any dynamical law. We introduce several important classes based on their operator-theoretical classification, which will be used extensively in subsequent studies. 

\subsection{Mapping properties. Jensen inequality} Let us discuss functional basics of environmental averages, and direct consequences of mapping properties stated in (ev2) and (ev3). 

First of all, order preserving maps \eqref{e:order} obey the maximum principle
\begin{equation}\label{e:max}
\min f \leq \ave{f}_\rho \leq \max f,
\end{equation}
and consequently are contractive on $L^\infty(\k_\rho)$:
\begin{equation}\label{e:Linftycontractive}
\| \ave{f}_\rho \|_\infty \leq \|f\|_\infty.
\end{equation}

Next, let us look into $L^\infty$-adjoint operator $\ave{\cdot}^*$. Technically, it maps $(L^\infty)^* \to (L^\infty)^*$ and if restricted to $L^1$ it still lands into $(L^\infty)^*$ from this general prospective. However, the extra structure of the averaging allows us to conclude more.
\begin{lemma}\label{l:L1toL1}
The operator $\ave{\cdot}_\rho^*$ has the following properties:
\begin{itemize}
\item[(1)] $\ave{\cdot}_\rho^*: L^1(\k_\rho) \to L^1(\k_\rho)$, and hence,  $\ave{\cdot}_\rho$ is weak$^*$-continuous on $L^\infty(\k_\rho)$;
\item[(2)] $\ave{\cdot}_\rho^*$ is order preserving;
\item[(3)] $\ave{\cdot}_\rho^*: L_+^1(\k_\rho) \to L_+^1(\k_\rho)$ is an isometry.
\end{itemize}
\end{lemma}
\begin{proof}
Let us fix $f\in L^1(\k_\rho)$ and for every measurable set $A$ define
\[
\nu_f (A) = \int_\O f \ave{\one_A}_\rho \dk_\rho.
\]
This defines finite $\s$-additive measure. Indeed, if $A = \cup_{i=1}^\infty A_i$, a disjoint union, then $\one_{\cup_{i=1}^N A_i } \to \one_{\cup_{i=1}^\infty A_i }$ in $L^2(\k_\rho)$. By (ev2), we then also have $\ave{\one_{\cup_{i=1}^N A_i }}_\rho \to \ave{ \one_{\cup_{i=1}^\infty A_i }}_\rho$ in $L^2(\k_\rho)$. Then up to a subsequence, the same convergence holds $\k_\rho$-a.e. By the Lebesgue dominated convergence theorem we obtain $\nu_f(\one_{\cup_{i=1}^N A_i }) \to \nu_f(\one_{\cup_{i=1}^\infty A_i }) $.

Furthermore, if $\k_\rho(A) = 0$, then by (ev2) $\ave{\one_A}_\rho= 0$ a.e., and hence $\nu_f(A) = 0$. This implies that $\nu_f$ is absolutely continuous with respect to $\k_\rho$. Hence, there exists a function $g \in L^1(\k_\rho)$ such that 
$\int_\O f \ave{\one_A}_\rho \dk_\rho = \int_\O g \one_A \dk_\rho$.
By approximation and continuity \eqref{e:Linftycontractive} we obtain the same relation $\int_\O f \ave{h}_\rho \dk_\rho = \int_\O g h \dk_\rho$,
for any $h \in L^\infty(\k_\rho)$. This means that $\ave{f}^*_\rho = g \in L^1(\k_\rho)$. We have proved (1).

Preservation of order (2) follows directly from (ev3) since if $f\in L^1_+(\k_\rho)$, then
$\int \ave{f}^*_\rho g\dk_\rho = \int f \ave{g}_\rho \dk_\rho \geq 0$ for all $g\in L^\infty_+(\k_\rho)$. Hence, $\ave{f}^*_\rho \geq 0$. Moreover, $\int \ave{f}^*_\rho \dk_\rho = \int f \ave{\one }_\rho \dk_\rho = \int f \dk_\rho $, which proves (3).

\end{proof}

As a consequence, we obtain the following point-wise Jensen inequality for averagings.

\begin{lemma}\label{}
For any $u\in L^\infty(\k_\rho)$ the following Jensen inequality holds $\k_\rho$-a.e.,
\begin{equation}\label{e:JI}
\psi( \ave{u}_\rho(x)) \leq \ave{\psi(u)}_\rho(x),
\end{equation}
where $\psi$ is a continuous convex even and monotonically increasing on $\R^+$ function.
\end{lemma}
\begin{proof}
By \lem{l:L1toL1}, for every $A\ss \O$, there exists $f_A\in L^1_+(\k_\rho)$, $\|f_A\|_1 = 1$, such that 
\[
 \frac{1}{\k_\rho(A)}\int_A \ave{u}_\rho \dk_\rho = \int_\O u  f_A \dk_\rho.
 \]
 Then by the classical Jensen inequality we have
 \begin{equation*}\label{}
\begin{split}
 \psi \left(  \frac{1}{\k_\rho(A)}\int_A \ave{u}_\rho \dk_\rho  \right) & =  \psi \left( \left| \int_\O u f_A \dk_\rho \right|  \right) \leq \psi \left( \int_\O |u| f_A \dk_\rho \right) \\
 & \leq \int_\O \psi(|u| ) f_A \dk_\rho = \int_\O \psi(u ) f_A \dk_\rho = \frac{1}{\k_\rho(A)} \int_A \ave{ \psi(u )}_\rho  \dk_\rho  .
 \end{split}
\end{equation*}
Since this holds for any $A$, by the Lebesgue differentiation theorem and continuity of $\psi$, as $A \to \{x\}$ for a.e. $x$ we obtain \eqref{e:JI}, as desired.
\end{proof}

One of the useful consequences of Jensen's inequality is extrapolation to $L^p$-spaces for $p<2$ and a bound on the $L^p$-norms. 

\begin{lemma}\label{l:extrap}
Suppose $\ave{1}_\rho^*  \in L^\infty(\k_\rho)$. Then $\ave{\cdot}_\rho : L^p(\k_\rho) \to L^p(\k_\rho)$ for all $1 \leq p \leq \infty$, and 
\begin{equation}\label{ }
\| \ave{\cdot}_\rho \|_{L^p(\k_\rho) \to L^p(\k_\rho)} \leq \| \ave{1}_\rho^* \|_\infty^{1/p}.
\end{equation}
\end{lemma}
\begin{proof}
For $p = \infty$ the result is simply the axiom (ev3). For $p<\infty$, we use Jensen's inequality
\[
\int_\O | \ave{u}_\rho|^p \dk_\rho \leq \int_\O  \ave{| u |^p}_\rho \dk_\rho = \int_\O |u|^p \ave{1}^*_\rho \dk_\rho \leq \| u\|_p^p \| \ave{1}_\rho^* \|_\infty,
\]
and the result follows.
\end{proof}

In some of our studies we will encounter the need to quantify boundedness of the weighted averages $\st_\rho \ave{\cdot}_\rho$ on $L^2(\rho)$. This is weaker than the previous mapping property thanks to the uniform boundedness of $\st_\rho$. Thus, a  weaker condition is required for it to hold.

\begin{lemma}\label{l:extraprho}
Suppose $\st_\rho \ave{\st_\rho^{p-1}}_\rho^*  \in L^\infty(\rho)$, $1 \leq p < \infty$. Then $\st_\rho \ave{\cdot}_\rho : L^p(\rho) \to L^p(\rho)$, and 
\begin{equation}\label{ }
\| \st_\rho \ave{\cdot}_\rho \|_{L^p(\rho) \to L^p(\rho)} \leq \|\st_\rho \ave{\st_\rho^{p-1}}_\rho^* \|_\infty^{1/p}.
\end{equation}
\end{lemma}
\begin{proof}
Using Jensen's inequality,
\begin{equation*}\label{}
\begin{split}
\| \st_\rho \ave{u}_\rho \|^p_{L^p(\rho)} = \int_\O | \st_\rho \ave{u}_\rho|^p \drho & \leq \int_\O  \ave{| u |^p}_\rho \st_\rho^p \drho = \int_\O  \ave{| u |^p}_\rho \st_\rho^{p-1} \dk_\rho  = \int_\O  | u |^p \ave{\st_\rho^{p-1} }^*_\rho \dk_\rho \\
& = \int_\O \st_\rho \ave{\st_\rho^{p-1} }^*_\rho  | u |^p\drho \leq \|\st_\rho \ave{\st_\rho^{p-1}}_\rho^* \|_\infty \| u\|^p_{L^p(\rho)}.
\end{split}
\end{equation*}
\end{proof}
We refer to \sect{s:rth} for further discussion.

\subsection{Reproducing kernel} \label{s:repkernel}
For material models we often deal with the weighted averaging $\st_\rho \ave{u}_\rho$ rather than the bare averaging $\ave{u}_\rho$. In most models the weighted average is an integral operator  represented by a kernel.
\begin{definition}\label{d:repker}
A reproducing kernel of the  model $\cM$ is a non-negative function $\phi_\rho \in L^1_+(\rho \otimes \rho)$, such that
\[
\int_\O \phi_\rho (x,y) \drho(y) = \st_\rho(x), \quad \rho\text{-a.e.}
\]
and so that  for all $u\in L^\infty(\rho)$
\begin{equation}\label{e:warep}
\st_\rho \ave{u}_\rho (x)= \int_\O \phi_\rho(x,y) u(y) \drho(y), \quad \rho\text{-a.e.}
\end{equation}
\end{definition}

A list of examples including our core models is provided in Table~\ref{t:kernels}.

\begin{table}
\begin{center}
\caption{Reproducing kernels}\label{t:kernels}
\begin{tabular}{  c |  c | c | c | c | c | c} 
 MODEL &   \ref{global}  &   \ref{CS}  &  \ref{MT}  &  \ref{Mb} &  \ref{Mf}  &  \ref{Mseg}  \\
  \hline
$\phi_\rho$ &  $1$ &  $\phi(x-y)$  &  $\displaystyle{ \frac{\phi(x-y)}{\rho_\phi(x)} }$ &$\displaystyle{ \frac{\phi(x-y)}{\rho^{1-\b}_\phi(x)} }$ & $\displaystyle{\int_\O \frac{\phi(x-z) \phi(y-z)}{\rho_\phi(z)} \dz}$ & $\displaystyle{\sum_{l=1}^L \frac{g_l(x) g_l(y)}{\rho(g_l)}}$  
 \end{tabular}
\end{center}
\end{table}

Generally, the kernel can be recovered from a right-stochastic reproducing kernel of the average itself,
\begin{equation}\label{e:rk}
\begin{split}
\Phi_\rho \in L^1_+(\k_\rho \otimes \k_\rho), & \quad  \int_\O \Phi_\rho (x,y) \dk_\rho(y) = 1, \quad \k_\rho\text{-a.e.}\\
\ave{u}_\rho & = \int_\O \Phi_\rho (x,y) u(y) \dk_\rho(y).
\end{split}
\end{equation}
The correspondence between the two is given by
\begin{equation}\label{e:warepker}
\phi_\rho(x,y) = \st_\rho(x)  \Phi_\rho(x,y)  \st_\rho(y).
\end{equation}
The representation of the adjoint averaging is given by
\begin{equation}\label{e:warepadj}
\st_\rho(y) \ave{v}_\rho^*(y) = \int_\O \phi_\rho(x,y) v(x) \drho(x), \quad \rho\text{-a.e.}
\end{equation}

Reproducing kernels are useful for many reasons. Not only do they provide more specific structure to the averaging operator, many properties of the averaging  that we will introduce later can be restated in terms of regularity of the kernel, see \sect{s:rth}.  The alignment forces that appear on all levels of description take a more conventional form:
\begin{subequations}\label{e:sys}
\begin{align}
 \st_i ( \ave{v}_i - v_i ) & = \sum_{j=1}^N m_j \phi_{\rho^N}(x_i,x_j) (v_j - v_i), \label{e:sys1}\\
 \st_\rho (\ave{u}_\rho -v) & = \int_\domain \phi_\rho(x,y) (w-v) f(y,w) \dw \dy \label{e:sys2} \\
\st_\rho (\ave{u}_\rho - u) & = \int_\O \phi_\rho(x,y) (u(y) - u(x) ) \drho(y). \label{e:sys3}
\end{align}
\end{subequations}

\subsection{Conservative models and contractivity} \label{s:cons}
Recall that due to \eqref{e:max} every alignment system that is based on an environmental averaging has a maximum/minimum principle and therefore tends to align. If one can quantify the rate of change of the amplitude of $u$ based on properties of the couple $(\k_\rho,\ave{\cdot}_\rho)$ one can potentially obtain an alignment  $u \to \bar{u}$ to some constant velocity vector $\bar{u}$. However, not every model has a predetermined $\bar{u}$. Typically $\bar{u}$ is uniquely defined by the initial condition if the system preserves the momentum.  This property is insured if the underlying model is conservative.

\begin{definition}\label{d:cons}
We say that the model $\cM$ is {\em conservative} if for any  $\rho \in \cP(\O)$, $u\in L^2(\k_\rho)$
\begin{equation}\label{e:cons}
\int_{\O}  u   \dk_\rho = \int_{\O} \ave{u}_\rho \dk_\rho.
\end{equation}
\end{definition}

At all levels of description   \eqref{e:ABS}, \eqref{e:VAintro},  \eqref{e:EAS},  conservative models preserve momentum,
\[
\ddt \bar{u}= 0, \quad \bar{u} = \int_{\O} \rho u \dx.
\]
 Since we assume that the total mass of a flock is $1$, this also predetermines the limiting average velocity from the initial condition $\bar{u} =  \int_{\O} \rho_0 u_0 \dx$.  Non-conservative models such as \ref{MT} may also align, see \sect{s:CST} below. However, for those models the limiting velocity emerges  dynamically and is not predetermined by the initial condition.

In  operator terms being conservative simply means that the adjoint average $\ave{\cdot}^*$ also preserves constants
\begin{equation}\label{e:*const}
\ave{\one_\O}^*_\rho = \one_\O, \qquad \k_\rho\text{-almost everywhere}, \ \forall \rho  \in \cP(\O).
\end{equation}
This in turn implies that the space of mean-zero fields 
\[
L^2_0(\k_\rho) = \left\{u\in L^2(\k_\rho): \int_\O u \dk_\rho = 0 \right\}
\]
is invariant for both $\ave{\cdot}_\rho$ and $\ave{\cdot}^*_\rho$.

Together with the positivity proved in \lem{l:L1toL1}, \eqref{e:*const}  implies that $\ave{\cdot}^*_\rho : L^\infty(\k_\rho) \to L^\infty(\k_\rho)$, and so the adjoint model $\cM^*$ consisting of pairs $(\k_\rho, \ave{\cdot}_\rho^*)$ fulfills all the requirements of environmental averaging. 

\begin{lemma}\label{l:consdual}
If $\cM$ is conservative, then $\cM^*$ also defines a conservative model. If $\cM_1$ and $\cM_2$ are conservative with the identical set of strength functions, then $\cM_2 \circ \cM_1$ and $\frac12 ( \cM_1 + \cM_2)$ are also conservative.
\end{lemma}

For a material model that possesses a reproducing kernel  being conservative is equivalent to $\Phi_\rho$ being doubly stochastic, or equivalently for $\phi_\rho$ to satisfy:
\begin{equation}\label{e:doublestoch}
\int_\O \phi_\rho (x,y) \drho(x)  = \st_\rho(y).
\end{equation}

A useful reformulation of conservative property can be done in terms of contractivity.

\begin{definition}\label{d:contr}
We say that the model $\cM$ is {\em $p$-contractive}, $1 \leq p \leq \infty$, if for any  $\rho \in \cP(\O)$, $u\in L^p(\k_\rho)$
\begin{equation}\label{e:contractive}
\| \ave{u}_\rho \|_{L^p(\k_\rho)} \leq \| u \|_{L^p(\k_\rho)}.
\end{equation}
\end{definition}

Note that straight from the definition part (ev3) all models are $\infty$-contractive. It is easy to show that contractivity is equivalent to being conservative.

\begin{lemma}\label{l:conscontr} The following are equivalent:
\begin{itemize}
\item[(i)] $\cM$ is conservative;
\item[(ii)] $\| \ave{\one_\O}^*_\rho \|_{L^\infty(\k_\rho)} \leq 1$  for all $\rho \in \cP(\O)$;
\item[(iii)]  $\cM$ is $p$-contractive for all $1\leq p\leq \infty$;
\item[(iv)]  $\cM$ is $1$-contractive.
\end{itemize}
\end{lemma}
\begin{proof} (i) $\Rightarrow$ (ii) is trivial.  Conversely, assume (ii). Then we have
\[
\int_\O \one_\O  \dk_\rho \geq \int_\O \ave{\one_\O}^*_\rho \dk_\rho  = \int_\O \one_\O \ave{\one_\O}_\rho \dk_\rho =  \int_\O \one_\O  \dk_\rho,
\]
which proves that $\ave{\one_\O}^*_\rho = \one_\O$ $\k_\rho$-almost everywhere.

The implication (ii)  $\Rightarrow$ (iii) is a direct consequence of \lem{l:extrap}. 

Since (iii) $\Rightarrow$ (iv) is trivial, let us now assume (iv). By duality $\ave{\cdot}_\rho^*$ is $\infty$-contractive, and hence (ii) holds. 
\end{proof}

Contractivity also implies that the alignment force is dissipative. For example, for the pressureless Euler-Alignment system, see \eqref{e:EASpresless} below, we obtain
\begin{equation}\label{e:EASlessenlaw}
\ddt \frac12 \int_{\O} \rho |u|^2 \dx =   \int_{\O}  [u \cdot \ave{u}_\rho  - |u|^2] \dk_\rho \leq 0.
\end{equation}

\subsection{Symmetric models}
Most of the models on our  list  are in fact {\em symmetric}: for all  $\rho \in \cP(\O)$ and $u', u'' \in L^2(\k_\rho)$
\begin{equation}\label{e:sym}
( u' , \ave{u''}_\rho )_{\k_\rho} = (\ave{u'}_\rho, u'' )_{\k_\rho},
\end{equation}
where we generally adopt the following notation for the inner-product relative to a measure $\k$:
\begin{equation}\label{ }
(f,g)_{\k}=\int_{\O} f g   \dk. 
\end{equation}

In other words, $\ave{\cdot}^*_\rho = \ave{\cdot}_\rho$. In terms of reproducing kernel, if one is available, symmetry is equivalent to $\Phi_\rho$ being symmetric. 
Setting $u'' = \one_\O$  we can see that every symmetric model is conservative. However, not every conservative model is automatically symmetric.  Plenty of examples are provided by defining the averages with non-symmetric doubly stochastic reproducing kernels.

For symmetric models the energy law \eqref{e:EASlessenlaw} takes a more explicit form
\begin{equation}\label{e:EASlessenlawsym}
\ddt  \int_{\O} \rho |u|^2 \dx =  -  \int_{\O} \phi_\rho(x,y) | u(x) - u(y) |^2 \drho(x) \drho(y).
\end{equation}
We can see that the dissipation burns energy for as long as communicating agents of the  flock are not yet aligned. This creates a mechanism for flocking behavior to be discussed in more detail in \sect{s:flocking}.

If $\cM$ is a conservative but not symmetric model, then  canonical ways to symmetrize it would be to consider the model $\frac12 ( \cM + \cM^*)$ or $\cM^*\circ \cM$.
According to \lem{l:consdual} those define  proper environmental averages.

\subsection{Galilean invariance}
We say that the model $\cM$ is Galilean invariant if for all $x \in \O$ and $v\in \R^n$
\begin{align}
\k_{\rho(\cdot + v)} (x) & = \k_\rho (x+v), \label{e:GIk} \\
\ave{ u(\cdot + v)}_{\rho(\cdot + v) }(x) &= \ave{u}_\rho (x+v).\label{e:GIu}
\end{align}
In terms of reproducing kernel, if one is available, the Galilean invariance is equivalent to 
\begin{equation}\label{e:GIm}
\begin{split}
\st_{\rho(\cdot + v)} (x) & = \st_\rho (x+v),  \\
\phi_{\rho(\cdot + v)}(x,y) &= \phi_\rho (x+v,y+v).
\end{split}
\end{equation}

For a particular differential system $\cM$ is involved in, this property implies the conventional Galilean invariance with respect to transformation
 \begin{equation}\label{e:Gi}
 x \to x + t V, \quad v  \to v + V, \quad u \to u+V.
\end{equation}

All the models considered above except for segregation and conditional expectation ones are Galilean invariant. The segregation protocols are planted into a given geography of the map and therefore are not translation invariant.

\subsection{Ball-positivity} \label{s:ball}
If an operator $T$ on a (real in our case) Hilbert space $H_\R$  is positive semi-definite, i.e.
\begin{equation}\label{e:Tpos}
(Tu, u) \geq 0,
\end{equation}
geometrically this means that $Tu$ and $u$  lie on the same side of the hyperplane  $u^\perp$. If $Tu$ lies in an even more restricted location, namely, in the ball $\frac12 B_{ \|u\|}( u)$, i.e.
\begin{equation}\label{e:Tbpos}
\left\| Tu - \frac12 u \right\| \leq \frac12 \|u\|,
\end{equation}
then we call $T$ {\em ball-positive}. A more useful definition of ball-positivity can be stated  equivalently as follows
\begin{equation}\label{e:bcT}
(Tu,u) \geq \| Tu\|^2, \qquad \forall u\in H_\R.
\end{equation}
In other words, it is  positivity \eqref{e:Tpos} that comes with a more coercive flavor. Although, as far as we can trace, there is no standard term associated with this property in the literature, such operators appeared for instance in \cite{LT1998} (with $\eta = 1$) and \cite{Tad2002}.

In the context of environmental averaging models, where $H_\R = L^2(\k_\rho)$, and $T = \ave{\cdot}_\rho$, the ball-positivity is stated as follows
\begin{equation}\label{e:bpenergies}
(u, \ave{u}_\rho)_{\k_\rho} \geq \| \ave{u}_\rho \|_{L^2(\k_\rho)}^2, \qquad \forall u\in L^2(\k_\rho).
\end{equation}
This property has profound implications to flocking behavior of the system as we will see later in \sect{s:lowenergy}.

We identify many ball-positive models on our list with the use of a simple lemma.
\begin{lemma}\label{l:symmbp}
If  $\cM$ is symmetric, then $\cM$ is ball-positive if and only if it is positive semi-definite.
\end{lemma}
\begin{proof}
The forward implication is trivial. Conversely, if $\cM$ is non-negative and symmetric, then $(u,v)_T = (Tu,v)$ defines a (possibly degenerate) inner product on the real Hilbert space $H_\R = L^2(\k_\rho)$. Hence, the \CS\ applies 
\begin{equation}\label{e:CauchySchwartz}
|(Tu,v)|  \leq   \sqrt{(Tu,u)}  \sqrt{(Tv,v)}.
\end{equation}
Taking supremum over all unit $v$ and using the contractivity of $T$, we obtain the result.
\end{proof}

\begin{corollary}\label{ }
If $\cM$ is conservative, then $\cM^* \circ \cM$ is ball-positive. 
\end{corollary}

Clearly, the conditional expectation model \ref{Mcond} is ball-positive because it consists of orthogonal projections.
For   \ref{Mseg} we have
\begin{equation}\label{ }
(u, \ave{u}_\rho)_{\rho} = \sum_{l=1}^L \frac{\rho(u g_l)^2}{\rho(g_l)} \geq 0.
\end{equation}
The classical Cucker-Smale model \ref{CS} is ball-positive, provided the kernel $\phi$ is Bochner-positive, i.e. $\phi = \psi \ast \psi$, for some smooth $\psi \geq 0$. We have
\begin{equation}\label{e:CSball}
(u, \ave{u}_\rho)_{\k_\rho} = \int_\O (u\rho) \cdot (u \rho)_\phi \dx =  \int_\O (u\rho)_\psi^2 \dx  \geq 0.
\end{equation}
The symmetric \ref{Mf} model is also ball-positive
\begin{equation}\label{ }
(u, \ave{u}_\rho)_{\rho} = \int_\O \frac{|(u\rho)_\phi|^2}{\rho_\phi} \dx \geq 0.
\end{equation}
The same argument shows that all \ref{Mfp}-models are ball-positive.

Among symmetric but not necessarily ball-positive models are the topological models \ref{CStopo}. Here, the kernel is not Bochner-positive to even imply sign definiteness of the averaging. Incidentally, ball-positivity does not imply symmetry either. This will be shown in  \appx{a:fm}. So, these two properties are completely independent.

Nonetheless, ball-positivity, does imply a host of other properties including of course positivity and $2$-contractivity.  The $2$-contractivity alone does not seem to be sufficient to imply conservation, in spite of \lem{l:conscontr},   it is still possible to show that all ball-positive models are conservative. The proof of this result is not so straightforward. We include it in \appx{a:fm}.
\begin{proposition}\label{p:bpcons}
Every ball-positive model is conservative.
\end{proposition}

Let us summarize the list of properties, relations between them, and examples.
\medskip
\[
\boxed{ \text{ball-positive} \, \Longleftarrow \left\{
\begin{array}
[c]{cccccccc}
\text{symmetric} & \Longrightarrow & \text{conservative} & \Longleftrightarrow & \text{contractive} \\
  & & \Uparrow & & & \\
 \text{positive semi-definite} & \Longleftarrow & \text{ball-positive } & & 
\end{array} \right. }
\]

\bigskip
\begin{center}
\begin{tabular}{  c |  c | c |c |c } 
 MODEL & conservative  & symmetric& ball-positive & Galilean invariant  \\ 
 \hline
\ref{Id} & $\checkmark$ & $\checkmark$  & $\checkmark$ & $\checkmark$\\ 
 \hline
 \ref{global} & $\checkmark$ & $\checkmark$ &$\checkmark$ & $\checkmark$ \\ 
  \hline
\ref{CS} & $\checkmark$  & $\checkmark$& $\checkmark$ if $\phi = \psi \ast \psi$ & $\checkmark$\\
  \hline
  \ref{CStopo} & $\checkmark$  & $\checkmark$& $\times$ & $\checkmark$\\
  \hline
 \ref{MT} &  $\times$ &  $\times$ & $\times$  & $\checkmark$\\
  \hline
 \ref{Mf} &$\checkmark$  &  $\checkmark$ &$\checkmark$ & $\checkmark$\\
   \hline
 \ref{Mseg} & $\checkmark$  &  $\checkmark$ &$\checkmark$ & $\times$
 \end{tabular}
\end{center}
\bigskip

The most important applications of ball-positivity will be seen in the context of flocking and spectral gap calculations to be discussed in \sect{s:lowenergy}.

\subsection{Thickness, regularity, and well-posedness of microscopic systems}\label{s:rth}

In order to develop a meaningful analysis of alignment  models it will be necessary to make a list of continuity and regularity assumptions. We state those in terms of representing kernels and strength functions, which appears to be the most economical way.

\subsubsection{Locality of communication}
First and foremost we assume that representing kernels  support communication at a short range, i.e. 
\begin{equation}\label{e:locker}
\phi_\rho(x, y)  \geq c_0 \one_{|x-y| <r_0}, \text{ for some } r_0>0, \text{ and all } \rho \in \cP(\O).
\end{equation}
Typically, for Favre-based models, such locality follows from the corresponding locality of the defining convolution kernel $\phi$:
\begin{equation}\label{e:lockerdef}
\phi(r) \geq c_0 \one_{r<r_0}.
\end{equation}

Many models in our list satisfy this condition automatically. For the classical Cucker-Smale, it simply means that $\phi >0$ near the origin. The Motsch-Tadmor model fulfills the same via 
\begin{equation}\label{e:MTkerlow}
\phi_\rho(x,y)= \frac{\phi(x-y)}{\rho_\phi(x)} \geq \frac{1}{\|\phi\|_\infty}\phi(x-y),
\end{equation}
since $\rho_\phi(x) \leq \|\phi\|_\infty$.  Similarly, for the \ref{Mf}-model, we have
\[
\phi_\rho(x,y) \geq \frac{1}{\|\phi\|_\infty} \phi \ast \phi(x-y) \geq c_0 \one_{|x-y| <r_0}.
\]

The locality also holds for the segregation model \ref{Mseg} on a compact environment $\O$. Indeed, since $\sum_{l=1}^L g_l(x) = 1$, for every $x$ there exists $l$ such that $g_l(x) \geq 1/L$. Using continuity and compactness, there exists a $r_0>0$ such that for any $|x-y| < r_0$ we have $g_l(y) > 1/2L$. Then, since $\rho(g_l) \leq 1$,
\[
\phi_\rho(x,y) = \sum_{l=1}^L \frac{g_l(x) g_l(y)}{\rho(g_l)} \geq \frac{1}{2L^2} = c_0, \qquad \forall x,y:  |x-y| < r_0.
\]
Thus, \eqref{e:locker} is satisfied.

A better way to express \eqref{e:locker} and similar conditions that follow is through the use of a smooth cut-off function. Let us fix $\chi \in C^\infty_0(B_1(0))$ such that $\chi(x) = 1$ for $x\in B_{1/2}(0)$ and $0\leq \chi \leq 1$ throughout.  We denote the rescaling of $\chi$ by $\chi_r(x) = \chi(x/r)$.  Thus, \eqref{e:locker} implies
\begin{equation}\label{e:locchi}
\phi_\rho(x, y)  \geq c_0 \chi_{r_0}(x-y).
\end{equation}

\subsubsection{Thickness}   Flock with a certain weight present throughout its support or even the entire environment are called thick. One can use masses of balls, $\rho(B_r(x))$ as a measure of thickness. This concept was adopted, for example, in  \cite{MPT2018}. While useful in many situations (see Sections \ref{s:chains}, \ref{s:lowenergy}) for some models, however, thickness takes more individual form which is easier to satisfy. For example, in the \ref{CS} case it is more natural to measure thickness as $\rho_\phi$, while for \ref{Mseg} the thickness can be measured in terms of masses of neighborhoods, $\rho(g_l)$.  We adopt the following general definition.
\begin{definition}\label{d:th}
A {\em thickness} of a density $\rho \in \cP(\O)$ is a function $\Th_\rho : \O \to \R^+$ satisfying the following conditions
\begin{itemize}
\item[(i)] $\Th(\rho,\cdot)$ is lower semi-continuous;
\item[(ii)] $\rho( \{ x: \Th(\rho,x) = 0 \} ) = 0$, for all $\rho \in L^1(\O) \cap \cP(\O)$;
\item[(iii)] There exists $c>0$ such that $ \Th(\rho,x) \geq c \min \rho$, for all $\rho\in \O$;
\item[(iv)] Continuity-in-$\rho$: there exists a $c>0$ such that for all $\rho',\rho''\in \cP(\O)$,
\begin{equation}\label{ }
\| \Th(\rho',\cdot ) - \Th(\rho'',\cdot)\|_\infty \leq c \|\rho'-\rho''\|_{\cP}.
\end{equation}
\item[(v)] Compatibility with the continuity equation: if $\rho$ satisfies 
\[
\p_t \rho + \n_x \cdot (u\rho) = 0,
\]
then  for every point $x\in \O$, the function $t \to \Th(\rho(t),x)$ satisfies
\begin{equation}\label{e:contiv}
\p_t \Th (\rho,x) \geq - c \| u\|_{L^2(\rho)},
\end{equation}
in distributional sense.

\end{itemize}
Thickness of the flock over a subset $S \ss \O$ is defined by
\begin{equation}\label{e:thS}
\Th(\rho,S)  = \inf_{x\in S} \Th(\rho,x) .
\end{equation}
If $S=\O$ we call $\Th(\rho,\O)$ the  {\em uniform thickness} of the flock.
\end{definition}

\begin{example}\label{ } If no specific structural information is known about $\phi_\rho(x,y)$ except for locality \eqref{e:locker} then a universal choice for the thickness  would be the mass of a smoothed ball at a scale $0<r<r_0$ (called {\em ball-thickness}):
\begin{equation}\label{e:thball}
\Th(\rho,x) = \orho_r(x) = \rho \ast \chi_r (x).
\end{equation}
Most properties are easy to verify: for (i) we even have $\Th \in C^\infty$, for  (ii) we observe that 
\[
\{ x: \Th(\rho,x) = 0 \} \cap \supp \rho = \emptyset,
\]
(iii) and (iv) are trivial,  and as to (v) we have
\begin{equation}\label{e:thickrdyn}
\p_t \orho_r(x) = - \n_x \cdot (u \rho)_{\chi_r} = - (u \rho)_{\n \chi_r}   \geq  - c  \| u\|_{L^2(\rho)}.
\end{equation}
\end{example}

\begin{example}\label{ }
Thickness associated with a local convolution type kernel $\phi$, \eqref{e:lockerdef}, is given by
\begin{equation}\label{e:ThCS}
\Th(\rho,x) = \rho_\phi (x).
\end{equation}
Here all the properties are trivial. Note that locality \eqref{e:lockerdef} is necessary for (ii). This choice will be  suitable for all Favre-based models.
\end{example}

\begin{example}\label{}
Another example is associated with the segregation model \ref{Mseg}:
\begin{equation}\label{e:segthick}
\Th(\rho,x) = \min_{l: x\in \supp g_l} \rho(g_l).
\end{equation}
Here we also assume for technical reasons that $| \p (\supp g_l) | = 0$. 

To show the lower semi-continuity, let $x\in \O$ be such that $\Th(\rho,x) >a$. Suppose $l_1,\dots,l_k$ is the list of indexes such that $x\not \in \supp g_{l_i}$. Then there exists $\d>0$ such that $B_\d(x) \cap \supp g_{l_i} = \emptyset$. Then for all $y\in B_\d (x)$ the list of $l$'s for which $y\in \supp g_l$ is a subset of the list of $l$'s corresponding to $x$. So, $\Th(\rho,y) \geq \Th(\rho,x)$, and hence the set $\{ x: \Th(\rho,x)>a\}$ is open. 

To show (ii) suppose we have $x: \Th(\rho,x) = 0$, hence there exists $l$ such that $x\in \supp g_l$ and $\rho(g_l)=0$. If $g_l(x) >0$, then $\rho(B_\e(x))=0$ for a small $\e$, hence $x\not\in \supp \rho$. Otherwise, $x\in \p (\supp g_l)$. So, 
\[
 \{ x: \Th(\rho,x) = 0 \}  \ss( \O \backslash \supp \rho ) \cup \p (\supp g_1) \cup \dots \cup \p (\supp g_L),
 \]
 and the $\rho$-measure of the set on the right hand side is $0$.
 
 (iii) and (iv) are trivial, as to (v) we have similar to \eqref{e:thickrdyn}
 \[
 \p_t \rho(g_l) = \int_\O u \cdot \n g_l \drho \geq - c  \| u\|_{L^2(\rho)}
 \]
 for any $l=1,\dots, L$. So, for any fixed $x\in \O$ there is a finite collection of $l$'s such that $x\in \supp g_l$. Denote it $L(x)$. Since the minimum is taken over a fixed compact set $L(x)$ at any moment of time,  Rademacher's lemma applies to deduce \eqref{e:contiv} is distributional sense.
\end{example}

\subsubsection{Regularity of $\cM$ and continuous dependence on $\rho$}

Let us discuss now regularity and continuity-in-$\rho$ of our models.  We will encounter two type of models -- ones whose regularity depends on thickness (and therefore can be violated if the density in question is not thick), and ones that are uniformly regular independently of thickness. 

Before we make these definitions precise, let us make an observation  -- in all our models the strength is bounded from below by the native and ball-mass thicknesses: there exists an non-decreasing continuous function $\st : \R^+ \to \R^+$ such that
\begin{equation}\label{e:stlow}
\st_\rho(x)  \geq \st(\Th(\rho,x)), \ \st_\rho(x)  \geq \st(\orho_{r_0}(x))\quad \text{ for all }  x\in \O.
\end{equation}

Let us recall the classical Kantorovich-Rubinstein distance between any two finite measures $\mu',\mu''$ over $\O$:
\begin{equation}\label{e:KR}
	W_1(\mu',\mu'') = \sup_{ \Lip(h) \leq 1} \left| \int_{\O} h(x) [ \dmu'(x) - \dmu''(x) ]\right|.
\end{equation}

\begin{definition}\label{d:r}
We say that a model $\cM$ with is {\em regular} if for every $R>0$ and $\rho, \rho',\rho'' \in \cP(B_R)$ we have
\begin{align}
\| \p^k \st_\rho \|_{L^\infty(B_R)}  + \| \p^k_{x} \phi_\rho \|_{L^\infty(B_R \times B_R)} +\| \p^k_{y} \phi_\rho \|_{L^\infty(B_R \times B_R)} & \leq C_{k,R}( \Th(\rho,B_R)), \quad k= 0,1,\dots,\label{e:r1} \\
\| \st_{\rho'} -\st_{\rho''}  \|_{L^\infty(B_R)}  + \| \phi_{\rho'}  -\phi_{\rho''} \|_{L^\infty(B_R \times B_R)} & \leq C_R(\Th(\rho',B_R),\Th(\rho'',B_R))  W_1(\rho',\rho'').\label{e:r2}
\end{align}
\end{definition}

 \begin{definition}\label{d:ur}
We say that a model $\cM$ with is {\em uniformly regular} if  for every $R>0$ and $\rho, \rho',\rho'' \in \cP(B_R)$ we have
\begin{align}
\| \p^k \st_\rho \|_{L^\infty(B_R)}  + \| \p^k_{x} \phi_\rho \|_{L^\infty(B_R \times B_R)} +\| \p^k_{y} \phi_\rho \|_{L^\infty(B_R \times B_R)} & \leq C_{k,R}, \quad k= 0,1,\dots,\label{e:ur1}\\
\| \st_{\rho'} -\st_{\rho''}  \|_{L^\infty(B_R)}  + \| \phi_{\rho'}  -\phi_{\rho''} \|_{L^\infty(B_R \times B_R)} & \leq C_R  W_1(\rho',\rho'').\label{e:ur2}
\end{align}
\end{definition}


If no information is known about the thickness of one of the densities involved in \eqref{e:r2}, some of the models still retain a level of continuity if at least the other density is thick: for every $R>0$ and $\rho',\rho'' \in \cP(B_R)$ one has
\begin{align}
\int_{\O} | \st_{\rho'}(x) - \st_{\rho''}(x)|^2 \drho''(x) & \leq C_R(\Th(\rho',B_R)) W^2_1(\rho',\rho''), \label{e:glob1} \\
\int_\O \int_\O | \phi_{\rho'}(x,y)  -\phi_{\rho''}(x,y)|^2 \drho''(x) \drho''(y) & \leq C_R(\Th(\rho',B_R)) W^2_1(\rho',\rho''). \label{e:glob2} 
\end{align}
This will be useful in the study of the hydrodynamic limits.

Let us go through the main examples on our list, identify their associated thicknesses and determine which level of regularity they satisfy and  under which conditions. Our findings are summarized in Table~\ref{t:rth}.

\begin{table}
\begin{center}
\caption{Regularity type and associated thickness of selected models.}\label{t:rth}
\begin{tabular}{  c |  c | c | c | c |c } 
\multirow{2}{*}{ $\mathrm{MODEL}$} & \multirow{2}{*}{  \ref{CS} }  & \multirow{2}{*}{  \ref{CStopo} }  & \multirow{2}{*}{\ref{MT} \& \ref{Mb}, $0< \b <1$ }& \multirow{2}{*}{ \ref{Mf}}  & \multirow{2}{*}{ \ref{Mseg}  } \\
& & & & \\
  \hline
  \multirow{2}{*}{$\Th(\rho,x)$}  &  \multirow{2}{*}{$\rho_\phi(x)$}  &\multirow{2}{*}{$\rho_\psi(x)$} & \multirow{2}{*}{  $\rho_\phi(x) $}  & \multirow{2}{*}{$\rho_\phi(x) $ } &  \multirow{2}{*}{  $ \min_{l: x\in \supp g_l} \rho(g_l)$ } \\ 
 & & & & \\
  \hline
  \multirow{2}{*}{regular}  &  \multirow{2}{*}{$\checkmark$}  & \multirow{2}{*}{$\checkmark$}  & \multirow{2}{*}{ $\checkmark$}  & \multirow{2}{*}{ $\checkmark$ on compact $\O$ } &  \multirow{2}{*}{  $\checkmark$ } \\ 
&  & & & & 
  \\
  \hline
\multirow{2}{*}{uniform}  &  \multirow{2}{*}{$\checkmark$}  & \multirow{2}{*}{$\checkmark$}  &  \multirow{2}{*}{$\phi >0$} &   $\phi  >0$  on compact  $\O$    &  \multirow{2}{*}{  $\supp g_l = \O$ } \\ 
&  & & & $\phi = \frac{1}{\lan x \ran^{n+\g}}$  on   $\R^n$ &  \\
 \hline
\multirow{2}{*}{\eqref{e:glob1}-\eqref{e:glob2}}  &  \multirow{2}{*}{$\checkmark$}  & \multirow{2}{*}{$\checkmark$}  & \multirow{2}{*}{ $c \one_{|x|<r_0} \leq  \phi(x) \leq C \one_{|x| < R_0}$}  & \multirow{2}{*}{ $\checkmark$ on compact $\O$ } &  \multirow{2}{*}{  $\checkmark$ } \\ 
& & & & &  
\end{tabular}
\end{center}
\end{table}

\begin{example}[\ref{CS}, \ref{CStopo}] 
The Cucker-Smale model is trivially uniformly regular with $\Th(\rho,x) = \rho_\phi$. While \ref{CStopo} is uniformly regular with $\Th(\rho,x) = \rho_\psi$.
\end{example}
\begin{example}[\ref{Mb}, $0\leq \b <1$]  The model has the same associated thickness $\Th(\rho,x) = \rho_\phi$. Under no conditions on $\phi$, the models is trivially regular.  If $\phi>0$, then 
\begin{equation}\label{e:densphilow}
\Th(\rho,x) \geq \inf_{r<2R} \phi(r)  = \d >0
\end{equation}
on $B_R$ for any $\rho \in \cP(B_R)$. So, in this case the model is uniformly regular, and all the estimates are straightforward.  

Let us assume that $\phi$ is local and satisfies \eqref{e:kerRr}. We will prove that in this case the model is  continuous in $\rho$, \eqref{e:glob1}-\eqref{e:glob2}.
Indeed,  as to \eqref{e:glob1}, by an elementary inequality,  we have
\begin{equation}\label{e:compstb}
| (\rho'_\phi)^\b - (\rho''_\phi)^\b | \leq C |\rho'_\phi |^{\b-1} | \rho'_\phi - \rho''_\phi | \leq C(\Th(\rho',x) )\|\n \phi\|_\infty W_1(\rho',\rho''),
\end{equation}
and \eqref{e:glob1} follows.  As to the kernel continuity \eqref{e:glob2}, we have
\begin{equation*}\label{}
\begin{split}
& \int_\O \int_{\O}  | \phi_{\rho'}(x,y)  -\phi_{\rho''}(x,y)|^2 \drho''(x) \drho''(y)  \\
 \leq  & \|\phi\|_\infty \int_\O \int_\O   \left| \frac{1}{(\rho'_\phi(x))^{1-\b}}   -\frac{1}{(\rho''_\phi(x))^{1-\b}} \right|^2 \phi(x-y)\drho''(x) \drho''(y)  \\
 =  &\|\phi\|_\infty  \int_\O \left| \frac{1}{(\rho'_\phi(x))^{1-\b}}   -\frac{1}{(\rho''_\phi(x))^{1-\b}} \right|^2 \rho''_\phi(x) \drho''(x)   \\
  =& \|\phi\|_\infty  \int_\O \frac{| (\rho''_\phi(x))^{1-\b} -(\rho'_\phi(x))^{1-\b} |^2}{(\rho'_\phi(x))^{2-2\b}}  (\rho''_\phi(x))^{2\b}  \frac{\drho''(x)}{\rho''_\phi(x)} \\
 \leq & C(\Th(\rho',B_R)) W^2_1(\rho',\rho'') \int_{B_R} \frac{\drho''(x)}{\rho''_\phi(x)}.
\end{split}
\end{equation*}
Note that 
\begin{equation}\label{ }
\int_{B_R} \frac{\drho''(x)}{\rho''_\phi(x)} =\frac{1}{\|\phi\|_1} \int_{B_{R+R_0}} \int_{B_R} \phi(x-y) \frac{\drho''(x)}{\rho''_\phi(x)} \dy = \frac{1}{\|\phi\|_1} \int_{B_{R+R_0}} \left( \frac{\rho''}{\rho''_\phi} \right)_\phi(y) \dy.
\end{equation}
According to \eqref{e:KMTrho}, the expression inside is uniformly bounded, and hence the whole integral is bounded by a constant depending only on $R,R_0$.
This proves \eqref{e:glob2}. 

Finally, we note that if $\phi$ is local, then the ball-thickness \eqref{e:thball} with $r\leq r_0$ can also be used in all the estimates. This observation will be useful in the relaxation study, see \sect{s:hypo}. However it should be noted that  $\rho_\phi (x) \geq c \orho_{r}(x)$, and so it is easier for densities to be natively thick than ball-thick.

\end{example}

\begin{example}[\ref{Mseg}]
The computation is quite similar for the segregation model \ref{Mseg}, where the thickness functional $\Th$ is given by \eqref{e:segthick}.  The model is clearly uniformly thick if $\supp g_l = \O$, $l = 1,\dots,L$, since then $\Th(\rho, \O) = 1$ for any $\rho$. Generally, the global thickness is given by $\Th(\rho,\O) = \min_l \rho(g_l)$. So, it is clear that regularity holds for this model as well.  All these conclusions hold for the ball-mass thickness \eqref{e:thball} where $r$ is a small radius so that for every $l$ there exists an $x_0 \in \O_l$ such that $\rest{g_l}{B_r(x_0)} \geq c_0$ for some fixed $c_0>0$.

Let us establish \eqref{e:glob2} relative to the native thickness (the ball-thickness  \eqref{e:thball} will not work here)
\begin{equation*}\label{}
\begin{split}
& \int_\O \int_\O  | \phi_{\rho'}(x,y)  -\phi_{\rho''}(x,y)|^2 \drho''(x) \drho''(y)  \\
 \leq  &  \int_\O \int_\O \sum_{l=1}^L g_l(x)g_l(y)  \left| \frac{1}{\rho'(g_l)}   -\frac{1}{\rho''(g_l)} \right|^2 \phi(x-y)\drho''(x) \drho''(y) \\
 \leq & \sum_{l=1}^L (\rho''(g_l))^2  \left| \frac{1}{\rho'(g_l)}   -\frac{1}{\rho''(g_l)} \right|^2 =\sum_{l=1}^L  \left| \frac{\rho'(g_l) - \rho''(g_l)}{\rho'(g_l)}  \right|^2 \leq C(\Th(\rho',\O)) W^2_1(\rho',\rho'') . 
\end{split}
\end{equation*}

\end{example}

\begin{example}[\ref{Mf}] Because of the non-local dependence on $\rho_\phi$ in the kernel, there doesn't seem to be another thickness quantity that would fulfill the local continuity and regularity assumptions. However, if we set $\Th(\rho,x) = \rho_\phi(x)$, the model becomes regular on any  compact environment $\O$ and for any kernel $\phi$. Also on compact $\O$, the model is uniformly regular when $\phi>0$.  Finally, the strong continuity-in-$\rho$, \eqref{e:glob1}-\eqref{e:glob2}, holds as well:
\begin{equation*}\label{}
\begin{split}
 &\int_\O \int_{\O}  | \phi_{\rho'}(x,y)  -\phi_{\rho''}(x,y)|^2 \drho''(x) \drho''(y) \\
 \lesssim & \int_\O \int_{\O} \int_{\O}\phi(x-z)\phi(y-z) \left| \frac{1}{\rho'_\phi(z)} -\frac{1}{\rho''_\phi(z)}  \right|^2 \dz \drho''(x) \drho''(y) \\
 = & \int_\O   \left| \frac{1}{\rho'_\phi(z)} -\frac{1}{\rho''_\phi(z)}  \right|^2 (\rho''(z))^2 \dz \leq \frac{|\O|}{ \Th^{2}(\rho',\O) } W^2_1(\rho',\rho'').
\end{split}
\end{equation*}

 On the open space, if $\phi$ is compactly supported then the model would fail to fulfill any regularity assumptions. However, for the integrable kernel $\phi$  satisfying the following conditions
\begin{equation}\label{e:algker}
\begin{split}
\phi  \in W^{k,1}(\O),& \quad \forall k\in \N, \\
\frac12 |y| \leq |x| \leq 2|y| \quad & \Rightarrow \quad \phi(x) \sim \phi(y),
\end{split}
\end{equation}
one can establish uniform regularity.
The choice $\phi = \frac{1}{\lan x \ran^{n+\g}}$, $\g>0$, $\jap{x} = (1+|x|^2)^{\frac12}$, is an example of such a kernel.

To see that let $\rho \in  \cP(B_R)$. We have
\[
\p_x^k \phi_\rho(x,y) = \int_\O \frac{\p^k\phi(x-z) \phi(y-z)}{\rho_\phi(z)} \dz =  \int_{B_{2R}} \frac{\p^k \phi(x-z) \phi(y-z)}{\rho_\phi(z)} \dz +  \int_{\O \backslash B_{2R}} \frac{\p^k \phi(x-z) \phi(y-z)}{\rho_\phi(z)} \dz.
\]
Inside the ball $B_{2R}$ we have $\rho_\phi(z) \geq \d$ by \eqref{e:densphilow}. So,
\[
\int_{B_{2R}} \frac{\p^k \phi(x-z) \phi(y-z)}{\rho_\phi(z)} \dz \leq C(R) \| \phi\|_{W^{k,1}}.
\]
For $z\in \O \backslash B_{2R}$ we have by \eqref{e:algker}
\[
\rho_\phi(z) = \int_{B_R} \phi(z - w) \drho(w) \gtrsim \phi(z) \int_{B_R} \drho(w) = \phi(z).
\]
On the other hand, by the same \eqref{e:algker}, since $y \in B_R$,
\[
\phi(y-z) \lesssim \phi(z).
\]
Thus,
\[
\int_{\O \backslash B_{2R}} \frac{\p^k \phi(x-z) \phi(y-z)}{\rho_\phi(z)} \dz \lesssim \int_{\O \backslash B_{2R}}| \p^k \phi(x-z) | \dz \leq \| \phi\|_{W^{k,1}}.
\]
Since the kernel is symmetric the same holds for $\p^k_y \phi_\rho$. We have proved \eqref{e:ur1}. To show \eqref{e:ur2} let us write
\begin{equation*}\label{}
\begin{split}
 | \phi_{\rho'}(x,y)  -\phi_{\rho''}(x,y)| & \leq  \int_\O \phi(x-z) \phi(y-z) \frac{|\rho''_\phi(z) - \rho'_\phi(z) |}{\rho'_\phi(z) \rho''_\phi(z)} \dz  \\
& = \int_{B_{2R}} \phi(x-z) \phi(y-z) \frac{|\rho''_\phi(z) - \rho'_\phi(z) |}{\rho'_\phi(z) \rho''_\phi(z)} \dz  \\
&+ \int_{ \O \backslash B_{2R}} \phi(x-z) \phi(y-z) \frac{|\rho''_\phi(z) - \rho'_\phi(z) |}{\rho'_\phi(z) \rho''_\phi(z)} \dz    
\end{split}
\end{equation*}
Using again that inside the ball $B_{2R}$, $\rho'_\phi(z), \rho''_\phi(z) \geq \d$, we obtain
\[
\lesssim W_1(\rho',\rho'') + W_1(\rho',\rho'')  \int_{ \O \backslash B_{2R}} \frac{\phi(x-z) \phi(y-z)}{\rho'_\phi(z) \rho''_\phi(z)} \sup_{w\in B_{R}}|\n_w \phi(z-w)|  \dz.
\]
Arguing as before we conclude that $ \frac{\phi(x-z) \phi(y-z)}{\rho'_\phi(z) \rho''_\phi(z)} $ is uniformly bounded on $\O \backslash B_{2R}$. At the same time, $\sup_{w\in B_{R}}|\n_w \phi(z-w)|\in L^1(\dz)$. This finishes the estimate.

The native thickness in all of the above can be replaced with the ball-thickness \eqref{e:thball} as well. 
\end{example}

\subsubsection{Well-posedness of agent-based systems}
Let us establish basic well-posedness of the agent based system as a consequence of the uniform regularity:
\begin{equation}\label{e:ABSdet}\left\{
\begin{split}
\dot{x}_i & = v_i \\
 \dot{v}_i & = \st_i ( \ave{v}_i - v_i )
 \end{split}\right.\quad i=1\dots N.
\end{equation}
Here, $\O$ can be any environment.    Recall that $\st_i$ and $ \ave{v}_i $ are defined in  \eqref{e:emp}.  The  maximum principle implies that $\max_i |v_i| \leq \max_i |v_i(0)| : = A$, and therefore, $\max|x_i| \lesssim t$ a priori.  So, in order to establish global existence by the standard fix point argument it suffices to check that the right hand side of \eqref{e:ABSdet} is locally Lipschitz on $\O^N \times \R^{nN}$.

So, let us assume that $\cM$ is uniformly regular.  Let us fix masses $m_1,\ldots,m_N$ and  two configurations 
\[
(x'_1,\dots,x'_N; v'_1,\dots,v'_N) \in B_R^{N} \times B_A^{N}, \quad (x''_1,\dots,x''_N; v''_1,\dots,v''_N) \in B_R^{N} \times B_A^{N}.
\]
We only need to show Lipschitzness of the momentum equation. We have
\[
|\st_i'[v']_i - \st'_i v_i' - \st_i''[v'']_i + \st''_i v_i'' | \leq |\st_i'[v']_i - \st_i''[v'']_i | + |\st'_i v_i' - \st''_i v_i'' | = I + II.
\]
As to II,
\[
II \leq |\st'_i - \st''_i| |v'_i| +|\st''_i|| v'_i -  v_i'' | \leq A  |\st'_i - \st''_i| + \oS | v'_i -  v_i'' |.
\]
Using \eqref{e:ur1}-\eqref{e:ur2},
\begin{equation*}\label{}
\begin{split}
|\st'_i - \st''_i| & \leq |\st_{\rho'}(x'_i) -\st_{\rho'}(x''_i)| + |\st_{\rho'}(x''_i)   - \st_{\rho''}(x''_i) | \leq \oC_1 |x'_i - x''_i| + \oC W_1(\rho' , \rho'') \\
& \leq \oC_1 |x'_i - x''_i| + \oC  \sum_j m_j |x_j' - x''_j| \lesssim \max |x'_j - x''_j|.
\end{split}
\end{equation*}
We now estimate the weighted averages term using the same regularity assumptions,
\begin{equation*}\label{}
\begin{split}
I & \leq \sum_j m_j | \phi_{\rho'}(x_i',x'_j) v'_j - \phi_{\rho''}(x''_i,x''_j) v''_j | \\
& \leq \sum_j m_j  \phi_{\rho'}(x_i',x'_j) | v'_j -v''_j|  + \sum_j m_j |  \phi_{\rho'}(x_i',x'_j) -\phi_{\rho''}(x''_i,x''_j)| |v''_j| \\
& \leq  \oC_0 \max |v'_j - v''_j| + A\sum_j m_j |  \phi_{\rho'}(x_i',x'_j) -\phi_{\rho''}(x'_i,x'_j)| +  A\sum_j m_j |  \phi_{\rho''}(x_i',x'_j) -\phi_{\rho''}(x''_i,x''_j)| \\
& \leq  \oC_0 \max |v'_j - v''_j| + A W_1(\rho' , \rho'') +  2 A \oC_1  \sum_j m_j |x_j' - x''_j| \\
& \lesssim \max |v'_j - v''_j| + \max |x'_j - x''_j|.
\end{split}
\end{equation*}
We have proved the following result.
\begin{proposition}
If $\cM$ is uniformly regular, then the system \eqref{e:ABSdet} is globally well-posed.
\end{proposition}

Note that this well-posedness result is robust -- the Lipschitzness is independent of the number of agents or their masses. That is why it can be extended to kinetic formulation as well, see \sect{s:mfldet}.  However, the well-posedness in a less robust form also extends to some non-regular models such as \ref{MT} if $\phi$ is finitely supported and satisfies \eqref{e:locker}.  This is based on the fact that for any atomic $\rho$ we have $\rho_\phi(x_i) \geq m_i \phi(0)$. So, there is a residual mass-dependent thickness of the flock left on its support. Indeed, we have in this case 
\[
|[v']_i - [v'']_i | \leq \frac{\|\phi\|_\infty}{\phi(0)m_i}\sum_j m_j  | v'_j -v''_j|  + A \sum_j m_j |  \phi_{\rho'}(x_i',x'_j) -\phi_{\rho''}(x''_i,x''_j)| ,
\]
and 
\[
  \phi_{\rho'}(x_i',x'_j) -\phi_{\rho''}(x''_i,x''_j) = \frac{ \phi(x'_i - x'_j) }{ \sum_k m_k \phi(x_i' - x_k')} -\frac{ \phi(x''_i - x''_j) }{ \sum_k m_k \phi(x''_i - x''_k)}  \lesssim \frac{1}{m_i^2} \max |x'_k - x''_k|.
\]
Similar computation works for \ref{Mb}. Models \ref{Mseg} and \ref{Mf} do not seem to have good well-posedness properties when it comes to agent-based systems with purely local communication kernels.

\begin{proposition}
The models \ref{MT}, \ref{Mb} are globally well-posed provided the defining kernel $\phi$ is locally supported \eqref{e:lockerdef}.
\end{proposition}

\subsubsection{Uniform mapping properties}  When studying well-posedness of kinetic models it will be essential to have a uniform boundedness of the weighted averages at the base level independent of $\rho$.  These can be readily stated in terms of Lebesgue integrability conditions on the kernel.  We will isolate two such conditions.

First, the uniform boundedness on $L^2(\rho)$:
\begin{equation}\label{e:unifL2L2}
\st_\rho \ave{\cdot}_{\rho} : L^2 (\rho) \to L^2 (\rho)
\end{equation}
can be stated using the result of \lem{l:extraprho}. It is guaranteed to hold under a simpler condition:
\begin{equation}\label{e:inftyuniform2}
\sup_{\rho\in \cP}  \| \st_\rho \ave{1}_\rho^* \|_\infty < \infty.
\end{equation}
Recalling the action of the adjoint in terms of reproducing kernel \eqref{e:warepadj}, \eqref{e:inftyuniform2} can be stated as $\phi_\rho \in L^\infty_y L^1_x(\rho)$ uniformly in $\rho$:
\begin{equation}\label{e:L2L2}
 \sup_{\rho \in \cP(\O)} \sup_{y\in \O} \int_\O \phi_\rho(x,y) \drho(x) <\infty.
\end{equation}
This condition was first documented in the context of \ref{MT}-model in \cite{KMT2013}. It holds trivially for all conservative models, see \eqref{e:doublestoch}. For \ref{Mb}, including the Motsch-Tadmor model \ref{MT}, this follows from \eqref{e:KMTrhobeta}. So, all the core models on our list satisfy \eqref{e:L2L2}.

Second, a stronger uniform boundedness
\begin{equation}\label{e:unifL2Linfty}
\st_\rho \ave{\cdot}_{\rho} : L^2 (\rho) \to L^\infty (\rho)
\end{equation}
is guaranteed by the membership $\phi_\rho \in L^\infty_x L^2_y(\rho)$ uniformly in $\rho$ (by the \HI):
\begin{equation}\label{e:L2Linfty}
 \sup_{\rho \in \cP(\O)} \sup_{x\in \O} \int_\O |\phi_\rho(x,y)|^2 \drho(y) <\infty.
\end{equation}

Examples on our list include all \ref{Mb} and \ref{Mbtopo} models for $\b \geq \frac12$, and  in particular, the classical Cucker-Smale model \ref{CS}.  Indeed, we have for \ref{Mb},
\[
\int_\O \left| \frac{\phi(x-y)}{\rho^{1-\b}_\phi(x)} \right|^2 \drho(y)\leq \|\phi\|_\infty \frac{\rho_\phi(x)}{\rho^{2-2\b}_\phi(x)} = \|\phi\|_\infty \rho^{2\b-1}_\phi(x) \leq \|\phi\|_\infty^2.
\]
Unfortunately, \ref{MT}, \ref{Mf}, and \ref{Mseg} are not regular enough to satisfy \eqref{e:L2Linfty} for arbitrary kernels. However, if $\inf \phi >0$, that is of course the case for  \ref{Mf} and all \ref{Mb}, and similarly if $\supp g_l = \O$ for \ref{Mseg}.

The results are summarized in Table~\ref{t:admis}.

\begin{table}
\begin{center}
\caption{Conditions under which models are uniformly bounded}\label{t:admis}
\begin{tabular}{  c |  c | c | c | c  |c} 
\multirow{2}{*}{ $\mathrm{MODEL}$} & \multirow{2}{*}{  \ref{CS} }  & \multirow{2}{*}{  \ref{CStopo} } & \multirow{2}{*}{\ref{MT} \& \ref{Mb}, $0< \b <1$ }& \multirow{2}{*}{ \ref{Mf}}  & \multirow{2}{*}{ \ref{Mseg}  } \\
& & & & &\\
  \hline
\multirow{2}{*}{$ L^2 (\rho) \to L^2 (\rho)$} &  \multirow{2}{*}{$\checkmark$ }  &  \multirow{2}{*}{$\checkmark$ } & \multirow{2}{*}{$c \one_{|x|<r_0} \leq  \phi(x) \leq C \one_{|x| < R_0}$} & \multirow{2}{*}{$\checkmark$ } & \multirow{2}{*}{$\checkmark$ } \\ 
 & & & & & \\
\hline
\multirow{2}{*}{$ L^2 (\rho) \to L^\infty (\rho)$} &  \multirow{2}{*}{$\checkmark$ }  &  \multirow{2}{*}{$\checkmark$ } & \multirow{2}{*}{ $\b \geq \frac12$  or  $\inf \phi >0$ } & \multirow{2}{*}{ $\inf \phi >0$} & \multirow{2}{*}{ $\supp g_l = \O$} \\
 & &  & &  &
\end{tabular}
\end{center}
\end{table}

\section{Flocking}\label{s:flocking}

\subsection{The Cucker-Smale Theorem} \label{s:CST}
We start with an extension of the classical Cucker-Smale Theorem that originally appeared in \cite{CS2007a} for the \ref{CS}-model. The result declares how strong the long-range communication must be in order to ensure alignment from any initial condition.   The discrete, kinetic, and hydrodynamic analogues of this result are proved in exact same way, due to essentially the same structure of the characteristic equations taking one of the forms \eqref{e:sys}, see \cite{Sbook} for a detailed account. We adhere to the context of kinetic Vlasov-Alignment model 
\begin{equation}\label{e:VA}
\p_t f + v \cdot \n_x f =  \n_v \cdot ( \st_\rho (v - \ave{u}_\rho) f ), \quad \diam (\supp f_0) < \infty,
\end{equation}
where
\[
\rho(x) = \int_{\R^n} f(x,v) \dv, \qquad u\rho(x) = \int_{\R^n} v f(x,v) \dv.
\]
It incorporates the agent based dynamics as a special case of a weak solution, and does not require any particular closure assumption, for more on this see \cite{Tadmor-notices,Tadmor-pressure}.  The pressureless Euler-alignment system allows the same treatment if written in Lagrangian coordinates, see \thm{t:CShydro} below. The main idea conveyed here is that the result does not require any special properties of the model and can be extended to any general material environmental averaging that has a reproducing kernel $\phi_\rho$.

We consider $\O$ to be an arbitrary environment, although the unbounded ones, such as $\R^n$, is where the result is most meaningful.    If $f$ is a measure-valued solution to \eqref{e:VA} starting from a compactly supported initial condition $f_0$, then at any point of time $f$ is given by the push-forward of $f_0$ along the characteristics (see  \sect{s:mfldet})
\begin{alignat}{2}
\ddt X(t,x,v) & = V(t,x,v), \quad  & X(0,x,v) & = x,\label{e:XK} \\
 \ddt V(t,x,v) & = \st_\rho(X)( [u]_\rho(X) - V), \quad &   V(0, x,v) & = v. \label{e:VK}
\end{alignat}
We abbreviate $\w = (x,v)$ for short. The representation formula \eqref{e:sys2} gives the $V$-equation a more specific form (using the characteristic change of coordinates)
\begin{equation}\label{e:charVrep}
\ddt V(t,\w) = \int_\domain \phi_\rho(X(t, \w),X(t,\w')) (V(t,\w')-V(t,\w)) \df_0(\w') ,
\end{equation}
from which the maximum principle for $V$-characteristics is evident. This fundamental principle holds  even for models without a representation kernel which we prove next.

\begin{lemma}[Maximum Principle]\label{l:maxpr}
Suppose $\cM$ is a material model, and $\supp f_0 \ss \O\times \R^n$ is compact. Then for any $\w \in \supp f_0$ and any $t>0$, we have
\[
V(t,\w) \in \conv \supp \left( \int_\O f_0(x,v) \dx \right).
\]
\end{lemma}
\begin{proof} 
The convex hull in question can be represented as the intersection of hyperspaces:
\[
\conv \supp \left( \int_\O f_0(x,v) \dx \right)  = \bigcap_{\ell \in F \ss \R^n} \{v:  \ell(v) \leq c_\ell\}.
\]
Let us fix an $\ell \in F$. Since the action of $\ell$ is just a linear combination of coordinates we have
\[
\ddt \ell(V(t,\w)) =  \st_\rho(X)( \ave{\ell(u)}_\rho(X) - \ell(V)) = \st_\rho(X) \ave{\ell(u) - \ell(V)}_\rho(X).
\]
By Rademacher's Lemma we can evaluate the above at a point $w\in \supp f_0$ where maximum of $\ell(V(t,\w))$ is achieved. Looking into the field under the average we have
\[
\ell(u)(y,t) - \ell(V) = \frac{ \int_{\R^n} \ell (w - V) f(t,y,w) \dw }{ \int_{\R^n} f(t,y,w) \dw }
\]
Now let $\psi_\d$ be a standard compactly supported mollifier. We have using the transport property
\[
\begin{split}
\int_{\R^n} \ell (w - V) f(t,y,w) \dw & = \lim_{\d \to 0} \int_{\R^n} \ell (w - V) \psi_\d(y-z) f(t,z,w) \dw \dz \\
& =  \lim_{\d \to 0} \int_{\R^n} \ell (V(t,\w') - V(t,\w)) \psi_\d(y-X(\w',t)) f_0(t,\w') \domega' \leq 0.
\end{split}
\]
Thus, $\ell(u) - \ell(V) \leq 0$ point-wise. By the order preserving property of the averages (ev3), we have
\[
\ddt \ell(V(t,\w)) \leq 0.
\]
In other words, $\max_{\supp f_0} \ell(V(t,\w)) \leq c_\ell$ for all times.  This finishes the proof.
\end{proof}

As a consequence, the macroscopic velocity $u$ of $f$ remains bounded by the initial condition:
\begin{equation*}\label{}
\begin{split}
| u(x,t)| = \lim_{\d \to 0} \left|  \frac{ \int_\domain V(t,z,w) \psi_\d(x-z) f_0(z,w) \dz \dw}{  \int_\domain  \psi_\d(x-z) f_0(z,w) \dz \dw} \right| \leq \| V\|_{L^\infty(\supp f_0)} \leq  \max |\supp f_0|.
\end{split}
\end{equation*}

\begin{theorem}[Kinetic Cucker-Smale]\label{t:CS}
Suppose there exists $\phi \in C^\infty$, a positive, non-increasing, radially symmetric kernel with fat tail, $\int_0^\infty \phi(r) \dr = \infty$, such that
\begin{equation}\label{e:Phiphi}
\phi_\rho(x,y)  \geq \phi(x-y), \quad \forall \rho\in \cP.
\end{equation}
Then any measure-valued solution to \eqref{e:VA} starting from a compactly supported initial condition $f_0$ aligns and flocks exponentially fast
\begin{align}
D(t) & = \max_{\w',\w'' \in \supp f_0} | X(t,\w') - X(t,\w'')| < C, \quad \forall t>0 \\
A(t) &=\max_{\w',\w'' \in \supp f_0} |V(t,\w')-V(t,\w'')| \leq Ce^{-\d t}, \label{e:ACS}
\end{align}
where $C,\d>0$ depend on the initial condition and the parameters of the model. Moreover, there exists $u_\infty \in \R^n$ such that
\begin{equation}\label{e:uinfty}
\max_{\w \in \supp f_0} |V(t,\w)-u_\infty | \leq Ce^{-\d t}.
\end{equation}
If the model $\cM$ is conservative then $u_\infty = \bar{u} = \int_\O u \rho \dx$, the total conserved momentum.
\end{theorem}

\begin{proof}
Following characteristics let us fix at any point of time a label $\w_{\pm} \in \supp f_0$ where $V^i$ achieves its maximum and minimum, respectively, $V^i_\pm$. So, by the Rademacher lemma, we have distributionally,
\begin{equation}\label{e:Vpm}
\ddt V^i_\pm  =  \int_{\domain}  \phi_\rho(X(t, \w_\pm),X(t,\w')) (V^i(t,\w')-V^i_\pm) \df_0(\w') .
\end{equation}
In view of \eqref{e:Phiphi},
\[
\ddt V^i_+  \leq  \int_{\domain}  \phi(X(t, \w_+)- X(t,\w')) (V^i(t,\w')-V^i_+) \df_0(\w')  \leq \phi(D)\int_{\domain} (V^i(t,\w')-V^i_+) \df_0(\w').
\]
 And similarly,
\[
\ddt V^i_-  \geq   \phi(D)\int_{\domain} (V^i(t,\w')-V^i_-) \df_0(\w').
\]
Subtracting the two, we obtain for the amplitude $A^i = V^i_+ - V^i_-$,
\[
\ddt A^i \leq  - \phi(D)   A^i.
\]
Taking the Euclidean amplitude $A = \sqrt{(A^1)^2+\dots+(A^n)^2}$, we obtain the system
\[
\ddt D \leq A, \quad \ddt A \leq  - \phi(D)   A.
\]
Following \cite{HL2009} we form the Lyapunov function 
\[
L = A + \int_0^D \phi(r) \dr,
\]
which remains bounded. Hence, in view of the fat-tail condition, $D$ remains bounded, and going back to the $A$-equation we obtain exponential decay on the amplitudes.

To conclude \eqref{e:uinfty} let us notice that as a consequence of \eqref{e:ACS}, we have 
\[
\max_{\w \in \supp f_0} |\dot{V}| \leq Ce^{-\d t}.
\]
So, every characteristic $V(\w,t)$ will converge exponentially fast to a limit $u_\infty(\w)$. In view of \eqref{e:ACS}, $u_\infty$ must be a constant vector.
\end{proof}

The alignment of characteristics stated in \thm{t:CS} implies corresponding behavior of the distribution $f$ itself by transport.  First, we can see that its $v$-marginal $f^v = \int_\O f(t,x,v)\dx$ converges weakly to Dirac, 
\[
f^v \to \d_0(v - u_\infty).
\]
Moreover, since $f$ is a push-forward of $f_0$ along \eqref{e:XK} - \eqref{e:VK}, the $v$-support of $f$ will belong to an exponentially shrinking ball around $u_\infty$. This implies uniform convergence of the macroscopic velocity
\[
u(x,t) - u_\infty = \frac{\int_{|v - u_\infty|\leq Ce^{-\d t}} (v - u_\infty) f(x,v,t) \dv}{ \int_{|v - u_\infty|\leq Ce^{-\d t}} f(x,v,t) \dv},
\]
so,
\begin{equation}\label{e:macroconvfat}
\sup_{x\in \supp \rho} |u(x,t) - u_\infty|  \leq Ce^{-\d t}.
\end{equation}
And it also implies exponential alignment in the energy sense, to be discussed in greater detail in \sect{s:flockingspec}:
\begin{equation}\label{ }
\bar{u} = \int_\O u\rho \dx,
\end{equation}
\begin{equation}\label{e:enpert}
\d\cE:= \frac12 \int_{\domain} |v - \bar{u}|^2 f \dv \dx \leq Ce^{-\d t}.
\end{equation}

Unfortunately the result doesn't seem to provide much insight into behavior of the macroscopic density $\rho$. See, however, \cite{ST2} for a convergence result to a traveling wave in 1D case.

The exact same result can be stated for the hydrodynamic alignment model without pressure, so called pressureless Euler-Alignment system (see \sect{s:mono} for derivation)
\begin{equation}\label{e:EASpresless}
\begin{split}
\rho_t + \n\cdot (u \rho) & = 0, \\
 u_t + u \cdot \n u & = \st_\rho (\ave{u}_\rho - u).
\end{split}
\end{equation}
If passed to Lagrangian coordinates
\begin{equation*}\label{}
\begin{split}
\dot{x}(\a,t) &= v(\a,t): = u(x(\a,t),t), \qquad \a\in  \O,\\
\dot{v}(\a,t) & = \int_\O \phi_\rho(x(t, \a),x(t,\a')) (v(t,\a')-v(t,\a)) \drho_0(\a') .
\end{split}
\end{equation*}
the system is structurally similar to \eqref{e:XK} - \eqref{e:VK}.  So, the proof goes through exactly as before. 
\begin{theorem}[Hydrodynamic Cucker-Smale]\label{t:CShydro}
Under the assumptions of \thm{t:CS},  any classical solution to the pressureless Euler-alignment system \eqref{e:EASpresless} with compactly supported initial $\rho_0$ aligns and flocks exponentially fast
\begin{equation}\label{ }
	\sup_{t\geq 0} (\diam(\supp \rho )) < \infty , \quad	\sup_{x \in\supp \rho} |u(t,x) - u_\infty|\leq C_0 e^{- \d t}.
\end{equation}
\end{theorem}

For the \ref{CS}-model where $\phi_\rho(x,y) = \phi(x-y)$ the statements above are classical. The kinetic and hydrodynamics versions appeared in \cite{CFRT2010} and \cite{TT2014}, respectively.

In the Motsch-Tadmor case, we can apply the same fat-tail condition on the defining kernel $\phi$ due to \eqref{e:MTkerlow}.
However, the limiting velocity $u_\infty$ is not determined by the initial condition and emerges dynamically. 

The theorem does not apply to either the over-mollified model $\cMfmt$ or  the segregation model $\cMseg$ as those are inherently local, which brings us to the next main question -- what conditions guarantee emergent behavior when communication is strictly local?

\subsection{Chain connectivity. The $\frac{1}{t^{1/4}}$ and $\frac{1}{t^{1/2}}$ results } \label{s:chains}

It is obvious that locality  \eqref{e:locker} itself is insufficient for unconditional alignment of the system.  In the open space $\R^n$ one can simply direct two agents away from each other starting at a distance larger than communication range. On $\T^n$ one can launch two agents with misaligned velocities along two parallel geodesics at a distance larger than communication range. So, it is clear that some kind of connectivity is necessary to obtain alignment. In this section we explore how to achieve this for symmetric models and for quantitatively thick flocks.  

\begin{definition}\label{}
We say that the flock $(u,\rho)$ is {\em chain connected at scale $r$} if 
for any two points $x',x'' \in \supp \rho$ there exists a chain 
\[
x'=x_1, x_2,\dots,x_K = x'' 
\]
 such that $x_i \in  \supp \rho$ and  $|x_i - x_j| <r$.
\end{definition}

Our main result shows alignment under connectivity assumption at a sub-local scale $r_0$ and proper thickness rate. Here we use the term thickness to refer to the ball-thickness defined in \eqref{e:thball}.

\begin{theorem}\label{t:conn} Let $\O = \R^n$, and $\cM$ is a symmetric model with kernel satisfying \eqref{e:locker}. If the flock remains chain connected at the scale $r = r_0/8$ for all time and has thickness satisfying $\orho_r(\supp \rho) \geq \frac{c}{t^{1/4}}$, then the flock aligns
\begin{equation}\label{e:logflock}
\sup_{\w,\w' \in\supp f_0} |V(t,\w) - V(t,\w')| \lesssim \frac{1}{\sqrt{\ln t}}.
\end{equation}
On the torus $\O = \T^n$ the result holds under weaker condition $\orho_r(\supp \rho) \geq \frac{c}{t^{1/2}}$.
\end{theorem}

In the case of the torus we can consider a non-vacuous flock $\rho_- = \min_{\T^n} \rho >0$.
Such a flock remains trivially connected at any scale and is uniformly thick $\orho_r(\supp \rho) \gtrsim \rho_-$.  So, one important consequence of the above theorem is a statement in terms of quantitative no-vacuum condition.

\begin{corollary}\label{c:squareroot}
Let $\O = \T^n$, $\cM$ is symmetric and material, and the kernel satisfies \eqref{e:locker}. If $\rho_- \geq \frac{c}{t^{1/2}}$, then the flock aligns \eqref{e:logflock}.
\end{corollary}

Before we get to the proof we first explore how one can reduce the number of links in a chain.

\begin{lemma}\label{l:chainK}
If the flock is chain connected at scale $r$, then  between any pair of points there is a $3r$-chain with the number of links limited to $K \leq \frac{2}{\orho_r(\supp \rho)}$.

If the diameter of the flock is bounded, then $K$ can be chosen independent of thickness but dependent on the diameter, $K \leq C(\diam(\supp \rho))$.
\end{lemma} 
\begin{proof}
Suppose we have a chain $x'=x_1, x_2,\dots,x_K = x''  \in \supp \rho$ with the properties listed in the definition. We now choose a subchain in the following manner. Let $x_{i_1} = x_1$. Then let us pick $i_2-1$ to be the largest index $>i_1$ for which $|x_{i_1} - x_{i_2-1}| < 2r$. So, all subsequent elements will stay at a distance at least $2r$ from $x_{i_1}$. In particular $|x_{i_1} - x_{i_2}| \geq 2r$, and yet since $|x_{i_2-1} - x_{i_2}|<r$, we have $|x_{i_1} - x_{i_2}| < 3r$.   Pick $i_3$ similarly to $i_2$, etc. Eventually $x_K$ will be selected last unconditionally. 

According to construction we have a new chain $y_j = x_{i_j}$, $j=1,\dots,J$, such that $|y_j - y_{j+1}| <3r$  and $|y_j - y_k| \geq 2r$ for any $j \neq k <J$.  Hence, the chain is connected at scale $3r$. At the same time, by disjointness
\[
\orho_r(\supp \rho) (J-1) \leq \sum_{j=1}^{J-1} \rho(B_r(y_j)) = \rho \left( \cup_{j=1}^{J-1}B_r(y_j)\right) \leq 1.
\]
Hence, $J \leq1+ \frac{1}{\orho_r(\supp \rho)}\leq  \frac{2}{\orho_r(\supp \rho)}$. 

Alternatively, if the flock is bounded, and the balls around $y_j$'s are disjoint,  $J$ is limited by volume to $c_n \diam(\supp \rho)^n / r^n$.  This proves the lemma.
\end{proof}

The primary technical use of this lemma will be in construction of chains with thick links. Specifically,  if the flock is $r$-connected then we find it also $3r$-connected by chains of size $K\leq \frac{2}{\orho_r(\supp \rho)}$, and since any ball $B_{4r}(x_i)$ contains the balls $B_r(x_{i-1}) \cup B_r(x_{i+1})$, then
 \begin{equation}\label{e:intersection}
 \rho(B_{4r}(x_i) \cap B_{4r}(x_{i+1})) \geq \orho_r(\supp \rho).
\end{equation}

\begin{proof}[Proof of \thm{t:conn}] Let us assume for now that $\O = \R^n$. 

By symmetry of the model we have the following energy law
\begin{equation}\label{ }
\ddt \cE = - \frac12 \int_\domain \phi_\rho(x, x') |v-v'|^2 f(t,\w')f(t,\w) \domega'\domega.
\end{equation}
Hence, in view of \eqref{e:locker},
\begin{equation}\label{ }
\int_0^\infty \int_{\{|x-x'|<r_0\} \times \R^n}  |v-v'|^2 f(t,\w')f(t,\w) \domega'\domega \dt < \infty.
\end{equation}

Consider the averages of macroscopic momenta over balls of radius $4r$:
\[
\bar{v}(x) = \frac{1}{\rho(B_{4r} (x))} \int_{ \R^n \times B_{4r} (x)  } w f(t,y,w) \dw \dy.
\]
The quadratic deviations from the averages are all subordinated to the dissipation rate:
\begin{equation*}\label{}
\begin{split}
 \int_{\R^n \times B_{4r} (x^*)}  | v - \barv(x^*)|^2 & f(t,x,v) \dv \dx  \\
 =  &\int_{\R^n \times B_{4r} (x^*)}  \left|  \frac{1}{\rho(B_{4r} (x^*))} \int_{\R^n \times B_{4r} (x^*)} (v - w )f(t,y,w) \dw \dy \right|^2  f(t,x,v) \dv \dx  \\
 \leq  &   \frac{1}{\rho(B_{4r} (x^*))}  \int_{\R^n \times B_{4r} (x^*) \times \R^n \times B_{4r} (x^*) }   |v - w|^2 f(t,y,w)  f(t,x,v) \dw \dy \dv \dx \\
\intertext{using that $|x-y| < 8 r = r_0$,}
& \leq \frac{1}{\rho(B_{4r} (x^*))}\int_{\{|x-x'|<r_0\} \times \R^n}  |v-v'|^2 f(t,\w')f(t,\w) \domega'\domega.
\end{split}
\end{equation*}
Thus, in view of \eqref{e:totaldiss}, and the fact that $\orho_r(\supp \rho) \leq \rho(B_{4r} (x^*))$,
\begin{equation}\label{e:totalave}
\int_0^\infty \sup_{x^*\in \O} \orho_r(\supp \rho) \int_{\R^n \times B_{4r} (x^*)}  | v - \barv(x^*)|^2  f(t,x,v) \dv \dx \dt <\infty.
\end{equation}

Let us now estimate the flattening near extremes. Let us fix one coordinate of $v \supp f$, say $v^i$ and denote by $v^i_+ = V^i(t,\w_+) = \max_{\w \in \supp f_0} V^i(t,\w)$, and $x_+ = X(t,\w_+)$. We drop superindex $i$ for shortness of notation. Then
\begin{equation*}\label{}
\begin{split}
\ddt v_+  & =   \int_{\domain} \phi_\rho(x_+, y) (w - v_+) f(t,y,w) \dy \dw \leq  c_0  \int_{B_{4r}(x_+)}  (w - v_+) f(t,y,w) \dy \dw \\
&= c_0 \rho(B_{4r}(x_+)) (\barv(x_+) - v_+) \leq c_0 \orho_r(\supp \rho) (\barv(x_+) - v_+).
\end{split}
\end{equation*}
Similarly,
\[
\ddt v_- \geq c_0 \orho_r(\supp \rho)(\barv(x_-) - v_-).
\]
Consequently,
\begin{equation}\label{e:totalext}
\int_0^\infty \orho_r(\supp \rho) [ (\barv(x_-) - v_-) + (v_+ - \barv(x_+))] \dt <\infty.
\end{equation}

Combining \eqref{e:totalave} and \eqref{e:totalext}, and fixing an $T'>0$ large enough we can ensure that for any $T>0$ there is a time $t\in [T,T+T']$ such that
\begin{equation}\label{e:ext2ave}
(\barv(x_-) - v_-) + (v_+ - \barv(x_+)) + \sup_{x^*\in \O} \int_{\R^n \times B_{4r} (x^*)}  | v - \barv(x^*)|^2  f(t,x,v) \dx \dv < \frac{1}{\orho_r(\supp \rho) t \ln t}.
\end{equation}
In particular, the extreme values are close to the averages around them. Let us now show that all the averages are close to each other, and this will finish the proof.

We have for any $x^* \in \O$,
\[
\int_{\R^n \times B_{4r} (x^*)}  | v - \barv(x^*)|^2  f(t,x,v) \dx \dv  \leq \frac{1}{\orho_r(\supp \rho) t \ln t}.
\]
Denote $\d =  \frac{2}{\Th_r(\rho)\sqrt{t \ln t}}$. Then by the Chebyshev inequality,
\begin{equation}\label{e:cheb}
f( \{ |v - \barv(x^*)| > \d\} \times  B_{4r}(x^*) ) \leq \frac{1}{\d^2}\int_{\R^n \times B_{4r} (x^*)}  | v - \barv(x^*)|^2  f(t,x,v) \dx \dv \leq  \frac14 \orho_r(\supp \rho).
\end{equation}

Let us now consider a $3r$-chain $x_1,\dots,x_K$ with $K< C / \orho_r(\supp \rho)$, which connects two points $x_-$ and $x_+$. According to \eqref{e:intersection}, $\rho(B_{4r}(x_i) \cap B_{4r}(x_{i+1})) \geq \orho_r(\supp \rho)$. Thus, $f( \R^n \times (B_{4r}(x_i) \cap B_{4r}(x_{i+1}))) \geq \orho_r(\supp \rho)$. Yet according to \eqref{e:cheb},
\[
f( (\{ |v - \barv(x_i)| > \d\} \times  B_{4r}(x_i)) \cup  (\{ |v - \barv(x_{i+1})| > \d\} \times  B_{4r}(x_{i+1}))  )  \leq  \frac12 \orho_r(\supp \rho).
\]
Consequently,
\[
B_{\d}(\barv(x_i)) \times  B_{4r}(x_i) \cap  B_{\d}(\barv(x_{i+1})) \times  B_{4r}(x_{i+1})  \neq \emptyset.
\]
Hence, 
\[
|  \barv(x_i) -  \barv(x_{i+1}) | \leq 2 \d.
\]
Summing up over all $i$, we obtain
\begin{equation}\label{e:vpm}
|  \barv(x_+) -  \barv(x_-) | \leq 2 \d K \lesssim \frac{1}{\orho^2_r(\supp \rho) \sqrt{t \ln t}} \sim \frac{1}{ \sqrt{ \ln t}} .
\end{equation}
Combining with \eqref{e:ext2ave} we have
\[
v_+ - v_- \leq \frac{1}{ \sqrt{ \ln t}}.
\]

Since this holds at time $t < T+T'$, it must hold at time $T+T'$ by the maximum principle. But since $t>T$,  $\frac{1}{ \sqrt{ \ln t}} \leq \frac{1}{ \sqrt{ \ln T}} \lesssim \frac{1}{ \sqrt{ \ln (T+T')}}$.  Since $T$ is arbitrary, this finishes the proof in the open space.

On the torus the diameter of the flock is uniformly bounded, and consequently, by \lem{l:chainK}, $K$ remains uniformly bounded. In this case the estimate \eqref{e:vpm} gets improved to the following 
\[
|  \barv(x_+) -  \barv(x_-) | \leq 2 \d K \lesssim \frac{1}{\orho_r(\supp \rho) \sqrt{t \ln t}} \sim \frac{1}{ \sqrt{ \ln t}},
\]
provided $\orho_r(\supp \rho) \gtrsim \frac{1}{t^{1/2}}$. The rest of the proof is the same.
\end{proof}

\begin{remark}
The exact same result holds for solutions of the pressureless Euler-Alignment System \eqref{e:EASpresless}, thanks to the fact that it has a similar form of the energy dissipation
\begin{equation}\label{e:totaldiss}
\int_0^\infty \int_{\O \times \O} \phi_\rho(x,y) |u(x,t) - u(y,t)|^2 \drho(y) \drho(y) \dt <\infty.
\end{equation}
\end{remark}

\subsection{Alignment in the energy sense. Spectral gaps} \label{s:flockingspec}

The alignment of characteristics stated in \thm{t:conn} implies alignment in the energy sense. Recalling that $\bar{u} = \int_\O u \rho \dx$, we have
\[
\d\cE= \frac12 \int_{\domain} |v - \bar{u}|^2 f \dv \dx \leq \int_{\domain\times \domain} |V - V'|^2 f_0 f_0' \domega \domega'  \lesssim \frac{1}{ \sqrt{ \ln t}}.
\]

In this section we explore alignment  in this weaker sense 
\begin{equation}\label{}
\d\cE \to 0,
\end{equation}
by appealing to the most basic energy law of the Vlasov-alignment equation \eqref{e:VA}.  

We will not make any special assumptions on the underlying model  $\cM$ except  that $\cM$ is material just to make sense of the strength function in equation \eqref{e:VA}.  In particular, the momentum $\bar{u}$ may not be conserved.

 In order to write the equation for $\d \cE$, let us note the identity
\[
\cE : =  \frac12 \int_{\domain} |v|^2 f \dv \dx  = \d \cE + \frac12 |\bar{u}|^2.
\]
The momentum satisfies
\begin{equation}\label{ }
\ddt \frac12 |\bar{u}|^2 = (\bar{u}, \ave{u}_\rho - u)_{\k_\rho},
\end{equation}
and the equation for total energy  is given by
\[
\ddt \cE = -   \int_{\domain}\st_\rho |v|^2 f  \dv \dx + (u,\ave{u}_\rho)_{\k_\rho}.
\]
Subtracting the two we obtain
\[
\ddt \d \cE = -   \int_{\domain} \st_\rho |v|^2 f  \dv \dx + (u,\ave{u}_\rho)_{\k_\rho} - (\bar{u}, \ave{u}_\rho - u)_{\k_\rho}.
\]

Let us further notice the identity
\[
\int_{\domain} \st_\rho |v|^2 f  \dv \dx = \int_{\domain}  \st_\rho|v - \bar{u}|^2 f \dv \dx  + 2 (u, \bar{u})_{\k_\rho} - (\bar{u}, \bar{u})_{\k_\rho} .
\]
Collecting the macroscopic terms together we obtain
\[
\ddt \d \cE = -   \int_{\domain}  \st_\rho  |v - \bar{u}|^2 f\dv \dx + (\d u,\ave{\d u}_\rho)_{\k_\rho}, \qquad \d u = u - \bar{u}.
\]

Next, let us decompose the energy on the right hand side into the internal and macroscopic part,
\[
  \int_{\domain}  \st_\rho  |v - \bar{u}|^2 f \dv \dx =   \int_{\domain} \st_\rho  |v - u|^2 f \dv \dx +  (\d u, \d u)_{\k_\rho} .
  \]
We obtain the energy law
\begin{equation}\label{e:devenlaw}
\ddt \d \cE = -   \int_{\domain} \st_\rho |v - u|^2 f  \dv \dx + (\d u, \ave{\d u}_\rho)_{\k_\rho} - (\d u, \d u)_{\k_\rho} .
\end{equation}

Naturally, we will seek to relate the right hand side back to the energy. This comes from two assumptions.  First, we require that the averaging operator has a numerical range separated from $1$, i.e. at any time there exists $\e = \e(t) \in (0,1)$ such that
\begin{equation}\label{e:sgap}
\sup \left\{ (u, \ave{u}_\rho)_{\k_\rho}: \,  u\in L^2(\k_\rho),\, \bar{u} = 0,\, \ \|u\|_{L^2(\k_\rho)}=1 \right\} \leq 1-  \e.
\end{equation}
This in turn implies
\begin{equation}\label{e:lgap}
(\d u, \d u)_{\k_\rho} - (\d u, \ave{\d u}_\rho)_{\k_\rho} \geq \e (\d u, \d u)_{\k_\rho}.
\end{equation}
Second, we require the strength function to have a positive lower bound
\begin{equation}\label{e:suniform}
\inf_{x \in \supp \rho} \st_\rho(x,t) = s(t).
\end{equation}

Plugging these back into \eqref{e:devenlaw} we obtain
\begin{equation}\label{e:enlawgap3}
\begin{split}
\ddt \d \cE  &\leq  -  \int_{\domain} \st_\rho |v - u|^2 f \dv \dx  -  \e (\d u, \d u)_{\k_\rho} \\
& \leq - \e  \left(  \int_{\domain}\st_\rho |v - u|^2 f   \dv \dx + (\d u, \d u)_{\k_\rho}\right)\\
& = - \e \int_{\domain}\st_\rho |v - \bar{u}|^2 f  \dv \dx \leq  - \e s \, \d \cE.
\end{split}
\end{equation}
This implies a general sufficient condition for alignment.
\begin{proposition}\label{p:enalign} 
Let  $\cM$ be a material model on an arbitrary environment $\O$.
The kinetic model \eqref{e:VA} aligns in the energy sense provide the following condition holds 
\begin{equation}\label{ }
\int_0^\infty \e(t)s(t) \dt = \infty.
\end{equation}
\end{proposition}

A few remarks are in order.

\begin{remark}\label{r:gaps}
Let us note that for symmetric models with $\st_\rho \equiv 1$, the space of vanishing momentum $L^2_0(\rho)$ is invariant under $\ave{\cdot}_\rho$, and the numerical range determines the range of the spectrum. So, condition \eqref{e:sgap}  is equivalent to a spectral gap between the trivial eigenvalue $1$ and the rest of the spectrum to the left 
\begin{equation}\label{e:spe}
\spec\{\ave{\cdot}_\rho; L^2_0(\rho)\} \ss (-\infty, 1-\e],
\end{equation}
where 
\[
L^2_0(\rho) = \left\{ u\in  L^2(\rho): \int_\O u \drho = 0 \right\}.
\]
For this reason, although in general \eqref{e:sgap} is not a spectral property, we still  refer to it as a {\em spectral gap}. 

In general, however, conservative models leave the null-space  
\[
L^2_0(\k_\rho) = \left\{ u\in  L^2(\k_\rho): \int_\O u \dk_\rho = 0 \right\}
\]
invariant. In this case it is possible to relate $\e$ to the actual spectral gap of $\ave{\cdot}_\rho$ on $L^2_0(\k_\rho)$ if $\st_\rho$ is bounded from below. Details are provided in Appendix~\ref{a:gap}.
\end{remark}

\begin{remark} \label{r:spgap} \prop{p:enalign}  can be viewed as a generalization of Tadmor's \cite{Tadmor-notices} to the non-symmetric case.  The argument there is slightly different in the interpretation of the spectral gap condition \eqref{e:lgap}. As opposed to \eqref{e:lgap} where all the inner products are related to the common $\k_\rho$-weight, one can make a more direct relation to the physical macroscopic energy, i.e. the $\rho$-weighted product:
\[
(\d u, \d u)_{\k_\rho} - (\d u, \ave{\d u}_\rho)_{\k_\rho} \geq \l (\d u, \d u)_{\rho}.
\]
The corresponding alignment statement in terms of $\l$ reads
\begin{equation}\label{e:propTad}
\int_0^\infty \min\{s(t),\l(t)\} \dt = \infty.
\end{equation}

Such $\l$ can be expressed in variational form as the second (approximate) eigenvalue of the alignment operator 
\begin{equation}\label{}
\cL_\rho u = \st_\rho(u - \ave{u}_\rho).
\end{equation}
We have
\begin{equation}\label{e:varprob}
\l = \inf_{ u \in L^2_0(\rho)} \frac{ (u, \cL_\rho u)_\rho}{ (u,u)_\rho}.
\end{equation}
The advantage of this approach consists in the fact that for symmetric models represented by a kernel  the formulation \eqref{e:varprob} takes a more explicit form:\begin{equation}\label{e:varprob2}
\l = \inf_{ u \in L^2_0(\rho), \|u\|_2 = 1} \int_{\O \times \O} |u(x) - u(y)|^2 \phi_\rho(x,y) \drho(y)\drho(x).
\end{equation}
Theorem 2 of \cite{Tadmor-notices} gives a kinematic estimate in terms of lower and upper bounds on the density, in case when $\O = \T^n$. Namely,  
\begin{equation}\label{e:Tad+-}
\l \gtrsim \frac{\rho_-^2}{\rho_+}.
\end{equation}
The result is proved under condition \eqref{e:philow} below, however it can be recast for physically local kernels \eqref{e:locker} as well. Let us reproduce the argument as it will be used later in Example~\ref{ex:MTgap}. 
\begin{proof}[Proof of \eqref{e:Tad+-}]
We obtain
\[
(u, \cL_\rho u)_\rho  \geq c_0   \rho_-^2 \int_{|x-y| < r_0 } |u(x) - u(y)|^2 \dx \dy.
\]
As shown in \cite[Lemma 2.1]{LS-entropy} this can be further estimated from below by
\[
\geq  c_0 c_1 \rho_-^2 \frac{c_0}{(2\pi)^n} \int_{\T^n} |u(x) - \Ave(u) |^2 \dx,
\]
where $c_1 = c_1(r_0)$, and $\Ave(u) = \frac{1}{(2\pi)^n} \int_{\O} u(x) \dx$. Recalling that $u$ has momentum zero, we finish with
\[
 \gtrsim  \frac{\rho_-^2}{\rho_+}   \int_{\T^n} |u(x) - \Ave(u) |^2 \drho(x)  \geq   \frac{\rho_-^2}{\rho_+}  (u,u)_\rho.
\]
\end{proof}

Estimate \eqref{e:Tad+-} shows that under global control on $\rho_+$ one obtains alignment under the root-assumption $\rho_- \gtrsim 1/ \sqrt{t}$, the same result as proved in \cor{c:squareroot} under no assumption on $\rho_+$.  The difference between the two approaches is fundamental -- dynamic vs kinematic. It appears that the dynamic approach is not sensitive to the density growth and gives a better result for symmetric models on the torus. However, as we will see later in \sect{s:lowenergy} the kinematic approach, although in somewhat different form than presented here, gives estimates  independent of $\rho_+$ as well, and in some cases can even  beat the root-result, see \prop{p:gaps}. Any bound on the spectral gap that does not rely on $\rho_+$ will prove to be a crucial in the study of  relaxation for kinetic Fokker-Planck models in \sect{s:hypo}.
\end{remark}

Let us present two applications of \prop{p:enalign} that are distinctly different from the root-result.  In both cases we assume $\O = \T^n$. 

\begin{example}[\ref{CS}-model]\label{ex:CSgap}
Let us assume that $\phi$ is a mollification kernel, $\phi \geq 0$, $\int_\O \phi \dx = 1$, local or not. Then its non-zero Fourier modes will necessarily be smaller than unit:
\begin{equation}\label{e:philow} 
c_0 = \sup_{k \in \Z^n \backslash \{0\} } | \hat{\phi}(k)| <1.
\end{equation}

Let us compute the spectral gap as defined by \eqref{e:sgap}. Using that $\int u \rho \dx = 0$ by the Plancherel identity,
\[
 (u, \ave{u}_\rho)_{\k_\rho} = \int_\O u \rho (u\rho)_\phi \dx  = \sum_{k \in \Z^n \backslash \{0\} } | \widehat{u \rho}(k) |^2 \re( \hat{\phi}(k)) \leq  c_0 \int_\O |u \rho|^2 \dx.
 \] 
We now relate it back to the $L^2(\k_\rho)$-norm:
\[
(u, \ave{u}_\rho)_{\k_\rho} \leq c_0 \int_\O |u|^2 \rho \rho_\phi  \frac{\rho}{\rho_\phi} \dx \leq c_0 \left\|\frac{\rho}{\rho_\phi} \right\|_{\infty} \|u\|_{L^2(\k_\rho)}^2 .
\]
Suppose now that $ \left\|\frac{\rho}{\rho_\phi} \right\|_{\infty} < \frac{1}{c_0}$. We define 
\begin{equation}\label{e:flatconc}
\e = 1 - c_0 \left\|\frac{\rho}{\rho_\phi} \right\|_{\infty} .
\end{equation}
 Naturally, $\e<1-c_0$ since at the point of maximum of $\rho$ we have $\rho \geq \rho_\phi$, and so the $L^\infty$-norm is at least $1$.  Also, note that if $\rho$ is convex in a ball $B_r(x)$, where $r$ is the range of the communication kernel, then $\rho(x) \leq \rho_\phi(x)$, and \eqref{e:flatconc} holds if restricted to that ball. So, the spectral gap \eqref{e:flatconc} essentially quantifies flatness of the density $\rho$ in those regions where it is not convex. 

Also, note that for $\e = 1-c_0$, the only flock that satisfies \eqref{e:flatconc} is the uniformly distributed one. So, the smaller the  $\e$ the more room there is for variations in distribution. However,  \eqref{e:flatconc} still ensures sufficient spread of the support across the domain (for otherwise the geodesic counterexample applies).

Now,  a lower bound on $\st_\rho$ can be interpreted as a measure of thickness, see \sect{s:rth},
\begin{equation}\label{e:CSthick}
s(t) = \Th(\rho,\supp \rho).
\end{equation}

Collecting the computations above and applying \prop{p:enalign} we obtain the following alignment result.
\begin{corollary}\label{c:CSgap}
For the Cucker-Smale model \ref{CS} a sufficient condition for alignment in the energy sense is the flatness \eqref{e:flatconc} and thickness \eqref{e:CSthick} to satisfy $\int_0^\infty \e(t) s(t) \dt = \infty$.
\end{corollary}

\end{example}

\begin{example}[\ref{MT}-model] \label{ex:MTgap} For the Motsch-Tadmor non-symmetric model  \ref{MT}  computation of the gap is more technical and require heavier assumptions on the density. 

Let us assume that the defining kernel $\phi$ is  local, \eqref{e:locker}, and $\int \phi \dx =1$.
We have
\begin{align}
(u, u)_\rho - (u, \ave{u}_\rho)_\rho & = \int_{\O \times \O} u(x) \cdot (u(x) - u(y)) \rho(x) \rho(y) \frac{\phi(x-y)}{\rho_\phi(x)} \dy \dx, \notag\\
\intertext{symmetrizing in $x$ and $y$}
&= \frac12 \int_{\O \times \O} |u(x) - u(y)|^2 \rho(x) \rho(y) \frac{\phi(x-y)}{\rho_\phi(x)} \dy \dx \label{e:diss2nd} \\
&+ \frac12 \int_{\O \times \O} u(y) \cdot (u(x) - u(y)) \rho(x) \rho(y) \left(\frac{1}{\rho_\phi(x)} - \frac{1}{\rho_\phi(y)} \right) \phi(x-y) \dy \dx.  \notag
\end{align}
Now, using that  $\rho_\phi(x) \leq \|\phi\|_\infty$ we bound the first term from below by a multiple of $(u, \cL_\rho u)_\rho$, which by \eqref{e:Tad+-} is bounded from below by $ c \frac{\rho_-^2}{\rho_+}  \|u\|_2^2$.  As to the second term, note that the component with the dot-product $u(y) \cdot u(x)$ vanishes by symmetry, and hence we are left with 
\begin{equation*}\label{}
\begin{split}
- \frac12 \int_{\O \times \O} |u(y)|^2  \rho(y) \rho(x) \left(\frac{1}{\rho_\phi(x)} - \frac{1}{\rho_\phi(y)} \right) \phi(x-y) \dy \dx  & = - \frac12 \int_{\O \times \O} |u(y)|^2  \rho(y) \left(\frac{\rho}{\rho_\phi}\right)_\phi(y) \dy + \frac12 \|u \|_2^2\\
&= \frac12 \int_{\O \times \O} |u(y)|^2  \rho(y) \left(1- \frac{\rho}{\rho_\phi}\right)_\phi(y) \dy .
\end{split}
\end{equation*}
We now impose the following condition on the smallness of variation
\begin{equation}\label{e:smallvar}
\rho_+ - \rho_- \leq c \frac{\rho_-^3}{\rho_+}.
\end{equation}
Then
\[
 \left(1- \frac{\rho}{\rho_\phi}\right)_\phi(y)\leq \frac{\rho_+ - \rho_-}{\rho_-} \leq  c\frac{\rho_-^2}{\rho_+}.
 \]
 Consequently, this term becomes less than half of the main dissipation term \eqref{e:diss2nd},
\[
(u, u)_\rho - (u, \ave{u}_\rho)_\rho \geq  \frac{c}{2} \frac{\rho_-^2}{\rho_+} (u, u)_\rho.
\]
So, similar to the symmetric case under the flatness assumption \eqref{e:smallvar}, the size of the spectral gap is still estimated at $\l = \e \gtrsim \frac{\rho_-^2}{\rho_+} $.

\begin{corollary}\label{c:MTspgap} There exists a $c>0$ which depends only on the parameters of the model such that 
any solution to the kinetic equation \eqref{e:VA} on $\T^n$ governed by the Motsch-Tadmor averaging aligns in the energy sense, provided
\[
\rho_+ - \rho_- \leq c \frac{\rho_-^3}{\rho_+}, \qquad 
\int_0^\infty \frac{\rho_-^2}{\rho_+} \ds = \infty.
\]
\end{corollary}

\end{example}

 \subsection{Spectral gap of a ball-positive model. Low energy method}\label{s:lowenergy}
 
As we have seen  the spectral gap condition \eqref{e:sgap} plays a central role in alignment dynamics and will be important in the study of relaxation, see \sect{s:hypo}. It will be essential to find bounds on $\e$ that are independent of  $\rho_+$, since the growth of the density cannot be controlled away from equilibrium. In this section we present the so-called  {\em low energy method} which allows one to obtain such bounds for ball-positive models on $\T^n$.

To describe the method let us first discuss energetics of ball-positive models.  Since $\|\ave{\cdot}_\rho \|_{L^2(\k_\rho)}\leq 1$, we obtain  a streak of three inequalities,
\begin{equation}\label{e:energies}
(u,u)_{ \k_\rho} \geq (u,\ave{u}_\rho)_{ \k_\rho} \geq (\ave{u}_\rho,\ave{u}_\rho)_{ \k_\rho}.
\end{equation}
This defines the hierarchy of three $\k$-energies  (not to be confused with the physical $\rho$-energies)
\begin{equation}\label{e:EEE}
\cE_0 = (u,u)_{ \k_\rho} , \quad \cE_1 =(u,\ave{u}_\rho)_{ \k_\rho}, \quad \cE_2 = (\ave{u}_\rho,\ave{u}_\rho)_{ \k_\rho}.
\end{equation}

As seen from \eqref{e:EASlessenlaw} the difference between the first two energies 
$\cA_0 = \cE_0 - \cE_1$ controls the rate of alignment in collective systems. The next difference $\cA_1 = \cE_1 - \cE_2$ is also non-negative by the very definition of ball-postivity, and in fact by the \CS\ one has the relation
\[
\cA_0 \geq \cA_1.
\]
So, it is clear that the strength of ball-positivity measured by $\cA_1$ bears direct relevance to alignment.

To adopt it for spectral gap calculations, we note that the spectral gap condition \eqref{e:sgap} can be expressed directly in terms of top tier energies
\begin{equation}\label{e:AE0}
 \cA_0 \geq \e \cE_0,\qquad \forall u\in L^2(\k_\rho), \ \bar{u} = \int_\O u \rho \dx= 0.
\end{equation}
The lower energy method seeks to achieve \eqref{e:AE0} through  comparison  between the two terms down in the hierarchy (low energies)
\begin{equation}\label{e:AE1}
\cA_1 \geq \e \cE_1, \qquad \forall u\in L^2(\k_\rho), \ \bar{u} = 0.
\end{equation}
Indeed, let us observe that \eqref{e:AE1}  is equivalent to 
\begin{equation}\label{e:lowen5}
(1-\e) (u, \ave{u}_\rho)_{\k_\rho} \geq (\ave{u}_\rho,\ave{u}_\rho)_{ \k_\rho}, \qquad \forall u\in L^2(\k_\rho), \ \bar{u} = 0,
\end{equation}
and hence 
\[
\| \ave{ u }_\rho\|_{L^2(\k_\rho)} \leq (1 - \e)\| u \|_{L^2(\k_\rho)},\qquad \forall u\in L^2(\k_\rho), \ \bar{u} = 0,
\]
which implies \eqref{e:sgap}$\sim$\eqref{e:AE0}. 

One can see from \eqref{e:lowen5} that the method is necessarily restricted to the class of ball-positive models.  
It turns out that estimating the low energy gap \eqref{e:AE1} sometimes gives substantial improvements over the direct approach \eqref{e:AE0} in the sense of giving a bound independent of $\rho_+$. Let us present several examples from our list.

Throughout we assume that the kernel in question is local \eqref{e:locker} - \eqref{e:lockerdef}, and the environment is periodic $\O = \T^n$. The summary of estimates to be obtained below is given in the following proposition. 

\def \rseg {r_{\mathrm{seg}}}

\begin{proposition}\label{p:gaps}
For each of the ball-positive models \ref{CS},  \ref{Mf},  \ref{Mseg}  we have the following bounds from below on the spectral gap up to a constant multiple:
\medskip
\begin{center}
\begin{tabular}{  c |  c | c |c } 
 MODEL &  \ref{CS} &  \ref{Mf} &  \ref{Mseg}  \\ 
 \hline
\multirow{2}{*}{spectral gap} &\multirow{2}{*}{ $ \orho^3_{r_0/2}(\O) $ }&  \multirow{2}{*}{$ \orho_{r_0/2}(\O)$ }& \multirow{2}{*}{$ \orho^{2L}_{\rseg}(\O)$, see \eqref{e:commball}}\\
 & & & 
 \end{tabular}
\end{center}
\medskip
In particular, if the kernel is all-to-all, $\inf \phi_\rho >0$, then the spectral gap is automatically uniform.
\end{proposition}

 \begin{proof}[Proof of \prop{p:gaps} for the \ref{Mf}-model]

For the \ref{Mf}-model  the following formula was proved in \cite{S-hypo}:
\begin{equation}\label{e:AMf}
\begin{split}
\cA_1 & = \frac12 \int_{\O \times \O} \rho_{\phi \phi}(x,y) | \uFavre(x) - \uFavre(y)|^2 \dx \dy, \\
\rho_{\phi \phi}(x,y) & = \int_{\O} \phi(x - \xi )\phi(y - \xi) \rho(\xi) \dxi,
\end{split}
\end{equation}
where $\uFavre$ is the Favre-filtration given by \ref{MT}.  The proof goes as follows
\begin{equation*}\label{}
\begin{split}
\cA_1 & = \int_{\O} (\rho_\phi |\uFavre|^2 - \rho |(\uFavre)_\phi |^2 ) \dx = \int_{\O} (\rho_\phi \uFavre \cdot \uFavre - \rho (\uFavre)_\phi \cdot (\uFavre)_\phi  ) \dx\\
& = \int_{\O} (\rho_\phi \uFavre - (\rho (\uFavre)_\phi)_\phi  )  \cdot \uFavre \dx = \int_{\O \times \O} \phi(x - \xi) \rho(\xi)( \uFavre(x) - (\uFavre)_\phi(\xi)  )  \cdot \uFavre(x) \dxi \dx\\
& = \int_{\O \times \O \times \O} \phi(x - \xi)\phi(y - \xi)  \rho(\xi)( \uFavre(x) - \uFavre(y)  )  \cdot \uFavre(x) \dxi \dx \dy\\
& = \int_{\O \times \O} \rho_{\phi \phi}(x,y)( \uFavre(x) - \uFavre(y)  )  \cdot \uFavre(x) \dx \dy = \frac12 \int_{\O  \times \O } \rho_{\phi \phi}(x,y) | \uFavre(x) - \uFavre(y)|^2 \dx \dy,
\end{split}
\end{equation*}
where in the last step we performed symmetrization in $x,y$.

We now estimate $\rho_{\phi \phi}$ from below:  let $|x-y|< r_0/2$, then
\begin{equation*}\label{}
\begin{split}
\rho_{\phi \phi}(x,y) & = \int_{\O} \phi(x -y + \xi )\phi(\xi) \rho(y- \xi) \dxi \geq \int_{|\xi|<r_0/2} \phi(x -y + \xi )\phi(\xi) \rho(y- \xi) \dxi \\
& \geq  c_0^2 \int_{|\xi|<r_0/2} \rho(y- \xi) \dxi  \geq c_0^2 \orho_{r_0/2}(\O).
\end{split}
\end{equation*}
Thus,
\begin{equation}\label{e:rhophiphi}
\rho_{\phi \phi}(x,y) \gtrsim  \orho_{r_0/2}(\O) \one_{|x-y|< r_0/2}.
\end{equation}

With this at hand we have
\[
\cA_1 \gtrsim  \orho_{r_0/2}(\O) \int_{|x-y|<r_0/2} | \uFavre(x) - \uFavre(y)|^2 \dx \dy
\]
by \cite[Lemma 2.1]{LS-entropy},
\[
\gtrsim  \orho_{r_0/2}(\O)  \int_{\O}  |\uFavre -   \Ave(\uFavre)|^2 \dx\gtrsim  \orho_{r_0/2}(\O) \int_{\O} \rho_\phi |\uFavre -   \Ave(\uFavre)|^2 \dx.
\]
 Using the vanishing momentum , $\Ave((u\rho)_\phi) = 0$, we continue
\begin{equation*}\label{}
=  \orho_{r_0/2}(\O) \left( \int_{\O} \rho_\phi |\uFavre |^2 \dx - \underbrace{ 2  \Ave(\uFavre)\cdot \Ave((u\rho)_\phi) }_{=0} + (2\pi)^n |\Ave(\uFavre)|^2 \right).
\end{equation*}
Noting that $ \int_{\O} \rho_\phi |\uFavre |^2 \dx = \cE_1$, we conclude
\[
\geq  \orho_{r_0/2}(\O)  \cE_1.
\]
So, we have a bound
\begin{equation}\label{e:Mfgap}
\e \geq c  \orho_{r_0/2}(\O),
\end{equation}
where $c>0$ depends only on the parameters of the model.

\end{proof}

We obtain the following improvement over the general root-result of \cor{c:squareroot}.
\begin{corollary}\label{c:Mfalign}
Under the \ref{Mf}-averaging protocol a solution to the Vlasov-Alignment equation \eqref{e:VA} aligns if $\rho_- \gtrsim \frac{1}{t}$.
\end{corollary}

Let us note that under this weak assumption on the density the only known alignment result  was established in \cite{ST-topo} for singular topological models.  And in 1D it was proved to hold automatically for  any non-vacuous solutions to the Euler-Alignment system \eqref{e:EASvelless} based on the metric or topological Cucker-Smale averaging protocol. For the system based on the \ref{Mf}-model such a bound is unknown a priori.

  \begin{proof}[Proof of \prop{p:gaps} for the \ref{CS}-model]

By the assumptions of ball-positivity and locality,  $\phi = \psi \ast \psi$, where $\psi$ is a non-negative smooth kernel satisfying
\begin{equation}\label{e:locker3}
\psi(x) \geq c_0 \one_{|x| \leq r_0}.
\end{equation}
 Let us apply the low energy method. We aim to prove the following bound:
\begin{equation}\label{e:CSbpgap}
\e \gtrsim \orho^3_{r_0/2}(\O).
\end{equation}

To prove \eqref{e:CSbpgap} we will quantify the alignment term $\cA_1$ in a way similar to the previous example.  To achieve this we notice that for the Bochner-positive $\phi$ the \ref{CS}-averaging is nothing but a nested application of two distinct Favre filtrations. Indeed, let us denote
\begin{equation}\label{e:vvarr}
v = \frac{(u \rho)_\psi}{\rho_\psi}, \quad \varrho = \rho_\psi.
\end{equation}
Then denoting $\vF= \frac{(v \varrho)_\psi}{\varrho_\psi}$,  we obtain
\begin{equation}\label{e:CSorder}
\ave{u}_\rho = \frac{(u \rho)_\phi}{\rho_\phi} = \frac{((u \rho)_\psi)_\psi}{\varrho_\psi} =  \frac{\left(\frac{(u \rho)_\psi}{\rho_\psi} \rho_\psi \right)_\psi}{\varrho_\psi}  =\vF.
\end{equation}
Observe that
\[
\cA_1 = \int_\O (u\rho)_\psi^2 \dx - \int_\O |\ave{u}_\rho|^2 \rho \rho_\phi \dx = \int_\O |v|^2 \varrho \rho_\psi \dx - \int_\O |\vF|^2 \rho \varrho_\psi \dx.
\]
Let us examine the second term now:  $|\vF|^2 \rho \varrho_\psi$. We use the fact that the Favre-filtration with respect to $\psi, \varrho$ is a symmetric  operation relative to the measure $\varrho \varrho_\psi$. So, we can write
\[
\int_\O |\vF|^2 \rho \varrho_\psi \dx =\int_\O \vF \cdot \left( \vF  \frac{\rho}{\varrho} \right) \varrho \varrho_\psi \dx = \int_\O v \cdot \left( \vF  \frac{\rho}{\varrho}\right)_{\mathrm{F}} \varrho \varrho_\psi \dx = \int_\O v \cdot (\vF \rho)_\psi \varrho \dx.
\]
Now let us factor out the common $v \varrho$ term:
\begin{align*}
\cA_1 &= \int_\O \varrho v \cdot ( \rho_\psi v - (\vF \rho)_\psi) \dx = \int_{\O^2} \varrho(x) \rho(y) v(x) \cdot ( v(x) - \vF(y)) \psi(x-y)  \dy \dx \\
\intertext{expanding further in $v_F(y)$, we obtain}
&= \int_{\O^2} \frac{\varrho(x) \rho(y)}{\varrho_\psi(y)} v(x) \cdot ( v(x) \varrho_\psi(y) - (v \varrho)_\psi (y)) \psi(x-y)  \dy \dx  \\
& = \int_{\O^3} \frac{\varrho(x) \rho(y) \varrho(z)}{\varrho_\psi(y)} v(x) \cdot ( v(x)  - v(z) )\psi (z-y) \psi(x-y) \dz \dy \dx \\
\intertext{symmetrizing in $x,z$,}
& = \frac12 \int_{\O^3} \frac{\varrho(x) \rho(y) \varrho(z)}{\varrho_\psi(y)} | v(x)  - v(z) |^2\psi (z-y) \psi(x-y) \dz \dy \dx .
\end{align*} 
Notice that the integral in $y$ represents the application of the variable doubling convolution to $\rho / \rho_\phi$ as in \eqref{e:AMf} using kernel $\psi$. So we obtain the following exact formula for $\cA_1$:
\begin{equation}\label{ }
\cA_1 = \frac12 \int_{\O^2}\varrho(x) \varrho(z) \left(\frac{\rho}{\rho_\phi}\right)_{\psi \psi}(x,z) | v(x)  - v(z) |^2 \dz \dx.
\end{equation}

 Since $\rho_\phi \leq c_1$ pointwise, we have, using \eqref{e:rhophiphi},
 \[
\left(\frac{\rho}{\rho_\phi}\right)_{\psi \psi} \geq c_1 \rho_{\psi \psi} \gtrsim  \orho_{r_0/2}(\O)\one_{|x-z|< r_0/2 }.
 \]
 So,
\begin{align*}
\cA_1&  \gtrsim  \orho_{r_0/2}(\O) \int_{|x-z|< r_0/2}\varrho(x) \varrho(z) | v(x)  - v(z) |^2 \dz \dx \gtrsim   \orho^3_{r_0/2}(\O) \int_{|x-z|< r_0/2} | v(x)  - v(z) |^2 \dz \dx\\
\intertext{proceeding as for the \ref{Mf}-model, }
&\geq \orho^3_{r_0/2}(\O) \int_\O |v(x) - \Ave(v)|^2 \dx \gtrsim \orho^3_{r_0/2}(\O) \int_\O \varrho |v(x) - \Ave(v)|^2 \dx \\
& \geq \orho^3_{r_0/2}(\O) \int_\O \varrho |v(x)|^2 \dx =  \orho^3_{r_0/2}(\O) \int_{\O} \frac{(u \rho)_\psi^2}{\rho_\psi} \dx \gtrsim \orho^3_{r_0/2}(\O) \int_{\O} (u \rho)_\psi^2 \dx =\orho^3_{r_0/2}(\O)  \cE_1.
\end{align*}

We arrive at \eqref{e:CSbpgap}.

\end{proof}

 \begin{proof}[Proof of \prop{p:gaps} for the \ref{Mseg}-model]
 Since this model is symmetric and non-negative definite it is automatically ball-positive by \lem{l:symmbp}. So, it is natural to apply the low-energy approach.  We start with the analogue \eqref{e:AMf} which in this case reads
\begin{equation}\label{ }
\cA_1 =  \frac12 \sum_{l, l'}  \rho(g_l g_{l'}) \left|  \frac{ \rho(u g_l) }{\rho(g_l)} - \frac{ \rho(u g_{l'}) }{\rho(g_{l'})}  \right|^2.
\end{equation}
Indeed, 
\begin{align*}
\cA_1 & =  \sum_l \frac{ ( \rho(u g_l) )^2}{\rho(g_l)} - \int_\O \left( \sum_l g_l  \frac{ \rho(u g_l) }{\rho(g_l)} \right)^2 \rho \dx \\
& = \sum_l \frac{ ( \rho(u g_l) )^2}{\rho(g_l)}  - \sum_{l,l'} \rho(g_l g_{l'}) \frac{ \rho(u g_l) }{\rho(g_l)} \frac{ \rho(u g_{l'}) }{\rho(g_{l'})}  \\
& = \sum_l \rho(u g_l) \left( \frac{ \rho(u g_l) }{\rho(g_l)} - \sum_{l'} \frac{  \rho(g_l g_{l'}) }{ \rho(g_l) }  \frac{ \rho(u g_{l'}) }{\rho(g_{l'})}  \right)\\
\intertext{noting that the coefficients $ \frac{  \rho(g_l g_{l'}) }{ \rho(g_l) }$ add up to $1$ over $l'$,}
& =  \sum_l \rho(u g_l) \sum_{l'} \frac{  \rho(g_l g_{l'}) }{ \rho(g_l) }  \left( \frac{ \rho(u g_l) }{\rho(g_l)} - \frac{ \rho(u g_{l'}) }{\rho(g_{l'})}  \right)\\
& = \sum_{l, l'}  \rho(g_l g_{l'})  \frac{ \rho(u g_l) }{\rho(g_l)} \left( \frac{ \rho(u g_l) }{\rho(g_l)} - \frac{ \rho(u g_{l'}) }{\rho(g_{l'})}  \right)\\
\intertext{ symmetrizing over $l,l'$,}
& = \frac12 \sum_{l, l'}  \rho(g_l g_{l'}) \left|  \frac{ \rho(u g_l) }{\rho(g_l)} - \frac{ \rho(u g_{l'}) }{\rho(g_{l'})}  \right|^2.
\end{align*}

The formula indicates that the energy keeps dissipating as long as discrepancies remain between local averages in adjacent and connected neighborhoods, $\rho(g_l g_{l'})>0$. To extract a working criterion out of it, we rewrite $\cA_1$ is a different way:
\[
\cA_1 =  \sum_l \frac{ ( \rho(u g_l) )^2}{\rho(g_l)}  -  \sum_{l,l'} G_{ll'} \frac{  \rho(u g_l) }{\sqrt{\rho(g_l)}}  \frac{\rho(u g_{l'})}{\sqrt{\rho(g_{l'})}},
\]
where 
\[
G_{ll'}=   \frac{\rho(g_l g_{l'})}{\sqrt{\rho(g_l) \rho(g_{l'})}}.
\]
Considering those as entries of the symmetric matrix $G = \{ G_{ll'} \}_{l,l'=1}^L$ and denoting the vector 
\[
X = \left(  \frac{  \rho(u g_1) }{\sqrt{\rho(g_1)}}, \ldots,  \frac{  \rho(u g_L) }{\sqrt{\rho(g_L)}} \right),
\]
 the above expression can be written as
\[
\cA_1 = |X|^2 - \lan G X, X \ran.
\]
The vanishing momentum condition means that the vector $X$ belongs to the hyperplane  orthogonal to the vector of roots $Y = (\sqrt{\rho(g_1)}, \ldots, \sqrt{\rho(g_L)})$, denoted $Y^\perp$. Such plane remains invariant under the action of $G$, while $GY= Y$. So, the low-energy bound \eqref{e:AE1} becomes equivalent to the spectral gap condition on $G$: 
\begin{equation}\label{e:gapG}
\spec\{ G; Y^\perp \} \leq 1-\e.
\end{equation}

It is not easy, however, to compute the spectrum of $G$ exactly.  A more practical approach to \eqref{e:gapG} would be to find a condition on the entries of $G$ that implies a bound like \eqref{e:gapG}.  To this end, let us assume that non-zero entries are uniformly bounded from below, i.e. the neighborhoods  have  `populated intersections':
\begin{equation}\label{e:segconn}
 \rho(g_{l} g_{l'} ) \geq \d \sqrt{\rho(g_l) \rho(g_{l'})}, \qquad \forall l,l': \, \supp g_l \cap \supp g_{l'} \neq \emptyset,
\end{equation}
for some $\d>0$. 

Under this condition let us consider the eigenvalue problem
\[
(1-\e) X = GX, \quad X \cdot Y = 0.
\]
Renormalizing $X = (X_1,\ldots,X_L)$ via $x_l = \frac{ X_l}{ \sqrt{\rho(g_l)}}$ we obtain the system
\begin{equation}\label{e:eigenx}
(1-\e) x_{l} = \sum_{l': \, \supp g_{l'} \cap \supp g_{l} \neq \emptyset} \frac{\rho(g_{l'} g_{l} )}{\rho(g_{l} )} x_{l'} .
\end{equation}
Note that the sum on the right represents a convex combination of coordinates.

Denote $x^+=x_{l^+}$ the positive maximal and $x^-=x_{l^-}$ the negative minimal values. Since $X\in Y^\perp$, those must be strictly signed.  Since $g$'s form a partition of unity, there is a sequence of indexes $l^+ = l_0, l_1,\dots, l_p = l^-$ with $p\leq L$ such that $\supp g_{l_i} \cap \supp g_{l_{i+1}} \neq \emptyset$.
Let us start with \eqref{e:eigenx} at $l = l_0$. Then $l_1$ is one of the neighbors. We can assume without loss of generality that $x_{l_1} < x^+$ for otherwise, we relabel and start with the first index $l_1$ having this property. 

We leave the $l_1$-term unchanged, and estimate rest of $x$'s by $x^+$ to obtain
\[
(1-\e) x^+ \leq \left(1-\frac{\rho(g_{l_0} g_{l_1} )}{\rho(g_{l_0} )}\right) x^+ + \frac{\rho(g_{l_0} g_{l_1} )}{\rho(g_{l_0} )} x_{l_1}.
\]
Solving for $x_{l_1}$    we obtain
\[
x_{l_1} \geq \left(1 - \frac{\e}{\frac{\rho(g_{l_0} g_{l_1} )}{\rho(g_{l_0} )}}\right) x^+.
 \]
Since $x_{l_1} < x^+$ it implies in particular that $\e >0$. It also follows from \eqref{e:segconn} that  $\frac{\rho(g_{l_0} g_{l_1} )}{\rho(g_{l_0} )} \geq \d^2$ and hence,
\[
x_{l_1} \geq \left(1 - \frac{\e}{\d^2}\right) x^+.
 \]

By the same computation centered this time at $x_{l_1}$ and with $\e$ reset to $\frac{\e}{\d^2}$ we obtain
\[
x_{l_2} \geq \left(1 - \frac{\e}{\d^4}\right) x^+.
\]
Continuing the process to the last term we obtain
\[
x^- \geq \left(1 - \frac{\e}{\d^{2p}}\right) u^+ \geq  \left(1 - \frac{\e}{\d^{2L}}\right) x^+.
\]
Recalling that $x^-<0$, it implies  $\e \geq \d^{2L}$. Thus, the spectral gap is estimated to be at least 
\begin{equation}\label{ }
\e = \d^{2L}.
\end{equation}
To estimate $\d$ in terms of thickness, let us observe that by continuity in any overlapping neighborhoods there exists a ball of fixed radius $\rseg >0$
\begin{equation}\label{e:commball}
B_{\rseg}(x) \ss \supp(g_l) \cap \supp(g_{l'})
\end{equation}
such that $g_l, g_{l'} \geq c_1$ on $B_{\rseg}(x)$ for some fixed $c_1>0$. Thus, we have using that $\rho(g_l) \leq 1$,
\[
 \frac{\rho(g_l g_{l'})}{\sqrt{\rho(g_l) \rho(g_{l'})}} \geq c^2_1 \orho_{\rseg}(\O).
 \]
So, $\d \gtrsim \orho_{\rseg}(\O)$.

\end{proof}

\subsubsection{Application of the low energy method to non-ball-positive models}
For non-ball-positive models such as Motsch-Tadmor, or more generally, for \ref{Mb} the low energy method can still  produce estimates on spectral gap for almost uniformly distributed densities, 
\begin{equation}\label{e:rhonearunif}
\left\|\rho - \frac{1}{|\O|} \right\|_1 \leq \d.
\end{equation}
Here, we make the same Bochner positivity assumption on the defining kernel $\phi = \psi \ast \psi$ and the locality \eqref{e:locker3}.  

Let us start  as in  \exam{ex:MTgap} by symmetrizing and using cancellation 
\begin{equation*}\label{}
\begin{split}
(u, u)_{\k_\rho} - (u, \ave{u}_\rho)_{\k_\rho} & = \int_{\O \times \O} u(x) \cdot (u(x) - u(y)) \rho(x) \rho(y) \frac{\phi(x-y)}{\rho^{1-\b}_\phi(x)} \dy \dx \\
&= \frac12 \int_{\O \times \O} |u(x) - u(y)|^2 \rho(x) \rho(y) \frac{\phi(x-y)}{\rho^{1-\b}_\phi(x)} \dy \dx \\
&- \frac12 \int_{\O \times \O}  | u(y)|^2 \rho(x) \rho(y) \left(\frac{1}{\rho^{1-\b}_\phi(x)} - \frac{1}{\rho^{1-\b}_\phi(y)} \right) \phi(x-y) \dy \dx = I + II.
\end{split}
\end{equation*}
First, note that
\[
I \geq c_1  \int_{\O \times \O} u(x) \cdot (u(x) - u(y)) \rho(x) \rho(y) \phi(x-y) \dy \dx = c_1[ (u, u)_{\rho_\phi \rho} - (u, \ave{u}_\rho)_{\rho_\phi \rho} ],
\]
which is exactly the spectral gap form that appears for the \ref{CS} model. So, using  \prop{p:gaps} and \eqref{e:rhonearunif} we obtain
\[
I \geq c_2  \orho^3_{r_0/2}(\O) (u, u)_{\rho_\phi \rho} \geq c_2( c_3 - \d)^3 (c_4 - \d)^{1-\b} (u, u)_{\rho^\b_\phi \rho} \geq c_5(u, u)_{\k_\rho} ,
\]
provided $\d <  \frac12 \min\{c_3,c_4\}$.  Next,
\begin{equation*}\label{}
\begin{split}
II = \frac12 \int_{\O \times \O}  | u(y)|^2  \rho(y)\rho^{\b}_\phi(y) \left[ 1 - \frac{1}{\rho^\b_\phi}  \left(  \frac{\rho}{\rho^{1-\b}_\phi} \right)_\phi \right]  \dy.
\end{split}
\end{equation*}
Using  again \eqref{e:rhonearunif},
\begin{equation*}\label{}
\begin{split}
 \frac{1}{\rho^\b_\phi}  \left(  \frac{\rho}{\rho^{1-\b}_\phi} \right)_\phi & \leq \frac{\rho_\phi}{c_4 - \d} \leq  \frac{c_4 + \d}{c_4 - \d} \leq 1 + \frac{2\d}{c_4 - \d} \\
  \frac{1}{\rho^\b_\phi}  \left(  \frac{\rho}{\rho^{1-\b}_\phi} \right)_\phi & \geq \frac{\rho_\phi}{c_4 + \d} \geq  \frac{c_4 - \d}{c_4 + \d} \geq 1 - \frac{2\d}{c_4 + \d}
\end{split}
\end{equation*}
 So, if $\d$ is small enough we have 
 \[
 \left| 1 - \frac{1}{\rho^\b_\phi}  \left(  \frac{\rho}{\rho^{1-\b}_\phi} \right)_\phi \right| \leq c_6 \d.
 \]
We arrive at 
\[
| II | \leq c_6 \d (u, u)_{\k_\rho} .
\]
Combining the two together we obtain
\[
(u, u)_{\k_\rho} - (u, \ave{u}_\rho)_{\k_\rho} \geq (c_5 - c_6 \d) (u, u)_{\k_\rho} \geq c_7 (u, u)_{\k_\rho},
\]
provided $\d < c_8$, where $c_8$ is an absolute constant depending only on the parameters of the model. We have thus proved a version of \prop{p:gaps} for \ref{Mb} models.

\begin{proposition}\label{p:gapsMb}
There exist  constants $\d, c_0 >0$ depending only on the parameters of the model \ref{Mb}, $0\leq \b \leq 1$, such that for any density satisfying \eqref{e:rhonearunif} the size of the spectral gap is estimated as $\e_0 > c_0$.
\end{proposition}

\section{Deterministic mean-field limit}\label{s:mfldet}

In this section we consider either the periodic or open environments $\O = \T^n,\ \R^n$. 

The goal of this section will be to derive  the Vlasov-Alignment equation \eqref{e:VA},
as the weak limit of empirical measures
\begin{equation}\label{e:empir}
\mu^N_t = \sum_{i=1}^N m_i \d_{x_i(t)} \otimes \d_{v_i(t)},
\end{equation}
where $(x_i,v_i)_{i=1}^N$ solve the agent based system \eqref{e:ABSdet}.
We will focus on the measure-valued solutions with bounded support. Although this is not a necessary assumption, it simplifies some of the technical issues considerably.
\begin{definition}\label{d:weakkin} We say that $\{\mu_t\}_{0\leq t<T} \in C_{w^*}([0,T); \cP(B_R \times B_R))$ is a measure-valued solution to  \eqref{e:VA} with initial condition $\mu_0$ if for any test-function $g\in C^\infty([0,T)\times \O \times \R^{n})$ one has, for all $0<t<T$,
	\begin{equation}\label{e:weakkin}
	\begin{split}
\int_{\domain} g(t,x,v)\dmu_t(x,v)  &=	\int_{\domain} g(0,x,v)\dmu_0(x,v) \\
&+ \int_0^t \int_{\domain}  ( \p_s g + v\cdot \n_x  g +  \st_{\rho_s} (\ave{u_s}_{\rho_s} - v) \cdot \n_v g)\dmu_s(x,v)\ds .
\end{split}
\end{equation}
\end{definition}
The definition makes sense provided $\st_{\rho_s}$ and $\st_{\rho_s}\ave{u_s}_{\rho_s}$ are bounded and continuous functions in $(s,x)$. This typically can be derived from the regularity of the model as stipulated in \sect{s:rth}. 
But since we cannot rely on any a priori thickness of solutions we must assume that the model $\cM$ is uniformly regular.

With this assumption, the continuity of $\st_{\rho_s}(x)$ in $x$ follows from \eqref{e:ur1}. Continuity in $s$ follows from \eqref{e:ur2}:
\[
\| \st_{\rho_{s'}} - \st_{\rho_{s''}} \|_\infty \lesssim W_1(\rho_{s'},\rho_{s''}) \leq W_1(\mu_{s'}, \mu_{s''}).
\]
Since for compactly supported measures $W_1$ determines the weak$^*$-convergence, the claim follows. As to the weighted averages, we have
\[
\st_{\rho_s}\ave{u_s}_{\rho_s} = \int_{B_R \times B_R} \phi_{\rho_s}(x,y) v \dmu_s(y,v).
\]
So, again the continuity in $x$ follows from \eqref{e:ur1}. In terms of time, we use 
\eqref{e:ur2}
\begin{equation*}\label{}
\begin{split}
& \int_{B_R \times B_R} \phi_{\rho_{s'}}(x,y) v \dmu_{s'}(y,v) - \int_{B_R \times B_R} \phi_{\rho_{s''}}(x,y) v \dmu_{s''}(y,v) \\
& =\int_{B_R \times B_R} \phi_{\rho_{s'}}(x,y) v [\dmu_{s'}(y,v) - \dmu_{s''}(y,v)] \\
& + \int_{B_R \times B_R}[ \phi_{\rho_{s'}}(x,y) - \phi_{\rho_{s''}}(x,y)] v \dmu_{s''}(y,v) \\
& \leq \oC_1 W_1(\mu_{s'}, \mu_{s''}) +  \oC R W_1(\rho_{s'},\rho_{s''}) \lesssim W_1(\mu_{s'}, \mu_{s''}) .
\end{split}
\end{equation*}

The crucial and elementary observation is that the empirical measure 
\eqref{e:empir} satisfies \eqref{e:weakkin} if and only if $\{(x_i,v_i)\}_i$ solve the agent-based system \eqref{e:ABSdet}.  As a consequence, solutions to \eqref{e:ABSdet} fall naturally into the framework of the Vlasov-Alignment equation.  Our goal will be to prove the following theorem by showing contractivity of the map $\mu_0 \to \mu_t$ on any finite time interval.

\begin{theorem}\label{t:mfl}
Suppose $\cM$ is uniformly regular. Let $\mu_0 \in \cP(\domain)$ be any measure with compact support. Then for any $T>0$ there exists a unique measure-valued solution  $\{\mu_t\}_{0\leq t<T} \in C_{w^*}([0,T); \cP(B_{R(T)})$ to \eqref{e:VA} which can be reconstructed from solutions to \eqref{e:ABSdet} as follows. Let all $(x_i^0,v_i^0) \in \cO$, where $\cO$ is some fixed neighborhood of $\supp \mu_0$ and such that $\mu^N_0 \to \mu_0$ weakly. Then $\mu^N_t \to \mu_t$ weakly uniformly on $[0,T)$.
\end{theorem}

As a corollary we obtain validity of the mean-field limit in all the cases listed in the last row of Table~\ref{t:rth}.

The theorem will be proved via a Lagrangian approach using the transport structure of \eqref{e:VAE}. 
To this end, we introduce the characteristic flow 
\begin{alignat}{2}
\ddt X(t,s,x,v) & = V(t,s,x,v), \quad  & X(s,s,x,v) & = x,\label{e:XKs} \\
 \ddt V(t,s, x,v) & = \st_\rho(X) ( \ave{u}_\rho(X) - V), \quad &   V(s,s, x,v) & = v. \label{e:VKs}
\end{alignat}
We also denote $X(t,0,x,v) = X(t,x,v)$, $V(t,0,x,v) = V(t,x,v)$, and $(x,v) = \w$. Note that the right hand side of \eqref{e:VKs} is Lipschitz in $(X,V)$, so the flow is well-defined on $[0,T]$. Define the test-function $g(s,\w) = h(X(t,s,\w),V(t,s,\w))$ for some $h \in C_0^\infty(\R^{2n})$, for which we have 
\[
\p_s g + v\cdot \n_x  g + \st_{\rho} (\ave{u}_{\rho} - v) \cdot \n_v g = 0 .
\]
So, plugging it into \eqref{e:weakkin} we obtain
\begin{equation}\label{e:consK}
\int_{\domain} h(\w) \dmu_t(\w) = \int_{\domain} h(X(t,\w),V(t,\w)) \dmu_0(\w). 
\end{equation}
This means that that $\mu_t$ is a {\em push-forward} of the initial measure $\mu_0$ along the flow-map $(X,V)$, $\mu_t = (X,V) \# \mu_0$.

The proof of the mean-field limit consists of two steps: establishing control over the deformation $(\n X, \n V)$ on a given time interval, and proving Lipschitzness of the push-forward map in the $W_1$-metric.

So, let us assume that on a time interval $[0,T]$ we have a solution $\mu_t \in \cP(B_R)$. By the maximum principle of \lem{l:maxpr} 
\begin{equation}\label{e:maxVO}
\|V(t)\|_{L^\infty(\cO)} \leq \max_{(x,v) \in \cO} |v| \leq \diam \cO.
\end{equation}

Let us fix a compact domain $\cO$ with $\supp \mu_0 \ss \cO$. Then
\[
\ddt \| \n X \|_{L^\infty(\cO)} \leq  \| \n V \|_{L^\infty(\cO)}.
\]
Next,
\begin{equation*}\label{}
\begin{split}
\ddt  \n V  \leq \n X^\top\n (\st_\rho \ave{u}_\rho)(X) + \n X^\top \n \st_\rho (X) V + \st_\rho (X) \n V,
\end{split}
\end{equation*}
so, in view of (ev4), \eqref{e:maxVO}, and \eqref{e:ur1}, we obtain the inequality up to a constant depending only on $R,m,\cO,\oS$, 
\begin{equation*}\label{}
\ddt \| \n V  \|_{L^\infty(\cO)} \leq \| \n X  \|_{L^\infty(\cO)}  + \|\n V \|_{L^\infty(\cO)}.
\end{equation*}
We thus conclude that 
\begin{equation}\label{e:XVRT}
\sup_{[0,T]} \| \n X  \|_{L^\infty(\cO)}  + \|\n V \|_{L^\infty(\cO)} \leq C(R,m,\cO,T).
\end{equation}

Let us now proceed to continuity estimates. Let us fix two measures $\mu'_t, \mu''_t \in \cP(B_R)$ for all $t\in [0,T]$. We also fix a common initial domain  $\cO$,  $\supp \mu'_0 \cup \supp \mu''_0 \ss \cO$.   Clearly, 
\begin{equation}\label{e:X'X''}
\ddt \|  X' - X'' \|_{L^\infty(\cO)} \leq \|  V' - V'' \|_{L^\infty(\cO)}.
\end{equation}

For velocities we have
\begin{equation*}\label{}
\begin{split}
\ddt (V' - V'') & = \st_{\rho'}(X')  \ave{u'}_{\rho'}(X') -\st_{\rho''}(X'') \ave{u''}_{\rho''}(X'')  +\st_{\rho''}(X'')   V'' - \st_{\rho'}(X')  V'.
\end{split}
\end{equation*}
So, from \eqref{e:ur1}-\eqref{e:ur2}, we have
\[
| \st_{\rho'}(X')  \ave{u'}_{\rho'}(X') -\st_{\rho''}(X'') \ave{u''}_{\rho''}(X'') | \lesssim W_1(\rho',\rho'')  + W_1(u'\rho',u''\rho'') +\|  X' - X'' \|_{L^\infty(\O)},
\]
and
\[
| \st_{\rho'}(X') - \st_{\rho''}(X'') | \lesssim  W_1(\rho',\rho'')+ \|  X' - X'' \|_{L^\infty(\O)}.
\]
Thus,
\begin{equation*}\label{}
\ddt \|  V' - V''\|_{L^\infty(\cO)} \lesssim W_1(\rho',\rho'')  + W_1(u'\rho',u''\rho'') +  \|  V' - V''\|_{L^\infty(\cO)}+ \|  X' - X'' \|_{L^\infty(\O)}.
\end{equation*}
But for any $\| g\|_\Lip \leq 1$ we have
\begin{equation*}\label{}
\begin{split}
\int_\O g(x) ( \drho_t' - \drho_t'') &= \int_\domain g(x) (\dmu_t' - \dmu_t'') = \int_\domain g(X') \dmu'_0 -  \int_\domain g(X'') \dmu''_0 \\
& =  \int_\domain g(X') (\dmu'_0 - \dmu_0'') +  \int_\domain (g(X') - g(X'')) \dmu'_0 \\
& \leq \|  \n X' \|_{L^\infty(\O)} W_1(\mu_0' , \mu_0'')  +  \|  X' - X'' \|_{L^\infty(\O)}
\end{split}
\end{equation*}
In view of \eqref{e:XVRT} we conclude that 
\begin{equation}\label{e:W1rho}
W_1(\rho'_t,\rho''_t) \lesssim W_1(\mu_0' , \mu_0'') +  \|  X' - X'' \|_{L^\infty(\O)}.
\end{equation}

Similarly, for any $\| g\|_\Lip \leq 1$ we have
\begin{equation*}\label{}
\begin{split}
\int_\O g(x) ( \mathrm{d}(u' \rho_t') - \mathrm{d}(u''\rho_t'') )&= \int_\domain g(x) v (\dmu_t' - \dmu_t'') = \int_\domain g(X') V' \dmu'_0 -  \int_\domain g(X'') V'' \dmu''_0 \\
& =  \int_\domain g(X') V' (\dmu'_0 - \dmu_0'') +  \int_\domain (g(X') V' - g(X'')V'') \dmu'_0 \\
& \leq (\diam{\cO} \|  \n X' \|_{L^\infty(\cO)} + \|g\|_{L^\infty(B_R)} \|\n V'\|_{L^\infty(\cO)}) W_1(\mu_0' , \mu_0'')  \\
& + m \diam{\cO} \|  X' - X'' \|_{L^\infty(\cO)} +  \|g\|_{L^\infty(B_R)} \|  V' - V'' \|_{L^\infty(\cO)}.
\end{split}
\end{equation*}
In view of \eqref{e:XVRT} we conclude that 
\begin{equation}\label{e:W1urho}
W_1(u'\rho', u'' \rho'') \lesssim W_1(\mu_0' , \mu_0'') +  \|  X' - X'' \|_{L^\infty(\cO)} +\|  V' - V'' \|_{L^\infty(\cO)}.
\end{equation}
Thus, we obtain
\begin{equation*}\label{}
\ddt \|  V' - V''\|_{L^\infty(\cO)} \lesssim W_1(\mu_0' , \mu_0'') +  \|  X' - X'' \|_{L^\infty(\cO)} +\|  V' - V'' \|_{L^\infty(\cO)}.
\end{equation*}
Combining with \eqref{e:X'X''} we conclude that 
\begin{equation}\label{e:stab'''}
 \|  X' - X'' \|_{L^\infty(\cO)} +\|  V' - V'' \|_{L^\infty(\cO)} \leq C(R,T) W_1(\mu_0' , \mu_0'') .
\end{equation}
 
Let us now fix a function $h$ with $\Lip (h) \leq 1$, and use the transport identity \eqref{e:consK}:
\begin{align*}
\int_{\domain}& h(\w) \dmu'_t - \int_{\domain} h(\w) \dmu''_t = \int_{\domain} h(X',V') \dmu'_0 - \int_{\domain} h(X'',V'') \dmu''_0 \\
& =  \int_{\domain} h(X',V') (\dmu'_0 - \dmu''_0 ) + \int_{\domain} [h(X',V') - h(X'',V'') ] \dmu''_0 \\
& \leq \Lip_\cO(h(X',V') ) W_1(\mu'_0,\mu''_0) +  \| X_\mu - X_\nu\|_{L^\infty(\cO)} + \|V_\mu - V_\nu\|_{L^\infty(\cO)}.
\end{align*}
Using that 
\[
\Lip_\cO(h(X',V') ) \leq  \| \n V' \|_{L^\infty(\cO)}  + \| \n X' \|_{L^\infty(\cO)}  ,
\]
and applying \eqref{e:XVRT}, \eqref{e:stab'''} we conclude the following bounds
\begin{equation}\label{e:stabt0}
W_1(\mu'_t,\mu''_t) \leq C(R,\cO,T) W_1(\mu'_0,\mu''_0).
\end{equation}
This immediately implies uniqueness and stability of measure-valued solutions.

So, we start now with an arbitrary measure $\mu_0$, and approximate it weakly with a sequence of empirical measures \begin{equation}\label{e:empir0}
\mu^N_0 = \sum_{i=1}^N m_i \d_{x_i} \otimes \d_{v_i},
\end{equation}
with all $(x_i,v_i) \in \cO$, where $\cO$ is some fixed neighborhood of $\supp \mu_0$. Then let us run the agent-based alignment model alignment \eqref{e:ABSdet}. For any time $T$, we have $\supp \mu_t^N \ss B_{|\cO|+T A_0} \times B_{A_0}$, $t<T$. Thus, 
according to \eqref{e:stabt0}, $\mu_t^N$ is weakly Cauchy, and hence $\mu_t^N \to \mu_t$ for some $\mu_t$.  To finish the proof \thm{t:mfl} we now prove a short lemma showing that the limit solves the Vlasov-alignment equation weakly.

\begin{lemma} Suppose a sequence of solutions  $\mu^N \in C_{w^*}([0,T); \cP(B_R))$ converges weakly pointwise, i.e. $\mu_t^N \to \mu_t$ for all $0\leq t<T$. Then $\mu \in C_{w^*}([0,T); \cP(B_R))$ is a weak solution to \eqref{e:VAE}.
\end{lemma}
\begin{proof}
	 The weak$^*$-continuity of the limit will follow immediately from \eqref{e:weakkin} once it is established. Clearly, all the linear terms in \eqref{e:weakkin} converge to their natural limits. As to the force  let us note for any $s<T$, we have (by computations done above)
\[
W_1(\rho^N_s, \rho^M_s) + W_1(u^N_s\rho^N_s, u^M_s \rho^M_s) \leq  C W_1(\mu^N_s,\mu^M_s) \leq  C W_1(\mu^N_0,\mu^M_0),
\]
since both are solutions to the Vlasov-alignment equation. Sending $M\to \infty$ we obtain
\[
W_1(\rho^N_s, \rho_s) + W_1(u^N_s\rho^N_s, u_s \rho_s) \leq  C W_1(\mu^N_0,\mu_0),
\]
which by continuity \eqref{e:ur2}  implies that 
\[
\| \st_{\rho^N_s} (\ave{u^N_s}_{\rho^N_s} - v) - \st_{\rho_s} (\ave{u_s}_{\rho_s} - v) \|_{L^\infty(B_R)} \to 0
\]
uniformly in $s$.  Together with the weak convergence assumed for $\mu^N_s$ we obtain
\[
 \int_0^t \int_{\domain}  ( \st_{\rho^N_s} (\ave{u^N_s}_{\rho^N_s} - v) )\dmu^N_s(x,v)\ds \to  \int_0^t \int_{\domain}  ( \st_{\rho_s} (\ave{u_s}_{\rho_s} - v) )\dmu_s(x,v)\ds .
 \]
This finishes the proof.
\end{proof}

Finally, let us discuss the implementation of \thm{t:mfl} to global well-posedness of smooth solutions. Since all solutions are transported according to \eqref{e:consK}  regularity of a solution will depend on the regularity of initial data and the parameters of the model.  First, let us notice that the Jacobian of the characteristic map, by the Liouville formula, is given by
\[
\det \n_{\w} (X,V) (t,\w) = \exp \left\{ - n \int_0^t \st_\rho(X(s,\w)) \ds \right\}.
\]
Then if $\mu_0 = f_0 \dw$, with $f_0\in C^k$, $k\in \N$ and compactly supported, then for any $t>0$,
\begin{equation}\label{ }
f(t, X(t,\w),V(t,\w)) = f_0(\w) \exp \left\{ n \int_0^t \st_\rho(X(s,\w)) \ds \right\}.
\end{equation}
Inverting the flow and noting that $(X,V)$ and $\st_\rho$ are $C^k$ implies $f \in C^k$ at all times with support in $v$ being confined to its original bounds and support in $x$ growing at most linearly.

\begin{theorem}\label{t:VAgwp}
Suppose $\cM$ is uniformly regular. Let $f_0 \in C^k_0(\domain)$ be any compactly supported distribution. Then for any $T>0$ there exists a unique  solution $f \in L^\infty([0,T); C^k_0)$ to \eqref{e:VA} which is supported on $\supp f_0 + B_{t A_0 } \times \{0\}$, where $A_0$ is the maximal initial velocity.
\end{theorem}

\section{Stochastic mean-field limit}\label{s:mflstoch} 

As discussed in \sect{s:CST} one of the main obstacles for  alignment on the torus $\T^n$  is existence of so-called locked states: solutions with agents locked on periodic orbits that stay at a positive distance greater than the communication length scale  $r_0$.  A natural way to avoid such unstable states is to introduce  stochastic noise 
\begin{equation}\label{e:ABSs}
\begin{split}
\dx_i & = v_i \dt \\
\dv_i & = \st_i ( \ave{v}_i - v_i )\dt + \sqrt{2\s \st_i } \dB_i,
\end{split}
\end{equation}
where $B_i$'s are independent Brownian motions in $\R^n$.
Note that the noise here is assumed to be ``material", i.e. it places stochasticity only within the influence of the flock. As $N \to \infty$ and assuming that the agents are indistinguishable, i.e. $m_1 = \dots = m_N = \frac{1}{N}$, the system comes in natural correspondence with what we call the {\em 
Fokker-Planck-Alignment equation}
\begin{equation}\label{e:FPA}
\p_t f + v \cdot \n_x f = \s \st_\rho \D_v f + \n_v( \st_\rho (v - \ave{u}_\rho) f ).
\end{equation}

A major advantage of using material noise is that the kinetic model \eqref{e:FPA}  possesses a family of thermodynamic equilibria
\begin{equation}\label{e:Maxwellian0}
\mu_{\s,\bar{u}} = \frac{1}{|\O|(2\pi \s)^{n/2}} e^{- \frac{|v - \bar{u}|^2}{2\s}}.
\end{equation} 
 If the underlying model $\cM$ is conservative every solution is centered around the constant averaged momentum $\bar{u}$, which predetermines the corresponding equilibrium and opens a possibility for potential relaxation towards that distribution. 
 The collective behavior interpretation of this result would say that, as expected, the noise disrupts the locked states and redistributes initial velocities symmetrically around the mean value $\bar{u}$.  Alignment is then restored in the sense of the vanishing noise limit:
\begin{equation}\label{e:ultimate}
\lim_{\s \to 0} \lim_{t \to \infty} f^\s(t)  = \frac{1}{|\O|} \d_{v = \bar{u}} \otimes \dx.
\end{equation}

The problem of relaxation and hypocoercivity will be discussed in \sect{s:hypo}.
In this section we provide a rigorous derivation of equation \eqref{e:FPA}  as a mean-field limit of solutions to the stochastic system \eqref{e:ABSs}.  To make this statement precise, let  us consider $f$  a solution to \eqref{e:FPA} on a time interval $[0,T]$ with initial distribution $f_0$.   Consider now $N$ independent identically distributed random variables $(x_i^0,v_i^0)$, $i\leq N$, with  $f_0 = \mathrm{law}(x_i^0,v_i^0)$, and let $(x_i,v_i)$ solve \eqref{e:ABSs}. Form the empirical measure-valued random variables
\[
\mu_t^N = \frac1N \sum_{i=1}^N \d_{x_i(t)} \otimes \d_{v_i(t)}.
\]
The mean-field limit consist of showing that for all $t \leq T$, we have $\mu_t^N \to f_t$ in law, i.e. for any Lipschitz function $h$ on $\domain$,
\begin{equation}\label{e:Emf}
\E \left|  \frac1N \sum_{i=1}^N h(x_i(t),v_i(t)) - \int_{\domain} h(x,v) f(t,x,v) \dx\dv \right|^2  \to 0.
\end{equation}
Note that $f \dx \dv$ in this context is considered as a constant random measure.

In general, the convergence \eqref{e:Emf} is equivalent to propagation of chaos,  see Sznitman \cite{sznitman}: if $f^{N}$ denotes the joint probability distribution of the process $(x_1,v_1,\dots, x_N,v_N)$ solving \eqref{e:ABSs}, then for any $k\geq 1$, the $k$-th marginal $f^{(k)}$ converges weakly to the product of $k$ copies of $f$, $f^{\otimes k}$, as $N\to \infty$:
\begin{equation}\label{e:prop}
\lan f^{(k)},  \f_1\otimes \ldots \otimes \f_k \ran = \lan f^{N}, \f_1\otimes \ldots \otimes \f_k \otimes 1 \otimes \dots \otimes 1 \ran \to \prod_{j=1}^k \lan f, \f_j \ran, \qquad \f_j \in C_b(\R^{2n}).
\end{equation}

The strategy of proving \eqref{e:Emf}  is based on the classical coupling method. Note that  if $(x_i,v_i)$'s were independent and identically distributed by $f$, then \eqref{e:Emf} would have been nothing but the Law of Large Numbers. So, to achieve the limit we couple \eqref{e:ABSs} with another system of separate $N$ copies of the characteristic processes for \eqref{e:FPA}:
\begin{equation}\label{e:ABSc}
\begin{split}
\dbx_i & = \bar{v}_i \dt \\
\dbv_i & = \st_\rho(\barx_i) ( \ave{u}_\rho(\barx_i)  - \barv_i ) \dt + \sqrt{2  \s \st_\rho(\barx_i)} \dB_i,
\end{split}
\end{equation}
with initial condition $(x_i^0,v_i^0)$. Here, $\rho$ and $u$ are the macroscopic values of $f$. Note that because the equations are decoupled, the pairs $(\barx_i,\barv_i)$ remain independent and identically distributed. By the It\^o formula, $f$  is their common law.

To establish \eqref{e:Emf} one can add and subtract the intermediate average of $h$ with $\barx_i(t),\barv_i(t)$ pairs:
\begin{equation}\label{}
\begin{split}
& \E \left|  \frac1N \sum_{i=1}^N h(x_i(t),v_i(t)) - \int_{\R^{2n}} h(x,v) f(t,x,v) \dx\dv \right|^2  \\
& \leq \E  \left|  \frac1N \sum_{i=1}^N h(x_i(t),v_i(t)) - \frac1N \sum_{i=1}^N h(\barx_i(t),\barv_i(t))  \right|^2  \\
& +  \E  \left|  \frac1N \sum_{i=1}^N h(\barx_i(t),\barv_i(t)) - \int_{\R^{2n}} h(x,v) f(t,x,v) \dx\dv \right|^2 .
\end{split}
\end{equation}
The second term goes to zero by the Law of Large Numbers, while the first can be estimated using symmetry by
\[
\E \left|  \frac1N \sum_{i=1}^N h(x_i(t),v_i(t)) - \frac1N \sum_{i=1}^N h(\barx_i(t),\barv_i(t))  \right|^2  \leq \| \n h\|_\infty \E[ |x_1 - \barx_1|^2 + |v_1 - \barv_1|^2].
\]
So the proof of  \eqref{e:Emf} reduces to obtaining control over separation of  characteristics:
\begin{equation}\label{e:Echar}
E(t) = \E[ |x_i - \barx_i|^2 + |v_i - \barv_i|^2] \to 0, \quad \text{as } N\to \infty .
\end{equation}

This approach was carried out by Bolley, et al., \cite{BCC2011} in the case of convolution-type alignment systems and with additive noise (no strength $\st_\rho$ thermalization). We now provide a proper extension that includes general environmental averaging models and material noise as stated.

Let us also note, following \cite{BCC2011}, that a bound on \eqref{e:Echar} entails a bound on the rate of decorrelation $f^{(k)} \to f^{\otimes k}$. Indeed,
\[
W_2^2(f^{(k)}, f^{\otimes k}) \leq \E\left[ \sum_{i=1}^k |x_i - \barx_i|^2 + |v_i - \barv_i|^2 \right] = k E(t) \to 0.
\]
where $W_2$ is the Wesserstein-2 distance.

\subsection{Law of large numbers}
We will work on the torus $\O = \T^n$ and assume that $\cM$ is uniformly regular in the sense of \defin{d:ur} with a minor modification.
For all our averaging models the $W_1$-metric used to define continuity in $\rho$ can in fact be replaced with a weaker  $W_1$ semi-metric determined by finitely many fixed Lipschitz functions: for $h_1,\ldots,h_K \in \Lip(\O)$ with $\|h_k\|_\Lip \leq 1$, 
\begin{equation}\label{e:W1weak}
W_1^{h_1,\ldots,h_K}(\rho',\rho'')  = \max_{k=1,\ldots,K} \left| \int_{\O} h_k(x) [ \drho'(x) - \drho''(x) ]\right|.
\end{equation}
Such is the case for all Favre-based models where $h = \phi$, or for  \ref{Mseg} where $h_l = g_l$. Thus, the uniform continuity can be understood as follows
\begin{equation}\label{e:contglobweak}
\| \st_{\rho'} -\st_{\rho''}  \|_{\infty} +  \| \phi_{\rho'}  -\phi_{\rho''} \|_{\infty}  \leq C W_1^{h_1,\ldots,h_K}(\rho',\rho'').
\end{equation}

Let us now discuss consequences of the assumed regularity of the model on the Law of Large Numbers.  The basic idea is that the model is compatible with the LLN in the averaged sense. Let us recall the classical law first, see \cite{sznitman}: for a sequence of i.i.d. random variables $X_j: \Sigma \to \R^d$ with bounded second momentum $\E|X_j|^2 \leq E_0$ and mean $\E X_j = m$ we have
\begin{equation}\label{ }
\E \left| \frac1N \sum_{j=1}^N X_j - m\right|^2 \leq \frac{E_0}{N}.
\end{equation}
Consequently, if $h\in C_b(\R^d)$ and $\mu$ is the law of $X_j$'s, then in terms of $\mu$ the above reads
\begin{equation}\label{e:LLNgen}
\int_{\R^{Nd}} \left| \frac1N \sum_{j=1}^N h(\w_j) -  \int_{\R^d} h(\w) \mu(\w) \right|^2 \dmu(\w_1) \dots \dmu(\w_N) \leq \frac{\|h\|_\infty^2}{N}.
\end{equation}
We will encounter \eqref{e:LLNgen} in two interpretations. Namely, for any $h\in C_b(\O)$ and $f\in \cP(\domain)$, we have
\begin{align}
\int_{\O^N} &\left| \frac1N \sum_{j=1}^N h(y_j) - \int_\O h(z) \drho(z) \right|^2 \drho(y_1)\ldots \drho(y_N) \leq  \frac{C\|h\|^2_\infty}{N}, \label{e:LLNh}\\
\int_{\O^N\times \R^{nN}}& \left| \frac1N \sum_{j=1}^N v_j h(y_j) - \int_\O h(z) u(z) \drho(z) \right|^2 \df(y_1,v_1)\ldots \df(y_N,v_N)  \leq  \frac{C \cE(f) \|h\|^2_\infty}{N}, \label{e:LLNhf}
\end{align} 
where $\cE(f) = \int_\domain |v|^2 \df $. 

The next two lemmas show that the analogue of these two laws of large numbers also holds with respect to the components of the model $\cM$. 
 
\begin{lemma}\label{l:aN} We have
\begin{equation}\label{e:LLN1}
\a_N = \sup_{\rho\in \cP(\O)} \int_{\O^N} \left| \st_{\rho^N} (y_i) - \st_\rho(y_i) \right|^2 \drho(y_1)\ldots \drho(y_N) \lesssim \frac1N, 
\end{equation}
where $\rho^N = \frac{1}{N} \sum_{j=1}^N \d_{y_j}$. Note that $\a_N$ is independent of $i$ by symmetry. 
\end{lemma}
\begin{proof}
 To see that we have by \eqref{e:contglobweak} and \eqref{e:LLNh},
\begin{equation*}\label{}
\begin{split}
\int_{\O^N} \left| \st_{\rho^N} (y_i) - \st_\rho(y_i) \right|^2 \drho(y_1)\ldots \drho(y_N) & \lesssim \sum_{k=1}^K \int_{\O^N} \left| \int_\O h_k(z)[ \drho^N -\drho ] \right|^2 \drho(y_1)\ldots \drho(y_N) \\
& = \sum_{k=1}^K \int_{\O^N} \left|  \frac1N \sum_{j=1}^N h_k(y_j) - \int_\O h_k(z) \drho(z) \right|^2 \drho(y_1)\ldots \drho(y_N) \\
& \leq \frac{C}{N} ( \| h_1\|_\infty + \dots \|h_K\|_\infty).
\end{split}
\end{equation*}
\end{proof}

\begin{lemma}\label{l:bN} We have
\begin{equation}\label{e:LLN2}
\b_N = \sup_{f: \cE(f)\leq \cE_0} \int_{\O^N\times \R^{nN}} \left| \st_{\rho^N} (y_i) \ave{ u^N }_{\rho^N}(y_i) - \st_\rho(y_i) \ave{u}_\rho (y_i) \right|^2 \df(y_1,v_1)\ldots \df(y_N,v_N) \lesssim \frac1N, 
\end{equation}
 where $\rho^N$ is as before, $u^N = \sum_{j=1}^N v_j \one_{\{y_j\}} $, and $\rho , u$ are the macroscopic density and velocity of $f$. 
\end{lemma}
\begin{proof}
Let us assume $i=1$ for definiteness.  We have
\begin{equation}\label{e:cross1}
\begin{split}
\st_{\rho^N} (y_1) \ave{ u^N }_{\rho^N}(y_1) - \st_\rho(y_1) \ave{u}_\rho (y_1) & = \int_\O (\phi_{\rho^N} (y_1,z)  - \phi_{\rho} (y_1,z)  ) u^N(z) \drho^N(z)\\
& + \int_\O  \phi_{\rho} (y_1,z)[u^N(z) \drho^N(z) - u(z)\drho(z)] = I + II.
\end{split}
\end{equation}
Let us examine $I$ first. We have by \eqref{e:W1weak},
\[
| I | \leq W_1^{h_1,\ldots,h_K}(\rho^N,\rho) \frac1N\sum_{j=1}^N |v_j|.
\]
Thus,
\[
\int_{\O^N\times \R^{nN}} |I|^2  \df(y_1,v_1)\ldots \df(y_N,v_N) \leq   \frac1N\sum_{j=1}^N \int_{\O^N\times \R^{nN}} |v_j|^2  (W_1^{h_1,\ldots,h_K}(\rho^N,\rho) )^2 \df(y_1,v_1)\ldots \df(y_N,v_N)
\]
and by symmetry in $j$,
\begin{multline*}\label{}
=  \int_{\O^N\times \R^{nN}} |v_1|^2 ( W_1^{h_1,\ldots,h_K}(\rho^N,\rho) )^2 \df(y_1,v_1)\ldots \df(y_N,v_N) \\ \lesssim
\sum_{k=1}^K  \int_{\O^N\times \R^{nN}} |v_1|^2  \left|  \frac1N \sum_{j=1}^N h_k(y_j) - \int_\O h_k(z) \drho(z) \right|^2 \df(y_1,v_1)\ldots \df(y_N,v_N).
\end{multline*}
Let us focus on one $k$th term. We single out the $j=1$ term from the rest:
\begin{multline*}\label{}
\frac{1}{N^2} \int_{\O^N\times \R^{nN}} |v_1|^2  | h_k(y_1)|^2 \df(y_1,v_1)\ldots \df(y_N,v_N) \\
+ 
\int_{\O^N\times \R^{nN}} |v_1|^2  \left|  \frac1N \sum_{j=2}^N h_k(y_j) - \int_\O h_k(z) \drho(z) \right|^2 \df(y_1,v_1)\ldots \df(y_N,v_N) \\
\leq \frac{1}{N^2} \|h_k\|_\infty^2 \cE_0 + \cE_0 \int_{\O^{N-1}\times \R^{n(N-1)}}  \left|  \frac1N \sum_{j=2}^N h_k(y_j) - \int_\O h_k(z) \drho(z) \right|^2 \drho(y_2)\ldots \drho(y_N)
\end{multline*}
The latter integral is $\lesssim \frac1N $ with a minor adjustment to $N\to N-1$ in the average. Thus, by \eqref{e:LLNh},
\[
\int_{\O^N\times \R^{nN}} |I|^2  \df(y_1,v_1)\ldots \df(y_N,v_N) \lesssim \frac1N .
\]

It remains to analyze $II$.  We will treat $y_1$ as a parameter, and let us denote $h_{y_1}(z) = \phi_\rho(y_1,z)$. By the regularity assumption, $h_{y_1}\in C_b(\O)$ and is even Lipschitz. We have
\[
II =  \frac1N \sum_{j=1}^N v_j h_{y_1}(y_j) - \int_\O h_{y_1}(z) u(z) \drho(z).
\]
Again, we single out the $j=1$ term, and by \eqref{e:LLNhf},
\begin{equation*}\label{}
\begin{split}
& \int_{\O^N\times \R^{nN}} |II|^2  \df(y_1,v_1)\ldots \df(y_N,v_N)  \leq \frac{1}{N^2} \int_{\O^N\times \R^{nN}} |v_1|^2  | h_{y_1}(y_1)|^2 \df(y_1,v_1)\ldots \df(y_N,v_N) \\
+& \int_{\O^N\times \R^{nN}} \left| \int_\O \frac{1}{N} \sum_{j=2}^N v_j h_{y_1}(y_j) -  \int_\O h_{y_1}(z) u(z) \drho(z)\right|^2 \df(y_1,v_1)\ldots \df(y_N,v_N)  \\
 \leq  & \frac{1}{N^2} \|\phi_\rho \|_\infty^2 \cE_0  + \frac{1}{N} \|\phi_\rho \|_\infty^2 \cE_0 
\end{split}
\end{equation*}
with a minor adjustment to the index $N \to N-1$ in the latter.
\end{proof}

\subsection{Main result}

As discussed earlier we now focus on obtaining an estimate on separations of characteristics to achieve \eqref{e:Echar}. The result holds on a finite time interval $[0,T]$ where $f$ is a smooth solution to \eqref{e:FPA} by which we mean existence of sufficiently many derivatives in weighted Sobolev spaces to sufficient to understand \eqref{e:FPA} classically, see \sect{s:WP}. 

\begin{theorem}\label{t:MFLs}
Suppose $\cM$ is uniformly regular satisfying \eqref{e:contglobweak}. Let $f$ be a classical solution to the Fokker-Planck-Alignment equation \eqref{e:FPA} on a time interval $[0,T]$ satisfying
\begin{equation}\label{e:stlowstoch}
\Th(\rho,\O) \geq \d, \quad \forall 0\leq t\leq T,
\end{equation}
and
\begin{equation}\label{e:expf1}
\int_{\O \times \R^n} e^{a|v|^2} f(x,v,t)\dx\dv \leq c_8, \quad \forall 0\leq t\leq T.
\end{equation}
Then for any solution to the particle system \eqref{e:ABSs} and \eqref{e:ABSc} on the time interval $[0,T]$ with i.i.d.\,initial datum $(x_i^0,v_i^0)$ distributed according to the law $f_0$ one has the following estimate
\begin{equation}\label{ }
\E[ |x_i - \barx_i|^2 + |v_i - \barv_i|^2] \leq C_1 \frac{1}{N^{e^{-C_2 t}}},
\end{equation}
for some $C_1,C_2>0$ depending on $T$ and all the constants involved in the assumptions above.  Consequently, the mean-field limit \eqref{e:Emf} holds.
\end{theorem}
\begin{proof}  We set $\s=1$  for simplicity.  First, we notice that  the solution has a uniformly bounded energy on $[0,T]$ and thus  \eqref{e:LLN2} applies uniformly on $[0,T]$.

Let us denote 
\begin{equation}\label{ }
\begin{split}
\E = \E_x &+ \E_v \\
\E_x = \E[ |x_i - \barx_i|^2],& \quad \E_v = \E[ |v_i - \barv_i|^2].
\end{split}
\end{equation}
Taking the derivative of the $x$-component we obviously obtain
\[
\ddt \E_x = 2  \E[(x_i - \barx_i) \cdot (v_i - \barv_i) ] \leq \E.
\]

For the velocity component we use the It\^o formula,
\[
\ddt \E_v = \E[ (v_i - \barv_i ) \cdot   ( \st_i ( \ave{v}_i - v_i ) - \st_\rho(\bar{x}_i) ( \ave{u}_\rho(\bar{x}_i)  - \bar{v}_i )  ) ] + \E  \left| \sqrt{2 \st_i} - \sqrt{2 \st_\rho(\bar{x}_i)} \right|^2 .
\]

Let us start with the noise term using \eqref{e:stlowstoch} and \eqref{e:stlow},
\[
 \E \left| \sqrt{2 \st_i} - \sqrt{2 \st_\rho(\bar{x}_i)} \right|^2  =  2 \E \left| \frac{\st_i - \st_\rho(\bar{x}_i)}{ \sqrt{\st_i} + \sqrt{\st_\rho(\bar{x}_i)}} \right|^2  \leq C\,\E \left| \st_i - \st_\rho(\bar{x}_i)\right|^2 .
\]
Recalling that $\st_i = \st_{\rho^N} (x_i)$, where $\rho^N = \frac{1}{N} \sum_{j=1}^N \d_{x_j} $ and denoting  $\barho^N = \frac{1}{N} \sum_{j=1}^N \d_{\barx_j} $, we add and subtract intermediate terms
\begin{equation}\label{e:auxsE}
 \E \left| \st_i - \st_\rho(\bar{x}_i)\right|^2   \lesssim  \E \left| \st_{\rho^N} (x_i) - \st_{ \barho^N} (\barx_i) \right|^2   +  \E \left| \st_{\barho^N } (\barx_i) -   \st_\rho(\bar{x}_i) \right|^2 
\end{equation}
By regularity of the strength function \eqref{e:ur1} and symmetry,
 \[
  \E  \left| \st_{\rho^N} (x_i) - \st_{\barho^N } (\barx_i)  \right|^2  \lesssim \E_x +  \E\left[ \frac{1}{N} \sum_{j=1}^N |x_j - \barx_j|^2 \right] =C \E_x.
  \]
The second term is bounded by $\a_N$ as defined in \eqref{e:LLN1}, since $\rho$ is the law of $\barx_i$ and the latter are independent,
 \[
 \E \left| \st_{ \barho^N } (\barx_i) -   \st_\rho(\bar{x}_i) \right|^2 = \int_{\O^N} | \st_{\frac{1}{N} \sum_{j=1}^N \d_{y_j} } (y_i) - \st_\rho(y_i) |^2 \drho(y_1)\ldots \drho(y_N) \leq \a_N.
 \]
 In conclusion, we obtain
 \[
 \E \left| \sqrt{2 \st_i} - \sqrt{2 \st_\rho(\bar{x}_i)} \right|^2  \leq C\E + \a_N.
 \]

 Let us now turn to the alignment term. By adding and subtracting several intermediate terms we expand it as follows 
 \begin{equation}\label{}
\begin{split}
& \E[ (v_i - \barv_i ) \cdot   ( \st_i ( \ave{v}_i - v_i ) - \st_\rho(\bar{x}_i) ( \ave{u}_\rho(\bar{x}_i)  - \bar{v}_i )  ) ]  = - \E[ \st_i  | v_i - \barv_i  |^2] \\
 &+ \E[ (v_i - \barv_i ) \cdot (\st_i  \ave{v}_i  -  \st_\rho( \bar{x}_i) \ave{u}_\rho(\bar{x}_i) )] + 
 \E[ (v_i - \barv_i ) \cdot \barv_i ( \st_\rho( \bar{x}_i) - \st_i) ].
\end{split}
\end{equation}

The first term is non-positive, so we simply drop it.  Let us estimate the last term. We fix an $R>0$ to be determined later and split the integrand as follows
\begin{equation}\label{e:Eaux2}
\begin{split}
 \E[ (v_i - \barv_i ) \cdot \barv_i ( \st_\rho( \bar{x}_i) - \st_i) ] & =  \E[ (v_i - \barv_i ) \cdot \barv_i \one_{|\barv_i|<R}( \st_\rho( \bar{x}_i) - \st_i) ] \\
 &+ \E[ (v_i - \barv_i ) \cdot \barv_i \one_{|\barv_i|\geq R}( \st_\rho( \bar{x}_i) - \st_i) ] \\
 & \leq R^2 \E  | v_i - \barv_i  |^2 + \E \left| \st_i - \st_\rho(\bar{x}_i)\right|^2   + C \E  | v_i - \barv_i  |^2 + C\E[ |\barv_i|^2 \one_{|\barv_i|\geq R} ] ,
\end{split}
\end{equation}
where in the last term we simply used the global boundedness of the strength functions. For the second term we use the same estimate as before \eqref{e:auxsE}, hence continuing 
\[
\leq (R^2+C) \E +  \a_N +C\E[ |\barv_i|^2 \one_{|\barv_i|\geq R} ] .
\]
Now, 
\[
\E[ |\barv_i|^2 \one_{|\barv_i|\geq R} ] \leq \E^{1/2}[  |\barv_i|^4] \E^{1/2}[   \one_{|\barv_i|\geq R} ] .
\]
Here the first term is bounded by the fourth moment of $f$ which is clearly bounded from the assumption \eqref{e:expf1}. And using again \eqref{e:expf1} we estimate the last term by
\[
 \E[   \one_{|\barv_i|\geq R} ] = \int_{\O} \int_{|v| \geq R} f_t(x,v)\dv \dx \leq \frac{c_4}{e^{aR^2}}.
\]
The latter remains bounded on the interval $[0,T]$ by a constant by assumption.
We thus obtained
\begin{equation}\label{e:2ndaux}
 \E[ (v_i - \barv_i ) \cdot \barv_i ( \st_\rho( \bar{x}_i) - \st_i) ] \leq  (R^2+C) \E(t) + \a_N +c_4 e^{-a R^2/2}. 
\end{equation}

Lastly, let us estimate the second term on the right hand side of \eqref{e:Eaux2}. We have, denoting
$u^N = \sum_{j=1}^N v_j \one_{\{x_j\}} $ and $\bar{u}^N = \sum_{j=1}^N \barv_j \one_{\{\barx_j\}} $,
\begin{equation}\label{e:auxEEE}
\begin{split}
&\E[ (v_i - \barv_i ) \cdot (\st_i  \ave{v}_i  -  \st_\rho( \bar{x}_i) \ave{u}_\rho(\bar{x}_i) )] \leq  \E_v \\
+& \E \left|  \st_{\rho^N} (x_i) \ave{ u^N }_{\rho^N}(x_i)  -  \st_{\barho^N} (x_i) \ave{ \baru^N}_{\barho^N}(x_i)  \right|^2  \\
+& \E \left|  \st_{\barho^N} (x_i) \ave{ \baru^N }_{\barho^N}(x_i)  -  \st_{\barho^N} (\barx_i) \ave{\baru^N }_{\barho^N}(\barx_i)  \right|^2  \\
+&   \E \left|  \st_{\barho^N} (\barx_i) \ave{ \baru^N  }_{\barho^N}(\barx_i)   -  \st_\rho( \bar{x}_i) \ave{u}_\rho(\bar{x}_i) \right|^2  .
\end{split}
\end{equation}
The last term here is bounded by $\b_N$, see \eqref{e:LLN2}. The elements in the first term are evaluated at the same point $x_i$. So, by a similar computation as in \eqref{e:cross1} we have 
\begin{equation}\label{e:aux2.5}
\begin{split}
 &\E \left|  \st_{\rho^N} (x_i) \ave{ u^N }_{\rho^N}(x_i)  -  \st_{\barho^N} (x_i) \ave{\baru^N }_{\barho^N}(x_i)  \right|^2 \\
 \lesssim &  \E \left[\left( \frac1N \sum_{j=1}^N |\barv_j|^2 \right) W_1^2(\rho^N, \barho^N) \right] +\E\left[ W_1^2\left(  u^N \rho^N , \baru^N \barho^N \right)\right] \\
 \leq & \frac{1}{N^2} \sum_{i, j=1}^N \E[ |\barv_j|^2  |x_i-\barx_i|^2] + \frac1N \sum_{i=1}^N \E[|\barv_i|^2 |x_i-\barx_i|^2 ] +  \E_v.
\end{split}
\end{equation}
Each term here will be estimated by the same splitting method as before:
\begin{equation}\label{e:aux3}
\begin{split}
\E[ |\barv_j|^2  |x_i-\barx_i|^2  ] & = \E[|\barv_j|^2\one_{|\barv_j| < R} |x_i-\barx_i|^2 ] +\E[|\barv_j|^2 \one_{|\barv_j| \geq R} |x_i-\barx_i|^2 ] \\
& \leq R^2 \E_x + \E^{1/2}[  |\barv_j|^4] \E^{1/2}[   \one_{|\barv_j|\geq R} ] \leq R^2 \E + C e^{-a R^2/2},
\end{split}
\end{equation}
and similarly for the middle term. Thus,
\[
\eqref{e:aux2.5} \leq R^2 \E + C e^{-a R^2/2}.
\]

It remains to estimate the middle term in \eqref{e:auxEEE}. Here we use the regularity of the kernel \eqref{e:ur1} and obtain
\[
 \E \left|  \st_{\barho^N} (x_i) \ave{ \baru^N }_{\barho^N}(x_i)  -  \st_{\barho^N} (\barx_i) \ave{\baru^N }_{\barho^N}(\barx_i)  \right|^2 \lesssim \frac1N \sum_{j=1}^N \E\left[ |x_i - \barx_i|^2   | \barv_j|^2   \right].
\]
This term becomes exactly as the previous one. So, the same estimate applies.

Putting the above estimates together and denoting $r= a R^2/2$ and $\g_N = \a_N + \b_N$, we arrive at
\begin{equation}\label{e:Efinal}
\ddt \E \leq C_1 (r+1) \E + C_2 \g_N + C_3  e^{- r}.
\end{equation}

Inequality \eqref{e:Efinal} is exactly the one that appeared in \cite{BCC2011}. Let us recap the conclusion for completeness. First, by choosing $r=1$ we see that $\cE$ remains uniformly bounded on $[0,T]$, $\E \leq \E_0$. Thus,  $-\ln (\E/e\E_0) \geq 1$. Denoting $v = \E/e\E_0$ and picking $r = -\ln v$ we obtain
\[
v' \leq - c_1 v \ln v + c_2 \g_N  \leq -c v \ln v + c \g_N,
\]
where $c = \max\{c_1,c_1\}$. Rescaling time $u(t) = v(t/c)$ we further obtain
\[
u' \leq - u \ln u + \g_N.
\]
Letting $w = u \g_N^{-e^{-t}}$ we conclude
\[
w' \leq - w \ln w + 1 \leq e^{-1} + 1.
\]
Thus, $w \leq T(e^{-1} + 1) = C_T$ and hence unwrapping the notation, $\E \leq C_1 \g_N^{e^{-C_2 t}}$ as claimed.

\end{proof}

\section{Fokker-Planck-Alignment equation}\label{s:WP}

 In this section we develop a well-posedness theory of classical solutions to  Fokker-Plank-Alignment equations \eqref{e:FPA} that is suitable for applications to flocking. This means that in addition to the standard regularity questions we will pay close attention to thickness as related to the spectral gap computations  discussed in \sect{s:flocking}. We will restrict ourselves to the periodic domain $\O = \T^n$ as that is the setting where most of our results will be used in the sequel.   We also set $\s=1$ as it plays no role in the analysis. So, we consider the FPA equation 
\begin{equation}\label{e:FPA1}
\p_t f + v\cdot \n_x f =  \st_\rho \D_v f + \n_v \cdot (\st_\rho(v - \ave{u}_\rho)f ).
\end{equation}

Classical solutions to \eqref{e:FPA1} are defined to be solutions that belong to a  high regularity weighted Sobolev class. For reasons that will be clarified later it is essential to distribute velocity weights in the manner defined as follows
\begin{equation}\label{e:Sobdef}
H^k_l(\domain) =  \left\{ f :  \sum_{k'\leq k} \sum_{|\bk'| = k'}  \int_\domain \jap{v}^{l + 2(k-k')}  |\p^{\bk'}_{x,v} f |^2 \dv\dx <\infty \right\},
\end{equation}
where $\jap{v} = (1+|v|^2)^{\frac12}$.  Some remarks are in order to elaborate on this choice. First, we note that the alignment term in  \eqref{e:FPA} prevents the persistence of a sub-Gaussian bound $f \leq C \mu$ if it holds initially. So, setting the problem traditionally in sub-Gaussian H\"older classes, c.f.  \cite{IM2021,AZ2021}, is not natural for the FPA equations. One exception is the class of perturbative solutions developed for particular models in \cite{Choi2016,DFT2010}.  Inclusion of the weights in \eqref{e:Sobdef} is necessary to achieve uniqueness primarily due to, again, the presence of alignment components, see however \cite{Villani} for the classical much weaker result.  The use of progressively increasing weights for lower order terms is required to control terms coming from the inhomogeneity in front of the Fokker-Planck operator, which prevents closing a priori estimates for any single-weight choice.  Single weight spaces, however, would have been sufficient for models with $\st_\rho=1$.

\subsection{Local well-posedness}\label{s:lwp}

Let us first discuss local well-posedness for thick data on compact domain.

\begin{theorem}\label{t:lwp}
Suppose the model $\cM$ is regular in the sense of \defin{d:r}, and $\O = \T^n$. Let $f_0\in H^k_l(\domain)$, $k,l \geq n+3$, be an initial condition such that 
\[
\Th(\rho_0,\O) > 0.
\]
Then there exists a unique local solution to \eqref{e:FPA1} on a time interval $[0,T)$, where $T >0$ depends only on the initial energy $\cE_0$ and thickness $\Th(\rho_0,\O)$,  in the regularity class
\begin{equation}\label{e:solclass}
f\in C_w ([0,T); H^k_l ), \quad  \n_{v} f \in L^2([0,T]; H^{k}_l).
\end{equation}
Moreover, if $f\in L^\infty_\loc ([0,T); H^k_l)$ is a given solution such that
\begin{equation}\label{e:rholow}
\inf_{[0,T)}\Th(\rho,\O) >0,
\end{equation}
then $f$ can be extended to an interval $[0,T+\e)$ in the same class.
\end{theorem}

We can view the right hand side of \eqref{e:FPA1}  as a sum of a weighted Fokker-Planck operator and a smooth drift
\begin{equation}\label{e:FPA2}
\p_t f + v\cdot \n_x f  =  \st_\rho \LFP f+ \rmw_\rho \cdot \n_v f,
\end{equation}
where 
\begin{equation}\label{ }
\LFP f = \n_v \cdot (\n_v f + v f), \quad \rmw_\rho = - \st_\rho \ave{u}_\rho.
\end{equation}
Let us first disassociate the weights $\st_\rho$ and $\rmw_\rho$ from the solution and consider the linear problem
\begin{equation}\label{e:FPAlinear}
\p_t f + v\cdot \n_x f =  \st(x,t) \LFP f+ \rmw(x,t) \cdot \n_v f,
\end{equation}
where  $\st,\rmw$ is a given smooth set of data on  $\O \times [0,T]$ with uniform bounds
\begin{equation}\label{e:sbelow1}
\st \geq c_0>0, \quad  \| \st\|_{C^k} + \| \rmw \|_{C^k} <C_0 \qquad  \text{ on } \O \times [0,T].
\end{equation}

\begin{lemma}\label{l:llwp}
Under the assumptions \eqref{e:sbelow1} for any initial condition $f_0 \in H_l^k$ there exists a unique solution to \eqref{e:FPAlinear} on $[0,T]$ with $f\in C_w([0,T]; H^k_l)$, $\n_{v} f \in L^2_\loc([0,T]; H^{k}_l)$, and moreover,
\[
\|f\|_{H^k_l} \leq \|f_0\|_{H^k_l} e^{C t},
\]
where $C$ depends only on $c_0$ and $C_0$.
\end{lemma}
\begin{proof}
To construct a solution to \eqref{e:FPAlinear}  from initial data $f_0 \in H_l^k$ one first considers a fully viscous regularization 
\begin{equation}\label{e:FPAe}
\p_t f + v\cdot \n_x f =  \st \LFP f + \rmw \cdot \n_v f + \e \D_{x,v} f.
\end{equation}
A local solution to \eqref{e:FPAe} on a time interval $[0,T_\e]$ is obtained via the standard fixed point argument, see \cite{Krylov-book}. In order to extend it to all of $[0,\infty)$ we provide a priori estimates for \eqref{e:FPAlinear} which automatically apply to \eqref{e:FPAe} independently of $\e$. As a result we obtain a bound on $\|f\|_{H^k_l}$ which depends only on its initial value, on $c_0$ and $C^k$-norms of $\st$ and $\rmw$.  So, we have a family of solutions $f^\e$ uniformly in $C([0,T];H^k_l)$,  and clearly also in $f_t^\e \in L^\infty([0,T];L^2)$. By the Aubin-Lions compactness lemma we can pass to the limit $\e\to 0$ in any $H^{k'}_{l'}$ for $k'<k$, $l'<l$ and weakly in $H^k_l$ extracting a subsequence converging to a solution to \eqref{e:FPAlinear}.  Weak continuity in $H^k_l$ also follows classically.

 Thus, the problem reduced to obtaining proper a priori bounds for solutions to \eqref{e:FPAlinear}.

Let us estimate the top $v$-derivative $\p_v^\bk f$ first (here and further on we use a less formal notation for the partials, only keeping track of the order)
\[
\p_t \p_v^\bk f + v\cdot \n_x  \p_v^\bk f + \p_v^{\bk-1} \p_x f = \st \LFP \p_v^\bk f +  \st  \p_v^{\bk} f + \rmw \cdot \n_v \p_v^\bk f.
\]
Testing with $ \jap{v}^l \p_v^\bk f$ we obtain
\begin{equation*}\label{}
\begin{split}
\ddt &  \int_\domain  \jap{v}^l    | \p_v^{\bk} f |^2\dv \dx +\frac12  \int_\domain \jap{v}^l  v\cdot \n_x | \p_v^\bk f |^2 \dv \dx +  \int_\domain \jap{v}^l \p_v^{\bk-1} \p_x f \p_v^{\bk}  f \dv \dx \\
 = & -  \int_\domain  \st \jap{v}^l    |\p_v^{\bk+1} f |^2 \dv \dx + \frac12  \int_\domain \st \D_v( \jap{v}^l)  | \p_v^{\bk} f |^2  \dv \dx\\
& + \int_\domain \st \p_v(\jap{v}^l v) | \p_v^{\bk} f |^2\dv \dx + \int_\domain \st \jap{v}^l| \p_v^{\bk} f |^2\dv \dx - \int_\domain \rmw \cdot \n_v ( \jap{v}^l )    | \p_v^{\bk} f |^2\dv \dx
\end{split}
\end{equation*}
While the first integral on the left hand side vanishes, the second is bounded by $\|f\|_{H^k_l}^2$. All the terms on the right hand side, using that $\p_v \jap{v}^p \lesssim \jap{v}^{p-1}$, are also bounded by $\|f\|_{H^k_l}^2$ except for the dissipation which has a uniform bound from below by \eqref{e:sbelow1}. Thus,
\[
\ddt   \int_\domain  \jap{v}^l    | \p_v^{\bk} f |^2\dv \dx \leq C \|f\|_{H^k_l}^2 - c_0  \int_\domain  \jap{v}^l    |\p_v^{\bk+1} f |^2 \dv \dx,
\]
where $C$ is a constant depending on all the $L^\infty$-norms of $\st, \rmw$.

Let us now estimate the rest of  the other top derivatives $\p_v^{\bk - \bk'} \p_x^{\bk'} f$, $\bk' >\bzero$. Sparing the tedious details, most of terms are all bounded by $C \|f\|_{H^k_l}^2$, where $C$  is an upper bound for the $C^k$-norms  of $\st, \rmw$.  The rest is given by

\begin{equation*}\label{}
\begin{split}
\ddt &  \int_\domain  \jap{v}^l    | \p_v^{\bk - \bk'} \p_x^{\bk'} f|^2\dv \dx \lesssim C \|f\|_{H^k_l}^2  -  c_0 \int_\domain   \jap{v}^l    |\p_v^{\bk - \bk'+1} \p_x^{\bk'} f |^2 \dv \dx \\
 + & \int_\domain \jap{v}^l \p_x \st \,  \D_v \p_v^{\bk - \bk'} \p_x^{\bk'-1} f  \p_v^{\bk - \bk'} \p_x^{\bk'} f  \dv \dx + \int_\domain \p_x \st \jap{v}^l  \,  v \cdot \n_v \p_v^{\bk - \bk'} \p_x^{\bk'-1} f  \p_v^{\bk - \bk'} \p_x^{\bk'} f  \dv \dx .
\end{split}
\end{equation*}
For the penultimate term we have
\begin{equation*}\label{}
\begin{split}
 \int_\domain \jap{v}^l \p_x \st \,  \D_v \p_v^{\bk - \bk'} \p_x^{\bk'-1} f  \p_v^{\bk - \bk'} \p_x^{\bk'} f  \dv \dx & \leq -  \int_\domain \jap{v}^l \p_x \st \,  \n_v \p_v^{\bk - \bk'} \p_x^{\bk'-1} f \cdot \n_v  \p_v^{\bk - \bk'} \p_x^{\bk'} f  \dv \dx \\
& \hspace{2in} + C\|f\|_{H^k_l}^2\\
 & = -  \frac12 \int_\domain \jap{v}^l \p_x \st \,  \p_x |\n_v \p_v^{\bk - \bk'} \p_x^{\bk'-1} f|^2 \dv \dx  + C\|f\|_{H^k_l}^2\\
 & =   \int_\domain \jap{v}^l \p^2_x \st \,  | \p_v^{\bk - \bk'+1} \p_x^{\bk'-1} f|^2 \dv \dx  + C\|f\|_{H^k_l}^2 \\
 &\leq C\|f\|_{H^k_l}^2.
\end{split}
\end{equation*}

In the remaining term we take advantage of the dissipation and the higher weight assigned to the lower order derivatives. Integrating by parts in $v$ we have
\begin{equation*}\label{}
\begin{split}
& \int_\domain \p_x \st \jap{v}^l  \,  v \cdot \n_v \p_v^{\bk - \bk'} \p_x^{\bk'-1} f  \p_v^{\bk - \bk'} \p_x^{\bk'} f  \dv \dx\\
 &= -  \int_\domain \p_x \st \n_v\cdot (\jap{v}^l   v) \p_v^{\bk - \bk'} \p_x^{\bk'-1} f  \p_v^{\bk - \bk'} \p_x^{\bk'} f  \dv \dx 
-  \int_\domain \p_x \st \jap{v}^l  \, \p_v^{\bk - \bk'} \p_x^{\bk'-1} f \, v \cdot \n_v \p_v^{\bk - \bk'} \p_x^{\bk'} f   \dv \dx \\
& \leq  C\|f\|_{H^k_l}^2 + \frac{c_0}{2} \int_\domain   \jap{v}^l    |\p_v^{\bk - \bk'+1} \p_x^{\bk'} f |^2 \dv \dx + C  \int_\domain \jap{v}^{l+2} | \p_v^{\bk - \bk'} \p_x^{\bk'-1} f|^2  \dv \dx \\
& \leq C\|f\|_{H^k_l}^2 + \frac{c_0}{2} \int_\domain   \jap{v}^l    |\p_v^{\bk - \bk'+1} \p_x^{\bk'} f |^2 \dv \dx.
\end{split}
\end{equation*}
As a result we obtain
\[
\ddt   \int_\domain  \jap{v}^l    | \p_v^{\bk - \bk'} \p_x^{\bk'} f|^2\dv \dx \leq C\|f\|_{H^k_l}^2 -  \frac{c_0}{2} \int_\domain   \jap{v}^l    |\p_v^{\bk - \bk'+1} \p_x^{\bk'} f |^2 \dv \dx.
\]

The same argument works to estimate any positive order derivatives, each time taking advantage of the higher weight put on one order below. It remains to estimate the zeroth-order term,
\begin{equation}\label{e:0thorder}
\begin{split}
\ddt &  \int_\domain  \jap{v}^{l+2 k}    |f|^2\dv \dx =  \frac12 \int_\domain \st \D_v \jap{v}^{l+2k} f^2  \dv \dx -  \int_\domain \st \jap{v}^{l+2k} | \n_v  f |^2  \dv \dx \\
- &\int_\domain \st \n_v \jap{v}^{l+2k} \cdot v  f^2 \dv \dx - \int_\domain \st \n_v\cdot(\jap{v}^{l+2k} v)  f^2  \dv \dx - \int_\domain  (\rmw \cdot \n_v \jap{v}^{l+2k} )  f^2  \dv \dx\\
& \leq C \int_\domain  \jap{v}^{l+2k}    |f|^2\dv \dx - c_0 \int_\domain \jap{v}^{l+2k} | \n_v  f |^2  \dv \dx.
\end{split}
\end{equation}
So, the estimate on the $0$-th order term closes on itself. 

We obtained
\[
\ddt \|f\|_{H^k_l}^2\leq C\|f\|_{H^k_l}^2 - \frac{c_0}{2} \|\n_v f\|_{H^k_l}^2 ,
\]
and the estimate stated in the lemma follows. It also proves uniqueness since the equation is linear.

\end{proof}

\begin{proof}[Proof of \thm{t:lwp}] To construct solutions to the fully non-linear problem \eqref{e:FPA2} we use iteration scheme based on solving a sequence of linear problems
\begin{equation}\label{e:FPAn}
\begin{split}
\p_t f^{m+1} + v\cdot \n_x f^{m+1}  & =  \st_{\rho^m} \LFP f^{m+1} + \rmw_{\rho^m} \cdot \n_v f^{m+1},\\
f^{m+1} (0) & = f_0,
\end{split}
\end{equation}
for $m = 0 ,1, \ldots$. In order to pass to the limit as $m \ra \infty$ we need to ensure that $f^{m+1}$ remain uniformly bounded in $H_l^k$ on a fixed time interval $[0,T]$. According to \lem{l:llwp} a bound on $f^{m+1}$ depends on smoothness of $\st_{\rho^m}$ and $\rmw_{\rho^m}$ and a lower bound on $\st_{\rho^m}$. Thanks to the regularity of $\cM$ and \eqref{e:stlow}, those can be controlled by the thickness $\Th(\rho^m,\O)$  and the energies
\[
\cE^m = \frac12 \int_\domain |v|^2 f^m \dv \dx.
\]

Let us show that there exists a common time interval $[0,T]$ on which all the energies are uniformly bounded and the all the densities $\rho^m$ are uniformly thick. 

Starting with the energy, testing \eqref{e:FPAn} with $\frac12 |v|^2$ we can see that the Fokker-Planck component yields a bound $\|\st_{\rho^m}\|_\infty \cE^{m+1} \leq \oS \cE^{m+1}$. Let us denote $2 \d = \Th(\rho_0,\O)$. Assuming for a moment that the $m$-th flock remains thick $\Th(\rho^m,\O) \geq \d$ then using the bound
\[
\| \rmw_{\rho^m} \|_\infty \leq C_0(\d) \sqrt{\cE^m},
\]
we obtain 
\begin{equation*}\label{}
\begin{split}
\ddt \cE^{m+1} & \leq C - \int_\domain \rmw^n \cdot v f^{m+1} \dv \dx  \\
& = C +  \int_\O \st^m \ave{u^m}_{\rho^m} \cdot u^{m+1} \rho^{m+1}  \dx \\
& \leq C  + C_0(\d) \sqrt{\cE^m} \int_\O |u^{m+1}| \rho^{m+1} \dx \leq  C + C_0(\d)\sqrt{\cE^m\cE^{m+1}}\\
& \leq C +  C_0(\d) \max\{ \cE^m, \cE^{m+1} \}.
\end{split}
\end{equation*}
Hence, denoting $\bar{\cE}^m = \max\{ \cE^0, \cE^1,\ldots, \cE^m\}$, we obtain from the above
\begin{equation}\label{e:En+1}
\ddt \bar{\cE}^{m+1} \leq C +  C_0(\d) \bar{\cE}^{m+1}.
\end{equation}
At the same time, by \eqref{e:contiv} 
\begin{equation}\label{e:rn+1}
\p_t \Th (\rho^m,\O) \geq - c \| u^m \|_{L^2(\rho^m)} \geq -c \sqrt{\bar{\cE}^{m+1}}.
\end{equation}

Let us argue by induction. The initial interval of existence for $m=0$ is $T_0=\infty$. On that interval $\Th (\rho_0,\O) = 2\d >\d$. Then from \eqref{e:En+1} we have
\[
\bar{\cE}^1 \leq \cE_0 e^{C_0(\d) t} + C_1(\d)  e^{C_0(\d) t}.
\]
So, for $t \leq \frac{\ln 2 }{ C_0(\d)} $ we have
\[
\bar{\cE}^1(t) \leq 2\cE_0 + 2C_1(\d).
\]
Using \eqref{e:rn+1} we conclude on the same time interval (recall that $\rho^m(0) = \rho^0$ initially)
\[
 \Th (\rho^1,\O)\geq  \Th (\rho_0,\O) - t C_1(\d) \sqrt{ 2\cE_0 + 2C_1(\d)} \geq 2 \d - t C_1(\d) \sqrt{ 2\cE_0 + 2C_1(\d)}.
\]
Consequently, for $t< \frac{\d}{  C_1(\d) \sqrt{2\cE_0 + 2C_1(\d)}}$  we have $ \Th (\rho^1,\O) \geq \d$.

Setting $T = \min\{\frac{\ln 2 }{ C_0(\d)}  , \frac{\d}{  C_1(\d) \sqrt{2\cE_0 + 2C_1(\d)}}\}$ we obtain exact same estimates for the next elements in the sequence:
\[
\bar{\cE}^2(t) \leq 2\cE_0 + 2C_1(\d), \ t<T
\]
and 
\[
 \Th (\rho^2,\O) \geq  2 \d - t C_1(\d) \sqrt{ 2\cE_0 + 2C_1(\d)} \geq \d, \ t<T.
\]
Continuing in the same manner it follows that  $\cE^m \leq  2\cE_0 + 2C_1(\d)$ and $\Th (\rho^m,\O) \geq \d$ on the same time interval $[0,T]$ for all $m \in \N$. 

\lem{l:llwp} implies that each solution in the sequence $f^m$ will exist and be uniformly bounded  in class $C_w([0,T]; H^k_l)$. By compactness we conclude that there exists a converging subsequence in any lower regularity class, and that the limit solves the equation \eqref{e:FPA2} classically by continuity properties of the model \eqref{e:r2}.

From the above we see that the local time of existence $T$ depends only on the initial energy $\cE_0$ and thickness $\d$. With this observation let us assume that we are given a solution on an interval $[0,T')$ in the Sobolev class $H^k_l$ and such that \eqref{e:rholow} holds for all $t<T'$. Then the estimate analogous to \eqref{e:En+1} shows that the energy $\cE(t)$ remains bounded on $[0,T')$ by a constant depending only on $\d$ and $\cE_0$.  Starting from $T'-\e$ where $\e>0$ is small we construct a solution on a time interval $[T'-\e, T'-\e + T)$ where $T$ depends only on $\d$ and $E_0$ and not on $\e$. This extends the solution beyond $T'$ by uniqueness, which we address next.

\def \tf {\tilde{f}}
\def \tu {\tilde{u}}
\def \tuF {\tilde{u}_{\mathrm{F}}}
\def \tr {{\tilde{\rho}}}

Let us have two thick solutions $f$ and $\tf$ in class \eqref{e:solclass} starting from the same initial condition $f_0$. Denote $g = f - \tf$. We will estimate evolution of this difference in the weighted class $L^2_l = H^0_l$ and show that estimates close if $l$ is large enough. Note that according to definition of $H^k_l$ for $k$ large as assumed, we have $\n_v f, \n_v \tf, \LFP f, \LFP \tf \in L^2_l $ uniformly. 

 Let us take the difference
\[
\p_t g + v \cdot \n_x g = \st_{\rho} \LFP g + (\st_\rho - \st_\tr) \LFP \tf + \rmw_\rho \cdot \n_v g + (\rmw_\rho - \rmw_\tr) \cdot \n_v \tf.
\]
Testing with $\jap{v}^l g$ and integrating the $x$-transport term drops out. The rest of the terms are estimated using continuity assumption \eqref{e:r2} and the usual energy estimates
\begin{equation*}\label{}
\ddt \|g\|_{L^2_l}^2  \lesssim  \|g\|_{L^2_l}^2 + \| \rho - \tr \|_1 \|\LFP \tf \|_{L^2_l}  \|g\|_{L^2_l} + (\| \rho  u- \tr \tu \|_1  +\| \rho - \tr \|_1 \cE) \|\n_v \tf \|_{L^2_l}  \|g\|_{L^2_l}.
\end{equation*}
Here, we replaced the $W_1$-metrics with $L^1$ since this is not essential.  The $\tf$ components and the energy are uniformly bounded as noted above. So, we have
\begin{equation*}\label{}
\ddt \|g\|_{L^2_l}^2  \lesssim  \|g\|_{L^2_l}^2 + \| \rho - \tr \|_1 \|g\|_{L^2_l} + \| \rho  u- \tr \tu \|_1   \|g\|_{L^2_l}
\end{equation*}

Now,
\begin{equation}\label{ }
 \| \rho - \tr \|_1 \leq \int_\domain |g| \dv \dx = \int_\domain \jap{v}^{l/2} |g| \jap{v}^{-l/2}\dv \dx  \lesssim \|g\|_{L^2_l},
\end{equation}
provided $l > n$. Similarly,
\begin{equation}\label{ }
 \| \rho  u- \tr \tu \|_1 \leq \int_\domain |v| |g| \dv \dx = \int_\domain \jap{v}^{l/2} |g| \jap{v}^{-l / 2  + 1}\dv \dx  \lesssim \|g\|_{L^2_l},
\end{equation}
provided $l  > n +2$.

So,  we arrive at
\begin{equation*}\label{}
\ddt \|g\|_{L^2_l}^2  \lesssim  \|g\|_{L^2_l}^2,
\end{equation*}
and uniqueness follows.

\end{proof}

\subsection{Spread of positivity} \label{s:spread}

In order to extend local solutions globally it is clear that we have to generate a lower bound on the macroscopic density. Since regularity of the local solution can deteriorate propagation of the thickness is impossible to prove with the local existence estimates. Instead we resort to what is called spread of positivity. 

Solutions to many kinetic equations tend to develop instantaneous spread of support  across the domain, in the sense of gaining a Gaussian bound
\begin{equation}\label{e:gaussab}
f(t,x,v) \geq b e^{-a |v|^2},
\end{equation}
see \cite{AZ2021,Kolm1934,Hormander1967,FranPoli2006,GIMV2019,AZ2021,DesVill2000,HST2020,Mouhot2005,IMS2020}. The constants $a,b$, however, depend on either the regularity of the solution on a given time interval or bounds on macroscopic quantities such as the mass-density, energy-density and entropy-density. 
Such bounds may deteriorate in time which puts constants $a,b$ in dependence on time as well. With a view towards  flocking and regularity the primary purpose of a bound like \eqref{e:gaussab} would be to translate into a global lower bound on the density $\rho \geq \rho_-$ dependent only on the basic quantities such as drift and entropy.  At this point we are essentially using the advantage of a compact environment. 
As a consequence, for those models where the drift and entropy can be controlled in time we can develop global existence and relaxation results. 

So, our primary goal in this section will be to establish the Gaussian bound \eqref{e:gaussab} with parameters that depend only on the entropy/energy and the drift. 

\begin{proposition}\label{p:Gauss}
For a given classical solution $f\in C_w ([0,T); H^k_l (\T^n))$ of \eqref{e:FPA1} on a  time interval $[0,T)$ there exists $a,b>0$ which depend only on the parameters of the model $\cM$, time $T$, and 
\begin{equation*}\label{}
\begin{split}
\oW & = \sup_{t\in [0,T)} \| \st_\rho \ave{u}_\rho \|_\infty,\\
\oH &=  \sup_{t\in [0,T)}  \int_{\T^n \times \R^n} |v|^2 f \dv \dx + \int_{\T^n \times \R^n}  f | \log f| \dv \dx,
\end{split}
\end{equation*}
such that
\begin{equation}\label{e:Gauss}
f(t,x,v) \geq b e^{-a |v|^2} , \qquad \forall (t,x,v) \in \T^n \times\R^n \times [T/2,T).
\end{equation}
\end{proposition}

Central to our proof will be the weak Harnack inequality proved in \cite{GI2021}. To state it we need to introduce some notation.  

We will be looking at solutions on kinetic cylinders defined by, for $z_0 = (t_0,x_0,v_0) \in \R \times \R^n \times \R^n$,
\[ 
Q_r(z_0) = \{z:  - r^2 < t - t_0 \leq 0,\ |x- x_0 - (t-t_0)v_0 | <r^3,\ |v-v_0| < r \}.
\]
One can define the Lie-group action on triplets $z$ by
\[
z_0 \circ z = (t_0+t, x_0 + x + tv_0, v_0+v).
\]
Then 
\[
z^{-1} = (- t, -x +t v, -v).
\]
And we define the kinetic multiplication by a scalar as
\[
r z = (r^2 t, r^3 x, r v).
\]
The cylinders $Q_r(z_0)$ can then be considered as the shift and rescaling of the $0$-centered cylinder $Q_r = Q_r(0)$
\[
Q_r(z_0) = z_0 \circ  Q_r.
\]
And by scaling, $rQ_1 = Q_r$.

We consider super-solutions to the following general Fokker-Planck equation
\begin{equation}\label{e:FPsuper}
\p_t f + v \cdot \n_x f \geq \n_v \cdot (\rmA \n_v f) + \rmB \cdot \n_v f.
\end{equation}
The equation has natural scaling invariance. If 
\[
f_{r,z_0}(z) = f( z_0 \circ rz),
\]
then $f_{r,z_0}$ satisfies
\begin{equation}\label{e:FPrescale}
\p_t f_{r,z_0}+ v \cdot \n_x f_{r,z_0} \geq \n_v \cdot(\rmA_{r,z_0} \n_v f_{r,z_0}) + \rmB_{r,z_0} \cdot \n_v f_{r,z_0},
\end{equation}
where 
\begin{equation}\label{e:ABrescale}
\rmA_{r,z_0} (z) = \rmA( z_0 \circ rz), \quad \rmB_{r,z_0} (z) = r \rmB ( z_0 \circ rz).
\end{equation}
Thus, the following rule applies:
\begin{claim} If $f$ solves \eqref{e:FPsuper} on $Q_{r'}(z')$, then $f_{r,z_0}$ solves the rescaled equation \eqref{e:FPrescale} on $Q_{r'/r}((z_0^{-1} \circ z')/r) $. Furthermore, if $A,B$ satisfy 
\begin{equation}\label{e:ABhypo}
\l \I \leq \rmA \leq \L \I, \qquad | \rmB| \leq \L, \quad z\in Q_{r'}(z'),
\end{equation}
for some $\l, \L >0$, then the new coefficients  $\rmA_{r,z_0}, \rmB_{r,z_0}$ satisfy the same bounds on $Q_{r'/r}((z_0^{-1} \circ z')/r)$, provided $r<1$.
\end{claim}

For $\w >0$ small let us introduce the following two non-overlapping cylinders
\[
Q^+_\w =Q_\w(1,0,0), \qquad Q^-_\w = Q_\w (\w^2, 0,0).
\] 
That is, $Q^+_\w $ is attached to the top of the basic cylinder $Q_1(1,0,0)$, and $Q^-_\w $ is lying on the bottom.

\begin{theorem}[Weak Harnack inequality, \cite{GI2021}] \label{t:wHarnack} There are constants $R_0,\w_0, p,C_0>0$ which depend only on $\l,L,n$ satisfying the following property. If $f$ is a super-solution to \eqref{e:FPsuper}, in the cylinder $Q^{R_0} = [0,1] \times \{ |x| <R_0\} \times \{|v| < R_0\}$  with $A,B$ satisfying \eqref{e:ABhypo}, then whenever 
\[
\left( \int_{Q_{\w_0}^-} f^p \dz  \right)^{1/p} \leq C_0 \inf _{Q_{\w_0}^+} f.
\]
\end{theorem}

We now turn to proving \prop{p:Gauss}. The proof goes in several steps.  First, we rewrite the FPA \eqref{e:FPA2} as follows
\begin{equation}\label{}
\p_t f + v\cdot \n_x f  =  \st_\rho \D_v f+  (\st_\rho v + \rmw_\rho) \cdot \n_v f + n \st_\rho f.
\end{equation}
Since the last term is non-negative, $f$ is a super-solution to the truncated equation
\begin{equation}\label{e:FPAsuper}
\p_t f + v\cdot \n_x f  \geq  \st_\rho \D_v f+  (\st_\rho v + \rmw_\rho) \cdot \n_v f,
\end{equation}
which has the structure of \eqref{e:FPsuper}. We will be mindful of the fact, however, that $\rmB = \st_\rho v + \rmw_\rho$ is unbounded in $v$, and this will be taken into account in due course.

In the subsequent course of the proof the various constants denoted 
\[
c_0,c_1, \ldots ,\ \w_0, \w_1, \ldots ,\ T_0,T_1,\ldots,\ r_0,r_1,\ldots, \ R_0,R_1, \ldots
\]
 depend only on the parameters of the model, $T$, and $\oW, \oH$.  We call such constants {\em admissible}.
 
\smallskip
\noindent{\sc Step 1: choosing domain of ellipticity.}  Let us recall from \eqref{e:stlow} that the strength function is supported from below by a measure of ball-thickness at scale $r_0$ across the domain $\O = \T^n$.  Since $\T^n$ has finite volume by a covering argument, there exists a constant $c_1$ depending on $n$,  and  there exists $x'\in \T^n$ such that 
\begin{equation}\label{e:r04}
\orho_{r_0/4}(0, x') \geq c_1,
\end{equation}
Consequently,
\[
\orho_{r_0}(0, x) \geq c_2 , \quad \forall x\in B_{r_0/2}(x').
\]
Next, notice that $\orho_{r_1}$ satisfies the following equation
\[
\p_t\orho_{r_0}= - \n_x \cdot ( u \rho)_{\chi_{r_0}} = - (u \rho)_{\n \chi_{r_0} } \geq - c_3 \|u\|_{L^2(\rho)} \geq - c_3 \oH.
\]
So, for any $t > 0$, and any $ x\in B_{r_0/2}(x')$, we have
\[
\orho_{r_0}(t, x)  \geq c_2 -  t c_3 \oH.
\]
This implies that on the time interval $t\in [0 , T_1]$, where $T_1 = \left( \frac{ c_2 }{ 2 c_3 \oH } \right) \wedge T$, we have
\[
\orho_{r_0}(t, x)  \geq c_2 / 2,\quad \forall x\in B_{r_0/2}(x'),
\]
and in view of \eqref{e:stlow}, 
\[
\st_{\rho}(t,x) \geq \st(c_2/2) = \l, \quad \forall (t,x) \in [0 , T_1] \times B_{r_0/2}(x').
\]

Let us come back to \eqref{e:r04} and extract a thick subdomain for $f$ not too far in $v$-direction. We have
\[
\begin{split}
c_1  & \leq \int_{ B_{r_0/4}(x') } \int_{\R^n} f(0, x,v) \dv \dx\\
&  =  \int_{ B_{r_0/4}(x') } \int_{|v|<R} f(0, x,v) \dv \dx +  \int_{ B_{r_0/4}(x') } \int_{|v| \geq R} f(0, x,v) \dv \dx\\
& \leq \int_{ B_{r_0/4}(x') } \int_{|v|<R} f(0, x,v) \dv \dx + \frac{\oH}{R^2}.
\end{split}
\] 
So, for $R= R_1 = 1 \vee \sqrt{\frac{2\oH}{c_1}}$ we have
\begin{equation}\label{e:fmass}
\int_{ B_{r_0/4}(x') } \int_{|v|<R_1} f(0, x,v) \dv \dx  \geq \frac{c_1}{2} = c_3.
\end{equation}

Let us define our domain of ellipticity $\overline{\O} = [0,T_1] \times B_{r_0/2}(x') \times B_{2R_1}(0)$, where we have
\begin{equation}\label{e:ellbounds}
\l \leq \st_\rho \leq \L,\qquad  |\st_\rho v| + |\rmw_\rho| \leq \L,
\end{equation}
where $\L = \max\{ \oS, 2\oS R_1+ \overline{W}\}$, and $\oS$ is the common bound on the strength function by (ev4). The constants $\l,\L$ determine $R_0,\w_0, p,C_0>0$ from \thm{t:wHarnack}, which depend only on $\l,\L,n$, so they are admissible.  

\smallskip
\noindent{\sc Step 2: finding the initial plateau.}

We want to find a center of inflation $(0, x_0,v_0)$ in such a way that the point $(x_0,v_0)$ lies within the interior subdomain $B_{r_0/4}(x') \times B_{R_1}(0)$ and a small $\w$-cylinder around it has a substantial presence of $f$. That cylinder will be blown into $Q_{\w_0}^-$ resulting in $f$ having a substantial $L^p$-mass in it.  At the same time the domain of ellipticity $\oO$ will be blown to engulf the needed wide cylinder $Q^{R_0}$ to fulfill the assumptions of \thm{t:wHarnack}. The theorem then applies to obtain an admissible lower bound on $f$ at a later time. 

Thanks to \eqref{e:fmass} by the standard covering argument, for any small $\w$ one can find a point $(x_0,v_0)\in B_{r_0/4}(x') \times B_{R_1}(0)$ such that 
\begin{equation}\label{e:inw}
\int_{ B_{\w^3}(x_0) \times B_{\w}(v_0)} f(0, x,v) \dv \dx  \geq  c_3 c_4 |B_{\w^3}(x_0) \times B_{\w}(v_0)| = c_5 \w^{4n},
\end{equation}
We will choose $\w$  later. Let us prove now that the initial weight of $f$ in a cylinder as in \eqref{e:inw} stretches in time on the natural scale $\w^2$.

\begin{lemma}\label{}
Suppose initially
\[
\int_{ B_{\w^3}(x_0) \times B_{\w}(v_0)} f(0, x,v) \dv \dx  \geq c_5 \w^{4n}.
\]
Then
\begin{equation}\label{e:fwc6}
\int_{ Q_{2 \w}(4 \w^2, x_0,v_0) } f(z) \dz \geq c_6 \w^{8n+2}.
\end{equation}
\end{lemma}
\begin{proof}
Let us fix a smooth cut-off function $h(r) = \one_{r<1}$ and $h(r) = 0$ for $r\geq 2$, bounded by $1$.  Let
\[
h_\w(x,v) = h(x/ \w^3) h(v/\w).
\]
Define the kinetic convolution
\[
g(t) = \int_{\domain} f(t,x_0+x+ t v_0, v+v_0) h_\w (x,v) \dv \dx
\]
Then initially, $g(0) \geq c_5 \w^{4n}$. Let us compute the derivative
\[
\ddt g = \int_{\domain} (\p_t f + v_0 \cdot \n_x f) h_\w \dv \dx = \int_{\domain} (\p_t f + (v_0 + v) \cdot \n_x f) h_\w  \dv \dx -  \int_{\domain} v  \cdot \n_x f  h_\w\dv \dx.
\]
Note that 
\[
\left| \int_{\domain} v  \cdot \n_x f  h_\w\dv \dx \right| = \left| \int_{\domain}  f v  \cdot \n_x h_\w\dv \dx \right| \leq c \| \n h \|_\infty  \w^{-2}.
\]
So,
\begin{equation}\label{}
\begin{split}
\ddt g & \geq \int_{\domain} [ \st_\rho\D_v f + (\st_\rho (v+v_0) + w_\rho)\cdot \n_v f ]h_\w\dv \dx - c \w^{-2} \\
& =  \int_{\domain} f [ \st_\rho  \D_v h_\w - n  \st_\rho  h_\w - (\st_\rho (v+v_0) + w_\rho) \cdot \n_v h_\w ]\dv \dx - c \w^{-2} \\
& \geq - \L \w^{-2}  - n \L - (\L 2R_1 + \oW) \w^{-1} - c \w^{-2} \geq - c_7  \w^{-2}.
\end{split}
\end{equation}
Hence,
\[
g(t) \geq c_5 \w^{4n} - c_7  \w^{-2} t.
\]
Integrating again we obtain
\[
\int_0^t g(s) \ds \geq c_5 \w^{4n} t - c_7  \w^{-2} \frac{ t^2}{2}
\]
Setting $t = c_5 \w^{4n+2}{c_7} \ll \w^2$ we obtain
\[
 \int_0^{\w^2} g(s) \ds \geq  c_6 \w^{8n +2}.
\]
Noting that $h_\w$ is supported on $B_{8\w^3} \times B_{2\w}$ and bounded by $1$, we obtain the desired result. 

\end{proof}

We now make a transformation 
\begin{equation}\label{e:trans0}
z\to z_0 \circ r z, \quad z_0 = (0,x_0,v_0), \quad r = \frac{\w_1}{\w_0}, \quad \w_1 = 2 \w.
\end{equation}
This insures that, whatever $\w$ is, the box $Q_{\w_0}^-$ gets transformed into our $Q_{2 \w}(4 \w^2, x_0,v_0)$. We now choose $\w$ such that the ambient domain $Q^{R_0}$ transforms inside our the lower half of the domain of ellipticity. Given that $x_0$ is within $B_{r_0/4}(x')$ and $v_0 \in B_{R_1}(0)$ it suffices to choose
\[
\w_1 = \w_0 \min\left\{  \sqrt{\frac{T_1}{4}} , \sqrt{\frac{r_0}{16 R_1}}, \frac{R_1}{2 R_0}\right\}.
\]
Under so defined rescaling the we have 
\[
Q_{\w_0}^- \to Q_{\w_1}( \w_1^2, x_0,v_0), \quad Q_{\w_0}^+ \to Q_{\w_1}( (\w_1/\w_0)^2, x_0 + v_0 (\w_1/\w_0)^2, v_0),
\]
and moreover $2 Q^{R_0} = [0,2] \times B_{2R_0} \times B_{2R_0}$ gets transformed inside the domain of ellipticity
\[
 2 Q^{R_0}  \hookrightarrow [0,T/2] \times B_{r_1/2}(x') \times B_{2R_1}(0) \ss \oO.
 \]
At the same time, the ellipticity bounds \eqref{e:ellbounds} remain the same (and in fact improve on the drift). Observe also that all the parameters involved so far are admissible.

In order to apply the weak Harnack inequality, we need to essentially interpolate the $L^1$-information on $f$ expressed by \eqref{e:fwc6} between $L^p$ and $L \log L$ in order to extract information on the $L^p$ level.

Since $\w_1= 2\w$ has been picked already and it is dependent only on the parameters of the model, and $T, \overline{W}, \overline{H}$, let us write  \eqref{e:fwc6} as follows
\begin{equation}\label{e:fwc8}
\int_{ Q_{\w_1}(\w_1^2, x_0,v_0) } f(z) \dz \geq c_8.
\end{equation}
We have
\[
\int_{Q_{\w_1}(\w_1^2, x_0,v_0)}  f | \log f| \dz \leq \w_1^2 \oH.
\]
Thus,
\[
\int_0^\infty | \{ f \geq \a\} \cap Q_{\w_1}(\w_1^2, x_0,v_0) | (|\log \a| + \sign(\a-1) ) \da \leq \w_1^2 \oH.
\]
Consequently, for $\a_0 >1$,
\[
\int_{\a_0}^\infty | \{ f \geq \a\}\cap Q_{\w_1}(\w_1^2, x_0,v_0)  | \da \leq \frac{1}{ \log \a_0} \w_1^2 \oH.
\]
Choosing $\a_0 = \exp\{ \frac{4\w_1^2 \oH}{c_8} \}$ we have
\[
\int_{\a_0}^\infty | \{ f \geq \a\}\cap Q_{\w_1}(\w_1^2, x_0,v_0)  | \da \leq \frac{c_8}{4}.
\]
At the same time for $\a_1 =  \frac{c_8}{4 \w_1^{2+ 4n}}$ we have
\[
\int_0^{\a_1} | \{ f \geq \a\}\cap Q_{\w_1}(\w_1^2, x_0,v_0)  | \da \leq \a_1 |Q_{\w_1}(\w_1^2, x_0,v_0)  | = \a_1 \w_1^{2+ 4n} = \frac{c_8}{4}.
\]
Consequently,
\[
\int_{\a_1}^{\a_0} | \{ f \geq \a\}\cap Q_{\w_1}(\w_1^2, x_0,v_0)  | \da \geq \frac{c_8}{4}.
\]
This implies that  
\[
 | \{ f \geq \a_1 \}\cap Q_{\w_1}(\w_1^2, x_0,v_0)  | \geq \frac{c_8}{4(\a_0 - \a_1)}: = c_9.
 \]

Note again that all the constants depend only on  the parameters of the model, and $T, \overline{W}, \overline{H}$.

 Using transformation \eqref{e:trans0} which has Jacobian $(\w_1/\w_0)^{4n+2}$ we obtain
 \[
  | \{  f_{r,z_0} \geq \a_1 \}\cap Q_{\w_0}^- | \geq (\w_1/\w_0)^{4n+2} c_9: = c_{10}.
  \]
Hence, by the Chebychev inequality,
\[
\left( \int_{Q_{\w_0}^-} f_{r,z_0}^p \dz  \right)^{1/p} \geq \left( \a_1^p |\{  f_{r,z_0} \geq \a_1 \}\cap Q_{\w_0}^- | \right)^{1/p} \geq \a_1 c_{10}^{1/p} : = c_{11}.
\]
\thm{t:wHarnack} applies to show that 
\[
 \inf _{Q_{\w_0}^+}  f_{r,z_0} \geq c_{12},
\]
 or in terms of the original function $f$,
 \[
  \inf _{Q_{\w_1}( (\w_1/\w_0)^2 , x_0 + v_0 (\w_1/\w_0)^2, v_0) }  f \geq c_{12}.
\]

\smallskip
\noindent{\sc Step 3: Harnack chains.}  It will be more efficient, in terms of notation, to remain in the new system of coordinates defined by \eqref{e:trans0}. Since the transformation involves only admissible parameters, any bound on $f$ obtained in the new system will translate into an admissible bound in the old system.

So, in the new coordinates, $f$ satisfies
\begin{equation}\label{e:superabw}
\p_t f + v\cdot \n_x f  \geq  \st(t,x) \D_v f+  (\rmb(t,x) v + \mathrm{w}(t,x)) \cdot \n_v f.
\end{equation}
We make another time-shift to make notation even simpler $z \to (1,0,0) \circ z$. Thus,
we have 
\begin{equation}\label{e:adriftell}
\l \leq \st \leq \L, \quad  |\rmb v| + |\rmw | \leq \L,
\end{equation}
on the new wide domain of ellipticity
\[
\oO = [-1,1] \times B_{2R_0} \times B_{2R_0}.
\]
Notice that the new quantities $\oW,\oH$ turn into another pair of admissible constants. 

On the previous step we established a bound, which in the new coordinate frame reads 
\begin{equation}\label{e:fQ0}
 \inf _{Q_{\w_0}}  f \geq c_{0},
\end{equation}
where $c_0$ is admissible. The goal now is to show that by the time $t = 1$ the solution spreads across the entire torus $\O$. 

It will also be more accommodating  to use \thm{t:wHarnack} where $Q_{\w_0}^-$ is replaced by $Q_{\w_0}$. This is clearly achievable by a slight rescaling and a shift which is allowed by our enlarged ellipticity domain $\oO$.  Also, notice that be rescaling the theorem also applies to the cylinders $Q_{\w_0/2}^\pm$ with $C_0$ being replaced with an absolute multiple of $C_0$, also admissible.

Now let us proceed with the construction of Harnack chains. The original idea goes back to  \cite{AS1967a,AS1967b} and has seen more recent adaptations for Fokker-Planck equation in \cite{AZ2021}. Our construction will be similar in spirit to the latter, although quite different in two technical aspects. First, we produce a chain that reaches the targeted velocity field in fewer steps, thus achieving the exact Gaussian tail on the first run. And second,  the estimates along the chain will take into account the loss of information that comes with the use of a weaker version of the Harnack inequality. 
\begin{lemma}\label{l:1stexp} Let \eqref{e:fQ0} hold. 
There exist admissible constants $a,b>0$ such that 
\begin{equation}\label{ }
f(t,x,v) \geq b e^{-a|v|^2},
\end{equation}
for all $(\w_0/4)^2 \leq t \leq (\w_0/2)^2$, $|x| \leq (\w_0 / 2)^3$ and all $v\in \R^n$.
\end{lemma}
\begin{proof}
Let us fix an $N\in \N$ to be determined later, and let $r = \frac{2|v|}{\w_0 N}$. Denote $\hat{v} = \frac{v}{|v|}$. Let us define the sequence of points
\[
z_0 = 0, \quad z_{l+1} = z_l \circ r (1,0, \frac{\w_0}{2} \hat{v} ), \qquad l=0, \ldots, N-1.
\]
In other words,
\[
z_l = (l r^2, l r^3 \frac{\w_0}{2} \hat{v}, l r \frac{\w_0}{2} \hat{v} ): = (t_l, x_l, v_l).
\]
Notice that the end-point
\[
z_N = (\frac{4 |v|^2}{N \w_0^2} ,\frac{4 |v|^3}{N^2 \w_0^2} \hat{v}, v)
\]
reaches the target velocity vector $v$ by cost of a small shift in time-space variable.  

Also notice the following embeddings of cylinders
\begin{equation}\label{e:embedQ0}
z_1 \circ r Q_{\w_0/2} \ss r Q_{\w_0}(1,0,0),
\end{equation}
which follows by direct verification.  Applying $z_l \circ$ from the left we obtain
\begin{equation}\label{e:embedQl}
z_{l+1} \circ r Q_{\w_0/2} \ss z_l \circ r Q_{\w_0}(1,0,0).
\end{equation}

We will be looking at the rescalings
\[
f_l(z) = f(z_l \circ r z).
\]
All these functions can be thought as defined on the same domain $\oO$ with the same ellipticity constants. 
Indeed, if $z = (s,y,w) \in \oO$, then
\[
z_l \circ r z = (t_l + r^2s,\ x_l + r^3 y + r^2 s v_l,\  r w + v_l).
\]
We have
\[
|t_l + r^2s| \leq \frac{4|v|^2}{\w^2_0 N} + \frac{4|v|^2}{\w^2_0 N^2} \leq  \frac{8 |v|^2}{\w^2_0 N}  <1
\]
provided $N \geq  \frac{8 |v|^2}{\w^2_0}$. Next,
\[
|x_l + r^3 y + r^2 s v_l| \leq \frac{4|v|^3}{\w^2_0 N^2 }  + 2R_0 \frac{8|v|^3}{\w^3_0 N^3 } +  \frac{4|v|^3}{\w^2_0 N^2 } \leq (2R_0+1) \frac{16|v|^3}{\w^3_0 N^2 } \leq 2R_0,
\]
provided $N^2 \geq \frac{2R_0+1}{2R_0} \frac{16|v|^3}{\w^3_0 }$.  This puts the $(t,x)$ pair into the box $[-1,1]\times B_{2R_0}$, and so, the ellipticity for $\st(z_l \circ r z)$ enjoys the same bounds \eqref{e:adriftell}. As to the drift term which gets rescaled to 
\[
\rmB_l = \rmb(t_l + r^2s,\ x_l + r^3 y + r^2 s v_l) r (r w + v_l)+ r\rmw(t_l + r^2s,\ x_l + r^3 y + r^2 s v_l) 
\]
notice that 
\[
|r w + v_l| \leq  2R_0 r+ |v| 
\]
so,
\[
|\rmB_l | \leq  \L r (2R_0 r+ |v| ) + r \L < \L,
\]
provided $N \geq c_1 |v|^2$, where $c_1$ is admissible.

The conclusion is that all functions $f_l$ if considered defined on $\oO$ satisfy the equation with the same ellipticity constants provided 
\[
N \geq c_2 \jap{v}^2,
\]
where $c_2$ is admissible.

Let us now start iteration of the weak Harnack inequality.  We have for $f_0(z) = f(r z)$ from the assumption \eqref{e:fQ0}, and since $r Q_{\w_0} \ss Q_{\w_0}$,
\[
\left( \int_{Q_{\w_0}} f^p_0(z) \dz  \right)^{1/p} \geq c_0 \w_0^{\frac{4n +2}{p}}.
\]
According to \thm{t:wHarnack},
\[
\inf_{Q_{\w_0}(1,0,0)} f_0 \geq C_0^{-1} c_0 (\w_0/2)^{\frac{4n +2}{p}},
\]
(we artificially divided  $\w_0$ by $2$ in order to fit with the general pattern later). According to \eqref{e:embedQ0} we have in particular
\[
\inf_{Q_{\w_0/2}} f_1 \geq C_0^{-1} c_0 (\w_0/2)^{\frac{4n +2}{p}}.
\]
Then by restricting to the cylinder $Q_{\w_0/2}$,
\[
\left( \int_{Q_{\w_0}} f^p_1(z) \dz  \right)^{1/p} \geq C_0^{-1} c_0 (\w_0/2)^{2 \frac{4n +2}{p}}.
\]
According to \thm{t:wHarnack},
\[
\inf_{Q_{\w_0}(1,0,0)} f_1 \geq C_0^{-2} c_0 (\w_0/2)^{2 \frac{4n +2}{p}}.
\]
We proceed in the same manner using \eqref{e:embedQl} and applying repeatedly \thm{t:wHarnack}.

On the last step we achieve the following bound
\[
\inf_{Q_{\w_0/2}} f_N \geq  c_0 [ (\w_0/2)^\frac{4n +2}{p} C_0^{-1} ]^N = c_0 c_3^N.
\]
In particular at the origin we obtain
\[
f_N(0,0,0) = f(z_N) = f(t_N,x_N,v) \geq c_0 c_3^N.
\]

Let us now fix a pair $(t,x)$ such that $(\w_0/4)^2 \leq t \leq (\w_0/2)^2$, $|x| \leq (\w_0 / 2)^3$ and consider the function
\[
g(z) = f((t-t_N,x-x_N,0) \circ z).
\]
This function satisfies the equation on the slightly shrunk domain of ellipticity $[-0.9, 0.9] \times B_{1.9R_0} \times B_{1.9R_0}$. At the same time
\[
\inf_{(t_N - t, x_N- x ,0) \circ Q_{\w_0}} g \geq c_0
\]
The same holds on the subcylinder  $Q_{\w_0/2} \ss (t_N - t, x_N- x ,0) \circ Q_{\w_0}$ (the inclusion follows from the assumptions on $(t,x)$). Applying the above proof to the new function $g$, we obtain
\[
g(t_N,x_N,v) = f(t,x,v) \geq c_0 c_4^N.
\]

Picking the minimal $N$ under which the above holds we find $N = c_5 \jap{v}^2$. Hence,
\[
 f(t,x,v) \geq c_0 e^{N \ln c_4} = c_0 e^{ - c_5 |\ln c_4| \jap{v}^2} ,
 \]
 and the proof is over.
\end{proof}

\smallskip
\noindent{\sc Step 4: spread of positivity in $x$.} Let us fix any point of time $(\w_0/4)^2 \leq t \leq (\w_0/2)^2$ and reset it to $0$. So, at the moment, we have 
\begin{equation}\label{e:fab}
f(0,x,v) \geq b e^{-a|v|^2},
\end{equation}
for all $|x| \leq (\w_0 / 2)^3: = r_3$ and all $v\in \R^n$. Note that $r_3$ is admissible.

The next goal is to establish spread of positivity across the entire periodic domain. Recall that after the rescaling \eqref{e:trans0} our distribution $f$ is defined on $L_0 \T^n \times \R^n$, where $L_0$ is an admissible new period. Also, recall that since the scaling parameter $r<1$, we still have global bounds on the coefficients 
\begin{equation}\label{}
|\st| \leq \oS,\quad |\rmb| \leq \oS, \quad |\rmw| \leq \oW.
\end{equation}

First, let us adopt a barrier construction from \cite{AZ2021} to our situation.

\begin{lemma}\label{l:barrier}
Suppose 
\[
f(0,x,v) \geq \d \one_{\{ |x| < r, \ |v| <R\}}.
\]
Then for any $\t>0$ we have
\[
f(t,x,v) \geq \frac{\d}{4}\one_{\{ |x - t v | < r/2, \ |v| <R/2\}}.
\]
for 
\begin{equation}\label{e:tlimit}
t \leq t_1: = \min\left\{ 1, \t, \frac{1}{8} \frac{1}{n  \oS ( \frac{1}{r^2} + \frac{1}{R^2} ) + (\oS R + \oW) (  \frac{ \t}{r} + \frac{1}{R} )} \right\}.
\end{equation}
\end{lemma}
\begin{proof}
Let us fix $A>0$ to be determined later and consider the barrier function
\[
\chi = - At + \d  \left( 1 - \frac{|x-tv|^2}{r^2} - \frac{|v|^2}{R^2} \right)
\]
Note that $f(0,x,v) \geq \chi(0,x,v)$, and also for all $t>0$, $f(t,x,v) \geq \chi(t,x,v) = 0$, on the boundary $1 = \frac{|x-tv|^2}{r^2} + \frac{|v|^2}{R^2}$.  So, we have $f\geq \chi$ on the parabolic boundary in question. We now need to show that $\chi$ is a sub-solution inside the ellipsoid $1 \geq \frac{|x-tv|^2}{r^2} + \frac{|v|^2}{R^2}$. By the classical comparison principle it implies $f \geq \chi$ on the same region.

So, differentiating we obtain
\begin{equation}\label{}
\begin{split}
\chi_t + v \cdot \n_x \chi & = -A,\\
|\st \D_v \chi |&=  \st \d \left|  \frac{2t^2n}{r^2} + \frac{2n}{R^2} \right| \leq 2n \d \oS \left( \frac{1}{r^2} + \frac{1}{R^2} \right),\\
|(\rmb v + \rmw) \cdot \n_v \chi |& \leq  \d (\oS R + \oW)  \left( 2 t \frac{|x-t v|}{r^2} + \frac{2 |v|}{R^2} \right) \leq \d (\oS R + \oW) \left(  \frac{2 \t}{r} + \frac{2}{R} \right) 
\end{split}
\end{equation}
Let
\[
A = 2n \d \oS \left( \frac{1}{r^2} + \frac{1}{R^2} \right) +\d (\oS R + \oW) \left(  \frac{2 \t}{r} + \frac{2}{R} \right) .
\]
In view of the bounds above this implies that $\chi$ is a sub-solution. 

It remains to observe that as long as $t \leq \frac{\d}{4A}$ and  $|x - t v | < r/2, \ |v| <R/2$, we have
$\chi \geq \frac{\d}{4}$.
\end{proof}

We will be applying \lem{l:barrier} for $r=r_3$. Let us pick $\t$ and $R$ now. Our aim is to make sure that the time limitation giving by the bound \eqref{e:tlimit} is long enough that every corner of the torus $L_0\T^n$ is reachable in that time with velocities from the ball $|v| \leq R/4$. In other words,  we ask for $t_1 R \geq 4 L_0$, or 
\begin{align}
\t R & \geq 4 L_0, \quad R \geq 4L_0 \label{e:tRL}\\ 
R & \geq 32 L_0 \left[n  \oS \left( \frac{1}{r_3^2} + \frac{1}{R^2} \right) + (\oS R + \oW) \left(  \frac{ \t}{r_3} + \frac{1}{R} \right)  \right]. \label{e:R32}
\end{align}
So, first we fix $\t = \frac{r_3}{2\oS} $. This ensures that the leading order term in \eqref{e:R32} has coefficient $\frac12$. Next, we fix the minimal $R = R_1$ satisfying both \eqref{e:tRL}  and \eqref{e:R32}. Note that $R_1$ is admissible.

Setting $\d = b e^{-a R_1^2}$, which is also admissible, in view of \eqref{e:fab} we have
\[
f(0,x,v) \geq \d \one_{\{ |x| < r_1, \ |v| <R_1\}}.
\]
Then 
\[
f(t,x,v) \geq \frac{\d}{4}\one_{\{ |x - t v | < r_3/2, \ |v| <R_1/2\}}, \quad t \leq t_1.
\]

Fix any $x_0 \in L_0\T^n$. Then at time $t_1$ there exists $|v_0| <  R_1/4$ such that $t_1 v_0 = x_0$. Notice that if 
\[
|x - x_0| < r_3/4, \quad |v - v_0 | < r_3 / 4,
\]
then $|x- t_1 v| = |x -x_0 + t_1(v_0 - v) | < r_3 / 2$, and certainly, $|v|<R_1/2$. So,
\[
f(t_1,x,v) \geq \frac{\d}{4} \one_{\{ |x - x_0 | < r_3/4, \ |v - v_0| <r_3/4\}}.
\]

Let us recall that we have started from any point of time $(\w_0/4)^2 \leq t \leq (\w_0/2)^2$, and obtained a time $t_1$ independent of $t$.  So, we found that for any $x_0 \in L_0 \T^n$ there exists a $v_0$, $|v_0|<R_1 / 4$, which depends only on $x_0$ such that 
\begin{equation}\label{e:platx0}
f(t,x,v) \geq \frac{\d}{4} \one_{\{ (\w_0/4)^2 < t- t_1 < (\w_0/2)^2, \ |x - x_0 | < r_3/4, \ |v - v_0| <r_3/4\}}.
\end{equation}

In particular,
\[
\rho(t,x_0) = \int_{\R^n} f(t,x_0,v) \dv \geq  \int_{|v - v_0| <r_3/4} f(t,x_0,v) \dv \geq \l_1,
\]
where $\l_1$ is admissible, and $(\w_0/4)^2 \leq t- t_1 \leq (\w_0/2)^2$. So, for all such times, the density has a uniform lower bound $\l_1$. At the same time there exists an admissible $\L_1$ such that 
\[
\st(t,x) + |\rmb(t,x) v + \rmw(t,x)| \leq \L_1,
\]
for all $(t,x,v) \in [(\w_0/4)^2 +t_1, (\w_0/2)^2 +t_1] \times L_0 \T^n \times B_{4R_1} = \oO_1$.

This implies that we have another initial plateau \eqref{e:platx0}, but now around an arbitrary point $x_0\in L_0 \T^n$, and inside a large domain of ellipticity $\oO_1$.  Applying \lem{l:1stexp} to shifted and if necessary rescaled solution $f$, we find a time $t_2 < (\w_0/2)^2 +t_1$ and admissible $\w_1, a_1,b_1>0$ such that 
\[
f(t_2,x,v) \geq b_1 e^{-a_1|v|^2} \one_{|x-x_0| < \w_1}.
\]
The obtained admissible constants are independent of $x_0$ by virtue of the argument on Step 3. Thus,
\begin{equation}\label{e:expab2}
f(t_2,x,v) \geq b_1 e^{-a_1|v|^2}.
\end{equation}

Now, let us go back to Step 1 and recall that we started with time $0$ and found an admissible time $0<t_2< \frac12$ such that \eqref{e:expab2} holds.  Starting at any other initial time $1-t_2>t>0$, we find that \eqref{e:expab2} holds at $t+t_2$. This finishes the proof.

\subsection{Entropy and global well-posedness}\label{s:gwp} The main implication of  \prop{p:Gauss} can be expressed in terms of lower bound on the density.

\begin{corollary}\label{c:denslow}
For a given classical solution $f\in C_w ([0,T); H^k_l (\T^n))$ of \eqref{e:FPA1} on a  time interval $[0,T)$ there exists $\rho_-$ which depends only on the parameters of the model $\cM$, time $T$, and $\oW$, $\oH$ such that 
\[
\rho(t,x) \geq \rho_-, \qquad \forall (t,x) \in [T/2,T) \times \T^n.
\]
\end{corollary}
So, controlling $\oW$ and $\oH$ over any finite time interval prevents formation of vacuum, which by
\thm{t:lwp} implies global extension.  For a special class of our models  control over $\oW$ and $\oH$ can indeed be given a priori in terms of energy.  We start with $\oH$.  

First we recall that $\oH$ is controlled by the true entropy
\[
\cH = \frac12 \int_\domain |v|^2 f \dv \dx + \int_\domain  f \log f \dv \dx.
\]
Indeed, by the classical inequality, \cite{Lions1994,GJV2004}, there is an absolute constant $C>0$ such that 
\begin{equation}\label{e:flogf}
\int_\domain | f \log f | \dv\dx \leq \int_\domain f \log f \dv\dx + \frac14 \int_\domain |v|^2 f \dv\dx + C \leq \cH + C.
\end{equation}
So,
\begin{equation}\label{e:HH}
\oH \leq 2\cH + C.
\end{equation}
We also have control over the energy
\begin{equation}\label{e:EH}
\cE \leq  2\cH + C.
\end{equation}

\begin{lemma}\label{l:Hcontrol}
Suppose $\cM$ satisfies \eqref{e:unifL2L2}. Then $\oH $ is finite on any finite time interval. Moreover, if the model $\cM$ is conservative, then $\oH $ is globally bounded by initial data
\begin{equation}\label{e:Hunif}
\oH \leq 2\cH_0 + C.
\end{equation}
\end{lemma}
\begin{proof}
We have directly from the equation
\begin{equation}\label{e:Hraw}
\ddt \cH = -  \int_{\domain} \st_\rho \left[  \frac{|\n_v f|^2}{f} + 2  ( v - \ave{u}_\rho) \cdot \n_v f +  v \cdot (v - \ave{u}_\rho) f \right] \dv  \dx.
\end{equation}
Using the identities
\[
 \int_{\domain}   \st_\rho \ave{u}_\rho \cdot \n_v f \dx \dv = 0, \qquad  \int_{\domain}   \st_\rho v \cdot \ave{u}_\rho f \dx \dv = (u,\ave{u}_\rho)_{\k_\rho},
 \]
and  replace $\ave{u}_\rho$ with $u$ in the second term and compute the third as follows
\begin{equation*}\label{}
\begin{split}
\int_{\domain}  \st_\rho v \cdot (v - \ave{u}_\rho) f  \dv   \dx  & = \int_{\domain}  \st_\rho  v \cdot (v - u)  f  \dv   \dx  +  \int_{\domain}  \st_\rho  v \cdot (u - \ave{u}_\rho)  f  \dv   \dx \\
& =  \int_{\domain}  \st_\rho  |v - u|^2  f  \dv   \dx  + \|u\|_{L^2(\k_\rho)}^2 - (u,\ave{u}_\rho)_{\k_\rho}.
\end{split}
\end{equation*}
We obtain
\begin{equation}\label{e:elaw1}
\ddt \cH = -  \int_{\domain}   \st_\rho \frac{|\n_v f + (v - u) f |^2}{f}  \dv\dx  - \|u\|_{L^2(\k_\rho)}^2 + (u,\ave{u}_\rho)_{\k_\rho}.
\end{equation}
We can see that $\cH \leq \cH_0$ for conservative models and \eqref{e:Hunif} follows.

Under  \eqref{e:unifL2L2} we use \eqref{e:EH} to conclude that $\dot{\cH} \leq C_1 \cH + C_2$. The conclusion follows. 
\end{proof}

Immediately from \lem{l:Hcontrol} we obtain control over $\oW$ as well under  $L^2\to L^\infty$ boundedness on the averages.

\begin{lemma}\label{l:HWcontrol}
Suppose $\cM$ satisfies  \eqref{e:unifL2Linfty}. Then $\oH, \oW$ are finite on any finite time interval. Moreover, if the model $\cM$ is conservative, then $\oH, \oW$ are uniformly bounded by a constant depending only on the initial condition.
\end{lemma}

\begin{theorem}\label{t:gwp}
Suppose the model $\cM$ is regular  and satisfies \eqref{e:unifL2Linfty}. Then any local solution $f$ to the Fokker-Planck-Alignment equation \eqref{e:FPA1} in class $H^k_l(\domain)$ extends globally in time.  Consequently, \eqref{e:FPA1} is globally well-posed for thick data 
\[
f_0\in H^k_l(\domain), \ k,l \geq n+3, \quad
 \Th(\rho_0,\O) >0.
\] 
If in addition $\cM$ is conservative, then there exists a $\rho_->0$ depending only on the initial entropy $\cH_0$ and the parameters of the model, such that 
\begin{equation}\label{e:denslow}
\rho(t,x) \geq \rho_-, \qquad \forall (t,x) \in [1,\infty) \times \T^n.
\end{equation}
\end{theorem}

A simple rescaling argument shows that in fact for any time $t_0>0$ there exists $\rho_->0$ depending on the initial entropy $\cH_0$, $t_0$, and the parameters of the model such that 
\begin{equation}\label{e:denslowinst}
\rho(t,x) \geq \rho_-, \qquad \forall (t,x) \in [t_0,\infty) \times \T^n.
\end{equation} 
So, vacuum disappears instantaneously. 

As shown on the third row of Table~\ref{t:rth} all our models are regular on compact environment, and hence the corresponding FPA are globally well-posed for thick data. In addition, the model  \ref{CS}, \ref{Mf}, and \ref{Mseg} due to being conservative, also enjoy the uniform bound from below on the density \eqref{e:denslow}.

\section{Global relaxation and hypocoercivity}\label{s:hypo}

The discussion in this section will be taking place on the compact domain $\O = \T^n$. 
The Fokker-Planck-Alignment equation
\begin{equation}\label{e:FPAhypo}
\p_t f + v \cdot \n_x f = \s \st_\rho \D_v f + \n_v \cdot ( \st_\rho (v - \ave{u}_\rho) f ), 
\end{equation}
has an obvious equilibrium 
\begin{equation}\label{e:Maxwellian}
\mu_{\s,\bar{u}} = \frac{1}{|\O|(2\pi \s)^{n/2}} e^{- \frac{|v - \bar{u}|^2}{2\s}},
\end{equation}
for any constant vector $\bar{u}$. In this section we demonstrate relaxation towards such equilibrium for large data. 

There are several issues that arise when comparing this result to the classical linear Fokker-Planck relaxation, see \cite{Villani}. First, the nonlinear alignment force pumps energy into the system as will be seen from \eqref{e:elaw0hypo}, which prevents direct sliding of the solution towards global Maxwellian. Second, the degeneracy of thermalization $\s \st_\rho$ needs to be avoided in order to retain uniform hypoellipticity of the equation. And third, since we are not assuming that $\cM$ is conservative, it is not immediately clear that the time dependent momentum $\bar{u}$ settles to a limiting vector $u_\infty$.  

We settle these issues in steps. Our first general result lists all the necessary functional requirements on the solution to ensure relaxation towards a moving Maxwellian. We then examine how these requirements are met in the context of  regularity properties stated in \sect{s:rth}  for specific classes of models and how the stabilization of momentum can be deduced.

\begin{proposition}\label{p:mainrelax}
Suppose $\cM$ is a material model. Let $f\in H^k_l(\domain)$ be a classical solution to \eqref{e:FPA} with density $\rho$ satisfying the following conditions uniformly in time
\begin{itemize}
\item[(i)] there exist constants $c_0,c_1,c_2>0$ such that $c_0 \leq \st_\rho \leq c_1$ and $\|\n \st_\rho\|_\infty \leq c_2$ for all $\rho \in \cD$;
\item[(ii)] there exists a constant $\e_0>0$ such that 
\begin{equation}\label{e:spagapii}
 \sup \left\{ (u, \ave{u}_\rho)_{\k_\rho}: \,  u\in L^2(\k_\rho),\, \bar{u} = 0,\, \ \|u\|_{L^2(\k_\rho)}=1 \right\} \leq 1-  \e_0.
\end{equation}
\item[(iii)]  $\| \st_\rho \ave{\cdot}_\rho\|_{ L^2(\rho) \to L^2(\rho)} + \| \n_x(\st_\rho \ave{\cdot}_\rho)\|_{ L^2(\rho) \to L^2(\rho)} \leq c_3$.
\end{itemize}
Then $f$ relaxes to the corresponding Maxwellian exponentially fast,
\begin{equation}\label{e:relaxexp}
\|f(t) - \mu_{\s,\bar{u}(t)}\|_{L^1(\domain)} \leq c_4 \sqrt{\s^{-1} \cI(f_0)}\ e^{- c_5 \s t},
\end{equation}
where $c_4, c_5 >0$ depend only on  the parameters of the model $\cM$ and $c_0,c_1,c_2,c_3,\e_0$. Here, $\cI(f_0)$ is the Fisher information defined in \eqref{e:fF}, and 
\[
\bar{u} = \int_\domain v f(t,x,v) \dv \dx.
\]
\end{proposition}

 \begin{proof}[Proof of \prop{p:mainrelax}]
 
We seek to estimate the relative entropy defined by
\begin{equation}\label{e:rele}
\cH  = \s \int_{\domain} f \log \frac{f}{\mu_{\s,\bar{u}}} \dv \dx.
\end{equation}
By the \CK, we have 
\[
c \s \|f - \mu_{\s,\bar{u}}\|_1^2 \leq \cH,
\]
for some absolute $c$. So, an exponential decay of the entropy would imply the desired result. Let us also recall that Sobolev densities $f\in H^k_l(\domain)$ have finite Fisher information (see below) which in turns control $\cH$, see \cite[Lemma 1]{TV2000}. We can therefore analyze $\cH$ classically. 

Since the model at hand is not assumed to be Galilean invariant or conservative the mean velocity $\bar{u}$ is time dependent and generally may not be assumed $0$ without changing the equation. It will, however, be beneficial  to pass to the reference frame centered at $\bar{u}$. So, we consider the change of variables
\[
\tilde{f}(x,v,t) = f(x,v+\bar{u},t), \quad \tilde{u} = u - \bar{u}, \quad \tilde{\rho} = \rho.
\]
In the new variables the equation becomes a system (dropping tildas)
\begin{equation}\label{e:FPAshift}
\begin{split}
\p_t f +(v+\bar{u}) \cdot \n_x f & = \bar{u}_t \cdot \n_v f + \s \st_\rho \D_v f + \n_v \cdot ( \st_\rho (v - \ave{u}_\rho) f ) \\
\bar{u}_t & = \int_\O (\ave{u}_\rho - u)\dk_\rho \\
\int_\O u \rho \dx & = 0.
\end{split}
\end{equation}
We also denote $\mu_\s = \mu_{\s,0}$. Again, let us note that the extra transport term $\bar{u}_t \cdot \n_v f $ appears because we do not assume that our model is conservative. We keep in mind that  $\bar{u}_t$ is  a constant vector at any point of time.

 The starting point in the proof is two forms of the entropy law.  One is \eqref{e:elaw1}, which after $\s$-rescaling reads
\begin{equation}\label{e:elaw1hypo}
\ddt \cH = -  \int_{\domain}   \st_\rho \frac{|\s \n_v f + (v - u) f |^2}{f}  \dv\dx  - \|u\|_{L^2(\k_\rho)}^2 + (u,\ave{u}_\rho)_{\k_\rho}.
\end{equation}
Using the spectral gap assumption (ii) we conclude
\[
\ddt \cH \leq - \e_0   \|u\|_{L^2(\k_\rho)}^2.
\]
And another form of entropy law  follows directly from \eqref{e:Hraw} (note that the extra transport term $\bar{u}_t \cdot \n_v f $ does not effect either of them)
 \begin{equation}\label{e:elaw0hypo}
\ddt \cH = -  \int_{\domain}   \st_\rho  \frac{|\s \n_v f + v f |^2}{f}  \dv\dx  + (u,\ave{u}_\rho)_{\k_\rho}.
\end{equation} 
Although this form is not dissipative, it gives access to the partial Fisher information
\[
\cI_{vv} =  \int_{\domain}  \frac{|\s \n_v f + v f |^2}{f}  \dv\dx .
\]
In view of (i) and (iii), we have
\begin{equation}\label{e:elaw01}
\ddt \cH \leq  -  c_0 \cI_{vv} + (u,\ave{u}_\rho)_{\k_\rho} \leq -  c_0 \cI_{vv} + c\|u\|_{L^2(\k_\rho)}^2.
\end{equation}
Combining with the previous form \eqref{e:elaw01} we obtain
\begin{equation}\label{e:entropyclean}
\ddt \cH \lesssim -   \cI_{vv} -   \|u\|_{L^2(\k_\rho)}^2.
\end{equation}

The next stage of the proof consists of showing that the classical hypocoercivity of the linear Fokker-Planck equation extends to the fully non-linear alignment model.  In contrast to the \ref{Mf}-model analyzed  in \cite{S-hypo} the general system \eqref{e:FPAshift} requires special attention due to presence of several additional ingredients such as inhomogeneity of diffusion and $\bar{u}$-shift in the transport term. These result in the slower exponential rate $\s$, as opposed to $\s^{1/2}$ for the \ref{Mf}-model.

Let us write the equation for the new distribution 
\[
h = \frac{f}{\mu_\s},
\]
\begin{equation}\label{e:FPAh}
h_t + (v + \bar{u}) \cdot \n_x h = \bar{u}_t \cdot \n_v h - \frac{v}{\s} \cdot \bar{u}_t h+ \st_\rho ( \s \D_v h - v \cdot \n_v h) + \st_\rho (\s^{-1}(\ave{u}_\rho \cdot v) h  - \ave{u}_\rho \cdot \n_v h ).
\end{equation}
The Fokker-Planck part of the equation \eqref{e:FPAh} has the traditional structure of an evolution semigroup. Denoting
\[
B = (v + \bar{u}) \cdot \n_x, \quad A =  \n_v, \quad A^* = (\frac{v}{\s}  -  \n_v) \cdot,
\]
where $A^*$ is understood relative to the inner product of the weighted space $L^2(\mu_\s)$,
we can write 
\begin{equation}\label{e:FPALN}
h_t =  - \s \st_\rho A^* A h- Bh+ \st_\rho A^*( \ave{u}_\rho h) - A^*(h \bar{u}_t).
\end{equation}

We consider Fisher information functionals
\[
\cI_{vv}(h) =  \s^2 \int_\domain \frac{|\n_v h|^2}{h} \dmu_\s,
\quad  \cI_{xv}(h) = \s^{3/2} \int_\domain \frac{\n_x h \cdot \n_v h}{h} \dmu_\s, \quad  \cI_{xx}(h) = \s \int_\domain \frac{|\n_x h|^2}{h} \dmu_\s,
 \]
 where $\dmu_\s = \mu_\s \dv \dx$. The full Fisher information defined by 
 \begin{equation}\label{e:fF}
\cI = \cI_{vv} +\cI_{xx} 
\end{equation}
dominates the relative entropy by the classical log-Sobolev inequality
\[
\cI_{vv} +\cI_{xx} \geq \l \cH.
\]

We now differentiate each of these functionals and obtain estimates on the obtained equations.  The coercivity will be restored by putting them together in a proper linear combination along with the entropy law \eqref{e:entropyclean}.

We will use the following notation: $(g)_\mu = \int_\domain g \dmu_\s$.

\begin{lemma}\label{}
We have
\[
\ddt \cI_{vv}(h)  \leq  -2 \s^3 \cD_{vv}-  2 c_0 \cI_{vv} -2 \s^{1/2} \cI_{xv} +  2  \|u\|_{L^2(\k_\rho)}^2,
\]
where
\[
\cD_{vv} = ( \st_\rho h |\n_v^2 \bar{h}|^2 )_\mu, \qquad \bar{h} = \log h.
\]
\end{lemma}
\begin{proof} 

Let us write $\cI_{vv} = ( \n_v h \cdot \n_v \bar{h} )_\mu$. Computing the derivative  we obtain
\begin{equation*}\label{}
 \frac{1}{\s^2} \ddt \cI_{vv} = 2 ( \n_v h_t \cdot \n_v \bar{h} )_\mu - ( |\n_v \bar{h} |^2 h_t )_\mu   =  J_A + J_B + J_u + J_{\bar{u}},
 \end{equation*}
 where
 \begin{equation*}\label{}
\begin{split}
J_A & =   -2 \s ( \st_\rho \n_v A^*A h \cdot  \n_v \bar{h} )_\mu + \s (\st_\rho  |\n_v \bar{h} |^2 A^*A h )_\mu \\
J_B &= -2 ( \n_v B h\cdot  \n_v \bar{h} )_\mu + ( |\n_v \bar{h} |^2 B h )_\mu \\
J_u & = 2(\st_\rho \n_v A^*(\ave{u}_\rho h)\cdot \n_v \bar{h} )_\mu - (\st_\rho  |\n_v \bar{h} |^2 A^*(\ave{u}_\rho h) )_\mu\\
J_{\bar{u}} & = - 2 (\n_v A^*( h \bar{u}_t) \cdot \n_v \bar{h})_\mu + ( \n_v | \n_v \bar{h}|^2 \cdot \bar{u}_t h)_\mu.
\end{split}
\end{equation*}

Let us start with the most straightforward transport term $B$. We have
\[
J_B = -2 ( \n_x h \cdot  \n_v \bar{h} )_\mu - 2 ( ((v+\bar{u}) \cdot \n_x h_{v_i}) \bar{h}_{v_i} )_\mu  + ( |\n_v \bar{h} |^2 (v+\bar{u}) \cdot \n_x h )_\mu 
=: J_B^1 + J_B^2 + J_B^3.
\]
Observe that 
\[
J_B^1 =  -2 \s^{-3/2} \cI_{xv}.
\]
Next, as to the second term:
\begin{multline*}
J_B^2 = - 2 ((v+\bar{u}) \cdot \n_x h_{v_i} {h}_{v_i}{h}^{-1} )_\mu = - ( (v+\bar{u}) \cdot \n_x |h_{v_i}|^2 h^{-1} )_\mu \\
= -( |h_{v_i}|^2 (v+\bar{u}) \cdot \n_x h h^{-2} )_\mu =- ( |\bar{h}_{v_i}|^2 (v+\bar{u}) \cdot \n_x h )_\mu= - J_B^3
\end{multline*}
and so the two cancel. We obtain
\[
J_B = -2 \s^{-3/2} \cI_{xv}.
\]

Let us turn to the dissipation term $J_A$.  Using the identity 
\[
\p_{v_i} (A^*A h) = A^*Ah_{v_i} + \s^{-1}h_{v_i},
\]
we have
\begin{equation*}\label{}
J_A  = -2 \s (\st_\rho  A^*Ah_{v_i} \bar{h}_{v_i} )_\mu - 2 (\st_\rho \n_v h \cdot \n_v \bar{h} )_\mu + \s (\st_\rho |\n_v \bar{h} |^2 A^*A h )_\mu.
\end{equation*}
Note that the term in the middle is bounded above by $- 2 c_0\s^{-2}\cI_{vv} $ in view of (i). In the other two we switch $A^*$ to the opposite side,

\begin{equation*}\label{}
\begin{split}
J_A  & \leq  -2 \s (\st_\rho  Ah_{v_i} \cdot A \bar{h}_{v_i} )_\mu - 2 c_0\s^{-2}  \cI_{vv} + \s (\st_\rho A |\n_v \bar{h} |^2\cdot  A h  )_\mu \\
&= -2\s (\st_\rho  h A\bar{h}_{v_i}\cdot  A \bar{h}_{v_i} )_\mu -2\s (\st_\rho \bar{h}_{v_i} Ah\cdot  A \bar{h}_{v_i}  )_\mu - 2 c_0 \s^{-2}\cI_{vv} + 2 \s (\st_\rho  \bar{h}_{v_i}A \bar{h}_{v_i} \cdot  A h  )_\mu.
\end{split}
\end{equation*}

The second and  last terms cancel, while the first is exactly $-2 \s \cD_{vv}$:
\[
J_A  \leq -2 \s \cD_{vv}  -  2 c_0 \s^{-2}\cI_{vv}.
\]

For the alignment term we obtain the following the exact identity 
\begin{equation}\label{e:Juexact}
J_u = 2 \s^{-2}( \ave{u}_\rho \cdot u )_{\k_\rho}.
\end{equation}
We note, however, that there is no advantage of keeping the low energy here as the full energy will emerge later in the proof. So, we replace it with the full energy
\begin{equation}\label{e:nonJuexact}
J_u \leq  2(1-\e_0) \s^{-2}  \|u\|_{L^2(\k_\rho)}^2 \leq 2 \s^{-2}  \|u\|_{L^2(\k_\rho)}^2.
\end{equation}

To prove \eqref{e:Juexact} we manipulate with the formula for $J_u$ as follows
\begin{equation*}\label{}
\begin{split}
J_{u} & = 2(\st_\rho \n_v A^*(\ave{u}_\rho h)\cdot  \n_v \bar{h} )_\mu -(\st_\rho |\n_v \bar{h} |^2 A^*(\ave{u}_\rho h))_\mu \\
& = 2( \st_\rho \n_v ( \s^{-1} v \cdot \ave{u}_\rho h - \ave{u}_\rho \cdot \n_v h)\cdot  \n_v \bar{h} )_\mu- ( \st_\rho \n_v |\n_v \bar{h} |^2 \cdot  \ave{u}_\rho h )_\mu \\
& = 2\s^{-1} ( \st_\rho \ave{u}_\rho h  \cdot  \n_v \bar{h} )_\mu + 2\s^{-1} ( \st_\rho (v \cdot \ave{u}_\rho) \n_v h \cdot  \n_v \bar{h} )_\mu - 2(\st_\rho \n^2_v h \ave{u}_\rho  \cdot  \n_v \bar{h} )_\mu - 2( \st_\rho \n_v^2 \bar{h}(\n_v \bar{h}) \cdot \ave{u}_\rho h )_\mu\\
& = : J_u^1 + J_u^2+J_u^3+J_u^4,
\end{split}
\end{equation*}
 where $\n_v^2 h$ is the Hessian matrix of $h$. 
 
 Observe that the first term is exactly the lower energy
 \[
J_u^1 =  2 \s^{-1}(\st_\rho  \ave{u}_\rho \cdot  \n_v h )_\mu = 2 \s^{-2}(\st_\rho  \ave{u}_\rho \cdot v h )_\mu = 2 \s^{-2}( \ave{u}_\rho \cdot u )_{\k_\rho}.
 \] 
 
Now comes the crucial observation that the remaining terms that cannot be controlled  cancel altogether
\[
  J_u^2+J_u^3+J_u^4 = 0.
 \]
 Indeed, using 
 \begin{equation}\label{e:hbarh}
  h_{v_i v_j} = h \bar{h}_{v_i v_j} + \frac{1}{h} h_{v_i}  h_{v_j}
\end{equation}
let us compute $J_u^3$,
\[
J_u^3 = -2  (\st_\rho  h_{v_i v_j} \ave{u_j}_\rho \bar{h}_{v_i} )_\mu  = - 2 (\st_\rho  h \bar{h}_{v_i v_j}  \ave{u_j}_\rho \bar{h}_{v_i} )_\mu - 2 (\st_\rho   \frac{1}{h} h_{v_i}  h_{v_j} \ave{u_j}_\rho \bar{h}_{v_i} )_\mu = J_u^4 - 2 (\st_\rho \n_v h \cdot \ave{u}_\rho | \n_v \bar{h} |^2   )_\mu
\]
Also note that 
\[
J_u^2 = 2\s^{-1} ( \st_\rho (v \cdot \ave{u}_\rho h ) |\n_v \bar{h}|^2 )_\mu.
\]
Then we have
\[
J_u^2 + J_u^3 = J_u^4 + 2\s^{-1} ( \st_\rho (v \cdot \ave{u}_\rho h ) |\n_v \bar{h}|^2 )_\mu - 2 (\st_\rho \n_v h \cdot \ave{u}_\rho | \n_v \bar{h} |^2   )_\mu = J_u^4 + 2 ( \st_\rho A^*( \ave{u}_\rho h)  | \n_v \bar{h} |^2 )_\mu .
\]
Switching $A^*$ in the last term we obtain
\[
2 ( \st_\rho A^*( \ave{u}_\rho h)  | \n_v \bar{h} |^2 )_\mu  = 2 ( \st_\rho h  \ave{u}_\rho \cdot \n_v  | \n_v \bar{h} |^2 )_\mu = -2 J_u^4.
\]
The obtained terms sum up to zero.

Finally, we show that the momentum term vanishes $J_{\bar{u}} = 0$. Let us expand
\[
J_{\bar{u}} =2( \n^2_v h \n_v \bar{h} \cdot \bar{u}_t )_\mu - 2 \left( \frac{v}{\s} \cdot \bar{u}_t \frac{|\n_v h|^2}{h} \right)_\mu + ( \n_v | \n_v \bar{h}|^2 \cdot \bar{u}_t h)_\mu =: J_{\bar{u}}^1+J_{\bar{u}}^2 + J_{\bar{u}}^3.
\]
Let us look into the first term,
\begin{equation*}\label{}
\begin{split}
J_{\bar{u}}^1 & = \left( \frac{\n_v |\n_v h|^2}{h} \cdot \bar{u}_t \right)_\mu = \left( A \left(\frac{|\n_v h|^2}{h} \right) \cdot \bar{u}_t \right)_\mu +\left( \frac{|\n_v h|^2}{h^2} \n_v h \cdot \bar{u}_t \right)_\mu  \\
& = \left(  \frac{|\n_v h|^2}{h}   \frac{v \cdot \bar{u}_t}{\s} \right)_\mu + \left( |\n_v \bar{h} |^2 \n_v h \cdot \bar{u}_t \right)_\mu .
\end{split}
\end{equation*}
So,
\[
J_{\bar{u}}^1 + J_{\bar{u}}^2 = - \left(  \frac{|\n_v h|^2}{h}   \frac{v \cdot \bar{u}_t}{\s} \right)_\mu + \left( |\n_v \bar{h} |^2 \n_v h \cdot \bar{u}_t \right)_\mu = -  ( |\n_v \bar{h} |^2 A^*( h \bar{u}_t ) )_\mu = -  ( \n_v |\n_v \bar{h} |^2 h \bar{u}_t )_\mu = -J_{\bar{u}}^3.
\]
Thus, $J_{\bar{u}}^1 + J_{\bar{u}}^2 +J_{\bar{u}}^3 = 0$.

\end{proof}

\begin{lemma}\label{}
We have
\[
\ddt \cI_{xv}(h)  \leq - \frac14 \s^{1/2} \cI_{xx} + c(\s^{-1/2}+1) \cI_{vv}+ \frac12  \s^3 \cD_{vv} + \frac12 \s^2 \cD_{xv} + c(\s^{-1/2}+1) \| u \|_2^2,
\]
where
\[
\cD_{xv} = ( \st_\rho h |\n_v \n_x \bar{h}|^2 )_\mu.
\]
\end{lemma}
\begin{proof}
Let us write
\[
\frac{1}{\s^{3/2}} \ddt \cI_{xv}(h) = (\n_x h_t \cdot \n_v \bar{h} )_\mu + ( \n_x \bar{h} \cdot \n_v {h}_t )_\mu - ( h_t \n_v \bar{h}\cdot \n_x \bar{h} )_\mu : = J_A + J_B + J_u+J_{\bar{u}},
\]
where as before $J_A,J_B,J_u,J_{\bar{u}}$ collect contributions from $A$, $B$, and alignment components, respectively.

For the $B$-term we have
\[
J_B = - ( \n_x ((v + \bar{u}) \cdot \n_x h) \cdot \n_v \bar{h} )_\mu - ( \n_x \bar{h} \cdot \n_v ((v + \bar{u}) \cdot \n_x h)  )_\mu + ( ((v + \bar{u}) \cdot \n_x h) \n_v \bar{h}\cdot \n_x \bar{h}  )_\mu: = J_B^1 + J_B^2 + J_B^3.
\]
For the middle term we expand
\[
J_B^2 =  - ( \n_x \bar{h} \cdot \n_x h )_\mu - ( \bar{h}_{x_i} (v_j +\bar{u}_j) h_{x_j v_i})_\mu.
 \]
The first term is exactly $- \s^{-1} \cI_{xx} $ and in the second integrating by parts in $x_j$, we obtain
 \[
 = - \s^{-1} \cI_{xx} + ( \bar{h}_{x_i x_j} (v_j +\bar{u}_j) h_{v_i} )_\mu
\]
using that $\bar{h}_{x_i x_j}= h^{-1} {h}_{x_i x_j} - \bar{ h}_{x_i} \bar{h}_{x_j}$,
\[
= - \s^{-1} \cI_{xx} +( {h}_{x_i x_j} (v_j +\bar{u}_j) \bar{h}_{v_i} )_\mu - ( \bar{h}_{x_i} \bar{h}_{x_j} (v_j +\bar{u}_j) h_{v_i} )_\mu = - \s^{-1} \cI_{xx} - J_B^1 - J_B^3.
\]
Hence,
\begin{equation}\label{e:JBxv}
J_B = -  \s^{-1}\cI_{xx} .
\end{equation}

Let us look into the $J_A$-term:
\[
\frac{1}{\s} J_A =   - (\n_x (\st_\rho A^*Ah) \cdot \n_v \bar{h} )_\mu - (\st_\rho  \n_x \bar{h} \cdot \n_v  A^*Ah )_\mu + (\st_\rho  A^*Ah \n_v \bar{h}\cdot \n_x \bar{h} )_\mu = J_A^1 + J_A^2 + J_A^3.
\]
For $J_A^1$ we obtain
\[
\begin{split}
J_A^1 & = - ( \st_\rho A^* A h_{x_i} \bar{h}_{v_i} )_\mu - ( (\st_\rho)_{x_i} A^* A h \bar{h}_{v_i} )_\mu = - ( \st_\rho \n_v h_{x_i} \cdot \n_v  \bar{h}_{v_i} )_\mu - ( (\st_\rho)_{x_i} \n_v h \n_v \bar{h}_{v_i} )_\mu \\
& = -  ( \st_\rho h \n_v \bar{h}_{x_i} \cdot \n_v  \bar{h}_{v_i} )_\mu- ( \st_\rho  \bar{h}_{x_i} \n_v h \cdot \n_v  \bar{h}_{v_i} )_\mu - \left( \frac{(\st_\rho)_{x_i}}{\st_\rho^{1/2}} \frac{\n_v h}{h^{1/2}} \cdot  \st_\rho^{1/2} h^{1/2} \n_v \bar{h}_{v_i} \right)_\mu .
\end{split}
\]
In view of assumption (i), 
\[
J_A^1 \leq  -  (\st_\rho  h \n_v \bar{h}_{x_i} \cdot \n_v  \bar{h}_{v_i} )_\mu- (\st_\rho  \bar{h}_{x_i} \n_v h \cdot \n_v  \bar{h}_{v_i} )_\mu + c \s^{-1} \sqrt{\cI_{vv} \cD_{vv}}.
\]

For $J_A^2$ we obtain
\[
J_A^2 = - \s^{-1} (\st_\rho \n_x \bar{h} \cdot \n_v h )_\mu - ( \st_\rho \bar{h}_{x_i} A^*A {h}_{v_i} )_\mu   = - \s^{-1} (\st_\rho \n_x \bar{h} \cdot \n_v h )_\mu -  ( \st_\rho h \n_v \bar{h}_{x_i} \cdot \n_v  \bar{h}_{v_i} )_\mu  -  ( \st_\rho  \bar{h}_{v_i} \n_v \bar{h}_{x_i} \cdot \n_v  h )_\mu  .
\] 
The two add up to 
\[
\begin{split}
J_A^1+J_A^2 & =  - \s^{-1} \left(\st_\rho \frac{\n_x {h}}{h^{1/2}} \cdot \frac{\n_v h}{h^{1/2}} \right)_\mu  - 2  (\st_\rho h \n_v \bar{h}_{x_i} \cdot \n_v  \bar{h}_{v_i} )_\mu - (\st_\rho  Ah \cdot A(\n_v \bar{h}\cdot \n_x \bar{h} ) )_\mu + c \s^{-1} \sqrt{\cI_{vv} \cD_{vv}} \\
& \leq  - c_0 \s^{-5/2} \sqrt{\cI_{xx} \cI_{vv}} + \sqrt{ \cD_{xv} \cD_{vv}} + c \s^{-1}\sqrt{\cI_{vv} \cD_{vv}} - J_A^3.
\end{split}
\]
Thus,
\[
J_A \leq c_0 \s^{-3/2}\sqrt{\cI_{xx} \cI_{vv}}  + \s \sqrt{ \cD_{xv} \cD_{vv}} +   c \sqrt{\cI_{vv} \cD_{vv}}.
\]

Let us examine the alignment term now,
\begin{equation*}\label{}
\begin{split}
J_u & = ( \n_x ( \st_\rho A^*(\ave{u}_\rho h)) \cdot \n_v \bar{h})_\mu + ( \n_x \bar{h} \cdot \n_v ( \st_\rho A^*(\ave{u} h))  )_\mu - (  \st_\rho A^*(\ave{u}_\rho h) \n_v \bar{h}\cdot \n_x \bar{h} )_\mu \\
& = (  (\st_\rho)_{x_i} A^*(\ave{u}_\rho h)  \bar{h}_{v_i})_\mu + ( \st A^*((\ave{u}_\rho)_{x_i} h) \bar{h}_{v_i})_\mu+ ( \st_\rho A^*(\ave{u}_\rho h_{x_i})  \bar{h}_{v_i})_\mu \\
& + (\st_\rho  \bar{h}_{x_i}   A^*(\ave{u}_\rho h_{v_i} ) )_\mu + \s^{-1} ( \st_\rho h \n_x \bar{h} \cdot \ave{u}_\rho)_\mu  - (  \st_\rho h \ave{u}_\rho  \cdot \n_v( \n_v \bar{h}\cdot \n_x \bar{h}) )_\mu \\
& = ( h (\st_\rho \ave{u}_\rho)_{x_i}  \cdot \n_v \bar{h}_{v_i})_\mu + ( \st_\rho h \ave{u}_\rho \bar{h}_{x_i} \cdot \n_v \bar{h}_{v_i})_\mu \\
& + (\st_\rho h \n_v \bar{h}_{x_i} \cdot   \ave{u}_\rho \bar{h}_{v_i}  )_\mu + \s^{-1} ( \st_\rho h \n_x \bar{h} \cdot \ave{u}_\rho)_\mu  - (  \st_\rho h \ave{u}_\rho  \cdot \n_v( \n_v \bar{h}\cdot \n_x \bar{h}) )_\mu
\end{split}
\end{equation*}
We can see that the 2rd, 3th, and 5th terms cancel, and we arrive at
\[
J_u = ( h (\st_\rho \ave{u}_\rho)_{x_i}  \cdot \n_v \bar{h}_{v_i})_\mu + \s^{-1} ( \st_\rho h \n_x \bar{h} \cdot \ave{u}_\rho)_\mu = J_u^1+J_u^2.
\]
We estimate $J_u^1$ using the assumption (iii), and the fact that $L^2(\k_\rho)$- and $L^2(\rho)$-norms are equivalent under (i),
\[
J_u^1  \leq c \| u \|_{L^2(\k_\rho)} \sqrt{\cD_{vv}}
\]
And again, by (iii),
\[
J_u^2 \leq \s^{-3/2} \| u \|_{L^2(\k_\rho)} \sqrt{\cI_{xx}} \leq \frac12 \s^{-1} \cI_{xx} + \s^{-2}  \| u \|_{L^2(\k_\rho)}^2 .
\]
Noticing that $\frac12 \s^{-1} \cI_{xx}$ is absorbed into \eqref{e:JBxv} and summing up all the terms we arrive at 
\begin{equation*}\label{}
\begin{split}
\ddt \cI_{xv}(h) & \leq - \frac12 \s^{1/2} \cI_{xx}  + c_0 \sqrt{\cI_{xx} \cI_{vv}} + c \s^{3/2}\sqrt{\cI_{vv} \cD_{vv}}+ \s^{5/2} \sqrt{ \cD_{xv} \cD_{vv}}+  c \s^{3/2} \| u \|_{L^2(\k_\rho)} \sqrt{\cD_{vv}} \\
&+ \s^{-1/2} \| u \|_{L^2(\k_\rho)}^2  \\
& \leq - \frac14 \s^{1/2} \cI_{xx} + c(\s^{-1/2}+1) \cI_{vv}+ \frac12  \s^3 \cD_{vv} + \frac12 \s^2 \cD_{xv} + c(\s^{-1/2}+1) \| u \|_{L^2(\k_\rho)}^2 .
\end{split}
\end{equation*}

Finally, let us look at the momentum term
\[
J_{\bar{u}} = - (A^*( \bar{u}_t \p_{x_i} h) \p_{v_i} \bar{h})_\mu - (\p_{x_i} \bar{h} \p_{v_i} A^*( \bar{u}_t h) )_\mu + (A^*( \bar{u}_t h)  \p_{v_i} \bar{h} \p_{x_i} \bar{h})_\mu := J_{\bar{u}}^1 + J_{\bar{u}}^2 + J_{\bar{u}}^3.
\]
Let us note the identity
\[
\p_{v_i} A^* = A^* \p_{v_i} + \frac{1}{\s} \Id,
\]
and expand on $J_{\bar{u}}^2$
\[
J_{\bar{u}}^2 = -  (\p_{x_i} \bar{h}  A^*( \bar{u}_t \p_{v_i} h) )_\mu - \s^{-1} (\p_{x_i} h  \bar{u}_t)_\mu.
\]
The last term vanishes since $\bar{u}_t$ is a constant vector.  Thus,
\[
J_{\bar{u}}^2= -  (\n_v \p_{x_i} \bar{h} \cdot  \bar{u}_t \p_{v_i} h )_\mu
\]
In the other two terms we switch $A^*$ as well
\[
\begin{split}
J_{\bar{u}}^1 &= - (\p_{x_i} h \bar{u}_t \cdot \n_v \p_{v_i} \bar{h})_\mu \\
J_{\bar{u}}^3 &= (h \bar{u}_t \cdot \n_v  (\p_{v_i} \bar{h} \p_{x_i} \bar{h}))_\mu
\end{split}
\]
The sum of the three is clearly zero by the product rule. So,
\[
J_{\bar{u}} =0.
\]
\end{proof}

\begin{lemma}\label{}
We have
\[
\ddt \cI_{xx}(h)  \leq   c \cI_{vv} - \s^2 \cD_{xv} + c \| u \|_{L^2(\k_\rho)}^2 .
\]
\end{lemma}
 \begin{proof}
 We have
 \[
\frac{1}{\s} \ddt \cI_{xx}(h) =  2( \n_x h_t \cdot \n_x \bar{h} )_\mu - ( |\n_x \bar{h} |^2 h_t )_\mu : = J_A+ J_B + J_u+J_{\bar{u}}.
\]

The $B$-term cancels entirely, 
\[
\begin{split}
J_B & = -2 ( \n_x((v+\bar{u}) \cdot \n_x h) \cdot \n_x \bar{h} )_\mu + ( |\n_x \bar{h} |^2 (v+\bar{u}) \cdot \n_x h )_\mu \\
& = -2 ( ((v+\bar{u}) \cdot \n_x h_{x_i}) h_{x_i} h^{-1} )_\mu + ( |\n_x \bar{h} |^2 (v+\bar{u}) \cdot \n_x h )_\mu \\
& = - ( ((v+\bar{u}) \cdot \n_x |\n_x h|^2 h^{-1} )_\mu + ( |\n_x \bar{h} |^2 (v+\bar{u}) \cdot \n_x h )_\mu \\
& =  - ( (v+\bar{u}) \cdot \n_x h  |\n_x h|^2 h^{-2} )_\mu + ( |\n_x \bar{h} |^2 (v+\bar{u}) \cdot \n_x h )_\mu = 0.
\end{split}
\]
So is $J_{\bar{u}}$,
\begin{equation*}\label{}
\begin{split}
J_{\bar{u}} & = - 2 (\n_x A^*(\bar{u}_t h) \cdot \n_x \bar{h})_\mu + ( |\n_x \bar{h} |^2 A^*(\bar{u}_t h))_\mu \\
& = - 2 ( \p_{x_i} h\bar{u}_t \cdot \n_v \p_{x_i}\bar{h})_\mu + ( \n_v |\n_x \bar{h} |^2 \cdot \bar{u}_t h)_\mu  = 0.
\end{split}
\end{equation*}

The $A$-term is given by
\[
\begin{split}
 J_A & = -2 ( \n_x (\st_\rho A^*A h) \cdot \n_x \bar{h} )_\mu + (\st_\rho  |\n_x \bar{h} |^2 A^*Ah )_\mu\\
& =-2 ( (\st_\rho)_{x_i} A h \cdot A \bar{h}_{x_i} )_\mu -2 (\st_\rho  A h_{x_i} \cdot A \bar{h}_{x_i} )_\mu+ (\st_\rho A |\n_x \bar{h} |^2 \cdot Ah)_\mu \\
& \leq c \s^{-1} \sqrt{\cD_{xv} \cI_{vv}} -2 (\st_\rho \n_v (h \bar{h}_{x_i}) \cdot \n_v \bar{h}_{x_i} )_\mu + (\st_\rho \n_v |\n_x \bar{h} |^2 \cdot \n_vh )_\mu\\
& = c \s^{-1} \sqrt{\cD_{xv} \cI_{vv}} -2 (\st_\rho h \n_v \bar{h}_{x_i} \cdot \n_v \bar{h}_{x_i} )_\mu -2  (\st_\rho  \bar{h}_{x_i} \n_v h \cdot \n_v \bar{h}_{x_i} )_\mu + (\st_\rho \n_v |\n_x \bar{h} |^2 \cdot \n_vh )_\mu\\
& =c \s^{-1} \sqrt{\cD_{xv} \cI_{vv}} -2 \cD_{xv} -  (\st_\rho \n_v |\n_x \bar{h} |^2 \cdot \n_vh )_\mu+ (\st_\rho \n_v |\n_x \bar{h} |^2 \cdot \n_vh )_\mu \\
& = c \s^{-1} \sqrt{\cD_{xv} \cI_{vv}}-2 \cD_{xv}.
\end{split}
\]
Thus,
\[
J_A \leq c  \sqrt{\cD_{xv} \cI_{vv}} -2 \s \cD_{xv}.
\]

Finally, the alignment term is given by
\[
J_u = 2( \n_x (\st_\rho A^*(\ave{u}_\rho h) ) \cdot \n_x \bar{h} )_\mu - (\st_\rho  |\n_x \bar{h} |^2 A^*(\ave{u}_\rho h) )_\mu .
\]
In the second term we switch the operator $A^*$: 
\begin{equation}\label{e:auxxx1}
- (\st_\rho |\n_x \bar{h} |^2 A^*(\ave{u}_\rho h))_\mu = - (\st_\rho \n_v |\n_x \bar{h} |^2 \ave{u}_\rho h)_\mu.
\end{equation}
For the first term we obtain
\begin{equation*}\label{}
\begin{split}
 2( \n_x (\st_\rho A^*(\ave{u}_\rho h) ) \cdot \n_x \bar{h} )_\mu & = 2(h (\st_\rho\ave{u}_\rho)_{x_i}   \cdot \n_v \bar{h}_{x_i} )_\mu  +  2(\st_\rho   h  \bar{h}_{x_i} \ave{u}_\rho \cdot \n_v \bar{h}_{x_i} )_\mu.
  \end{split}
\end{equation*}
We can see that the last term cancels with \eqref{e:auxxx1}. We thus obtain, using assumption (iii),
\[
J_u =  2(h (\st_\rho\ave{u}_\rho)_{x_i}   \cdot \n_v \bar{h}_{x_i} )_\mu  \leq c \| u \|_{L^2(\k_\rho)} \sqrt{\cD_{xv}} .
\]

Putting together the obtained bounds we have
\[
\ddt \cI_{xx}(h)  \leq    c \s \sqrt{\cD_{xv} \cI_{vv}} -2 \s^2 \cD_{xv}+ c \s\| u \|_{L^2(\k_\rho)} \sqrt{\cD_{xv}}  \leq c \cI_{vv} - \s^2 \cD_{xv} + c \| u \|_{L^2(\k_\rho)}^2 .
\]

\end{proof}

To conclude the proof, let us combine together all the Fisher functionals in the following format
\[
\tilde{\cI} = \cI_{vv} + \d \s^{1/2} \cI_{xv}+\frac{c_0}{c}\cI_{xx},
\]
where $\d>0$ is small but dependent only on the parameters of the assumptions (i), (ii), (iii). Since $\s < 1$, for $\d$ small enough we have $\tilde{\cI} \sim \cI$. Then,
\[
\ddt \tilde{\cI} \leq -c_0 \cI_{vv} - c_1 \s  \cI_{xx} + C \|u\|_2^2.
\]
Invoking the entropy inequality \eqref{e:entropyclean}, we obtain with a properly chosen constant $C>0$,
\[
\ddt (\tilde{\cI} + C \cH) \leq -c_0 \cI_{vv} - c_1 \s  \cI_{xx} \leq -c_2 \s(\cI_{vv} + \cI_{xx}) \sim - \s (\tilde{\cI} + C \cH) .
\]
Hence,
\begin{equation}\label{e:relaxlast}
\tilde{\cI} + C \cH \leq c_3 \cI(f_0) e^{-c_2 \s t},
\end{equation}
 and the result follows. 

\end{proof}

\subsection{Applications}

Let us now  explore how \prop{p:mainrelax} implies relaxation for various situations.

First, we have  convergence near equilibrium for all models whose spectral gap is under control for densities near uniform.

\begin{proposition} \label{p:sid}
Suppose that $\cM$ is regular and has a uniformly positive spectral gap for any densities close to uniform 
\begin{equation}\label{ }
\inf\{ \e_0(\rho): \ \| \rho - 1/|\O| \|_1 \leq \d_0 \} >0.
\end{equation}
Then there exists a constant $c>0$ depending only on the parameters of the model such that for any initial condition $f_0\in H_l^k(\O)$ satisfying
\begin{equation}\label{e:Ismall2}
\cI(f_0) \leq  c \s \d_0,
\end{equation}
there exists a global classical solution converging to the Maxwellian exponentially fast.
\end{proposition}
\begin{proof}
By \defin{d:th} (iii),(iv) we can further reduce the size of $\d_0$ if necessary to have not only  the uniform spectral gap condition but also the uniform thickness condition satisfied, $\Th(\rho,\O) > c_4$. By the regularity assumption \eqref{e:r1}, the assumption on the spectral gap, and \eqref{e:stlow}, such densities fulfill all the conditions (i), (ii), (iii) of  \prop{p:mainrelax}, with constants $c_0,c_1,c_2,c_3,\e_0$ depending only on $\d_0$. And according to \thm{t:lwp} such data give rise to local classical solutions $f\in H_l^k$.

If \eqref{e:Ismall} holds, then by the \CK\ $\|\rho_0 -  \frac{1}{|\O|}\|_{L^1} \leq c_5 c \d_0$ for some absolute $c_5>0$. If $c< \frac{1}{2 c_5}$  then by continuity we have $\|\rho(t) -  \frac{1}{|\O|}\|_{L^1} < \d_0$ at least on some short time interval $[0,T)$. Let $T$ be the maximal time of existence of the local solution which satisfies the above. Note that the solution cannot blowup before it reaches the equality $\|\rho(t) -  \frac{1}{|\O|}\|_{L^1} = \d_0$, due to the continuation criterion \eqref{e:rholow}. Hence, if $T$ is finite it is only because we have  $\|\rho(T) -  \frac{1}{|\O|}\|_{L^1} =\d_0$ for the first time. 

The \prop{p:mainrelax}  then applies on $[0,T]$.  As a consequence,  $\|\rho(T) -  \frac{1}{|\O|}\|_{L^1} \leq c_6 c \d_0$ for all $t\leq T$ and some $c_6$ depending on the parameters of the model only. Assuming further that $c <  \frac{1}{2 c_6}$ we conclude that $T$ cannot be finite. Thus, the solution exists globally and satisfies $\|\rho(t) -  \frac{1}{|\O|}\|_{L^1} < \d_0$ for all time.   \prop{p:mainrelax} applies again to conclude the result.
\end{proof}

Before we state application of this result to particular models, let us address the issue of the time-dependent momentum for non-conservative case.  It turns out that  for all our core non-conservative models \ref{Mb}, including the Motsch-Tadmor model, the average momentum stabilizes.

\begin{lemma}
Suppose $\cM$ satisfies the assumptions of \prop{p:sid} and 
\begin{equation}\label{e:adjunif}
\ave{1}^*_{\frac{1}{|\O|} } = 1.
\end{equation}
For any solution $f$ that relaxes at an exponential rate in the sense of relative entropy there exists $u_\infty \in \R^n$ such that $\bar{u}(t) \to u_\infty$ exponentially fast, and consequently, 
\begin{equation}\label{e:relaxexp2}
\|f(t) - \mu_{\s,u_\infty}\|_{L^1(\domain)} \leq c_9 e^{- c_{10} t}.
\end{equation}
In particular, the conclusion applies to all  conservative and all \ref{Mb} models.
\end{lemma}
\begin{proof}
By the \CK\ we have $\| \rho - \frac{1}{|\O|} \|_1 \lesssim e^{-c t}$. Thus, as we argued in the proof of \prop{p:sid}, the density eventually enters into a class satisfying all the functional requirements of  \prop{p:mainrelax} uniformly in time. 

We have
\begin{equation*}\label{}
\begin{split}
\bar{u}_t &= \int_\O \int_\O \phi_\rho(x,y) u(y) \rho(y) \dy \rho(x) \dx -  \int_\O \st_\rho(y) u(y) \rho(y) \dy \\
& =  \int_\O \int_\O (\phi_\rho(x,y) - \phi_{\frac{1}{|\O|}} (x,y) )  u(y) \rho(y) \dy \rho(x) \dx  \\
& +  \int_\O \int_\O \phi_{\frac{1}{|\O|}} (x,y)   u(y) \rho(y) \dy \left[ \rho(x) - \frac{1}{|\O|} \right] \dx \\
& +  \int_\O \left[  \frac{1}{|\O|} \int_\O \phi_{\frac{1}{|\O|}} (x,y) \dx  -\st_{\frac{1}{|\O|}}(y)  \right]  u(y) \rho(y) \dy \\
& + \int_\O \left[ \st_{\frac{1}{|\O|}}(y) - \st_\rho(y)  \right]  u(y) \rho(y) \dy .
\end{split}
\end{equation*}
By continuity assumptions all the terms on the right hand side are bounded by a constant multiple of $\| \rho - \frac{1}{|\O|} \|_1 \sqrt{\cE}$. Since the energy remains uniformly bounded all these terms  are exponentially decaying. This proves the exponential convergence $\bar{u}(t) \to u_\infty$ for some $u_\infty \in \R^n$.

Next, we have
\[
\int_{\domain} f \log \frac{f}{\mu_{\s,u_\infty}} \dv \dx =  \int_{\domain} f \log \frac{f}{\mu_{\s,\bar{u}}} \dv \dx +   \int_{\domain} f \log \frac{\mu_{\s,\bar{u}}}{\mu_{\s,u_\infty}} \dv \dx.
\]
The last term is a constant multiple of 
\[
\int_{\domain} f( |u_\infty - v|^2 - |\bar{u} - v|^2) \dv \dx =  |u_\infty|^2 - |\bar{u} |^2 + 2 \int_\O ( \bar{u} - u_\infty ) \cdot u \rho \dx \lesssim e^{-c t}.
\]
This finishes the proof.
\end{proof}

According to our computations of spectral gaps stated in \prop{p:gaps} and \prop{p:gapsMb}, we can apply the above results to conclude local relaxation for all core models. Let us gather all this in one statement. 

\begin{theorem}[Relaxation near equilibrium] \label{t:relaxlocal} Suppose that $\cM$ is a regular model satisfying
\begin{equation}\label{ }
\inf\{ \e_0(\rho): \ \| \rho - 1/|\O| \|_1 \leq \d_0 \} >0,
\end{equation}
and 
\[
\ave{1}^*_{\frac{1}{|\O|} } = 1.
\]
Then there exists a constant $c>0$ depending only on the parameters of the model such that for any initial condition $f_0\in H_l^k(\O)$ satisfying
\begin{equation}\label{e:Ismall}
\cI(f_0) \leq  c \s \d_0,
\end{equation}
there exists a global classical solution $f$ and there exists $ u_\infty \in \R^n$ such that 
\begin{equation}\label{e:relaxexp2}
\|f(t) - \mu_{\s,u_\infty}\|_{L^1(\domain)} \leq c_9 e^{- c_{10} t}.
\end{equation}
In particular, the conclusion applies to all core models \ref{CS}, \ref{MT}, \ref{Mb}, \ref{Mf}, \ref{Mseg}.
\end{theorem}

We now gather a set of conditions which guarantees global relaxation.

\begin{theorem}[Global relaxation] \label{t:relax}
Suppose $\cM$ is a regular conservative model satisfying  \eqref{e:unifL2Linfty} and such that 
\begin{equation}\label{ }
\inf\{ \e_0(\rho): \orho_{r}(\O) > \d_0 \} >0,
\end{equation}
for some $r,\d_0 >0$. Then any classical global solution to the Fokker-Planck-Alignment equation \eqref{e:FPAhypo} relaxes to equilibrium as stated in \prop{p:mainrelax} with $\bar{u} = \bar{u}_0$.

In particular, global relaxation holds for the following models
\begin{itemize}
\item the Cucker-Smale model \ref{CS} with a Bochner-positive kernel $\phi = \psi \ast \psi$;
\item the \ref{Mf}-model with $\inf \phi > 0$;
\item the segregation model \ref{Mseg} with $\supp g_l = \O$ for all $l= 1,\ldots, L$.
\end{itemize}
\end{theorem}

\begin{proof} According to  \thm{t:gwp} under the given conditions on $\cM$ any classical solution gains a uniform bound on the density from below  
$\rho(x,t) \geq \rho_-$, for $(x,t) \in \O \times [1,\infty)$. This automatically puts the solution into a class satisfying the assumptions of \prop{p:mainrelax} uniformly. 
\end{proof}

Let us remark that the only requirement that prevents global relaxation for \ref{Mf} and \ref{Mseg} models with general local kernels is the uniform $L^2 \to L^\infty$ boundedness \eqref{e:unifL2Linfty}, which is needed to control $\oW$.  However, this control can be regained if the solution is known to have uniformly bounded macroscopic velocities
\begin{equation}\label{ }
\sup_{t \geq 0} \| u(t) \|_\infty <\infty.
\end{equation}
This is precisely the result obtained for \ref{Mf}  in \cite{S-hypo}.

\section{Hydrodynamic limits}\label{s:hydrolim}

Supplementing the Vlasov equation \eqref{e:VA} with a strong penalization force 
\begin{equation}\label{e:VAF}
\p_t f^\e + v \cdot \n_x f^\e =  \n_v \cdot ( \st_\rho (v - \ave{u^\e}_{\rho^\e}) f^\e ) + \frac{1}{\e} F(f^\e),
\end{equation}
one can achieve regimes in which the distribution $f$ asymptotically takes a special form explicitly expressible in terms of the macroscopic quantities $u,\rho$. The limiting system satisfied by $u,\rho$ is called the Euler-alignment system \eqref{e:EAS}, in which the pressure law depends on the particular force $F$ used in the limit. In this study we will cover  two types of limits --  monokinetic and Maxwellian.

The monokinetic limit is achieved by enforcing strong local alignment $F = \n_v\cdot [ (v - u^\e) f^\e]$. The force penalizes deviation from the Dirac concentrated on $u^\e$, which drives the solution towards  monokinetic distribution $f = \rho(x,t) \d_{u(x,t)}(v)$, where $\rho, u$ solve the pressureless EAS
\begin{equation}\label{e:EASvelless}
\begin{split}
\p_t \rho + \n\cdot (u \rho) & = 0, \\
\p_t u + u \cdot \n u &  = \st_\rho (\ave{u}_\rho - u).
\end{split}
\end{equation}
Solutions to \eqref{e:EASvelless} will always be understood in smooth regularity classes such as
\begin{equation}\label{e:macroclass}
	(u,\rho) \in C_w([0,T); H^m \times (H^k\cap L^1_+)) \cap \Lip ([0,T); H^{m-1} \times (H^{k-1} \cap L^1_+)), 
\end{equation}
for $ m \geq k+1> \frac{n}{2} + 2$.  Local and global well-posedness theory for such solutions can be established for a variety of models and data, see \cite{Sbook} for a detailed analysis. Because of the maximum principle on $u$ which applies to solutions of \eqref{e:EASvelless} any initially compact flock $\supp \rho_0 \ss B_{R_0}$ will remain compactly supported on any finite time interval
\begin{equation}\label{e:comptsuppEAS}
\supp \rho(t) \ss B_{R(t)}, \quad R(t) \leq R_0 + A_0 t. 
\end{equation}

The history of this limit goes back to \cite{MV2008,KV2015} where the alignment term in \eqref{e:VA} is considered centered around zero velocity. In the settings of the classical Cucker-Smale model the hydrodynamic limit was studied in \cite{FK2019}. In both studies the force $F = \n_v \cdot [ (v - u^\e) f^\e]$ includes the rough macroscopic field $u^\e$ causing issues  with  uniqueness of characteristics of \eqref{e:VAF} and subsequently the transport of $f^\e$. These issues have been dealt with in  \cite{FK2019}  by imposing no vacuum condition $\rho >0$ and restricting analysis to the periodic domain. A more recent remake of Figalli-Kang's argument done in \cite{Sbook} avoids all these issues by replacing $u^\e$ with a mollified version if it, $u^\e_\d$, based on the \ref{Mf}-protocol. Such change allows to extend the limit to vacuous and compactly supported flocks on either $\T^n$ or $\R^n$. 

 In the context of the general environmental averaging models  this result can be broadly extended to include all uniformly regular models. Moreover,   in contrast to the previous studies  the convergence $f^\e \to f$ can be upgrade quantitatively to  Wasserstein-2 metric, see \thm{t:hydrolim}.

In the Maxwellian regime the force $F$ is given by the Fokker-Planck-Alignment operator 
\[
F = \D_v f^\e + \n_v\cdot[ (v - u^\e) f^\e].
\]
  The local thermodynamic equilibrium becomes the Maxwellian 
\[
f = \frac{ \rho(x,t) }{(2\pi)^{n/2}} e^{- \frac{|v - u(x,t)|^2}{2}},
\]
and so the corresponding macroscopic  model is given by the EAS with isothermal pressure
\begin{equation}\label{e:EASiso}
\begin{split}
\p_t \rho + \n\cdot (u \rho) & = 0, \\
\p_t (\rho u) + \n \cdot (\rho u \otimes u) + \n \rho &  = \rho\st_\rho   (\ave{u}_\rho - u).
\end{split}
\end{equation}
In the Cucker-Smale settings, this limit was justified in \cite{KMT2015} via the relative entropy method. Again, because of the roughness of $u^\e$ the result had to be cast in the settings of a special  class of weak solutions  established in \cite{KMT2013}, see also \cite{KMT2014} for the justification of a local alignment limit.  The work \cite{S-hypo} implemented similar method to prove hydrodynamic limit in the context of the \ref{Mf}-model. 

Now, we can cast the Maxwellian limit in the framework of general environmental averaging models with the additional implementation of the mollified local alignment field $u^\e_\d$ -- the same methodology we will be using in the monokinetic case. This allows to work in the class of classical solutions as stated in \thm{t:lwp} and \thm{t:gwp}. The limiting solution must be non-vacuous and the domain is restricted to the torus $\O = \T^n$. \thm{t:hydroMax} shows convergence $f^\e \to f$ in the relative entropy sense, which implies stronger convergence in $L^1$ by the \CK.

\subsection{Monokinetic limit}\label{s:mono}

In this section we discuss the monokinetic limit.  The analysis will be carried out on any environment $\O$, compact or not under the assumption of uniform regularity of $\cM$.

Let us consider solutions to the following Vlasov model with forced local alignment
\begin{equation}\label{e:VAe}
\p_t f^\e + v \cdot \n_x f^\e = \n_v \cdot ( \st_{\rho^\e} (v - \ave{u^\e}_{\rho^\e}) f^\e) + \frac{1}{\e}  \n_v\cdot( (v - u^\e_\d) f^\e),
\end{equation}
where subscript $\d$ designates a special mollification. 
To define it let us fix a smooth mollifier $\psi_\d(x) = \frac{1}{\d^n} \psi(x/\d)$, where $\psi >0$ on $\O$ and in the case of $\O = \R^n$ we assume that $\psi$ satisfies the algebraic decay condition \eqref{e:algker}.  Then let  $u_\d$ be the average of $u$ based on the $\cM_{\psi_\d}$-protocol,
\begin{equation}\label{e:fMTd}
u_\d =  \left( \frac{(u \rho)_{\psi_\d}}{\rho_{\psi_\d}} \right)_{\psi_\d}.
\end{equation}
Formally, \eqref{e:VAe} corresponds to the Vlasov equation \eqref{e:VA} based on the model given by
\[
\st^\e_\rho = \st_\rho + \frac{1}{\e}, \quad \ave{u}_\rho^{\e,\d} = \frac{\e \st_\rho}{\e \st_\rho +1} \ave{u}_\rho +\frac{1}{\e \st_\rho +1} u_\d.
\]
Clearly, since $\cM_{\psi_\d}$ and $\cM$ are uniformly regular, then so is the model above. 
Consequently, the global existence of classical compactly supported solutions to \eqref{e:VAe} is warranted in this case by \thm{t:VAgwp}.  Moreover, characteristics of  \eqref{e:VAe}  satisfy the usual maximum principle for velocities \lem{l:maxpr}. Hence, $|X^\e(t)| \leq R_0+t A_0$, where $R_0$ is the initial radius of the support in $x$ and $A_0$ is the maximal initial velocity. Thus, on any time interval $[0,T]$, the family $f^\e$ will be supported on a bounded region uniformly in $\e$ if initial $f_0$ is compactly supported.



Before we focus on the main convergence result let is go back to the defined mollification $u_\d$ and note another remarkable approximation property -- if $u$ is a smooth field, then $u_\d$ approximates $u$ with a quantitative bound {\em independent} of any regularity of $\rho$.  This allows to implement it to situations where the only information known on $\rho$ is its mass. The following is a generalization of such approximation property presented in \cite[Lemma 5.1]{Sbook}.

\begin{lemma}\label{l:mollest}
For any $u\in \Lip$ and for any $1\leq p <\infty$ one has
\begin{equation}\label{e:mollest}
\| u_\d - u \|_{L^p(\rho)} \leq C \d \|u\|_{\Lip},
\end{equation}
where $C>0$ is a constant depending only on the kernel $\psi$ and $p$. The estimate also holds for all $1 \leq p \leq \infty$ with $C$ independent of $p$ if $\psi$ is compactly supported.
\end{lemma}
\begin{proof}
Let us fix a test-function $f\in L^q(\rho)$, where $q^{-1} + p^{-1} = 1$.  Then, let us split
\[
\int_\O f (u_\d - u) \rho \dx  =\int_\O f (u_\d - u_{\psi_\d}) \rho \dx + \int_\O f ( u_{\psi_\d} - u) \rho \dx : = I_1 + I_2.
\]
For $I_2$ we simply use the standard approximation property of mollification
\[
I_2 \leq \d \|u\|_{\Lip} \| f\|_{L^1(\rho)} \leq \d \|u\|_{\Lip} \| f\|_{L^q(\rho)}.
\] 
For $I_1$ we have, using Minkowskii and \HI,
\begin{equation*}\label{}
\begin{split}
I_1 &= \int_\O (f \rho)_{\psi_\d}  \frac{(u \rho)_{\psi_\d}}{\rho_{\psi_\d}} \dx + \int_\O (f \rho)_{\psi_\d}  \frac{u \rho_{\psi_\d}}{\rho_{\psi_\d}} \dx\\
&= \int_\O \frac{ (f \rho)_{\psi_\d} ( (u \rho)_{\psi_\d} - u \rho_{\psi_\d})}{\rho_{\psi_\d}} \dx = \int_\O \frac{ (f \rho)_{\psi_\d} }{\rho^{1/p}_{\psi_\d}}  \frac{ (u \rho)_{\psi_\d} - u \rho_{\psi_\d} }{\rho^{1/q}_{\psi_\d}} \dx \\
& \leq \left(  \int_\O \frac{| (f \rho)_{\psi_\d} |^q}{\rho^{q/p}_{\psi_\d}}  \dx \right)^{1/q} \left(  \int_\O \frac{| (u \rho)_{\psi_\d} - u \rho_{\psi_\d}  |^p}{\rho^{p/q}_{\psi_\d}}  \dx \right)^{1/p} \\
& \leq \left(  \int_\O | (|f|^q \rho)_{\psi_\d}  \dx \right)^{1/q} \left(  \int_{\O\times\O} |u(y) - u(x)|^p \rho(y) \psi_\d(x-y) \dy  \dx \right)^{1/p}  \\
& \leq \|f\|_{L^q(\rho)} \|u\|_{\Lip} \left(  \int_{\O\times\O} |x-y|^p \rho(y) \psi_\d(x-y) \dy  \dx \right)^{1/p}  = \d \|f\|_{L^q(\rho)} \|u\|_{\Lip}  C^{1/p}_{p,\psi},
\end{split}
\end{equation*}
where $C_{p,\psi} = \int_\O |x|^p \psi(x) \dx$.  This implies \eqref{e:mollest} for all $p<\infty$. If however $\psi$ is compactly supported, then $C_{p,\psi} \leq (\diam \supp \psi)^p$, and so the estimate holds also in the limit as $p\to \infty$.
\end{proof}

The main convergence result of this section will be quantified in terms of  $W_2$-metric:
\[
W_2^2(f_1,f_2) = \inf_{\g \in \Pi(f_1,f_2)} \int_{\O^2 \times \R^{2n}} |\w_1 - \w_2|^2 \dg(\w_1,\w_2),
\]
where $\Pi(f_1,f_2)$ is the set of probability measures with marginals $f_1$ and $f_2$, respectively.

\begin{theorem}\label{t:hydrolim}Suppose $\cM$ is a uniformly regular model.  Let $(\rho,u)$ be a classical solution to \eqref{e:EASvelless} on the time interval $[0,T)$ with compact support \eqref{e:comptsuppEAS}, and let $f = \rho(x,t) \d_{u(x,t)}(v)$. Suppose $f_0^\e \in C^k_0(\domain)$ is a family of initial conditions satisfying 
\begin{itemize}
	\item[(i)] $\supp f_0^\e \ss \{ |w| < R_0\}$;
	\item[(ii)] $W_2(f_0^\e, f_0) \leq \e$.
\end{itemize}
Then there exists a constant $C$ such that for all $t<T$ one has
\begin{equation}\label{e:W1ed}
W_2(f^\e_t, f_t)   \leq C  \sqrt{\e + \frac{\d}{\e}}.
\end{equation}
\end{theorem}

\begin{remark}
Let us note that the scaling regime $\d = \e^2$ appears to be the most optimal:  if $\d \ll \e^2$, the model becomes over-resolved without improvement on convergence rate of solutions, if $\d \gg \e^2$, the model is under-resolved and the convergence rate  slows down. We obtain in this case the optimal rate of $\sqrt{\e}$:
\begin{equation}\label{e:W1e}
W_2(f^\e, f)   \leq C  \sqrt{\e}.
\end{equation}
\end{remark}

\begin{remark}
Not that $W_2(f^\e, f) \to 0$ also implies convergence of densities, simply because $\rho$'s are marginals of $f$'s: $W_2(\rho^\e, \rho) \leq W_2(f^\e, f)$. Similarly, since all distributions are confined to a bounded set, we also have $W_1(u^\e \rho^\e, u \rho) \leq C W_1(f^\e, f) \leq C W_2(f^\e, f) $. So, this also implies  the convergence of momenta.
\end{remark}

\begin{remark} The theorem applies to a range of core models listed in Table~\ref{t:rth}.
However, we also note that the uniform regularity  is only needed to facilitate global existence of solutions. The actual assumptions that are needed to run the argument for a given family of solutions are  \eqref{e:unifL2L2}, \eqref{e:glob1} - \eqref{e:glob2}, where $\rho' = \rho$ is the limiting density, and $\| \p_y \phi_\rho \|_\infty <C$. Thus, if the limiting density is known to be thick and the model is simply regular and satisfies  \eqref{e:unifL2L2}, \eqref{e:glob1} - \eqref{e:glob2},  then the theorem applies just as well to putative solutions and extends to a much wider class of models listed in the last row of Table~\ref{t:rth}.
\end{remark}

\begin{proof} Let us first note that since all densities are compactly supported the model satisfies all the estimates 
\eqref{e:ur1}-\eqref{e:ur2}  for $\rho$ and $\rho^\e$ uniformly on $[0,T]$.

Denoting
\[
\cE_\e  = \frac{1}{2} \int_{\domain} | v |^2 f^\e(x,v) \dx \dv,
\]
we have the following energy balance relation  for solutions of \eqref{e:VAe}:
\begin{equation}\label{e:ee1}
\ddt \cE_\e  = - \int_{\domain} \st_{\rho^\e} | v |^2 f^\e \dv \dx + (u^\e , \ave{u^\e}_{\rho^\e} )_{\k_{\rho^\e}} +  \frac{1}{\e} \int_{\O} \frac{ | (\rho^\e u^\e)_{\psi_\d} |^2}{\rho^\e_{\psi_\d}} \dx - \frac{2}{\e} \cE_\e.
\end{equation}
Noting that 
\[
 \int_{\domain} \st_{\rho^\e} | v |^2 f^\e \dv \dx \geq  (u^\e , u^\e )_{\k_{\rho^\e}}
\]
we obtain
\begin{equation}\label{e:ee2}
\ddt \cE_\e  \leq   (u^\e , \ave{u^\e}_{\rho^\e} - u^\e )_{\k_{\rho^\e}} +  \frac{1}{\e} \int_{\O} \frac{ | (\rho^\e u^\e)_{\psi_\d} |^2}{\rho^\e_{\psi_\d}} \dx - \frac{2}{\e} \cE_\e.
\end{equation}
Obviously the last two terms store a lot of dissipative information. The crucial observation is that they control internal energies of $f^\e$ both the native one relative to the local field $u^\e$ and relative to the filtered field $u^\e_\d$. To see that let us note the following two identities
\begin{align*}
\en(f^\e | u^\e) &: =	\frac{1}{2} \int_{\domain} | v - u^\e|^2 f^\e \dv \dx =\ \cE_\e - 	\frac{1}{2} \int_{\R^{n}} \rho^\e |u^\e|^2 \dx, \\
\en(f^\e | u^\e_\d) &: =	\frac{1}{2} \int_{\domain} | v - u^\e_\d|^2 f^\e \dv \dx = \ \cE_\e -  \int_{\O} \frac{ | (\rho^\e u^\e)_{\psi_\d} |^2}{\rho^\e_{\psi_\d}} \dx  + \frac12 \int_{\O}  \rho^\e |u^\e_\d|^2 \dx.
\end{align*}
Summing up we obtain
\[
\en(f^\e | u^\e) + \en(f^\e | u^\e_\d) = 2\cE_\e -  \int_{\O} \frac{ | (\rho^\e u^\e)_{\psi_\d} |^2}{\rho^\e_{\psi_\d}} \dx + \frac12 \int_{\R^{n}} \rho^\e |u^\e_\d|^2  \dx - 	\frac{1}{2} \int_{\R^{n}} \rho^\e |u^\e|^2 \dx,
\]
and since the $\cMfmt$-model is contractive, the last two terms add up to a non-positive value. Thus,
\[
2\cE_\e -  \int_{\O} \frac{ | (\rho^\e u^\e)_{\psi_\d} |^2}{\rho^\e_{\psi_\d}} \dx \geq  \en(f^\e | u^\e) + \en(f^\e | u^\e_\d).
\]
Consequently, plugging this pack into \eqref{e:ee2} we obtain
\begin{equation}\label{e:ened}
\ddt \cE_\e \leq  (u^\e , \ave{u^\e}_{\rho^\e} - u^\e )_{\k_{\rho^\e}} - \frac{1}{\e} [\en(f^\e | u^\e) + \en(f^\e | u^\e_\d)].
\end{equation}
 
The energy inequality \eqref{e:ened} already shows that the solution concentrates to a monokinetic form near its own macroscopic field. However, the quantity that controls how far that concentration is from $u$, is the modulated kinetic energy:
\[
\en(f^\e | u) = \frac12 \int_{\domain} | v - u|^2 f^\e \dv \dx.
\]
This quantity plays a key role in the argument. It should be noted that it controls the corresponding macroscopic relative entropy
\begin{equation}\label{e:mrl}
\begin{split}
	\int_{\O}  \rho^\e|u^\e - u|^2 \dx & =  \int_{\O}(  |u^\e|^2 \rho^\e - 2 u^\e \cdot u \rho^\e + |u|^2 \rho^\e) \dx \\
	&\leq \int_{\domain} (|v|^2 f^\e - 2 u \cdot v f^\e + |u|^2 f^\e )\dx \dv = \en(f^\e|u).
\end{split}
\end{equation}

According to (ii) we can fix an initial $\g_0 \in \Pi(f_0^\e,f_0)$ such that 
\[
\int_{\O^2 \times \R^{2n}} |\w_1 - \w_2|^2 \dg_0(\w_1,\w_2) \leq 2\e^2.
\]
Let us now propagate $\g_0$ along the direct product of characteristic maps of \eqref{e:VAe} and \eqref{e:VA}, i.e. let $\g_t$ by the measure-valued solution to the transport equation
\[
\p_t \g + v_1 \cdot \n_{x_1} \g + v_2 \cdot \n_{x_2} \g + \n_{v_1}[ \g (\st_{\rho^\e} (v_1 - \ave{u^\e}_{\rho^\e}) + \frac{1}{\e} (v_1 - u^\e_\d)) ] + \n_{v_2}[ \g \st_\rho( v_2 - \ave{u}_\rho)] = 0.
\]
Integrating upon pairs $(x_1,v_1)$ and $(x_2,v_2)$ we can see that the marginals of $\g$ satisfy the same transport equations as $f$ and $f^\e$, respectively. Consequently, by uniqueness, $\g_t \in \Pi(f^\e_t,f_t)$ for all time. This means that the cost of $\g_t$ dominates the $W_2$-distance at any time, 
\[
W : = \int_{\O^2 \times \R^{2n}} |\w_1 - \w_2|^2 \dg_t(\w_1,\w_2)  \geq W_2^2(f^\e,f).
\]

Let us split $W$ into potential and kinetic components
\[
W =  \int_{\O^2 \times \R^{2n}} |v_1 - v_2|^2 \dg +  \int_{\O^2 \times \R^{2n}} |x_1 - x_2|^2 \dg : =  W_v + W_x.
\]
Evolution of the potential component is easily estimated using the transport of $\g$
\begin{multline*}
\ddt W_x = \ddt  \int_{\O^2 \times \R^{2n}} |X^\e(\w_1,t) - X(\w_2,t) |^2 \dg_0 \\
 = 2 \int_{\O^2 \times \R^{2n}} (X^\e(\w_1,t) - X(\w_2,t)) \cdot  (V^\e(\w_1,t) - V(\w_2,t))  \dg_0 \leq W_x+ W_v.
\end{multline*}

Instead of writing the evolution equation for $W_v$ we subordinate it to the internal energy, and trace its evolution. Let us make the following estimate
\begin{equation*}\label{}
\begin{split}
W_v & \leq  \int_{\O^2 \times \R^{2n}} |v_1 - u(x_1) |^2 \dg + \int_{\O^2 \times \R^{2n}} |u(x_1) - u(x_2)|^2 \dg +  \int_{\O^2 \times \R^{2n}} |u(x_2) - v_2|^2 \dg  \\
& \leq \int_{\O \times \R^{n}} |v - u(x) |^2 f^\e(x,v)\dv\dx + C  \int_{\O^2 \times \R^{2n}} |x_1 - x_2|^2 \dg + 0
\end{split}
\end{equation*}
where the last term canceled thanks to the monokinetic nature of $f$,
\[
= \en(f^\e|u) + C W_x.
\]

We have obtained so far
\begin{equation}\label{e:syseW}
\begin{split}
\ddt W_x & \leq \en(f^\e|u) + c_1 W_x, \\
W_v & \leq \en(f^\e|u) + c_2 W_x.
\end{split}
\end{equation}
To complete this system we now investigate evolution of the internal energy itself.

Before we write the equation for the modulated energy $\en(f^\e|u)$, let us recall that we are dealing with smooth solutions to both so all the computations are legitimate. From \eqref{e:VAe} we can read off  the macroscopic system for the $\e$-density and momentum
\[
\left\{
\begin{split}
\rho^\e_t + \n \cdot (\rho^\e u^\e) & = 0, \\
(\rho^\e u^\e)_t + \n_x \cdot (\rho^\e u^\e \otimes u^\e + \cR_\e) &= \k_{\rho^\e} (  \ave{u^\e}_{\rho^\e} - u^\e ) + \frac{1}{\e} \rho^\e ( u^\e_\d - u^\e),
\end{split} \right. 
\]
where the Reynolds stress is given by
\[
\cR_\e = \int_{\R^n}( v - u^\e) \otimes (v - u^\e) f^\e(x,v,t) \dv.
\]

Let us expand $\en(f^\e|u)$ into three parts
\[
	\en(f^\e|u)  = \cE_\e - \int_{\R^{n}} \rho^\e u^\e \cdot u \dx + \frac12\int_{\R^{n}} \rho^\e  |u|^2 \dx.
\]
From the energy inequality \eqref{e:ened} we will only retain the alignment component (to be used later) and the  native internal energy
\begin{equation}\label{e:enDpure}
\ddt  \cE_\e  \leq (u^\e , \ave{u^\e}_{\rho^\e} - u^\e )_{\k_{\rho^\e}}   -  \frac{1}{\e} \en(f^\e | u^\e).
\end{equation}
Let us work out the equation for the macroscopic part:
\begin{align}
 \ddt \int_{\O} \rho^\e u^\e \cdot u \dx  &= \int_{\O} \p_t (\rho^\e u^\e) \cdot u \dx + \int_{\O} \rho^\e u^\e \cdot \p_t u \dx \label{e:eeu} \\
&= \int_{\R^{n}}  ( \rho^\e u^\e \otimes u^\e + \cR_\e ) : \n u \dx  - \int_{\O} \rho^\e u^\e \otimes u : \n u \dx & & \text{(inertia)}  \notag\\
& + ( \ave{u^\e}_{\rho^\e} - u^\e , u )_{\k_{\rho^\e}}  +  \int_\O \rho^\e u^\e \cdot ( \ave{u}_\rho - u) \st_\rho \dx & & \text{(native alignment)} \notag \\
&+ \frac{1}{\e} \int_{\O} \rho^\e ( u^\e_\d - u^\e) \cdot u \dx  & & \text{(local alignment)}  \notag\\
\ddt  \frac12\int_{\O} \rho^\e |u|^2 \dx & = \int_{\O} \rho^\e  u \cdot \p_t u \dx +  \frac12\int_{\O} \p_t \rho^\e |u|^2  \dx \label{e:euu}\\
& = - \int_{\O} \rho^\e u \otimes u : \n u \dx +  \int_{\O} \rho^\e u \otimes u^\e : \n u \dx & & \text{(inertia)}  \notag \\
& +  \int_\O \rho^\e u \cdot ( \ave{u}_\rho - u) \st_\rho \dx  & & \text{(native alignment)}. \notag
\end{align}
Putting the two equations together and collecting all the inertia terms and using \eqref{e:mrl} we obtain
\[
- \int_{\O}  \rho^\e (u^\e - u) \otimes (u^\e - u) : \n u \dx \leq \|\n u\|_\infty  \int_{\O}  \rho^\e|u^\e - u|^2 \dx \lesssim  \en(f^\e|u).
\]

The Reynolds stress is estimated similarly
\[
 \int_{\O}   \cR_\e  : \n u \dx \leq \|\n u\|_\infty  \int_{\domain} | v - u^\e(x,t)|^2 f^\e(x,v,t) \dx \dv \lesssim \en(f^\e|u^\e).
\]

As to the local alignment term, we use the symmetry and approximation property of  the \ref{Mf}-averaging used to define $u^\e_\d$, which is crucially independent of regularity of $\rho^\e$. Namely, by \lem{l:mollest} with $p=2$, we have
\begin{multline*}
	\int_{\O} \rho^\e ( u^\e_\d - u^\e) \cdot u \dx = \int_{\O} \rho^\e u^\e_\d \cdot u \dx - \int_{\O} \rho^\e  u^\e \cdot u \dx =  \int_{\O} \rho^\e u^\e \cdot u_\d \dx - \int_{\O} \rho^\e  u^\e \cdot u \dx \\
	= \int_{\O} \rho^\e u^\e \cdot ( u_\d- u) \dx \leq C \|u^\e\|_{L^2(\rho^\e)}  \d \|\n u\|_\infty.
\end{multline*}
Thus, the local alignment term can be estimated by 
\begin{equation}\label{e:amplify}
A_\loc = \frac{1}{\e} \int_{\R^{n}} \rho^\e ( u^\e_\d - u^\e) \cdot u \dx    \lesssim  \|u^\e\|_{L^2(\rho^\e)} \frac{\d}{\e}.
\end{equation}
Note that the energy $  \|u^\e\|_{L^2(\rho^\e)}$ remains uniformly bounded in $\e$, so,
\begin{equation}\label{e:amplify2}
A_\loc   \lesssim  \frac{\d}{\e}.
\end{equation}

Let us collect the obtained estimates \eqref{e:enDpure}, \eqref{e:eeu}, \eqref{e:euu}, and simplify the native alignment components
\begin{equation}\label{}
 \ddt \en(f^\e|u)  \lesssim  \en(f^\e|u) + \frac{\d}{\e} +(u^\e - u , \ave{u^\e}_{\rho^\e} - u^\e )_{\k_{\rho^\e}} +  \int_\O \rho^\e (u - u^\e) \cdot ( \ave{u}_\rho - u) \st_\rho \dx .
\end{equation}
It remains to estimate the alignment terms. Let us rearrange them as follows
\begin{equation*}\label{}
\begin{split}
\d A & = (u^\e - u , \ave{u^\e}_{\rho^\e} - u^\e )_{\k_{\rho^\e}} +  \int_\O \rho^\e (u - u^\e) \cdot ( \ave{u}_\rho - u) \st_\rho \dx  \\
& = \int_\O \rho^\e (u - u^\e) \cdot ( \st_{\rho^\e} \ave{u^\e}_{\rho^\e} -   \st_\rho \ave{u}_\rho +\st_\rho u - \st_{\rho^\e} u^\e) \dx \\
& = (u^\e - u , \ave{u^\e - u }_{\rho^\e}  )_{\k_{\rho^\e}} +  \int_\O \rho^\e(u - u^\e) \cdot ( \st_{\rho^\e} \ave{u}_{\rho^\e} -   \st_\rho \ave{u}_\rho ) \dx  \\
& - \int_\O \rho^\e |u - u^\e|^2 \st_{\rho^\e} \dx + \int_\O \rho^\e(u - u^\e) \cdot u ( \st_{\rho^\e} -\st_{\rho} )\dx : = I + II + III + IV.
\end{split}
\end{equation*}
Note that I and III would add up to a non-positive constant had we assumed that our model was contractive. Instead, we simply drop III and  use the uniform boundedness   \eqref{e:unifL2L2}  which is implied by \eqref{e:ur1} to estimate I by 
the macroscopic relative entropy
\[
I \lesssim \int_{\O}  \rho^\e|u^\e - u|^2 \dx.
\]
Since  $u\in W^{1,\infty}$, using \eqref{e:ur2} the second term is bounded by
\begin{equation*}\label{}
\begin{split}
 \int_\O \rho^\e(u - u^\e) \cdot ( \st_{\rho^\e} \ave{u}_{\rho^\e} -   \st_\rho \ave{u}_\rho ) \dx& \lesssim \int_{\O}  \rho^\e |u^\e - u|^2 \dx +  \int_{\O}   \left| \int_\O (\phi_{\rho^\e} (x,y) \rho^\e(y) -\phi_{\rho} (x,y) \rho(y)) u(y) \dy \right|^2 \rho^\e(x)\dx  \\
 &\leq \int_{\O}  \rho^\e |u^\e - u|^2 \dx+ \int_{\O}   \left| \int_\O (\phi_{\rho^\e} (x,y) -\phi_{\rho} (x,y)) \rho^\e(y) u(y) \dy \right|^2 \rho^\e(x)\dx \\
 & + \int_{\O}   \left| \int_\O \phi_{\rho} (x,y) (\rho^\e(y) -  \rho(y)) u(y) \dy \right|^2 \rho^\e(x)\dx\\
 & \leq  \int_{\O}  \rho^\e |u^\e - u|^2 \dx + \|u\|_\infty^2  \int_{\O}  \int_\O | \phi_{\rho^\e} (x,y) -\phi_{\rho} (x,y)|^2 \rho^\e(y)  \rho^\e(x) \dy \dx\\
 & + \| \p_y ( \phi_{\rho} u)\|_\infty^2 W^2_1(\rho^\e,\rho)   ,
\end{split}
\end{equation*}
which by \eqref{e:ur1}-\eqref{e:ur2} (or in fact by a weaker assumption \eqref{e:glob1} - \eqref{e:glob2}) is bounded further by 
\[
\lesssim  \int_{\O}  \rho^\e |u^\e - u|^2 \dx + W^2_1(\rho^\e,\rho) .
\]
Finally, by \eqref{e:ur2} the last term is bounded by the same quantity
\[
\int_\O \rho^\e(u - u^\e) \cdot u ( \st_{\rho^\e} -\st_{\rho} )\dx \lesssim \int_{\O}  \rho^\e |u^\e - u|^2 \dx +  \int_{\O}  \rho^\e | \st_{\rho^\e} - \st_\rho|^2 \dx \leq   \int_{\O}  \rho^\e |u^\e - u|^2 \dx + W^2_1(\rho^\e,\rho). 
\]
In summary, the alignment term is bounded by 
\begin{equation}\label{e:Amono}
\d A \lesssim \int_{\O}  \rho^\e|u^\e - u|^2 \dx +W^2_1(\rho^\e,\rho)  \leq \en(f^\e|u) +  W^2_2(\rho^\e,\rho).
\end{equation}

Collecting all the estimates together we obtain 
\[
\ddt   \en(f^\e|u)  \lesssim   \en(f^\e|u)  + \frac{\d}{\e} + W^2_2(\rho^\e,\rho).
\]
Note that since the $(x_1,x_2)$-marginal of $\g$ belongs to $\Pi(\rho^\e,\rho)$ we further find $W^2_2(\rho^\e,\rho) \leq W_x$. So, we have obtained the system
\begin{equation*}\label{}
\begin{split}
\ddt W_x & \leq \en(f^\e|u) + c_1 W_x, \\
\ddt   \en(f^\e|u) & \leq c_2 \left( \en(f^\e|u) + W_x + \frac{\d}{\e} \right).
\end{split}
\end{equation*}

Note that the initial value of $\en(f^\e|u)  + W_x$ is bounded by a constant multiple of $\e$ in view of the choice of $\g_0$ for $W_x$ (even $\e^2$ in this case), and
\[
\en(f^\e_0|u_0) = \int_\domain |v - u_0 |^2 \df_0^\e =  \int_\domain |v - u_0 |^2 [ \df_0^\e - \df_0] \leq C W_1(f_0^\e,f_0) \leq C W_2(f_0^\e,f_0) \leq \e.
\]
\GL\  implies $\en(f^\e|u) + W_x \lesssim \e + \frac{\d}{\e}$, and thanks to \eqref{e:syseW},
\[
W_v \leq \e + \frac{\d}{\e}.
\]
We have established \eqref{e:W1ed}.

\end{proof}

\subsection{Maxwellian limit}\label{s:Max}

In this section we provide a derivation of the Euler-alignment system with isothermal pressure for material models on the torus $\O = \T^n$,
\begin{equation}\label{e:macroEAS}
\begin{split}
\rho_t  + \n \cdot (u \rho) & = 0\\
(\rho u)_t + \n \cdot (\rho u \otimes u) +  \n \rho &=  \rho\st_\rho ( \ave{u}_\rho - u).
\end{split}
\end{equation}
Well-posedness of this system has been established for non-vacuous solutions for various models, see \cite{Choi2019,CDS2019}. 

As outlined in the beginning of this section our strategy will be to consider the equation with strong Fokker-Planck penalization force
\begin{equation}\label{e:FPAe2}
\p_t f^\e + v\cdot \n_x f^\e = \frac{1}{\e} [ \D_v f^\e + \n_v \cdot ((v - u^\e_\d)f^\e ) ] + \n_v \cdot (\st_{\rho^\e}(v - \ave{u^\e}_{\rho^\e})f^\e )  ,
\end{equation}
where $u^\e$ is the macroscopic velocity field associated with $f^\e$, and $u^\e_\d$ is the same mollification as defined in the previous monokinetic study. 

Let us briefly discuss regularity of \eqref{e:FPAe2}. In what follows we will study solutions of \eqref{e:FPAe2} that exist on a common tine interval $[0,T]$ independent of $\e$. Unfortunately the local existence result alone stated in \thm{t:lwp} will not provide such  solutions, because the energy bounds (or entropy for that matter) will deteriorate with $\e$. So, the only way to ensure common existence is to guarantee global well-posedness of \eqref{e:FPAe2}. According to \thm{t:gwp} the equation is globally well-posed for thick data if both models -- the native $\cM$ and the mollification $u_\d$ based on $\cM_\psi$ -- are regular and satisfy  \eqref{e:unifL2Linfty}.   Assuming that $\supp \psi = \O$ the model $\cM_\psi$ will fulfill these conditions, and as to the defining model $\cM$, we will make it as an assumption. The focus will now be turned to establishing convergence of the hydrodynamic limit for a given family of solutions.

Let us write out the corresponding macroscopic system
\begin{equation}\label{e:macroe}
\begin{split}
\rho^\e_t  + \n \cdot (u^\e \rho^\e) & = 0\\
(\rho^\e u^\e)_t + \n \cdot (\rho^\e u^\e \otimes u^\e) +  \n \rho^\e + \n_x \cdot \cR_\e &=  \rho^\e \st_{\rho^\e}( \ave{u^\e}_{\rho^\e} - u^\e) + \frac{1}{\e}\rho^\e(u^\e_\d - u^\e)\\
\cR_\e & = \int_{\R^n} ( (v -u^\e) \otimes (v -u^\e) - \I ) f^\e \dv.
\end{split}
\end{equation}
Here, $\I$ is the identity matrix.

We measure the distance between pairs $(u^\e, \rho^\e)$ and $(u, \rho)$ by using the relative entropy between the corresponding local Maxwellians:
\begin{equation}\label{e:MaxLim}
\mu =\frac{ \rho(x,t) }{(2\pi)^{n/2}} e^{- \frac{|v - u(x,t)|^2}{2}}, \quad  \mu^\e =\frac{ \rho^\e(x,t) }{(2\pi)^{n/2}} e^{- \frac{|v - u^\e(x,t)|^2}{2}}.
\end{equation}
In fact such entropy is encoded into the total relative entropy between $f^\e$ and $\mu$:
\[
\cH(f^\e | \mu) = \int_{\domain} f^\e\log \frac{f^\e}{\mu} \dv \dx.
\]
Indeed, the following identity holds,
\begin{equation}\label{e:H1}
\cH(f^\e | \mu) = \cH(f^\e| \mu^\e) + \cH(\mu^\e | \mu),
\end{equation}
\begin{equation}\label{e:H2}
\cH(\mu^\e| \mu) = \frac12 \int_{\O} \rho^\e | u^\e - u|^2 \dx +  \int_{\O}  \rho^\e \log( \rho^\e /\rho) \dx.
\end{equation}
So, if $\cH(f^\e | \mu) \to 0$, then also $\cH(\mu^\e | \mu) \to 0$. Recall that by the classical \CK, see for example \cite{Sbook}, the relative entropy controls $L^1$-distance between the probability densities,
\[
\cH(f|g) \geq c \|f-g\|_{L^1}^2.
\]
So, vanishing of the relative entropy $\cH(\mu^\e | \mu) \to 0$ implies strong limits
\begin{equation}\label{e:macrolimit}
\begin{split}
\rho^\e & \to \rho,\\
\rho^\e u^\e & \to \rho u, \\
\rho^\e |u^\e|^2 & \to \rho |u|^2.
\end{split}
\end{equation}
 in $L^1(\O)$.

\begin{theorem}\label{t:hydroMax} Suppose $\cM$ is a regular model on $\T^n$ satisfying \eqref{e:glob1} - \eqref{e:glob2} and \eqref{e:unifL2Linfty}. Let $(u,\rho)$ be a given smooth non-vacuous solution to \eqref{e:macroEAS} on a time interval $[0,T]$. Suppose that initial distributions $f_0^\e \in H^k_l$ converge to $\mu_0$ in the sense of entropies as $\e \to 0$:
\[
\cH(f^\e_0 | \mu_0) \to 0.
\]
Then for all $\e$ small enough there exists a unique global solution $f^\e \in H^k_l$, and  as long as $\d = o(\e)$, we have
\begin{equation}\label{e:vanrelent}
\sup_{t\in [0,T]} \cH(f^\e | \mu) \to 0.
\end{equation}
\end{theorem}

\begin{remark}
Going back to the discussion of \sect{s:rth} we can see that the theorem applies to many core models on our list. Specifically, we have it for \ref{CS} and \ref{CStopo} unconditionally,  for \ref{Mb} we have it for all local kernels if $\b \geq\frac12$ and for all-to-all kernels $\phi>0$ for any $\b$, the \ref{Mf} requires $\phi >0$ as well, and \ref{Mseg} requires $\supp g_l = \O$.
\end{remark}
\begin{remark}
We also note that in the course of the proof, just like in the monokinetic case, the regularity and  \eqref{e:unifL2Linfty} conditions are only needed to facilitate global existence of solutions, while the bounds  \eqref{e:unifL2L2}, \eqref{e:glob1} - \eqref{e:glob2}  are used in the actual estimates. So, for putative classical solutions  to \eqref{e:FPAe2} the result extends to a wider range of models listed on the last row of Table~\ref{t:rth}.
\end{remark}

\begin{proof} First let us notice that by the \CK,
\[
\cH(f^\e_0 | \mu_0) \geq \| \rho_0^\e - \rho_0\|_1^2.
\]
Since, $\rho_0 >0$ on $\O$, it implies that $\Th(\rho_0,\O) >0$ by \defin{d:th} (iii), and by (iv) we have
\[
|\Th(\rho^\e_0,\O) - \Th(\rho_0,\O)| \leq c \| \rho_0^\e - \rho_0\|_1 \to 0,
\]
so starting from some $\e_0$ we have $\Th(\rho^\e_0,\O) >\d >0$, for $\e<\e_0$. Such initial conditions give rise to global solutions by  \thm{t:gwp}.

Let us break down the relative entropy into kinetic and macroscopic parts:
\begin{equation}\label{}
\begin{split}
\cH(f^\e | \mu) & = \cH_\e + \cG_\e \\
\cH_\e & = \int_{\domain} \left( f^\e \log f^\e + \frac12 |v|^2 f^\e \right)\dv \dx +  \frac{n }{2} \log(2\pi)\\
\cG_\e & = \int_{\O} \left( \frac12 \rho^\e |u|^2 - \rho^\e u^\e \cdot u - \rho^\e \log \rho \right) \dx.
\end{split}
\end{equation}

Let us state the energy bounds for each component.  In the sequel we denote for short $\k^\e = \k_{\rho^\e}$.
\begin{lemma}\label{l:elawepsilon} There are constants $c_1,c_2,c_3$ that depend only on the model such that we have the following entropy law:
\begin{align}
\cH_\e, \cE_\e & \in L^\infty([0,T]) \text{ uniformly in } \e, \label{e:HM} \\ 
\ddt \cH_\e  & \leq - \frac{1}{\e} \cI_\e +c_2 \,\e\, \en(f^\e | u^\e)     - \|u^\e\|^2_{L^2(\k^\e)} + ( u^\e, \ave{u^\e}_{\rho^\e})_{\k^\e} , \label{e:Hident}
\end{align}
where 
\[
\cI_\e = \int_{\domain}\frac{|\n_v f^\e+ (1+ \e \st_{\rho^\e} / 2)   (v -u^\e) f^\e|^2}{f^\e}  \dv\dx.
\]
\end{lemma}
\begin{proof}
Differentiating,
\begin{equation}\label{e:He}
\begin{split}
\ddt \cH_\e  = &- \frac{1}{\e} \int_{\domain}\left[ \frac{|\n_v f^\e|^2}{f^\e} + 2 \n_v f^\e \cdot (v -u^\e_\d) + |v -u^\e_\d|^2 f^\e \right]  \dv\dx \\
 &- \frac{1}{\e} [ (u^\e_\d,u^\e)_{\rho^\e} - (u^\e_\d,u^\e_\d)_{\rho^\e}]\\
& - \int_{\domain} \st_{\rho^\e}[\n_v f^\e \cdot (v - \ave{u^\e}_{\rho^\e})  + v\cdot (v - \ave{u^\e}_{\rho^\e}) f^\e ] \dv\dx.
\end{split}
\end{equation}
To prove \eqref{e:HM} we simply dismiss the first information term, and recall that the $\d$-mollification constitutes the \ref{Mf}-averaging which is ball-positive. So, the second term, according to \eqref{e:bpenergies} is also non-negative and we dismiss it too. We estimate the third term as follows
\begin{equation*}\label{}
\begin{split}
& - \int_{\domain}  \st_{\rho^\e} [\n_v f^\e \cdot (v - \ave{u^\e}_{\rho^\e})  + v \cdot (v - \ave{u^\e}_{\rho^\e}) f^\e ] \dv \dx \\
& = n  \int_{\domain}  \st_{\rho^\e} f^\e \dv \dx -  \int_{\domain}  \st_{\rho^\e} |v|^2 f^\e \dv\dx +( u^\e, \ave{u^\e}_{\rho^\e})_{\k^\e}  \leq C - \int_\O   \st_{\rho^\e} \rho^\e |u^\e|^2 \dx +( u^\e, \ave{u^\e}_{\rho^\e})_{\k^\e} \\
& =  c_1 - \|u^\e\|^2_{L^2(\k^\e)} + ( u^\e, \ave{u^\e}_{\rho^\e})_{\k^\e}.
\end{split}
\end{equation*}
Now, according to \eqref{e:unifL2L2} the averaging operators are uniformly bounded on $L^2(\k_{\rho^\e})$. So, we obtain
\[
\ddt \cH_\e \leq c_1 + c_2 \cE_\e, \qquad \cE_\e = \frac12 \int_{\domain} |v|^2 f^\e \dv \dx,
\]
and according to \eqref{e:EH},
\[
\ddt \cH_\e \leq c_3 + c_4 \cH_\e.
\]
This proves  \eqref{e:HM}.

To show \eqref{e:Hident} we replace all the macroscopic velocities in \eqref{e:He} with the native one $u^\e$. Indeed, in the information term we have 
\begin{equation*}\label{}
\begin{split}
& -  \frac{1}{\e} \int_{\domain}\left[ \frac{|\n_v f^\e|^2}{f^\e} + 2 \n_v f^\e \cdot (v -u^\e_\d) + |v -u^\e_\d|^2 f^\e \right]  \dv\dx \\
 =& - \frac{1}{\e} \int_{\domain}\left[ \frac{|\n_v f^\e|^2}{f^\e} + 2 \n_v f^\e \cdot (v -u^\e) + |v -u^\e|^2 f^\e + |u^\e - u^\e_\d|^2 f^\e \right]  \dv\dx \\
 \leq &  - \frac{1}{\e} \int_{\domain}\left[ \frac{|\n_v f^\e|^2}{f^\e} + 2 \n_v f^\e \cdot (v -u^\e) + |v -u^\e|^2 f^\e \right]  \dv\dx \\
  = &- \frac{1}{\e}  \int_{\domain}\frac{|\n_v f^\e+  (v -u^\e) f^\e|^2}{f^\e}  \dv\dx .
\end{split}
\end{equation*}
For the alignment term we obtain similarly,
\[
\begin{split}
&- \int_{\domain}    \st_{\rho^\e} [\n_v f^\e \cdot (v - \ave{u^\e}_{\rho^\e})  + v \cdot (v - \ave{u^\e}_{\rho^\e}) f^\e ] \dv\dx  \\
  = & - \int_{\domain}    \st_{\rho^\e}  \n_v f^\e \cdot (v - u^\e) \dv\dx  - \int_{\domain} \st_{\rho^\e} [ v \cdot (v - u^\e) f^\e + v\cdot(u^\e - \ave{u^\e}_{\rho^\e}) f^\e ] \dv\dx \\ 
=  &-   \int_{\domain}    \st_{\rho^\e} (  \n_v f^\e \cdot (v - u^\e)+  |v - u^\e |^2 f^\e ) \dv\dx  - \|u^\e\|^2_{L^2(\k^\e)} + ( u^\e, \ave{u^\e}_{\rho^\e})_{\k^\e} .
\end{split}
\]
Combing the two expressions and completing the squares 
\begin{equation*}\label{}
\begin{split}
\ddt \cH_\e & \leq  - \frac{1}{ \e}  \int_{\domain}  \frac{|\n_v f^\e+ (1+ \e \st_{\rho^\e} / 2) (v -u^\e) f^\e|^2}{f^\e}  \dv\dx  + \frac{\e}{4}   \int_{\domain}    \st_{\rho^\e}|v - u^\e|^2 f^\e  \dv\dx \\
& \hspace{3.2in}- \|u^\e\|^2_{L^2(\k^\e)} + ( u^\e, \ave{u^\e}_{\rho^\e})_{\k^\e}\\
& \leq -\frac{1}{\e} \cI_\e  + c \e \en(f^\e | u^\e)     - \|u^\e\|^2_{L^2(\k^\e)} + ( u^\e, \ave{u^\e}_{\rho^\e})_{\k^\e} 
\end{split}
\end{equation*}
We have obtained \eqref{e:Hident}.
\end{proof}

\begin{lemma}\label{}
We have the following inequality
\begin{equation}\label{e:Ge}
\ddt \cG_\e \leq C \cH(f^\e | \mu) + C \sqrt{ \cI_\e} + C\e   + \frac{\d}{\e}+   \|u^\e\|^2_{L^2(\k^\e)} - ( u^\e, \ave{u^\e}_{\rho^\e})_{\k^\e} ,
\end{equation}
where $C$ is independent of $\e$.
\end{lemma}
\begin{proof} Let us compute the derivative of each component of $\cG_\e$
\[
\begin{split}
\ddt \frac12 \int_{\O} \rho^\e |u|^2 \dx & = \int_\O [ \rho^\e (u^\e - u) \cdot \n u \cdot u - \rho^\e  u \cdot \n \log \rho  + \rho^\e \st_\rho   ( \ave{u}_\rho - u) \cdot u ] \dx \\
\ddt \int_{\O^n} \rho^\e u^\e \cdot u  \dx& = \int_\O[ \rho^\e (u^\e - u) \cdot \n u \cdot u^\e + \rho^\e \n \cdot u - \rho^\e u^\e \cdot \n \log \rho - \n u: \cR_\e \\
&+  \rho^\e \st_{\rho^\e}  (\ave{u^\e}_{\rho^\e} - u^\e) \cdot u +  \rho^\e \st_\rho (\ave{u}_\rho - u) \cdot u^\e + \frac{1}{\e}\rho^\e(u^\e_\d - u^\e)\cdot u ] \dx\\
\ddt \int_{\O^n} \rho^\e \log \rho \dx &= \int_\O [ \rho^\e u^\e \cdot \n \log \rho - \rho^\e  u \cdot \n \log \rho - \rho^\e \n \cdot u ] \dx.
\end{split}
\]
Thus,
\[
\ddt \cG_\e =\int_\O[ \n u: \cR_\e  -  \rho^\e (u^\e - u) \cdot \n u \cdot (u^\e - u)  ] \dx + A + A_\loc,
\]
where $A_\loc$ is the same local alignment terms as appeared in the previous section, and
\begin{equation*}\label{}
\begin{split}
A & =  \int_\O[ \rho^\e \st_\rho   ( \ave{u}_\rho - u) \cdot u -  \rho^\e \st_{\rho^\e}  (\ave{u^\e}_{\rho^\e} - u^\e) \cdot u -  \rho^\e  \st_\rho (\ave{u}_\rho - u)\cdot u^\e  ]\dx \\
& = \d A + \|u^\e\|^2_{L^2(\k^\e)} - ( u^\e, \ave{u^\e}_{\rho^\e})_{\k^\e},
\end{split}
\end{equation*}
where $\d A$ is again the same alignment term that appeared in the previous section. We estimate $A_\loc$ and $\d A$ as before using \eqref{e:amplify2} and the intermediate estimate in \eqref{e:Amono}. We  recall that only  \eqref{e:glob1} - \eqref{e:glob2} and regularity of the kernel $\phi_\rho$ are necessary to prove \eqref{e:Amono}. Since $\rho \geq  \rho_->0$ we have both by the assumptions. Thus,
\begin{equation}\label{ }
\d A + A_\loc \lesssim  \int_{\O}  \rho^\e|u^\e - u|^2 \dx +W^2_1(\rho^\e,\rho) + \frac{\d}{\e}.
\end{equation}
Keeping in mind that both the macroscopic relative entropy $ \int_{\O}  \rho^\e|u^\e - u|^2 \dx$ and $W_1^2(\rho^\e,\rho) \leq \| \rho^\e - \rho\|_1^2$ are controlled by $\cH(f^\e|\mu)$, see \eqref{e:H1}, \eqref{e:H2} we obtain
\begin{equation}\label{ }
\d A + A_\loc \lesssim  \cH(f^\e|\mu) + \frac{\d}{\e}.
\end{equation}

Next, given that $u$ is smooth we have
\begin{equation}\label{e:Ge1}
\left| \int_\O \rho^\e (u^\e - u) \cdot \n u \cdot (u^\e - u)  ] \dx \right|  \lesssim \int_{\O}  \rho^\e |u^\e - u|^2  \dx \leq \cH(f^\e|\mu).
\end{equation}

As to the Reynolds stress, we will use a well-known estimate from \cite{MV2008} that establishes  a bound in terms of information and energy. Let us rerun this argument to account for the $\e$-correction. We simply note that 
\[
\cR_\e = \int_{\R^n}  [ 2 \n_v \sqrt{f^\e} + (v -u^\e) \sqrt{f^\e}] \otimes [ (v -u^\e) \sqrt{ f^\e}] \dv.
\]
then we reinsert the $\e$-correction to obtain
\begin{equation*}\label{}
\cR_\e  = \int_{\R^n}  [ 2 \n_v \sqrt{f^\e} + (1+ \e \st_{\rho^\e}/2)(v -u^\e) \sqrt{f^\e}] \otimes [ (v -u^\e) \sqrt{ f^\e}]  \dv - \e \st_{\rho^\e}/2 \int_{\R^n}   (v -u^\e)\otimes  (v -u^\e)  f^\e   \dv .
\end{equation*}
So, 
\[
 \int_{\O}|\cR_\e | \dx \lesssim  \sqrt{\en(f^\e | u^\e)  \cI_\e} + \e \en(f^\e | u^\e)  \lesssim  \sqrt{ \cI_\e} + \e.
\]
Collecting the obtained estimates together we obtain \eqref{e:Ge}.
\end{proof}

Combining the equations on $\cH_\e$ and $\cG_\e$, \eqref{e:Hident}, \eqref{e:Ge},  
we see that the residual alignment-energy terms cancel and we obtain
\[
\ddt \cH(f^\e | \mu)  \lesssim \cH(f^\e | \mu)   - \frac{1}{\e} \cI_\e  +  \e + \frac{\d}{\e} +  \sqrt{  \cI_\e} \leq \cH(f^\e | \mu)  - \frac{1}{2\e} \cI_\e + 2 \e + \frac{\d}{\e} \lesssim  \cH(f^\e | \mu) +\e + \frac{\d}{\e}.
\]
By the \GL\ we obtain
\[
\cH(f^\e | \mu) \leq \cH(f^\e_0 | \mu_0) e^{C T} + C (\e + \frac{\d}{\e}) e^{CT}, \quad \forall t \leq T,
\]
where $C$ depends only on the parameters of the model and the regularity of $(u,\rho)$. This finishes the proof.

\end{proof}

\begin{remark}\label{}
The same observation can be made here as in the monokinetic case. If we quantify the initial entropy 
\[
\cH(f^\e_0 | \mu_0) \leq \e,
\]
then the proof produces the bound
\[
\cH(f^\e | \mu) \leq \e +  \frac{\d}{\e}.
\]
So, again, the optimal convergence is achieved when $\d \sim \e^2$. However, unlike in the monokinetic case, here we do not loose on the magnitude of the entropy at positive times. 
\end{remark}

\subsection{Remarks on the pressureless Euler Alignment System}\label{s:EASnew} We will leave discussion of the well-posedness of macroscopic systems that arise from general models $\cM$ to a future research, see \cite{TT2014,CCTT2016,Sbook,LS-uni1,MT2014,HeT2017} for the literature on this problem specifically for smooth communication models.  The most clear-cut result obtained in \cite{CCTT2016} pertains to the regularity of the 1D pressureless EAS based on the Cucker-Smale protocol
\[
\p_t u + uu_x = \rho_\phi u - (u\rho)_\phi.
\]
Here, one finds an additional conserved quantity
\[
e = u_x + \rho_\phi
\]
which controls $u_x$ and hence regularity of the system. In fact, $e$ satisfies
\[
\p_t e + \p_x( u e ) = 0
\]
or in Lagrangian coordinates associated with $u$,
\[
\ddt e =  e \rho_\phi - e^2 = e(  \rho_\phi  - e),
\]
which is  a non-homogeneous logistic ODE.  The critical threshold for regularity becomes $e_0 \geq 0$.

In multi-D, the law of $e$ is given by 
\begin{equation}\label{e:eD}
\begin{split}
e & = \diver u + \rho_\phi, \\
\partial_t e+\n \cdot(u  e) & =(\n \cdot u )^{2}-\Tr[(\nabla u )^{2}].
\end{split}
\end{equation}
Although the right hand side in this case involves $\n u$, in some cases this still allows to obtain partial regularity results in multi-D, for example for small data or for unidirectional flocks, see \cite{TT2014,HeT2017,LS-uni1}.  The latter are solutions of the form
\[
u = (u(x_1,\ldots,x_n),0,\ldots,0).
\]
For these the right hand side of \eqref{e:eD} vanishes.

While the existence of $e$ is attributed to the particular commutator structure of the alignment forcing of the Cucker-Smale model, in general, it can be seen  as a consequence of another property of the model -- transport of the specific strength function $\st_\rho$ itself. Indeed, let us notice that in the \ref{CS}-case we have
\begin{equation}\label{e:scons}
\p_t \st_\rho + \n\cdot ( \st_\rho \ave{u}_\rho) = 0,
\end{equation}
simply because $\rho_\phi$ is transported by the Favre-filtration $\uF =(u\rho)_\phi / \rho_\phi$.  This turns out to be the general reason for the conservation of $e$. 

\begin{lemma}\label{ }
If for any solution of the pressureless EAS \eqref{e:EASvelless} the strength function satisfies \eqref{e:scons}, then
$e = \diver u + \st_\rho$ satisfies \eqref{e:eD}. In particular, $e$ is conserved for all solutions in 1D and unidirectional solutions in multi-D.
\end{lemma}
\begin{proof}\label{ }
By direct verification.
\end{proof}

The above observation motivates to consider a system where the strength is not fixed  but rather evolves according to the `natural law' \eqref{e:scons}, whereby the strength itself becomes another unknown. This leads to the following system
\begin{equation}\label{e:EASentropic}
\begin{split}
\p_t \rho & + \n\cdot (u \rho)  = 0, \\
\p_t \st &+ \n\cdot ( \st \ave{u}_\rho)  = 0,\\
\p_t u &+ u \cdot \n u   = \st (\ave{u}_\rho - u).
\end{split}
\end{equation}
All such systems will satisfy the $e$-law by design, where $e = \diver u + \st$.

For example, if we start from the initial Favre-based model, $\ave{u}_\rho = \uF$ and set $\st_0 = 1$, like for instance in the \ref{MT}-model,  the future value of strength will be determined by the transport along the averaged velocity $\ave{u}_\rho$, rather than being forcefully set at $\st=1$ for all times.  Given that both $\st$ and $\rho_\phi$ solve the same continuity equation in this case, we also have transport of the ratio
\[
\p_t \frac{\st}{\rho_\phi} + \uF \cdot \n_x \frac{\st}{\rho_\phi} = 0.
\]
This implies 
\[
c_1 \rho_\phi \leq \st \leq c_2\rho_\phi, \quad \forall (t,x) \in [0,\infty) \times \O,
\]
if initially so. In particular $\st$ remains uniformly bounded regardless of the regularity of $u_F$ (!).

A thorough study of this model has been recently completed for Favre-based modes in \cite{STT2023} during the review of this present work. 

\section{Appendix: proof of \eqref{e:KMTrhobeta}}

We start as in \cite[Lemma 5.2]{KMT2013}. Let us fix $x\in \O$ and consider a cover of the ball $B_R(x)$ by balls of radius $r/2$:
\[
B_R(x) \ss \bigcup_{i=1}^I B_{r/2}(x_i),
\]
where $N$ depends only on $n$ and $R/r$. Then assuming \eqref{e:kerRr} we obtain
\begin{equation*}\label{}
\begin{split}
\left( \frac{\rho}{\rho^{1-\b}_\phi}\right)_\phi (x) = \int_{B_R(x)} \frac{\rho(y) \phi(x-y)}{\left( \int_\O \rho(z) \phi(y-z) \dz \right)^{1-\b}} \dy \leq \sum_{i=1}^I \int_{B_{r/2}(x_i)} \frac{\rho(y) \phi(x-y)}{\left( \int_{B_{r/2}(x_i)} \rho(z) \phi(y-z) \dz \right)^{1-\b}} \dy .
\end{split}
\end{equation*}
Since $y,z\in B_{r/2}(x_i)$, we have $|y-z|\leq r$, and by the lower bound on the kernel we obtain
\[
\left( \frac{\rho}{\rho^{1-\b}_\phi}\right)_\phi (x) \lesssim \sum_{i=1}^I \int_{B_{r/2}(x_i)} \frac{\rho(y) \phi(x-y)}{\left( \int_{B_{r/2}(x_i)} \rho(z)  \dz \right)^{1-\b}} \dy.
\]
If $\b = 0$, we remove the kernel by $\|\phi\|_\infty$, and the rest adds up to $I$. 

In the case when $\b>0$ we have
\[
\left( \frac{\rho}{\rho^{1-\b}_\phi}\right)_\phi (x) \lesssim \sum_{i=1}^I \int_{B_{r/2}(x_i)} \frac{\rho(y) \phi(x-y)}{\int_{B_{r/2}(x_i)} \rho(z)  \dz} \left( \int_{B_{r/2}(x_i)} \rho(z)  \dz \right)^{\b} \dy.
\]
Treating 
\[
\one_{B_{r/2}(x_i)}(y) \frac{\rho(y)}{\int_{B_{r/2}(x_i)} \rho(z)  \dz} \dy
\]
as a probability measure for each integral, by the \HI, we obtain
\begin{equation*}\label{}
\begin{split}
\left( \frac{\rho}{\rho^{1-\b}_\phi}\right)_\phi (x) & \lesssim \sum_{i=1}^I \left( \int_{B_{r/2}(x_i)} \frac{\rho(y) \phi^{\frac{1}{\b}}(x-y)}{\int_{B_{r/2}(x_i)} \rho(z)  \dz} \int_{B_{r/2}(x_i)} \rho(z)  \dz  \dy \right)^{\b} \\
& = \sum_{i=1}^I \left( \int_{B_{r/2}(x_i)} \rho(y) \phi^{\frac{1}{\b}}(x-y)\dy \right)^{\b} \\
& \leq \|\phi\|_\infty^{1-\b} \sum_{i=1}^I \left( \int_{B_{r/2}(x_i)} \rho(y) \phi(x-y)\dy \right)^{\b} \leq I \|\phi\|_\infty^{1-\b} \rho_\phi^\b(x).
\end{split}
\end{equation*}

\section{Appendix: averagings on finite sets and proof of \prop{p:bpcons}}\label{a:fm}

In this section we will prove \prop{p:bpcons}.

To achieve it we first study properties of models on finite sets -- to which as we will  see the result is reduced. We will use the notation of \exam{ex:finite} below.

For models on finite sets being conservative is equivalent to 
\[
A^{\top} \k = \k, \qquad \k = (\k_1, \dots, \k_N).
\]
Denoting $K = \diag\{ \k_1, \dots, \k_N\}$ one can see that being symmetric is equivalent to the matrix $KA$ being symmetric,
\[
(KA)^{\top} = KA.
\]
Similarly, the ball-positivity is equivalent to the matrix $A$ being ball-coercive relative to the inner product $(\cdot,\cdot)_K = (K \cdot, \cdot)$:
\[
( A u, u)_K \geq  (Au,Au)_K.
\]
\begin{lemma}\label{}
If $\cM$ is ball-positive on a 2-point set, then $\cM$ is symmetric.
\end{lemma}
\begin{proof}
The result reduces to showing that $\k_1 a_{12} = \k_2 a_{21}$ for any ball-coercive model.  Coercivity is equivalent to 
\[
\k_1 a_{11} u_1^2 + (\k_1 a_{12} + \k_2 a_{21}) u_1 u_2 + \k_2 a_{22} u_2^2 \geq \k_1(a_{11} u_1 + a_{12} u_2)^2 + \k_2(a_{21} u_1 + a_{22} u_2)^2 .
\]
Collecting coefficients in front of each monomial we obtain
\[
\a u_1^2 +  \b u_1 u_2 +  \g u_2^2 \geq 0,
\]
where 
\[
\a = \k_1 a_{11} -\k_1 a_{11}^2 - \k_2 a_{21}^2, \quad \b = \k_1 a_{12} + \k_2 a_{21} - 2 \k_1 a_{11} a_{12} - 2\k_2 a_{21} a_{22}, \quad \g = \k_2 a_{22} - \k_1 a_{12}^2 - \k_2 a_{22}^2.
\]
This means that the determinant of the quadratic form is non-negative
\[
4 \a \g \geq \b^2.
\]
Using stochasticity of $A$ and after a long but elementary computation, the above condition reduces to
\[
(\k_1 a_{12} - \k_2 a_{21})^2 \leq 0,
\]
which proves the result.
\end{proof}

\begin{proof}[Proof of \prop{p:bpcons}]

Since the averages act coordinatewise it is sufficient to prove the result for scalar fields $u$. 

Let us fix $\rho$. Let us pick any partitioning of $\O$ into two sets $A,B$ and assume that $\nu(A), \nu(B) >0$. Let us denote
\begin{align*}
a_{11} & = \frac{1}{\k_\rho(A)}\int_A \ave{\one_A}_\rho \dk_\rho, \qquad a_{12}  = \frac{1}{\k_\rho(A)}\int_A \ave{\one_B}_\rho \dk_\rho; \\
a_{21} & = \frac{1}{\k_\rho(B)}\int_B \ave{\one_A}_\rho \dk_\rho, \qquad a_{22}  = \frac{1}{\k_\rho(B)}\int_B \ave{\one_B}_\rho\dk_\rho.
\end{align*}
Note that the matrix $A = ( a_{ij} )_{i,j = 1}^2$ is right stochastic.  Denoting $\k_1 = \k_\rho(A)$, $\k_2 = \k_\rho(B)$ and verifying coercivity on functions of the form $u = u_1 \one_A + u_2 \one_B$ we obtain
\[
\k_1 a_{11} u_1^2 + (\k_1 a_{12} + \k_2 a_{21}) u_1 u_2 + \k_2 a_{22} u_2^2  \geq \int_\O | u_1 \ave{\one_A}_\rho + u_2 \ave{\one_B}_\rho |^2 \dk_\rho.
\]
Breaking down the integral and using the \HI, we obtain
\begin{equation*}\label{}
\begin{split}
 \int_\O | u_1 \ave{\one_A}_\rho + u_2 \ave{\one_B}_\rho |^2 \dk_\rho & =  \int_A | u_1 \ave{\one_A}_\rho + u_2 \ave{\one_B}_\rho |^2 \dk_\rho +  \int_B | u_1 \ave{\one_A}_\rho + u_2 \ave{\one_B}_\rho |^2 \dnu \\
& \geq \frac{1}{\k_\rho(A)} \left|  \int_A (u_1 \ave{\one_A}_\rho + u_2 \ave{\one_B}_\rho )\dk_\rho \right|^2 +\frac{1}{\k_\rho(B)} \left|  \int_B (u_1 \ave{\one_A} _\rho+ u_2 \ave{\one_B}_\rho ) \dk_\rho\right|^2 \\
& = \k_1(a_{11} u_1 + a_{12} u_2)^2 + \k_2(a_{21} u_1 + a_{22} u_2)^2,
\end{split}
\end{equation*}
which implies that the $2$-point model with $A$ and $\k$ defined above is ball-positive. The previous lemma implies that 
\begin{equation}\label{e:ABsymm}
\int_A \ave{\one_B}_\rho \dk_\rho = \int_B \ave{\one_A}_\rho \dk_\rho.
\end{equation}
We further conclude
\begin{equation*}\label{}
\begin{split}
\int_\O \ave{\one_A}_\rho \dk_\rho & = \int_A \ave{\one_A}_\rho \dk_\rho + \int_B \ave{\one_A}_\rho \dk_\rho \\
& =  \int_A \ave{\one_A}_\rho \dk_\rho + \int_A \ave{\one_B}_\rho \dk_\rho = \int_A \ave{\one_\O}_\rho \dk_\rho = \k_\rho(A).
\end{split}
\end{equation*}
In other words, the conservative property holds for all characteristic functions. Since it is also linear and the average, by our assumption, is a bounded operator on $L^2(\k_\rho)$ we obtain the result by the standard approximation.
\end{proof}

It would seem like \eqref{e:ABsymm} is suggestive of symmetry as it holds for any pair of partitioning sets. However, to prove general symmetry one would have to make the same conclusion for any pair of disjoint sets not necessarily partitioning $\O$, or for any triple of partitioning sets. The above argument fails to do it, and in fact the implication 
`` ball-positive  $\Rightarrow$ symmetric " 
is generally not true.  A finite dimensional example can be found via a $3$-point construction.

\begin{example}\label{}
Let us assume for simplicity that $\k = \one = (1,1,1)$. Then we are looking for a matrix that is non-symmetric yet doubly stochastic, $A\one = A^\top \one = \one$, and ball-positive.  

Thanks to stochasticity, $A$ leaves the space $X = \one^\perp$ invariant, and so it is enough to properly define $A$ on the 2-dimensional space $X$ only. Let us fix a non-orthogonal basis in $X$:  $e_1 = (1,-1,0)$, $e_2 = (1,0,-1)$, and complement it to $e_3 = \one$. We define 
\[
A e_1 = \l_1 e_1, \qquad Ae_2 = \l_2 e_2,
\]
where $1>\l_i >0$ and $\l_1 \neq \l_2$. This choice guarantees that the matrix $A$ is not symmetric. Now, we need to make sure that $A$ is ball-positive.  Again, by stochasticity, ball-positivity reduces to that of the restriction $\rest{A}{X}$. The latter is equivalent to the condition 
\[
(e_1 + t e_2) \cdot (\l_1 e_1 + t \l_2 e_2) \geq | \l_1 e_1 + t \l_2 e_2|^2,
\]
for all $t\in \R$. Expanding we obtain
\[
(\l_2 - \l_2^2) t^2 + [ \frac12(\l_1+ \l_2) - \l_1 \l_2 ] t + \l_1 - \l_1^2 \geq 0.
\]
This is equivalent to
\begin{equation}\label{e:l1}
(\l_1+ \l_2 - 2 \l_1 \l_2 )^2 \leq 16 (\l_2 - \l_2^2)(\l_1 - \l_1^2).
\end{equation}
In addition we need to ensure that all the entries of the matrix $A$ in the original system of coordinates are non-negative. We can write down these entries explicitly:
\[
A = \frac13 
\begin{pmatrix}
1+\l_1 + \l_2 & 1+ \l_2 - 2 \l_1 & 1+\l_1 - 2\l_2 \\
1-\l_1 & 1 + 2\l_1 & 1-\l_1 \\
1-\l_2 & 1 - \l_2 &  1+2\l_2 
\end{pmatrix}.
\]
So the only conditions to guarantee are
\begin{equation}\label{e:l2}
1+ \l_2 - 2 \l_1 \geq 0, \quad  1+\l_1 - 2\l_2 \geq 0.
\end{equation}
There are plenty of choices to fulfill both \eqref{e:l1} and \eqref{e:l2}. For example, $\l_1 = \frac12$, $\l_2 = \frac13$. This concludes the construction.
\end{example}

\section{Appendix: on spectral gaps} \label{a:gap}

With regard to the discussion of \rem{r:gaps}, we prove a lemma that establishes equivalence of numerical ranges on the space of zero-momenta and the mean-zero functions.  

\begin{lemma}\label{l:sgap}
Suppose $\cM$ is conservative and satisfies the following
\begin{align}
c_0 \leq \st_\rho(x) \leq c_1,\quad \forall x \in \supp \rho,& \label{e:sbounds} \\
\sup \left\{ (u, \ave{u}_\rho)_{\k_\rho}: \,   u\in L^2_0(\k_\rho),\, \ \|u\|_2=1 \right\} & \leq 1-  \e. \label{e:sgap0}
\end{align}
Then
 \begin{equation}\label{e:sgap1}
\sup \left\{ (u, \ave{u}_\rho)_{\k_\rho}: \,  u\in L^2(\k_\rho),\, \bar{u} = 0,\, \ \|u\|_2=1 \right\} \leq 1-  \e \frac{c_0}{c_0 + c_1}.
\end{equation}
Conversely, if
 \begin{equation}
\sup \left\{ (u, \ave{u}_\rho)_{\k_\rho}: \,  u\in L^2(\k_\rho),\, \bar{u} = 0,\, \ \|u\|_2=1 \right\} \leq 1-  \d,
\end{equation}
then
\begin{equation}
\sup \left\{ (u, \ave{u}_\rho)_{\k_\rho}: \,  u\in L^2_0(\k_\rho),\,  \|u\|_2=1 \right\} \leq 1-  \d \frac{c_0}{c_0 + c_1}.
\end{equation}
\end{lemma}
\begin{proof} First let us observe that the bounds on $\st_\rho$, \eqref{e:sbounds}, imply bounds on $\k_\rho$-masses
\begin{equation}\label{e:massmass}
c_0 \leq \k_\rho(\O) \leq c_1.
\end{equation}

Let us denote $\P : L^2(\k_\rho) \to \R^n$ the orthogonal projection onto the space of constant fields. We have for all $u$ with $\bar{u} = 0$,
\[
\left | \int_\O (u - \P u ) \rho \dx \right| = |\P u|  = \frac{1}{\sqrt{\k_\rho(\O)}} \| \P u\|_2.
\]
On the other hand, by (i),
\[
\left | \int_\O (u - \P u ) \rho \dx \right|  = \left | \int_\O (u - \P u ) \frac{1}{\st_\rho} \dk_\rho \right| \leq \frac{\sqrt{\k_\rho(\O)}}{c_0} \| u - \P u\|_2.
\] 
Using  compatibility of masses \eqref{e:massmass}, 
\[
\| u - \P u\|_2 \geq \frac{c_0}{c_1}  \| \P u\|_2.
\]
 Hence,
 \[
 \| u\|_2^2 = \| u - \P u\|_2^2 + \| \P u\|_2^2 \geq (1 + \frac{c_0}{c_1}  ) \| \P u\|_2^2,
 \]
 or
 \begin{equation}\label{e:Pue1}
\| \P u\|_2^2 \leq \frac{c_1}{c_0 + c_1} \| u\|_2^2.
\end{equation}

 Now, let us compute the numerical range, noting that $\ave{\P u}_\rho = \P u $, 
 \[
 (u,\ave{u}_\rho)_{\k_\rho} = ( u - \P u,\ave{u - \P u}_\rho )_{\k_\rho}  + (u - \P u, \P u )_{\k_\rho} +( \P u,\ave{u- \P u}_\rho )_{\k_\rho} + \| \P u \|_2^2.
\]
The second term vanishes due to orthogonality. For the third term we observe that due to the conservative property of the average integrating against a constant field produces the same result as integrating without the average.  So,
\[
( \P u,\ave{u- \P u}_\rho )_{\k_\rho} = ( \P u, u- \P u )_{\k_\rho} = 0.
\]
Using the spectral gap condition for the first term and \eqref{e:Pue1} for the last one,  we obtain
\begin{multline*}
 (u,\ave{u})_{\k_\rho} \leq (1-\e_0) \| u - \P u \|_2^2 + \|\P u \|_2^2 = (1-\e_0) \|u\|_2^2 + \e_0  \|\P u \|_2^2 \\ \leq \left(1-\e_0 +\e_0 \frac{c_1}{c_0 + c_1}  \right) \|u\|_2^2 =  \left(1- \e_0 \frac{c_0}{c_0 + c_1} \right) \|u\|_2^2 .
 \end{multline*}
 
 To obtain the converse statement, apply the same argument replacing the roles of $\rho$ and $\k_\rho$, and note that $1/c_1 \leq 1/\st_\rho \leq 1/c_0$.
 \end{proof}

\section{Appendix: categorial considerations}

Environmental averagings form an 'ecosystem' of models. On a more formal level they can be thought of as a category of objects and we can discuss relationships between them.  

For a couple of models $\cM'$, $\cM''$ defined over $\O'$ and $\O''$, respectively, a morphism $\cM' \to \cM''$ is defined by a volume preserving homeomorphism 
$\t : \O' \to \O''$ such that if $\rho''\circ \t = \rho'$ and $u'' \circ \t = u'$, then
\[
\ave{u''}_{\rho''}'' \circ \t = \ave{u'}_{\rho'}',
\]
and there exist two constants $c,C>0$ such that 
\[
c \dk'_{\rho'} \leq \dk''_{\rho''} \circ \t \leq C\dk'_{\rho'}.
\]
For material models, the latter can be restated in terms of specific strengths
\[
c \st'_{\rho'} \leq \st''_{\rho''} \circ \t \leq C \st'_{\rho'}.
\]
We have tacitly employed this concept in Appendix~\ref{a:fm} when discussing models on finite sets.

On a given environment $\O$ all models can be partially ordered is several ways. The most straightforward definition of $\cM'  \preceq \cM''$ is
\[
\ave{\ave{u}_\rho'}''_\rho = \ave{\ave{u}''_\rho}'_\rho = \ave{u}'_\rho, \quad \forall u\in L^\infty(\O). 
\]
For example, among rough segregation models we have $\cM_{\cF'} \preceq \cM_{\cF''}$ provided $\cF' \ss \cF''$. The identity model \ref{Id} is the finest of all material ones  (although if we defined it to be $\ave{u} = u$ irrelevant of the $\supp \rho$, then it would have become the finest of all). At the same time \ref{global} is the coarsest among all conservative ones with $\st_\rho = 1$. 

A more refined definition of order can be given on classes of equivalence where we say $\cM' \sim \cM''$ if there exist intermediate averagings $\cM_1, \ldots, \cM_n$ such that for any $\rho \in \cP$ there exist  $\rho_1, \ldots, \rho_n \in \cP$ such that 
\[
\ave{ \dots \ave{\ave{u}''_\rho}^1_{\rho_1} \dots }^n_{\rho_n}=\ave{u}'_\rho, \quad \forall u\in L^\infty(\O),
\]
and there exist intermediate averagings $\cM_{n+1}, \ldots, \cM_{n+m}$ such that for any $\rho \in \cP$ there exist  $\rho_{n+1}, \ldots, \rho_{n+m} \in \cP$ such that 
\[
\ave{ \dots \ave{\ave{u}'_\rho}^{n+1}_{\rho_{n+1}} \dots }^{n+m}_{\rho_{n+m}}=\ave{u}''_\rho, \quad \forall u\in L^\infty(\O).
\]
Then for a pair of models representing their equivalence classes we say $\cM' \preceq \cM''$ if only one half of the definition above holds, namely, there exist intermediate averagings $\cM_1, \ldots, \cM_n$ such that for any $\rho \in \cP$ there exist  $\rho_1, \ldots, \rho_n \in \cP$ such that 
\[
\ave{ \dots \ave{\ave{u}''_\rho}^1_{\rho_1} \dots }^n_{\rho_n}=\ave{u}'_\rho, \quad \forall u\in L^\infty(\O).
\]

Under this partial ordering, more subtle examples emerge. For instance, for Cucker-Smale models with Bochner-positive kernels, it can be seen from the identity \eqref{e:CSorder} that if $\phi = \psi \ast \psi$, and assuming that $\int \psi = 1$, then the \ref{CS}-model based on $\psi$ is finer than that based on $\phi$, $\cM_{\mathrm{CS}}^\phi \preceq \cM_{\mathrm{CS}}^\psi$. The same applies for \ref{MT}-models as those are based on the same averaging.

One can build new averaging models from old ones by superimposing averages as long as they are defined over the same strength measures. So, if 
\[
\cM_i = \{ (\k_\rho, \ave{\cdot}^i_\rho): \rho \in \cP(\O) \}, \quad i=1,2
\]
are two averaging models, then
\begin{equation}\label{e:superM}
\cM_2 \circ \cM_1= \left\{ (\k_\rho, \ave{ \ave{\cdot}^1_\rho }_\rho^2): \rho \in \cP(\O) \right\}
\end{equation}
defines another averaging model.

Certain compositions preserve special properties. For example, if $\cM_i$ are ball-positive and symmetric the conjugation $(\k_\rho, \ave{\ave{\ave{\cdot}^1_\rho}^2_\rho}^1_\rho)$ is also ball-positive and symmetric.


\begin{thebibliography}{VCBJ{\etalchar{+}}95}

\bibitem[ABF{\etalchar{+}}19]{ABFHKPPS}
G.~Albi, N.~Bellomo, L.~Fermo, S.-Y. Ha, J.~Kim, L.~Pareschi, D.~Poyato, and
  J.~Soler.
\newblock Vehicular traffic, crowds, and swarms: {F}rom kinetic theory and
  multiscale methods to applications and research perspectives.
\newblock {\em Math. Models Methods Appl. Sci.}, 29(10):1901--2005, 2019.

\bibitem[Aok82]{Aok1982}
I.~Aoki.
\newblock A simulation study on the schooling mechanism in fish.
\newblock {\em Bull. Japanese Society of Scientific Fisheries},
  48(8):1081--1088, 1982.

\bibitem[AS67a]{AS1967a}
D.~G. Aronson and James Serrin.
\newblock Local behavior of solutions of quasilinear parabolic equations.
\newblock {\em Arch. Rational Mech. Anal.}, 25:81--122, 1967.

\bibitem[AS67b]{AS1967b}
D.~G. Aronson and James Serrin.
\newblock A maximum principle for nonlinear parabolic equations.
\newblock {\em Ann. Scuola Norm. Sup. Pisa Cl. Sci. (3)}, 21:291--305, 1967.

\bibitem[Axe97]{Axel97}
Robert Axelrod.
\newblock {\em The Complexity of Cooperation: Agent-Based Models of Competition
  and Collaboration}.
\newblock Princeton University Press, 1997.

\bibitem[AZ21]{AZ2021}
Francesca Anceschi and Yuzhe Zhu.
\newblock On a spatially inhomogeneous nonlinear fokker-planck equation: Cauchy
  problem and diffusion asymptotics, 2021.

\bibitem[BCC{\etalchar{+}}08]{Bal2008}
M.~Ballerini, N.~Cabibbo, R.~Candelier, A.~Cavagna, E.~Cisbani, I.~Giardina,
  V.~Lecomte, A.~Orlandi, G.~Parisi, A.~Procaccini, M.~Viale, and
  V.~Zdravkovic.
\newblock Interaction ruling animal collective behavior depends on topological
  rather than metric distance: evidence from a field study.
\newblock {\em Proc. Natl Acad. Sci. USA}, 105:1232--1237, 2008.

\bibitem[BCnC11]{BCC2011}
Fran\c{c}ois Bolley, Jos\'{e}~A. Ca\~{n}izo, and Jos\'{e}~A. Carrillo.
\newblock Stochastic mean-field limit: non-{L}ipschitz forces and swarming.
\newblock {\em Math. Models Methods Appl. Sci.}, 21(11):2179--2210, 2011.

\bibitem[BFK15]{Bo2015}
M.~Bongini, M.~Fornasier, and D.~Kalise.
\newblock ({U}n)conditional consensus emergence under perturbed and
  decentralized feedback controls.
\newblock {\em Discrete \& Continuous Dynamical Systems - A}, 35:4071, 2015.

\bibitem[BN05]{Ben2005}
E.~Ben-Naim.
\newblock Opinion dynamics: Rise and fall of political parties.
\newblock {\em Europhys. Lett.}, 69:671--677, 2005.

\bibitem[Car33]{Carl1933}
Torsten Carleman.
\newblock Sur la th\'{e}orie de l'\'{e}quation int\'{e}grodiff\'{e}rentielle de
  {B}oltzmann.
\newblock {\em Acta Math.}, 60(1):91--146, 1933.

\bibitem[CCG{\etalchar{+}}12]{CCGPS2012}
M.~Camperi, A.~Cavagna, I.~Giardina, G.~Parisi, and E.~Silvestri.
\newblock Spatially balanced topological interaction grants optimal cohesion in
  flocking models.
\newblock {\em Interface Focus}, 2:715--725, 2012.

\bibitem[CCP17]{CCP2017}
J.~A. Carrillo, Y.-P. Choi, and S.~P. Perez.
\newblock A review on attractive-repulsive hydrodynamics for consensus in
  collective behavior.
\newblock In {\em Active particles. {V}ol. 1. {A}dvances in theory, models, and
  applications}, Model. Simul. Sci. Eng. Technol., pages 259--298.
  Birkh\"{a}user/Springer, Cham, 2017.

\bibitem[CCTT16]{CCTT2016}
J.~A. Carrillo, Y.-P. Choi, E.~Tadmor, and C.~Tan.
\newblock Critical thresholds in 1{D} {E}uler equations with non-local forces.
\newblock {\em Math. Models Methods Appl. Sci.}, 26(1):185--206, 2016.

\bibitem[CDM{\etalchar{+}}07]{DOrsogna}
Yao-li Chuang, Maria~R. D'Orsogna, Daniel Marthaler, Andrea~L. Bertozzi, and
  Lincoln~S. Chayes.
\newblock State transitions and the continuum limit for a 2{D} interacting,
  self-propelled particle system.
\newblock {\em Phys. D}, 232(1):33--47, 2007.

\bibitem[CDS20]{CDS2019}
P.~Constantin, T.~D. Drivas, and R.~Shvydkoy.
\newblock Entropy hierarchies for equations of compressible fluids and
  self-organized dynamics.
\newblock {\em SIAM J. Math. Anal.}, 52(3):3073--3092, 2020.

\bibitem[CFPT15]{CEPT2015}
M.~Caponigro, M.~Fornasier, B.~Piccoli, and E.~Tr\'{e}lat.
\newblock Sparse stabilization and control of alignment models.
\newblock {\em Math. Models Methods Appl. Sci.}, 25(3):521--564, 2015.

\bibitem[CFRT10]{CFRT2010}
J.~A. Carrillo, M.~Fornasier, J.~Rosado, and G.~Toscani.
\newblock Asymptotic flocking dynamics for the kinetic {C}ucker-{S}male model.
\newblock {\em SIAM J. Math. Anal.}, 42(1):218--236, 2010.

\bibitem[CFTV10]{CFTV2010}
J.~A. Carrillo, M.~Fornasier, G.~Toscani, and F.~Vecil.
\newblock Particle, kinetic, and hydrodynamic models of swarming.
\newblock In {\em Mathematical modeling of collective behavior in
  socio-economic and life sciences}, Model. Simul. Sci. Eng. Technol., pages
  297--336. Birkh\"{a}user Boston, Boston, MA, 2010.

\bibitem[Cho16]{Choi2016}
Young-Pil Choi.
\newblock Global classical solutions of the {V}lasov-{F}okker-{P}lanck equation
  with local alignment forces.
\newblock {\em Nonlinearity}, 29(7):1887--1916, 2016.

\bibitem[Cho19]{Choi2019}
Y.-P. Choi.
\newblock The global {C}auchy problem for compressible {E}uler equations with a
  nonlocal dissipation.
\newblock {\em Math. Models Methods Appl. Sci.}, 29(1):185--207, 2019.

\bibitem[CK23]{CK2023}
Young-Pil Choi and Jeongho Kim.
\newblock Rigorous derivation of the {E}uler-alignment model with singular
  communication weights from a kinetic {F}okker--{P}lanck-alignment model.
\newblock {\em Math. Models Methods Appl. Sci.}, 33(1):31--65, 2023.

\bibitem[CKPP19]{CKPP2019}
Y.-P. Choi, D.~Kalise, J.~Peszek, and A.~A. Peters.
\newblock A collisionless singular {C}ucker-{S}male model with decentralized
  formation control.
\newblock {\em SIAM J. Appl. Dyn. Syst.}, 18(4):1954--1981, 2019.

\bibitem[CS07a]{CS2007a}
F.~Cucker and S.~Smale.
\newblock Emergent behavior in flocks.
\newblock {\em IEEE Trans. Automat. Control}, 52(5):852--862, 2007.

\bibitem[CS07b]{CS2007b}
F.~Cucker and S.~Smale.
\newblock On the mathematics of emergence.
\newblock {\em Jpn. J. Math.}, 2(1):197--227, 2007.

\bibitem[DeG74]{DeGroot}
M.~H. DeGroot.
\newblock Reaching a consensus.
\newblock {\em J. Amer. Statist. Assoc.}, 69:118--121, 1974.

\bibitem[DFP06]{FranPoli2006}
Marco Di~Francesco and Sergio Polidoro.
\newblock Schauder estimates, {H}arnack inequality and {G}aussian lower bound
  for {K}olmogorov-type operators in non-divergence form.
\newblock {\em Adv. Differential Equations}, 11(11):1261--1320, 2006.

\bibitem[DFT10]{DFT2010}
Renjun Duan, Massimo Fornasier, and Giuseppe Toscani.
\newblock A kinetic flocking model with diffusion.
\newblock {\em Comm. Math. Phys.}, 300(1):95--145, 2010.

\bibitem[DKRT18]{DKRT2018}
Tam Do, Alexander Kiselev, Lenya Ryzhik, and Changhui Tan.
\newblock Global regularity for the fractional {E}uler alignment system.
\newblock {\em Arch. Ration. Mech. Anal.}, 228(1):1--37, 2018.

\bibitem[DS19]{DS2019}
H.~Dietert and R.~Shvydkoy.
\newblock On {C}ucker-{S}male dynamical systems with degenerate communication.
\newblock {\em Anal. Appl. (Singap.)}, 19(4):551--573, 2019.

\bibitem[DV00]{DesVill2000}
Laurent Desvillettes and C\'{e}dric Villani.
\newblock On the spatially homogeneous {L}andau equation for hard potentials.
  {I}. {E}xistence, uniqueness and smoothness.
\newblock {\em Comm. Partial Differential Equations}, 25(1-2):179--259, 2000.

\bibitem[EK01]{Edel2001}
Leah Edelstein-Keshet.
\newblock Mathematical models of swarming and social aggregation.
\newblock 2001.

\bibitem[Fav83]{Favre}
A.~Favre.
\newblock Turbulence: Space-time statistical properties and behavior in
  supersonic flows.
\newblock {\em The Physics of Fluids}, 26(10):2851--2863, 1983.

\bibitem[FK19]{FK2019}
A.~Figalli and M.-J. Kang.
\newblock A rigorous derivation from the kinetic {C}ucker-{S}male model to the
  pressureless {E}uler system with nonlocal alignment.
\newblock {\em Anal. PDE}, 12(3):843--866, 2019.

\bibitem[GI23]{GI2021}
Jessica Guerand and Cyril Imbert.
\newblock Log-transform and the weak {H}arnack inequality for kinetic
  {F}okker-{P}lanck equations.
\newblock {\em J. Inst. Math. Jussieu}, 22(6):2749--2774, 2023.

\bibitem[GIMV19]{GIMV2019}
Fran\c{c}ois Golse, Cyril Imbert, Cl\'{e}ment Mouhot, and Alexis~F. Vasseur.
\newblock Harnack inequality for kinetic {F}okker-{P}lanck equations with rough
  coefficients and application to the {L}andau equation.
\newblock {\em Ann. Sc. Norm. Super. Pisa Cl. Sci. (5)}, 19(1):253--295, 2019.

\bibitem[GJV04]{GJV2004}
Thierry Goudon, Pierre-Emmanuel Jabin, and Alexis Vasseur.
\newblock Hydrodynamic limit for the {V}lasov-{N}avier-{S}tokes equations. {I}.
  {L}ight particles regime.
\newblock {\em Indiana Univ. Math. J.}, 53(6):1495--1515, 2004.

\bibitem[H\"67]{Hormander1967}
Lars H\"{o}rmander.
\newblock Hypoelliptic second order differential equations.
\newblock {\em Acta Math.}, 119:147--171, 1967.

\bibitem[HHK10]{HHK2010}
S.-Y. Ha, T.~Ha, and J.-H. Kim.
\newblock Emergent behavior of a cucker-smale type particle model with
  nonlinear velocity couplings.
\newblock {\em IEEE Trans. on Automatic Control}, 55(7), 2010.

\bibitem[HKK14]{HKK2014}
S.-Y. Ha, M.-J. Kang, and B.~Kwon.
\newblock A hydrodynamic model for the interaction of {C}ucker-{S}male
  particles and incompressible fluid.
\newblock {\em Math. Models Methods Appl. Sci.}, 24(11):2311--2359, 2014.

\bibitem[HL09]{HL2009}
S.-Y. Ha and J.-G. Liu.
\newblock A simple proof of the {C}ucker-{S}male flocking dynamics and
  mean-field limit.
\newblock {\em Commun. Math. Sci.}, 7(2):297--325, 2009.

\bibitem[HR17]{HRCST2017}
S.-Y. Ha and T.~Ruggeri.
\newblock Emergent dynamics of a thermodynamically consistent particle model.
\newblock {\em Arch. Ration. Mech. Anal.}, 223(3):1397--1425, 2017.

\bibitem[HST20]{HST2020}
Christopher Henderson, Stanley Snelson, and Andrei Tarfulea.
\newblock Self-generating lower bounds and continuation for the {B}oltzmann
  equation.
\newblock {\em Calc. Var. Partial Differential Equations}, 59(6):Paper No. 191,
  13, 2020.

\bibitem[HT08]{HT2008}
S.-Y. Ha and E.~Tadmor.
\newblock From particle to kinetic and hydrodynamic descriptions of flocking.
\newblock {\em Kinet. Relat. Models}, 1(3):415--435, 2008.

\bibitem[HT17]{HeT2017}
S.~He and E.~Tadmor.
\newblock Global regularity of two-dimensional flocking hydrodynamics.
\newblock {\em C. R. Math. Acad. Sci. Paris}, 355(7):795--805, 2017.

\bibitem[HT21]{HeTmulti}
Siming He and Eitan Tadmor.
\newblock A game of alignment: collective behavior of multi-species.
\newblock {\em Ann. Inst. H. Poincar\'{e} C Anal. Non Lin\'{e}aire},
  38(4):1031--1053, 2021.

\bibitem[IM21]{IM2021}
Cyril Imbert and Cl\'{e}ment Mouhot.
\newblock The {S}chauder estimate in kinetic theory with application to a toy
  nonlinear model.
\newblock {\em Ann. H. Lebesgue}, 4:369--405, 2021.

\bibitem[IMS20]{IMS2020}
Cyril Imbert, Cl\'{e}ment Mouhot, and Luis Silvestre.
\newblock Gaussian lower bounds for the {B}oltzmann equation without cutoff.
\newblock {\em SIAM J. Math. Anal.}, 52(3):2930--2944, 2020.

\bibitem[Jac08]{Jackson2008}
Matthew~O. Jackson.
\newblock Social and economic networks.
\newblock 2008.

\bibitem[JJ15]{P-ES2015}
Pierre-Emmanuel Jabin and St\'{e}phane Junca.
\newblock A continuous model for ratings.
\newblock {\em SIAM J. Appl. Math.}, 75(2):420--442, 2015.

\bibitem[KMT13]{KMT2013}
T.~K. Karper, A.~Mellet, and K.~Trivisa.
\newblock Existence of weak solutions to kinetic flocking models.
\newblock {\em SIAM. J. Math. Anal.}, 45:215--243, 2013.

\bibitem[KMT14]{KMT2014}
T.~K. Karper, A.~Mellet, and K.~Trivisa.
\newblock On strong local alignment in the kinetic {C}ucker-{S}male model.
\newblock In {\em Hyperbolic conservation laws and related analysis with
  applications}, volume~49 of {\em Springer Proc. Math. Stat.}, pages 227--242.
  Springer, Heidelberg, 2014.

\bibitem[KMT15]{KMT2015}
T.~K. Karper, A.~Mellet, and K.~Trivisa.
\newblock Hydrodynamic limit of the kinetic {C}ucker-{S}male flocking model.
\newblock {\em Mathematical Models and Methods in Applied Sciences},
  25(1):131--163, 2015.

\bibitem[Kol34]{Kolm1934}
A.~Kolmogoroff.
\newblock Zuf\"{a}llige {B}ewegungen (zur {T}heorie der {B}rownschen
  {B}ewegung).
\newblock {\em Ann. of Math. (2)}, 35(1):116--117, 1934.

\bibitem[Kry96]{Krylov-book}
N.~V. Krylov.
\newblock {\em Lectures on elliptic and parabolic equations in {H}\"{o}lder
  spaces}, volume~12 of {\em Graduate Studies in Mathematics}.
\newblock American Mathematical Society, Providence, RI, 1996.

\bibitem[KV15]{KV2015}
M.-J. Kang and A.~Vasseur.
\newblock Asymptotic analysis of {V}lasov-type equations under strong local
  alignment regime.
\newblock {\em Math. Models Methods Appl. Sci.}, 25(11):2153--2173, 2015.

\bibitem[Lio94]{Lions1994}
P.-L. Lions.
\newblock Compactness in {B}oltzmann's equation via {F}ourier integral
  operators and applications. {I}, {II}.
\newblock {\em J. Math. Kyoto Univ.}, 34(2):391--427, 429--461, 1994.

\bibitem[LRS21]{LRS-friction}
Daniel Lear, David~N. Reynolds, and Roman Shvydkoy.
\newblock Grassmannian reduction of cucker-smale systems and dynamical opinion
  games.
\newblock {\em Discrete Contin. Dyn. Syst.}, 41(12):5765--, 2021.

\bibitem[LRS22]{LRS-topo}
Daniel Lear, David~N. Reynolds, and Roman Shvydkoy.
\newblock Global solutions to multi-dimensional topological {E}uler alignment
  systems.
\newblock {\em Ann. PDE}, 8(1):Paper No. 1, 43, 2022.

\bibitem[LS16]{LS2016}
T.~M. Leslie and R.~Shvydkoy.
\newblock The energy balance relation for weak solutions of the
  density-dependent {N}avier-{S}tokes equations.
\newblock {\em J. Differential Equations}, 261(6):3719--3733, 2016.

\bibitem[LS19a]{LS-uni1}
D.~Lear and R.~Shvydkoy.
\newblock Existence and stability of unidirectional flocks in hydrodynamic
  {E}uler {A}lignment systems.
\newblock {\em Anal. PDE}, 15(1):175--196, 2019.
\newblock https://arxiv.org/abs/1911.10661.

\bibitem[LS19b]{LS-entropy}
T.~M. Leslie and R.~Shvydkoy.
\newblock On the structure of limiting flocks in hydrodynamic {E}uler
  {A}lignment models.
\newblock {\em Math. Models Methods Appl. Sci.}, 29(13):2419--2431, 2019.

\bibitem[LT98]{LT1998}
Doron Levy and Eitan Tadmor.
\newblock From semidiscrete to fully discrete: stability of {R}unge-{K}utta
  schemes by the energy method.
\newblock {\em SIAM Rev.}, 40(1):40--73, 1998.

\bibitem[Mar18]{Mar2018}
I.~Markou.
\newblock Collision-avoiding in the singular {C}ucker-{S}male model with
  nonlinear velocity couplings.
\newblock {\em Discrete Contin. Dyn. Syst.}, 38(10):5245--5260, 2018.

\bibitem[MMP20]{MMP2020}
Piotr Minakowski, Piotr~B. Mucha, and Jan Peszek.
\newblock Density-induced consensus protocol.
\newblock {\em Math. Models Methods Appl. Sci.}, 30(12):2389--2415, 2020.

\bibitem[Mou05]{Mouhot2005}
Cl\'{e}ment Mouhot.
\newblock Quantitative lower bounds for the full {B}oltzmann equation. {I}.
  {P}eriodic boundary conditions.
\newblock {\em Comm. Partial Differential Equations}, 30(4-6):881--917, 2005.

\bibitem[MP18]{MP2018}
P.~B. Mucha and J.~Peszek.
\newblock The {C}ucker-{S}male equation: singular communication weight,
  measure-valued solutions and weak-atomic uniqueness.
\newblock {\em Arch. Ration. Mech. Anal.}, 227(1):273--308, 2018.

\bibitem[MPT19]{MPT2018}
J.~Morales, J.~Peszek, and E.~Tadmor.
\newblock Flocking with short-range interactions.
\newblock {\em J. Stat. Phys.}, 176(2):382--397, 2019.

\bibitem[MT11]{MT2011}
S.~Motsch and E.~Tadmor.
\newblock A new model for self-organized dynamics and its flocking behavior.
\newblock {\em J. Stat. Phys.}, 144(5):923--947, 2011.

\bibitem[MT14]{MT2014}
S.~Motsch and E.~Tadmor.
\newblock Heterophilious dynamics enhances consensus.
\newblock {\em SIAM Rev.}, 56(4):577--621, 2014.

\bibitem[MV08]{MV2008}
A.~Mellet and A.~Vasseur.
\newblock Asymptotic analysis for a {V}lasov-{F}okker-{P}lanck/compressible
  {N}avier-{S}tokes system of equations.
\newblock {\em Comm. Math. Phys.}, 281(3):573--596, 2008.

\bibitem[NMG14]{NIIZATO201462}
Takayuki Niizato, Hisashi Murakami, and Yukio-Pegio Gunji.
\newblock Emergence of the scale-invariant proportion in a flock from the
  metric-topological interaction.
\newblock {\em Biosystems}, 119:62 -- 68, 2014.

\bibitem[{Olf}06]{OS2006}
R.~{Olfati-Saber}.
\newblock Flocking for multi-agent dynamic systems: algorithms and theory.
\newblock {\em IEEE Transactions on Automatic Control}, 51(3):401--420, 2006.

\bibitem[PEG09]{Darwin}
L.~Perea, P.~Elosegui, and G.~Gomez.
\newblock Extension of the {C}ucker-{S}male control law to space flight
  formations.
\newblock {\em Journal of Guidance, Control, and Dynamics}, 32:526 -- 536,
  2009.

\bibitem[Pes15]{Pe2015}
J.~Peszek.
\newblock Discrete {C}ucker-{S}male flocking model with a weakly singular
  weight.
\newblock {\em SIAM J. Math. Anal.}, 47(5):3671--3686, 2015.

\bibitem[Rey87]{Rey1987}
C.~W. Reynolds.
\newblock Flocks, herds and schools: A distributed behavioral model.
\newblock {\em {ACM} {SIGGRAPH} Computer Graphics}, 21:25--34, 1987.

\bibitem[RS20]{RS-topoloc}
D.~N. Reynolds and R.~Shvydkoy.
\newblock Local well-posedness of the topological {E}uler alignment models of
  collective behavior.
\newblock {\em Nonlinearity}, 33(10):5176--5215, 2020.

\bibitem[SB14]{Shang2014}
Yilun Shang and Roland Bouffanais.
\newblock Consensus reaching in swarms ruled by a hybrid metric-topological
  distance.
\newblock {\em The European Physical Journal B}, 87(12):294, Dec 2014.

\bibitem[Shv21]{Sbook}
Roman Shvydkoy.
\newblock {\em Dynamics and analysis of alignment models of collective
  behavior}.
\newblock Ne\v{c}as Center Series. Birkh\"{a}user/Springer, Cham, [2021]
  \copyright 2021.

\bibitem[Shv22]{S-hypo}
Roman Shvydkoy.
\newblock Global hypocoercivity of kinetic {F}okker-{P}lanck-{A}lignment
  equations.
\newblock {\em Kinet. Relat. Models}, 15(2):213--237, 2022.

\bibitem[ST17a]{ST1}
Roman Shvydkoy and Eitan Tadmor.
\newblock Eulerian dynamics with a commutator forcing.
\newblock {\em Trans. Math. Appl.}, 1(1):26, 2017.

\bibitem[ST17b]{ST2}
Roman Shvydkoy and Eitan Tadmor.
\newblock Eulerian dynamics with a commutator forcing {II}: {F}locking.
\newblock {\em Discrete Contin. Dyn. Syst.}, 37(11):5503--5520, 2017.

\bibitem[ST18]{ST3}
Roman Shvydkoy and Eitan Tadmor.
\newblock Eulerian dynamics with a commutator forcing {III}. {F}ractional
  diffusion of order {$0<\alpha<1$}.
\newblock {\em Phys. D}, 376/377:131--137, 2018.

\bibitem[ST19]{ShuT2019anti}
R.~Shu and E.~Tadmor.
\newblock Anticipation breeds alignment.
\newblock {\em Arch. Ration. Mech. Anal.}, 240(1):203--241, 2019.

\bibitem[ST20a]{ShuT2019}
R.~Shu and E.~Tadmor.
\newblock Flocking hydrodynamics with external potentials.
\newblock {\em Arch. Ration. Mech. Anal.}, 238(1):347--381, 2020.

\bibitem[ST20b]{ST-topo}
Roman Shvydkoy and Eitan Tadmor.
\newblock Topologically based fractional diffusion and emergent dynamics with
  short-range interactions.
\newblock {\em SIAM J. Math. Anal.}, 52(6):5792--5839, 2020.

\bibitem[ST21]{STmulti}
Roman Shvydkoy and Eitan Tadmor.
\newblock Multiflocks: emergent dynamics in systems with multiscale collective
  behavior.
\newblock {\em Multiscale Model. Simul.}, 19(2):1115--1141, 2021.

\bibitem[ST23]{STT2023}
Roman Shvydkoy and Trevor Teolis.
\newblock Well-posedness and long time behavior of the {E}uler {A}lignment
  system with adaptive communication strength.
\newblock {\em accepted at the Abel Proceedings}, 2023.

\bibitem[Szn91]{sznitman}
Alain-Sol Sznitman.
\newblock Topics in propagation of chaos.
\newblock In {\em \'{E}cole d'\'{E}t\'{e} de {P}robabilit\'{e}s de
  {S}aint-{F}lour {XIX}---1989}, volume 1464 of {\em Lecture Notes in Math.},
  pages 165--251. Springer, Berlin, 1991.

\bibitem[Tad02]{Tad2002}
E.~Tadmor.
\newblock From semidiscrete to fully discrete: stability of {R}unge-{K}utta
  schemes by the energy method. {II}.
\newblock In {\em Collected lectures on the preservation of stability under
  discretization ({F}ort {C}ollins, {CO}, 2001)}, pages 25--49. SIAM,
  Philadelphia, PA, 2002.

\bibitem[Tad21]{Tadmor-notices}
Eitan Tadmor.
\newblock On the mathematics of swarming: emergent behavior in alignment
  dynamics.
\newblock {\em Notices Amer. Math. Soc.}, 68(4):493--503, 2021.

\bibitem[Tad23]{Tadmor-pressure}
Eitan Tadmor.
\newblock Swarming: hydrodynamic alignment with pressure.
\newblock {\em Bull. Amer. Math. Soc. (N.S.)}, 60(3):285--325, 2023.

\bibitem[TT14]{TT2014}
E.~Tadmor and C.~Tan.
\newblock Critical thresholds in flocking hydrodynamics with non-local
  alignment.
\newblock {\em Philos. Trans. R. Soc. Lond. Ser. A Math. Phys. Eng. Sci.},
  372(2028):20130401, 22, 2014.

\bibitem[TV00]{TV2000}
G.~Toscani and C.~Villani.
\newblock {\em On the trend to equilibrium for some dissipative systems with
  slowly increasing a priori bounds}, volume~98.
\newblock 2000.

\bibitem[VCBJ{\etalchar{+}}95]{VCBCS1995}
T.~Vicsek, A.~Czir\'ok, E.~Ben-Jacob, I.~Cohen, and O.~Shochet.
\newblock Novel type of phase transition in a system of self-driven particles.
\newblock {\em Physical Review Letters}, 75(6):1226--1229, 1995.

\bibitem[Vil09]{Villani}
C\'{e}dric Villani.
\newblock Hypocoercivity.
\newblock {\em Mem. Amer. Math. Soc.}, 202(950):iv+141, 2009.

\bibitem[VZ12]{VZ2012}
T.~Vicsek and A.~Zefeiris.
\newblock Collective motion.
\newblock {\em Physics Reprints}, 517:71--140, 2012.

\end{thebibliography}

\newcommand{\etalchar}[1]{$^{#1}$}

\end{document}